\documentclass[a4paper]{article}
\usepackage[margin=1in]{geometry}
\usepackage{vmargin}
\usepackage{appendix}
\usepackage[hidelinks]{hyperref}
\usepackage{amssymb}
\usepackage[shortalphabetic]{amsrefs}
\usepackage{amsmath}
\usepackage{amsrefs}
\usepackage{upgreek}
\usepackage{amsthm}
\usepackage{mathdots}
\usepackage{url}
\usepackage{enumerate}
\usepackage{mathtools}
\usepackage{amsxtra}
\usepackage{xcolor}
 \usepackage[refpage]{nomencl}
\usepackage{enumitem}
\usepackage{graphicx}
\usepackage{comment}

\makenomenclature

\newtheorem{thm}{Theorem}[section]
\newtheorem{lem}[thm]{Lemma}
\newtheorem{cor}[thm]{Corollary}
\newtheorem{prop}[thm]{Proposition}

\newcommand{\wt}{\widetilde}
\newcommand{\pipsitilde}{\pi^{\thicksim}_{\psi}}
\newcommand{\GL}{\mathrm{GL}}
\newcommand{\U}{\mathrm{U}}
\newcommand{\Sp}{\mathrm{Sp}}
\newcommand{\SO}{\mathrm{SO}}

\newcommand{\Tr}{\mathrm{Tr}}
\newcommand{\Trans}{\mathrm{Trans}}
\newcommand{\R}{\mathbb R}
\renewcommand{\H}{\mathbb H}
\newcommand{\Z}{\mathbb Z}
\newcommand{\C}{\mathbb C}
\newcommand{\etaA}{\eta^{\mathrm{Ar}}}
\newcommand{\etaABV}{\eta^{\mathrm{ABV}}}
\newcommand{\etaloc}{\eta^{\mathrm{loc}}}
\newcommand{\etamic}{\eta^{\mathrm{mic}}}
\newcommand{\chG}{{}^\vee G}
\newcommand{\ch}[1]{{}^\vee#1}
\newcommand{\chGL}{{}^\vee \mathrm{GL}}
\newcommand{\Gparameters}{\Xi(\chG^\Gamma)}

\newcommand{\GLparameters}{\Xi(\chGL_N^\Gamma)}

\newcommand{\X}{ X(\chG^\Gamma)}
\newcommand{\XGL}{ X(\chGL^\Gamma)}

\newcommand{\PiABV}{\Pi^{\mathrm{ABV}}}
\newcommand{\PiArthur}{\Pi^{\mathrm{Ar}}}
\newcommand{\Xtwisted}{X(\chGL_N^\Gamma,\vartheta)}
\newcommand{\Lift}{\mathrm{Lift}}
\newcommand{\LG}{{^\vee}G^{\Gamma}}
\newcommand{\LGL}{\ch\GL_N^{\Gamma}}
\newcommand{\g}{\mathfrak g}
\newcommand{\gl}{\mathfrak{gl}_{N}}
\newcommand{\h}{\mathfrak h}
\renewcommand{\O}{\ch \mathcal O}
\newcommand{\D}{\mathcal D}
\newcommand{\Xcal}{\mathcal X}

\newcommand{\Stab}{\mathrm{Stab}}
\newcommand{\Norm}{\mathrm{Norm}}
\newcommand{\Cent}{\mathrm{Cent}}

\newcommand{\stable}{\mathrm{st}}

\newcommand\inv{^{-1}}

\newcommand{\XX}{\Xcal_{\ch\rho}^w\times \ch\Xcal^{ww_0}_\lambda}

\newcommand{\GLext}{\GL_N(\R)\rtimes\langle\vartheta\rangle}

\newcommand{\chimic}{\chi^{\mathrm{mic}}}

\title{Equivalent definitions of Arthur packets for real classical
  groups}

\author{J. Adams, N. Arancibia Robert, P. Mezo}

\begin{document}
\maketitle
\begin{abstract}
  Arthur has conjectured the existence of what are now known as
  Arthur packets of representations of reductive algebraic groups over
  local and global fields.  In the case of classical groups he
  subsequently gave a definition of these packets, using local and
  global methods.  For general real groups, an alternative approach to
  the definition of Arthur packets has been given by
  Adams-Barbasch-Vogan. This construction is purely local and uses
  geometric methods. Our main result is that these two definitions agree
  in the case of real classical groups.

\end{abstract}

\tableofcontents

\section{Introduction}
\label{intro}

Guided by the theory of the trace formula, Arthur conjectured a
classification of automorphic representations of a connected reductive
algebraic group $G$ in terms of \emph{A-parameters} (\cite{Arthur84},
\cite{Arthur89}).  These A-parameters are global objects, which
conjecturally restrict to local A-parameters. The local
part of Arthur's conjectures states that associated to an
A-parameter $\psi_{G}$ over a local field $F$ is a finite set $\Pi_{\psi_{G}}$
of irreducible representations of $G(F)$, which now are commonly referred to as Arthur packets.
Arthur conjectured that these packets
satisfy conditions to be given in Problems A-E below.

Assume for the moment that  $G$ is either the split form of
$\mathrm{GL}_{N}$, or a  quasisplit form of $\mathrm{Sp}_{N}$ or $\mathrm{SO}_{N}$, $N \geq 2$
over a local field $F$ of characteristic $0$.   For these groups,
Arthur defines a set,
which we denote $\PiArthur_{\psi_G}$, and proves that it satisfies most of
the stated conditions \cite{Arthur}. His approach uses harmonic analysis, and
both local and global methods.

On the other hand, for a general connected real reductive group Adams, Barbasch
and Vogan define a set which we denote by $\PiABV_{\psi_{G}}$, and prove
these satisfy most
of Arthur's conditions \cite{ABV}. Their methods are quite
different, being based on equivariant sheaf theory on a flag variety.

It has long been expected that the two  definitions agree
 when both are defined (\cite{Arthur08}*{Section 8}), \emph{i.e.}
 for real split $\mathrm{GL}_N$ and the real quasisplit groups $\mathrm{Sp}_{N}$ and $\mathrm{SO}_{N}$.
The main result of this paper is that, aside from the case of even
rank special orthogonal groups,  it is indeed true that
$$\PiArthur_{\psi_{G}}=\PiABV_{\psi_{G}}.$$
The case of even rank special orthogonal groups
requires a slightly modified identity, which we give later in the
introduction.

We state Arthur's original conjectures as a set 
of problems.  We then describe the two
approaches to these problems.  We assume that the reader is somewhat familiar
with the theory of endoscopy for tempered representations
\cite{Shelstad08}.  We follow the
notation of \cite{ABV} and also provide references from \cite{Arthur84} and  \cite{Arthur89}.

\nomenclature{$\psi_{G}$}{A-parameter for $G$}
\nomenclature{$\ch G$}{Langlands dual of $G$}

Let $\Gamma\nomenclature{$\Gamma$}{real Galois group}$ be the Galois group of $\mathbb{C}/\mathbb{R}$ and let $\LG = \chG \rtimes \Gamma$  be the
Galois
form of the L-group of $G$.
An \emph{A-parameter for} $G$, is a group homomorphism
\begin{equation}
\label{aparameter}
\psi_{G}: W_{\mathbb{R}} \times \mathrm{SL}_{2} \rightarrow {^\vee}G^{\Gamma}
\end{equation}
such that  $\psi_G|_{W_\R}$ is a tempered L-parameter and
$\psi_{G}|_{\mathrm{SL}_{2}}$ is algebraic.  For the definition of $W_\R$
and tempered L-parameters
see \cite{borel}.  For the sake of simplicity, we assume that $G$ is quasisplit in the following problems.

\begin{description}
\item[\hypertarget{proba}{Problem A}] Associate to $\psi_{G}$ a finite linear combination
   of irreducible characters $\eta_{\psi_{G}}$ of $G(\mathbb{R})$ which is
  a \emph{stable} distribution (\cite{shelstad}, \cite{ABV}*{Definition 18.2}).
\end{description}
\nomenclature{$\eta_{\psi_{G}}$}{stable virtual character defining the Arthur packet}
Problem A  was first conjectured in \cite{Arthur84}*{Conjecture 1.3.3}.  A revised conjecture was given in \cite{Arthur89}*{\emph{pp.} 26-27}, where $\eta_{\psi_{G}}(f)$ is written as $f^{G}(\psi_{G})$ for $f \in C_{c}^{\infty}(G(\mathbb{R}))$.
Assuming a solution to Problem A,  the Arthur packet $\Pi_{\psi_G}$ is defined to be the set of irreducible characters occurring in $\eta_{\psi_{G}}$.
Let $A_{\psi_{G}}$
\nomenclature{$A_{\psi_{G}}$}{Component group}
 be the component group of
 the centralizer  in ${^\vee}G$ of the image of $\psi_{G}$. The groups in question are finite and abelian (\cite{Arthur}*{\emph{p}. 32}).
\begin{description}
  \item[\hypertarget{probb}{Problem B}]  Associate to each $\pi \in \Pi_{\psi_{G}}$ a non-zero
    finite-dimensional  representation $\tau_{\psi_{G}}(\pi)$ of
    $A_{\psi_{G}}$.
    \nomenclature{$\tau_{\psi_{G}}(\pi)$}{representation of $A_{\psi_{G}}$}
  \item[\hypertarget{probc}{Problem C}] Prove that
    $$\eta_{\psi_{G}} = \sum_{\pi \in \Pi_{\psi_{G}}}  \varepsilon_{\pi}
    \dim(\tau_{\psi_{G}}(\pi)) \, \pi$$
    for some  $\varepsilon_{\pi} = \pm 1$.
\end{description}

Problems B and C are actually special cases of the next problem, but deserve to be singled out as they highlight the representations $\tau_{\psi_{G}}(\pi)$ and their relation to $\eta_{\psi_{G}}$.  These problems first appeared in a preliminary form as \cite{Arthur84}*{(1.3.4)}, where $\tau_{\psi_{G}}(\pi)$ was written as $<\cdot, \pi >$.  They appeared in a revised form in \cite{Arthur89}*{(4.1)-(4.2)}, where the character of $\tau_{\psi_{G}}(\pi)$ is written as $< \cdot, \pi|\uprho>$.

\begin{description}
  \item[\hypertarget{probd}{Problem D}]  Prove that the stable distributions $\eta_{\psi_{G}}$ satisfy
    analogues of Shelstad's theorem on endoscopic lifting for tempered
    representations \cite{ABV}*{Chapter 26}.
\end{description}
Problem D is stated rather vaguely here.  Let us attempt to give its gist, following \cite{Arthur89}.   To a semisimple element $s \in {^\vee}G$ centralizing $\psi_{G}(W_{\mathbb{R}})$ one may define an \emph{endoscopic group} $G'$.  The A-parameter $\psi_{G}$ factors  to an A-parameter $\psi_{G'}$ for $G'$, and Shelstad has defined a map $\mathrm{Trans}_{G'}^{G}$ from stable distributions on $G'(\mathbb{R})$ to distributions on $G(\mathbb{R})$.  Arthur conjectured that
\begin{equation}
\label{arthur4.1}
\mathrm{Trans}_{G'}^{G}(\eta_{\psi_{G'}}) = \sum_{\pi \in \Pi_{\psi_{G}}} \uprho(s_{\psi} s) \ \mathrm{Tr}\left( \tau_{\psi_{G}}(\pi) (s_{\psi_{G}}\bar{s})\right)  \, \pi
\end{equation}
for some non-vanishing complex-valued function $\uprho$ on $A_{\psi_{G}}$ and a specified element $s_{\psi_{G}} \in A_{\psi_{G}}$. In addition, the function $\uprho$ and each $\tau_{\psi_{G}}(\pi)$ are conjectured to take values $\pm 1$ on $s_{\psi_{G}}$ (\cite{Arthur89}*{\emph{p.} 27}).  Problem C is recovered from (\ref{arthur4.1}) by taking $s = 1$.

Although the solution to Problem D in \cite{ABV}*{Theorem 26.25} is a recognizable cognate of  (\ref{arthur4.1}), the harmonic analysis underlying $\mathrm{Trans}_{G'}^{G}$ is replaced by sheaf theory.  The bridge between the two perspectives will be discussed shortly.

\begin{description}
 \item[\hypertarget{probe}{Problem E}]  Prove that the irreducible representations of
   $\Pi_{\psi_{G}}$ are all unitary.
\end{description}

Problem E is stated in \cite{ABV}, and it is implicit in Arthur's
original conjectures, in that these representations should be local
components of automorphic representations.
As observed in \cite{Arthur89}, the conditions of Problems A-E are not enough to determine these packets uniquely. 

The virtual character $\eta_{\psi_{G}}^{\mathrm{ABV}}$ constructed in \cite{ABV} solves Problems A-D. It is not known in general whether  the corresponding Arthur packet $\Pi_{\psi_{G}}^{\mathrm{ABV}}$ consists of unitary representations (Problem E).

For the remainder of this section we assume $G$ is either the split form of $\mathrm{GL}_{N}$, or a quasisplit form of  $\mathrm{Sp}_{N}$ or
$\mathrm{SO}_{N}$. In this case Arthur  constructs virtual characters $\eta_{\psi_{G}}^{\mathrm{Ar}}$ and proves they solve most of Problems A-E
 \cite{Arthur}*{Theorem 2.2.1}, with the following exceptions.
In the case of
$G = \mathrm{SO}_{2N}$  Arthur proves slightly weaker results,
due to the
existence of an outer automorphism.
Also he does not prove the properties of the
signs $\varepsilon_{\pi}$
of \hyperlink{probc}{Problem C}.

The main idea of Arthur's approach is to express a  symplectic or
special orthogonal group $G$ as
a \emph{twisted endoscopic
  group} of $(\mathrm{GL}_{N}, \vartheta)$ \cite{KS}*{Section 2}.  In this pair
$\vartheta$ is the outer automorphism of $\mathrm{GL}_{N}$ of order
two defined by
\nomenclature{$\vartheta$}{outer automorphism of $\mathrm{GL}_{N}$}
\begin{equation}
  \label{varthetadef}
  \vartheta(g) = \tilde{J} \, (g^{-1})^{\intercal} \, \tilde{J}^{-1}, \quad
  g \in \mathrm{GL}_{N},
\end{equation}
where $\tilde{J}$ is the anti-diagonal matrix
\nomenclature{$\tilde{J}$}{anti-diagonal matrix}
\begin{equation*}
\label{tildej}
\tilde{J} = \scriptsize \begin{bmatrix}0 & & & 1\\
  & &-1 & \\
  & \iddots & & \\
  (-1)^{N-1} & & & 0\end{bmatrix}\normalsize .
  \end{equation*}
The semidirect product $\mathrm{GL}_{N} \rtimes \langle \vartheta
\rangle$ is a disconnected algebraic group with non-identity component
$\mathrm{GL}_{N} \rtimes \vartheta$. The group $G$ is attached to the
pair $(\mathrm{GL}_{N}, \vartheta)$ through the existence of an
element $s \vartheta \in
\mathrm{GL}_{N} \rtimes \vartheta$ for which  ${^\vee}G$ is the
identity component of the fixed-point set $({^\vee}\mathrm{GL}_{N})^{s
  \vartheta}$.  Furthermore, there is a natural inclusion
\begin{equation}
\label{preepsilon}
\epsilon : {^\vee}G^{\Gamma} \hookrightarrow
{^\vee}\mathrm{GL}_{N}^{\Gamma}
\end{equation}
which allows us to define the A-parameter
\begin{equation}
\label{prepsitilde}
\psi = \epsilon \circ \psi_{G}
\end{equation}
for $\mathrm{GL}_{N}$ using (\ref{aparameter}).
\nomenclature{$\psi$}{A-parameter for $\mathrm{GL}_{N}$}

As in \hyperlink{probd}{Problem D} there are theorems on
\emph{twisted} endoscopic
lifting for tempered representations (\cite{Shelstad12}, \cite{Mezo},
\cite{Mezo2}).  One may therefore extend \hyperlink{probd}{Problem D} to
\begin{description}
\item[\hypertarget{probd'}{Problem D}$'$] Prove that the stable distributions $\eta_{\psi_{G}}$ satisfy
    analogues of both standard and twisted endoscopic lifting for tempered
    representations.
\end{description}
The solution to Problem D$'$ in the twisted setting above opens a
path towards defining $\eta_{\psi_{G}}$.  We may take for granted the
existence of an irreducible character $\pi_{\psi}$ of
$\mathrm{GL}_{N}(\mathbb{R})$  (\cite{Arthur}*{\emph{p.} 64}) such
that $\pi_{\psi}$ solves Problems A-E, \emph{i.e.}
$$\eta_{\psi} = \pi_{\psi}.$$

Now suppose $\pipsitilde$ is an extension of $\pi_\psi$ to
$\GL_N(\R)\rtimes \langle \vartheta \rangle$.  Let
$\Tr_{\vartheta}(\pipsitilde)$ be
the {\it twisted trace} of $\pipsitilde$, which is obtained by
restricting the distribution character of $\pipsitilde$ to the
non-identity component  $\GL_N\rtimes\vartheta$.
The extension $\pipsitilde$ is not unique; we choose it following
\cite{Arthur}*{\emph{pp.} 62-63} by fixing a Whittaker datum.
Towards a solution to \hyperlink{proba}{Problem A}, Arthur defines a
stable virtual character $\etaA_{\psi_G}$ using a
twisted endoscopic transfer identity
\begin{equation}
  \label{spectrans}
  \Tr_{\vartheta}(\pipsitilde)=\Trans_G^{\GL_N \rtimes \vartheta} ( \etaA_{\psi_{G}}).
\nomenclature{$\Trans_G^{\GL_N \rtimes \vartheta}$}{ endoscopic transfer map}
\end{equation}
The endoscopic transfer map $\Trans_{G}^{\GL_N \rtimes \vartheta}$ is
defined on the space of stable virtual characters of $G(\mathbb{R})$,
and $ \etaA_{\psi_{G}}$ is fixed under the action of any outer automorphisms.
This characterizes $\etaA_{\psi_G}$ uniquely,
and $\Pi_{\psi_G}^{\mathrm{Ar}}$ is defined to be the irreducible representations
occurring in $\etaA_{\psi_G}$.
Arthur proves the existence of $\etaA_{\psi_G}$ satisfying
\eqref{spectrans} using the solution to
\hyperlink{probd'}{Problem D}$'$ in the tempered setting.

Adams, Barbasch and Vogan use completely different methods to study
Problems A-E.  They first construct a pairing between characters and
equivariant sheaves. They then apply techniques from microlocal
geometry to these sheaves and use the pairing to transfer these back
to the world of virtual characters.  An outline of their methods is
given in the introduction to \cite{ABV}. Here we summarize the main
ideas, specialized to the case of quasisplit classical groups.

Adams, Barbasch
and Vogan introduce a complex variety $\X$
equipped
with a ${^\vee}G$-action \cite{ABV}*{Section 6}, so that the ${^\vee}G$-orbits are in
bijection with the equivalence classes of L-parameters.  The advantage
to working with orbits of $X({^\vee}G^{\Gamma})$ lies in the
additional topological structure.  The orbits provide a stratification
of $X({^\vee}G^{\Gamma})$ which naturally leads to the notions of local
systems and constructible sheaves.
We define a  {\it complete geometric parameter}  to be a pair
\begin{equation*}
\label{cgp}
\xi = (S, \mathcal{V})
\end{equation*}
consisting of an orbit $S \subset
X({^\vee}G^{\Gamma})$, together with a $\chG$-equivariant local system
$\mathcal{V}$ on $S$
(\cite{ABV}*{Definition 7.6}).
The set of complete geometric parameters is denoted by $\Gparameters$.
This definition ignores more general local systems \cite{ABV}, which are
equivariant for an algebraic cover of ${^\vee}G$. These aren't needed here, and this simplifies the discussion.
By \cite{ABV}*{Theorem 10.11} there is  a canonical bijection

\begin{equation}
\label{Gparameters}
  \Xi({^\vee}G^{\Gamma})\longleftrightarrow\Pi(G/\R)
\end{equation}
The set on the right is the set of (equivalence classes of)
irreducible representations of
certain real forms of $G$, including a fixed quasisplit form.
We write bijection (\ref{Gparameters}) as
$$
\xi\mapsto \pi(\xi).
$$
Each irreducible representation $\pi(\xi)$ is the unique irreducible quotient
of a standard representation $M(\xi)$, so we also have a bijection
$$
\xi\mapsto M(\xi)
$$
between $\Gparameters$ and a set of standard representations.
 Let $K\Pi(G/\R)$ be the Grothendieck group of
 the admissible representations of the real forms
 of $G$ appearing on the right of (\ref{Gparameters}).
This Grothendieck group has two bases, namely $\{\pi(\xi)\}$ and $\{M(\xi)\}$ for $\xi\in \Xi(^{\vee}G^{\Gamma})$.

There is a parallel construction  in terms of  sheaves for the dual
group ${^\vee}G$. Suppose
$\xi\in\Gparameters$. The local system of this complete geometric
parameter is a $\chG$-equivariant sheaf on $S$.
Applying the functors of extension by zero to the closure of $S$, and then taking the direct image gives an
irreducible $\chG$-equivariant constructible sheaf $\mu(\xi)$ on
$X(\ch G^{\Gamma})$. This defines a bijection
$$
\xi\mapsto \mu(\xi)
$$
between complete geometric parameters and irreducible
$\chG$-equivariant constructible sheaves.
Alternatively, one may apply the functors of intermediate extension
and direct image. This
defines an irreducible ${^\vee}G$-equivariant perverse sheaf
$P(\xi)$, and a bijection
$$
\xi\mapsto P(\xi)
$$
between complete geometric parameters and irreducible
$\chG$-equivariant perverse sheaves.
The Grothendieck groups of the categories of ${^\vee}G$-equivariant
constructible and perverse
sheaves are isomorphic (\cite{ABV}*{Lemma 7.8}, \cite{bbd}).  We
identify the  two Grothendieck groups and write them as
$K X(\ch G^{\Gamma})$.  The sets $\{\mu(\xi)\}$ and $\{ P(\xi)\}$ for $\xi \in
\Xi({^\vee}G^{\Gamma})$ each form a basis of $K X(\chG^{\Gamma})$.

We now define a pairing
\begin{equation}
\label{prepair}
\langle \cdot, \cdot \rangle_{G} : K \Pi(G/\R) \times K X(\ch
G^{\Gamma}) \rightarrow
\mathbb{Z}
\end{equation}
using the bases of standard representations and constructible sheaves:
\begin{equation*}
\label{standpair}
\langle M(\xi), \mu(\xi') \rangle_{G} = e(\xi) \, \delta_{\xi, \xi'},
\quad \xi,\xi' \in \Xi(\ch G^{\Gamma}).
\end{equation*}
Here $e(\xi)$ is the Kottwitz sign (\cite{ABV}*{Definition 15.8}), and
$\delta_{\xi, \xi'}$ is the Kronecker
delta.
It is natural to ask what the formula for  this pairing is in terms of the bases
of irreducible representations and
perverse sheaves. It is a deep fact that in these alternative bases
the pairing is
also, up to signs, given by the Kronecker delta function.  More precisely
\begin{equation*}
\label{irredpair}
\langle \pi(\xi), P(\xi') \rangle_{G} = e(\xi)\, (-1)^{d(\xi)} \,
\delta_{\xi, \xi'}, \quad \xi,\xi' \in \Gparameters
\end{equation*}
where $d(\xi)$ is the dimension of the orbit $S$ in $\xi = (S,
\mathcal{V})$ (\cite{ABV}*{Theorem 1.24}).

Using the pairing \eqref{prepair} we may regard virtual characters as
$\mathbb{Z}$-valued linear functionals on $K X({^\vee}G^{\Gamma})$. Of particular
importance are the stable virtual characters.
The theory of microlocal geometry provides a family of linear functionals
\begin{equation}
\label{mmm}
\chi^{\mathrm{mic}}_{S} : K X(\ch G^{\Gamma}) \rightarrow \Z
\end{equation}
parameterized by  the ${^\vee}G$-orbits $S \subset
X({^\vee}G^{\Gamma})$.
These microlocal multiplicity maps appear in the
theory of \emph{characteristic cycles} (\cite{ABV}*{Chapter 19},
\cite{Boreletal}), and are associated with ${^\vee}G$-equivariant
local systems on a conormal bundle over $X({^\vee}G^{\Gamma})$ (\cite{ABV}*{Section 24}, \cite{GM}).
The virtual characters associated by the pairing to these linear functionals
are stable (\cite{ABV}*{Theorems 1.29 and 1.31}).

Now we return to the  Arthur parameter $\psi_G$ given in
(\ref{aparameter}). Associated to $\psi_G$ is a Langlands parameter
$\phi_{\psi_G}\nomenclature{$\phi_{\psi_G}$}{Langlands parameter associated to $\psi_G$}$ \cite{Arthur89}*{Section 4} defined by
\begin{equation}
\label{phipsi}
\phi_{\psi_G}(w)=\psi_G \left(w,
\begin{bmatrix}
  |w|^{\frac12}&0\\0&|w|^{-\frac12}
\end{bmatrix} \right), \quad w\in W_\R.
\end{equation}
Let $S_{\psi_G}\subset X(\ch G^{\Gamma})$ be the $\ch G$-orbit of $\phi_{\psi_G}$.
We define $\eta^{\mathrm{mic}}_{\psi_G}$ to be the virtual
character associated to $\chi_{S_{\psi_{G}}}^{\mathrm{mic}}$ by the pairing. That is,
$\eta^{\mathrm{mic}}_{\psi_G}$ is  the unique virtual character
satisfying
$$\langle \eta^{\mathrm{mic}}_{\psi_{G}}, \mu \rangle_{G} = \chi_{S_{\psi_{G}}}^{\mathrm{mic}}(\mu),
\quad \mu \in KX({^\vee}G^{\Gamma}).
$$
As a distribution, the stable virtual character $\eta^{\mathrm{mic}}_{\psi_{G}}$
is supported on real forms of $G$ which include the quasisplit form
$G(\mathbb{R})$.  In this more general context,
Adams, Barbasch and Vogan show that $\eta^{\mathrm{mic}}_{\psi_{G}}$ satisfies the
conditions of Problems
B-D.   For our purposes however, it suffices to consider the restriction
\begin{equation}\label{eq:etaABVIntro}
\etaABV_{\psi_{G}} = \eta^{\mathrm{mic}}_{\psi_{G}}|_{G(\mathbb{R})}
\end{equation}
of $\eta^{\mathrm{mic}}_{\psi_{G}}$ to the quasisplit form
$G(\mathbb{R})$, and let $\Pi_{\psi_G}^{\mathrm{ABV}}$ be the set of irreducible representations
occurring in
$\etaABV_{\psi_{G}}$.

Having sketched the construction of $\etaA_{\psi_{G}}$ and
$\etaABV_{\psi_{G}}$, we come to the main result.
\begin{thm}
\label{main}
  Let $G$ be either the real split form of $\mathrm{GL}_{N}$ or a real quasisplit form of $\mathrm{Sp}_{N}$ or
  $\mathrm{SO}_{2N+1}$.
  Suppose $\psi_G$ is an Arthur parameter for $G$. Then
  $$
  \etaA_{\psi_G}= \etaABV_{\psi_G}.
  $$
  In particular
  $$
  \Pi_{\psi_G}^{\mathrm{Ar}}=\Pi_{\psi_G}^{\mathrm{ABV}}.
  $$
\end{thm}
In the case of $G=\mathrm{SO}_{2N}$, $\Pi_{\psi_G}^{\mathrm{Ar}}$ is the union
of $\Pi_{\psi_G}^{\mathrm{ABV}}$ and its twist by the outer automorphism
of $G$. 
For the precise statement
see Theorem \ref{finalthm1}.
We continue by giving an outline of the proof under the assumption that $G$ is
not a special orthogonal group of even rank.

Arthur's definition of $\etaA_{\psi_{G}}$ is given in terms of the twisted
endoscopic transfer map $\mathrm{Trans}_{G}^{\mathrm{GL}_{N} \rtimes
  \vartheta}$ appearing in (\ref{spectrans}).  The first step in the proof of Theorem \ref{main} is
to compare $\mathrm{Trans}_{G}^{\mathrm{GL}_{N} \rtimes \vartheta}$
with the  analogous twisted endoscopic lifting map
$\mathrm{Lift}_{0}$  defined in \cite{Christie-Mezo}*{Section 5}.   We
wish to prove
\begin{equation}
  \label{Lift0trans}
  \Lift_0=\Trans_G^{\GL_N \rtimes \vartheta}.
\end{equation}
The construction of the map  $\mathrm{Lift}_{0}$ follows the construction in \cite{ABV}*{Section
  26} and is given in terms of a pairing analogous
to \eqref{prepair} in
the setting of twisted characters and sheaves.
Associated to the involution $\vartheta$
is a $\mathbb{Z}$-module of {\it twisted characters}
$K\Pi(\GL_N(\R),\vartheta)$ \cite{AVParameters} . On
the dual side we have a $\mathbb{Z}$-module of
{\it twisted sheaves} $K\Xtwisted$ \cite{LV2014}.
We wish to define a pairing
\begin{equation}
\label{prepairtwist}
\langle \cdot, \cdot \rangle: K \Pi(\mathrm{GL}_{N}(\mathbb{R}),
\vartheta) \times KX
({^\vee}\mathrm{GL}_{N}^{\Gamma}, \vartheta ) \rightarrow \mathbb{Z}.
\end{equation}
One of the technical difficulties in defining this pairing lies in
making canonical choices of extensions.
 Suppose
$\xi\in\GLparameters$ (see \eqref{Gparameters}), with associated
standard representation $M(\xi)$. If $M(\xi)$ is fixed by $\vartheta$ then it is induced from an irreducible $\vartheta$-fixed representation which
extends in two ways to an irreducible representation of a disconnected Levi subgroup.  Inducing these two extended irreducible representations one obtains two extensions of $M(\xi)$ to
$\GL_N(\R)\rtimes \langle \vartheta \rangle$.
Each of the two resulting representations restricts to
$\mathrm{GL}_{N}(\mathbb{R}) \rtimes \vartheta$ to give a twisted
character. The $\mathbb{Z}$-module $K\Pi(\GL_N(\R),\vartheta)$ is defined so
that if $M(\xi)^\pm$
are the two extensions, then  $M(\xi)^-=-M(\xi)^+$ in
$K\Pi(\GL_N(\R),\vartheta)$.

The literature offers two techniques to choose a preferred extension of $M(\xi)$.  As mentioned earlier, Arthur uses
Whittaker data to fix a preferred extension which we
denote $M(\xi)^{\thicksim}$ and call the \emph{Whittaker extension} (\cite{Arthur}*{\emph{pp.} 62-63}). On the
other hand \cite{AVParameters}*{(20), Section 5}
gives an extension which we label $M(\xi)^{+}$ and call the
\emph{Atlas extension}.  The techniques differ in the type of induction employed.  By taking quotients, we obtain extensions $\pi(\xi)^{\thicksim}$ and $\pi(\xi)^{+}$ of $\pi(\xi)$.

One may also show that the $\chGL_N$-equivariant constructible sheaf $\mu(\xi)$
extends in two ways
to a $(\chGL_N\rtimes\langle\vartheta\rangle)$-equivariant constructible sheaf
on $\XGL$.  The  two extensions $\mu(\xi)^\pm$ again differ by sign in
$KX(\chGL_N^\Gamma,\vartheta)$.
In order to choose a preferred extension, we use a special property of
irreducible ${^\vee}\mathrm{GL}_{N}$-equivariant sheaves, namely that
they are all constant sheaves.  We define $\mu(\xi)^{+}$ to be the
irreducible $({^\vee}\mathrm{GL}_{N} \rtimes \langle \vartheta
\rangle)$-equivariant constant sheaf on
$X({^\vee}\mathrm{GL}_{N}^{\Gamma})$.  Similar considerations yield a preferred extension $P(\xi)^{+}$ of the perverse sheaf $P(\xi)$.

Having chosen the extensions we define pairing (\ref{prepairtwist}) by
\begin{equation}
\label{pairing1}
\langle M(\xi)^{\thicksim},\mu(\xi')^+\rangle=\delta_{\xi,\xi'}.
\end{equation}

The endoscopic lifting map $\mathrm{Lift}_{0}$ is defined using the 
pairings (\ref{prepair}), (\ref{prepairtwist}) and the
 map $\epsilon$ (see \eqref{preepsilon}) as follows.  The map
 $\epsilon$ naturally induces a map
$$
X(\chG^{\Gamma})\rightarrow X(\chGL_N^\Gamma).
$$
The usual inverse image functor on constructible sheaves then induces
a homomorphism
$$
\epsilon^*:K_{\C}X(\chGL_N^\Gamma,\vartheta)\rightarrow K_{\C}X(\chG^{\Gamma})
$$
on the complexifications of the $\mathbb{Z}$-modules.
The adjoint of $\epsilon^{*}$ with respect to the pairings is the homomorphism
$$
\epsilon_{*}:  K_{\C}\Pi(G/\R) \rightarrow K_{\C}\Pi(\mathrm{GL}_N(\R),\vartheta)
$$
defined by
$$
\langle \epsilon_{*}(\eta),\mu\rangle = \langle
\eta,\epsilon^*(\mu)\rangle_G, \quad \eta \in  K_{\C}\Pi(G/\R), \quad \mu \in
K_{\C}X(\chGL_N^\Gamma,\vartheta).
$$
Here, the pairings on the left and right are \eqref{prepairtwist} and
\eqref{prepair}, respectively.
Finally, the endoscopic lifting map
$$
\Lift_0:K_\C\Pi(G(\R))^{\mathrm{st}} \rightarrow K_\C\Pi(\mathrm{GL}_N(\R),\vartheta)
$$
is defined to be the restriction of $\epsilon_{*}$ to the stable
subspace of $ K_{\C}\Pi(G(\R))$, the complex virtual characters of $G(\mathbb{R})$.
That is,  if $\eta \in K_{\C}\Pi(G/ \mathbb{R})$ is stable then
$\Lift_0(\eta)$ is defined by
\begin{equation}
  \label{Lift0}
\langle \Lift_0(\eta),\mu\rangle = \langle \eta,\epsilon^*(\mu)\rangle_G
\end{equation}
for all $\mu\in KX(\chGL_N^\Gamma,\vartheta)$.

Now that $\Lift_0$ is  defined, we may proceed to check the equality
\eqref{Lift0trans}.
Fix a $\chG$-orbit $S_G\subset X$. The local multiplicity function
taking a constructible sheaf to the dimension of a stalk at a point in
$S_G$ is a linear functional on $KX(\ch G^{\Gamma})$. By the pairing
\eqref{prepair} this defines an element of
$K\Pi(G/\R)$. This is a stable virtual character denoted by
$\etaloc_{S_{G}}$.  It is the sum of
the standard representations in what is sometimes called  a \emph{pseudopacket}.

Let $S \subset \XGL$ be the ${^\vee}\mathrm{GL}_{N}$-orbit containing
$\epsilon(S_G)$, and let $M(S,1)$ be the standard representation
defined by the trivial local system on $S$.  By Proposition \ref{twistimlift}

\begin{subequations}
  \renewcommand{\theequation}{\theparentequation)(\alph{equation}}
  \label{compare}
\begin{equation}
  \Lift_0(\etaloc_{S_G})=(-1)^{\ell^I(S,1)-\ell^I_\vartheta(S,1)} \  \mathrm{Tr}_{\vartheta} (M(S,1)^{+})
\end{equation}
The terms in the exponent are defined in Section \ref{pairings} and we use $\mathrm{Tr}_{\vartheta}$ to denote the twisted character as in (\ref{spectrans}).\footnote{We drop the notation $\mathrm{Tr}_{\vartheta}$ in the body of the paper.}
On the other hand Arthur defines a stable character $\eta'_{S_G}$ by
\begin{equation}
\Trans_G^{\GL_N\rtimes\vartheta}(\eta'_{S_G})= \mathrm{Tr}_{\vartheta}(M(S,1)^{\thicksim}).
\end{equation}
According to \cite{AMR1} $\eta'_{S_G}=\etaloc_{S_G}$.
The two extensions of $M(S,1)$ are related by

\begin{equation}
\mathrm{Tr}_{\vartheta} (M(S,1)^{\thicksim}) =(-1)^{\ell^I(S,1)-\ell^I_\vartheta(S,1)} \  \mathrm{Tr}_{\vartheta}( M(S,1)^{+})
\end{equation}
\end{subequations}
(see  Proposition \ref{wasign2}). Taken together, (\ref{compare}(a-c))
give
\begin{equation}
\label{abvpseudo}
\mathrm{Lift}_{0}\left( \eta^{\mathrm{loc}}_{S_{G}}\right)  = \mathrm{Tr}_{\vartheta}( M(S,1)^{\thicksim}) = \mathrm{Trans}_{G}^{\mathrm{GL}_{N} \rtimes \vartheta}( \eta^{\mathrm{loc}}_{S_{G}} ).
\end{equation}
Identity \eqref{Lift0trans} follows from the fact that the $\eta^{\mathrm{loc}}_{S_{G}}$ form a basis of the stable virtual characters.

Going back to \eqref{spectrans}, and using \eqref{Lift0trans}  we see
$\eta_{\psi_G}^{\mathrm{Ar}}$ is determined by
$$
\Lift_0(\eta_{\psi_G}^{\mathrm{Ar}})= \mathrm{Tr}_{\vartheta}(\pi_{\psi}^{\thicksim}).
$$
Therefore to prove Theorem \ref{main} it is enough to show
$$
\Lift_0(\eta_{\psi_G}^{\mathrm{ABV}})= \mathrm{Tr}_{\vartheta} (\pi_{\psi}^{\thicksim}).
$$
According to (\ref{Lift0}), this identity is equivalent to
\begin{equation}
  \label{abvpitilde}
\langle \pi_{\psi}^{\thicksim}, P(\xi')^{+}\rangle = \langle
\eta_{\psi_G}^{\mathrm{ABV}}, \epsilon^{*}(P(\xi')^{+}) \rangle_{G}.
\end{equation}
%
Recall from \eqref{pairing1} that the pairing on the left-hand side is
defined in terms of
standard representations and constructible sheaves.
However, on the left of (\ref{abvpitilde}) we require a corresponding
formula in terms of irreducible representations and
perverse sheaves.  It turns out that the exact formula required is
\begin{equation}
  \label{twist1.24}
\langle \pi(\xi)^{+},P(\xi')^{+} \rangle =
(-1)^{\ell^I(\xi)-\ell^I_\vartheta(\xi)} \delta_{\xi,\xi'}.
\end{equation}
We prove this formula from an identity involving the twisted Kazhdan-Lusztig-Vogan
polynomials.  We are in the setting of \cite{LV2014} and \cite{AVParameters}, so we have all of the tools of the Hecke algebra available.
The proof of (\ref{twist1.24}) is carried out in Section \ref{pairings}, and Theorem \ref{main} then follows.

We now provide further details by running
through the remaining sections in sequence.
Section \ref{llc} begins with an outline of the local Langlands
correspondence appearing in \cite{ABV}.  One of the features in this
correspondence is the parameterization of inner forms of
$G(\mathbb{R})$ using \emph{strong involutions}, and the subsequent
inclusion of \emph{representations of strong involutions} in the
correspondence.  Unlike the overview above, we shall be keeping track
of the infinitesimal characters of these representations.  As a
result, the variety $X({^\vee}G^{\Gamma})$ in the overview is replaced
by $X(\O,{^\vee}G^{\Gamma})$, where $\O$ is an
infinitesimal character.  We assume all infinitesimal characters to be
regular until Section \ref{equalapacketsing}.  Another important theme
of Section \ref{llc} is the equivalence of complete geometric
parameters with
the  parameters used in \cite{AVParameters}, which
we refer to as {\it Atlas} parameters.
Atlas parameters are indispensable
in defining the Atlas extensions, and in the ensuing Hecke operator
computations of Section \ref{pairings}.  For these reasons
it is important to make the connection between the approach in \cite{ABV} and \cite{AVParameters} precise in the case of $\mathrm{GL}_N(\R)$. 
The section closes with a
discussion on twisted characters, and the $\mathbb{Z}$-module $K
\Pi(\O, \mathrm{GL}_{N}(\mathbb{R}), \vartheta)$ which
contains them.

Section \ref{sheaves} is devoted to ${^\vee}G$-equivariant sheaves,
and their relationship with $\mathcal{D}$-modules and characteristic cycles.   We
recall a category of sheaves extended by an automorphism $\upsigma$ of
${^\vee}\mathrm{GL}_{N}$ (\cite{ABV}*{(25.7)}).  The automorphism is of
the form
$$\upsigma = \mathrm{Int}(s) \circ \vartheta$$
where $s \in {^\vee}\mathrm{GL}_{N}$.  The element $s$ plays no
meaningful role in this section, but becomes important in the theory
of endoscopy (Section \ref{endosec}).  The category of extended
sheaves is the counterpart to the category of representations on
$\mathrm{GL}_{N} (\mathbb{R}) \rtimes \langle \vartheta \rangle$.  We
define the canonical extended sheaves $\mu(\xi)^{+}$ and $P(\xi)^{+}$
in Lemma \ref{cansheaf}.
The twisted characters that one obtains from representations of
$\mathrm{GL}_{N}(\mathbb{R}) \rtimes \langle \vartheta \rangle$ find
their counterpart as microlocal traces (\cite{ABV}*{(25.1)}) which are
supported on extensions of irreducible sheaves.  The
$\mathbb{Z}$-module counterpart to  $K \Pi(\O,
\mathrm{GL}_{N}(\mathbb{R}), \vartheta)$ is defined in
(\ref{twistsheafgroth}) and is denoted by  $K (X(\O,
       {^\vee}\mathrm{GL}^{\Gamma}), \upsigma)$.   The pairings (\ref{prepair}) and (\ref{prepairtwist}) are also defined in this section.

 We provide a terse summary of $\mathcal{D}$-modules and their
       relationship to equivariant sheaves, characteristic cycles, the
       microlocal multiplicity maps (\ref{mmm}), and the definition of
       $\eta_{\psi_{G}}^{\mathrm{ABV}}$.   The set of irreducible characters
in the support of $\eta_{\psi_{G}}^{\mathrm{ABV}}$ is
       denoted by $\Pi_{\psi_{G}}^{\mathrm{ABV}}$ and is called the
       \emph{ABV-packet} of $\psi_{G}$.

The main objective of Section \ref{pairings} is to prove the
equivalence of the twisted pairings (\ref{prepairtwist}) and
(\ref{twist1.24}).  Our proof is an adaptation of the proof of the
equivalence for ordinary pairings (\cite{ABV}*{Sections 16-17}) using the tools of
\cite{AVParameters}.  As noted earlier, Hecke operators are among
these tools.  A conspicuous difference between \cite{ABV} and
\cite{AVParameters} is in the objects upon which Hecke operators act.
In \cite{ABV} Hecke operators are defined on both characters and
sheaves.  By contrast, the Hecke operators of \cite{AVParameters}*{Section 7}
are defined only on (twisted) characters.   The links between
characters and sheaves in the Hecke actions are the Riemann-Hilbert and  Beilinson-Bernstein
correspondences (\cite{ABV}*{Theorems 7.9 and 8.3}).  In Sections \ref{bbvd} and \ref{duality} we
describe these correspondences as a bijection
$$P(\xi) \longleftrightarrow \pi( {^\vee}\xi), \quad \xi \in
\Xi(\O, {^\vee}G^{\Gamma}),$$
where $\pi({^\vee}\xi)$ is the \emph{Vogan dual} of $\pi(\xi)$ (as the
equivalence class of a Harish-Chandra module) (6.1
\cite{AVParameters}).  For $G = \mathrm{GL}_{N}$ the correspondence is
extended to
$$P(\xi)^{+} \longleftrightarrow \pi({^\vee}\xi)^{+}$$
for $\vartheta$-fixed complete geometric parameters $\xi$.  Once
sheaves are aligned with characters in this manner, the rest of the
proof of the equivalence of the twisted pairings follows \cite{ABV}
without incident.

Subsection \ref{tKLV} is included in Section \ref{pairings} only
because it uses the same machinery.  This subsection presents an
argument from the twisted Kazhdan-Lusztig-Vogan algorithm
(\cite{LV2014}, \cite{adamsnotes}) which is crucial to the comparison
of Whittaker and Atlas extensions  in Section
\ref{whitsec}.

In Section \ref{endosec} we describe the theory of endoscopy, both
standard and twisted, for $\mathrm{GL}_{N}$ using the framework of
\cite{ABV}.  The standard theory of endoscopy  in Subsection
\ref{standend} is simply a specialization of \cite{ABV}*{Section 26} to $G =
\mathrm{GL}_{N}$.  It is included primarily to motivate the twisted
theory, but is also used in Proposition \ref{prop:singletonGLN}
further on.  The twisted theory of endoscopy in Subsection
\ref{twistendsec} is a specialization of \cite{Christie-Mezo}*{Section 5.4} to
$\mathrm{GL}_{N}$.  In this subsection the
${^\vee}\mathrm{GL}_{N}$-equivariant sheaves of $K (X(\O,
    {^\vee}\mathrm{GL}^{\Gamma}), \upsigma)$ are recast as
    $({^\vee}\mathrm{GL}_{N} \rtimes \langle \upsigma
    \rangle)$-equivariant sheaves.  The endoscopic lifting map takes
    the form
$$\mathrm{Lift}_{0} : K_{\mathbb{C}}\Pi(\O_{G}, G(\mathbb{R}))^{\mathrm{st}} \rightarrow K_{\mathbb{C}} \Pi(\O,
    \mathrm{GL}_{N}(\mathbb{R}), \vartheta).$$
The precursor to (\ref{abvpseudo}) is Proposition \ref{twistimlift},
where the Atlas extension is used instead of the Whittaker extension.
The endoscopic lifting $\mathrm{Lift}_{0}(\eta_{\psi_{G}}^{\mathrm{ABV}})$ is
described as an element $\eta_{\psi}^{\mathrm{ABV}+} \in K
\Pi(\O, \mathrm{GL}_{N}(\mathbb{R}), \vartheta)$, which
reduces to $\eta_{\psi}^{\mathrm{ABV}}$ when restricted to
$\mathrm{GL}_{N}(\mathbb{R})$ (Theorem \ref{thm:etaplus1}).  The
endoscopic lifting map is proved to be injective for
$\mathrm{GL}_{N}$-regular infinitesimal character $\O_{G}$.

In Section \ref{glnpacket}  we prove that for \emph{any} A-parameter
$\psi$ of $\mathrm{GL}_{N}$ (not necessarily of the form
(\ref{prepsitilde})), there is only one irreducible character in the
support of $\eta_{\psi}^{\mathrm{ABV}}$.  This implies that
$$\Pi_{\psi}^{\mathrm{ABV}} = \{ \pi_{\psi} \} = \Pi_{\psi}^{\mathrm{Ar}}.$$
It also implies that
$\eta_{\psi}^{\mathrm{ABV}+}$ is supported on a single irreducible
twisted character obtained by extension from $\pi_{\psi}$.
The proof begins under the assumption that $\psi$ is an
A-parameter studied by Adams and Johnson (\cite{Adams-Johnson}).
Adams and Johnson defined packets for these parameters, and it is
easily shown that their packets are singletons for $\mathrm{GL}_{N}$.
The anticipated equality of the Adams-Johnson packets with the
ABV-packets is proven in \cite{arancibia_characteristic}.  The proof that
ABV-packets are singletons for arbitrary A-parameters $\psi$
of $\mathrm{GL}_{N}$ follows from a decomposition of $\psi$ in
terms of Adams-Johnson A-parameters of smaller general linear groups,
and an application of all standard endoscopic lifting  from the direct
product of these smaller general linear groups (Proposition
\ref{prop:singletonGLN}).

The purpose of Section \ref{whitsec} is the proof of Equation
(\ref{compare}(c)).  This equation gives a precise relationship between
the Whittaker and Atlas extensions of $\pi(\xi)$ in terms of the
\emph{integral lengths} $l^{I}(\xi)$ and $l^{I}_{\vartheta}(\xi)$
(\cite{ABV}*{(16.6)}, (\ref{intlength}), (\ref{thetalength})).  This
identity is peculiar in that a Whittaker extension is inherently an
analytic object, whereas an Atlas extension is inherently algebraic.
When $\pi(\xi)$ is the Langlands quotient of a (standard) principal
series representation, the difference between the two extensions may
be attributed to differences in the extensions of quasicharacters of
the diagonal subgroup $H$.  This reduction for principal series
furnishes an easy proof of (\ref{compare}(c)) (Lemma \ref{prinsame}).

In some sense (see the proof of Proposition \ref{conjq}), $\pi(\xi)$
is furthest from a Langlands quotient of principal series when
$\pi(\xi)$ is \emph{generic}, \emph{i.e.} has a Whittaker model.  The
proof of (\ref{compare}(c)) for generic representations is the key to the
general proof, in that irreducible generic representations occur as
subrepresentations of  standard representations (Lemma
\ref{uniquegeneric}), and determine the Whittaker extensions of
standard representations.   If one knows the (signed) multiplicity
with which an irreducible twisted generic character $\pi(\xi_{0})^{+}$
appears in the decomposition of a twisted standard principal series
representation $M(\xi)^{+}$, then one can use the knowledge of
(\ref{compare}(c)) for $\pi(\xi)$ to prove (\ref{compare}(c)) for
$\pi(\xi_{0})$.  This desired multiplicity is computed in Proposition
\ref{conjq}, and the proof of (\ref{compare}(c)) for generic
$\pi(\xi_{0})$ occurring in the standard principal representation $M(\xi)$ is
Proposition \ref{wasign1}.

It is implicit in the previous paragraph that the parameters and
representations are all $\vartheta$-stable.  However not every
$\vartheta$-stable generic representation $\pi(\xi_{0})$ is a
subrepresentation of a $\vartheta$-stable principal series
representation.  Therefore, the strategy of the previous paragraph does
not provide an exhaustive proof of (\ref{compare}(c)).  Most of Section
\ref{whitsec} is dedicated to the description of a $\vartheta$-stable standard
representation which plays the part of the principal series
representation.  In Lemma \ref{biglem}  we prove that every
$\vartheta$-stable generic
representation $\pi(\xi_{0})$ which has \emph{integral infinitesimal
  character} is a subrepresentation of a $\vartheta$-stable standard
representation satisfying (\ref{compare}(c)).  We remove the restriction
of integrality on the infinitesimal character in Lemma \ref{biglem1}.
We then follow the strategy of the previous paragraph using the
$\vartheta$-stable standard representation of Lemma \ref{biglem1} to prove
(\ref{compare}(c)) in general.  This is the last result needed to apply the
twisted pairing (\ref{twist1.24}) to the computation of (\ref{abvpitilde}).

The theorems comparing $\etaA_{\psi_{G}}$ with $\eta_{\psi_{G}}^{\mathrm{ABV}}$
are  found in Sections \ref{equalapacketreg} and
\ref{equalapacketsing}.  Section \ref{equalapacketreg} is presented
under the assumption that the infinitesimal character
$\O_{G}$ is regular in $\mathrm{GL}_{N}$.  This regularity
condition is removed in Section \ref{equalapacketsing} by applying
the \emph{Jantzen-Zuckerman translation principle} to the (twisted) characters and to
the pairings.  The main theorem, Theorem \ref{finalthm1}, states that
\begin{equation}
  \label{main1}
\etaA_{\psi_{G}} = \eta_{\psi_{G}}^{\mathrm{ABV}} \mbox{ and }
\Pi_{\psi_{G}}^{\mathrm{Ar}} = \Pi_{\psi_{G}}^{\mathrm{ABV}}
\end{equation}
when $G$ is not isomorphic to $\mathrm{SO}_{N}$
for even $N$, and that
\begin{equation}
  \label{main2}
\etaA_{\psi_{G}} = \frac{1}{2} \left(\eta_{\psi_{G}}^{\mathrm{ABV}} +
  \eta^{\mathrm{ABV}}_{\mathrm{Int}(\tilde{w}) \circ \psi_{G}} \right) \mbox{ and
  }\Pi_{\psi_{G}}^{\mathrm{Ar}} = \Pi_{\psi_{G}}^{\mathrm{ABV}} \cup \,
  \Pi_{\mathrm{Int}(\tilde{w})   \circ   \psi_{G}}^{\mathrm{ABV}}
\end{equation}
when $G \cong \mathrm{SO}_{N}$ with $N$ even.

In light of this theorem, we again look back to Problems B-E in Section
\ref{btoe}.  Problem E is solved immediately in the affirmative.  Problem C
has a partial solution in \cite{Arthur} and a complete solution in \cite{ABV}.  When $G$ is an odd rank special orthogonal group the solutions are easier to write and compare.  In this context the solutions in \cite{Arthur} and \cite{ABV}  take the shape
$$\eta_{\psi_{G}}^{\mathrm{Ar}} = \sum_{\pi \in \Pi_{\psi_{G}}^{\mathrm{Ar}}} \mathrm{Tr}
\left(\tau_{\psi_{G}}(\pi)(s_{\psi_{G}})\right) \pi$$
(\emph{cf.} (\ref{arthur4.1})), and
$$\eta_{\psi_{G}}^{\mathrm{ABV}} = \sum_{\pi \in \Pi_{\psi_{G}}^{\mathrm{ABV}}}
(-1)^{d(\pi) - d(S_{\psi_{G}})}
\ \dim\left(\tau^{\mathrm{ABV}}_{\psi_{G}}(\pi) \right) \, \pi$$
respectively.  The right-hand sides of the two equations are equal by (\ref{main1}).  By the linear independence of characters on $G(\mathbb{R})$, the coefficients of the right-hand sides are also equal.  One would therefore expect that
\begin{equation}
\label{sametaus}
\tau_{\psi_{G}}(\pi) = \tau^{\mathrm{ABV}}_{\psi_{G}}(\pi) \mbox{ and } \epsilon_{\pi}= (-1)^{d(\pi) - d(S_{\psi_{G}})} = \tau_{\psi_{G}}(\pi)(s_{\psi}).
\end{equation}
These identities are indeed shown to be true through a comparison of Problem D in Section \ref{btoe}.

It follows immediately from (\ref{sametaus}) that each irreducible subrepresentation of $\tau_{\psi_{G}}(\pi)$ has the common value $\epsilon_{\pi}$ on $s_{\psi_{G}}$.   This is related to a larger unsolved problem, which is to determine whether the representations $\tau_{\psi_{G}}(\pi)$ are irreducible or not.  This irreducibility problem has been explored in \cite{MR1}, \cite{MR3}, \cite{MR4}.  It was also explored in \cite{Moeglin-11} for various $p$-adic groups, where the representations have been shown to be irreducible.  Theorem \ref{finalthm1} opens up the possibility of using techniques from microlocal geometry to settle this problem. 

Our work also connects with the study of Adams-Johnson packets  (\cite{Adams-Johnson}).  These packets have been proven to equal Arthur's packets in \cite{AMR}.
Moreover, as shown in \cite{arancibia_characteristic}, they are precisely the ABV-packets with regular and integral infinitesimal character.

A natural question for future consideration is how the packets for quasisplit unitary groups, established by Mok \cite{Mok}, compare with the microlocal packets of \cite{ABV}.  The methods developed here appear to be equally applicable to the setting of quasisplit unitary groups.  Furthermore, the context of pure inner forms in which we work, ought also to allow for easy comparison with the related work of \cite{Kaletha-Minguez} and \cite{Arthur}*{Chapter 9}.

Another natural question is whether similar comparisons between $p$-adic Arthur-packets can be made.  In the $p$-adic context $\eta_{\psi_{G}}^{\mathrm{Ar}}$ is also defined in \cite{Arthur}.  The beginnings of $\etaABV_{\psi_{G}}$ in the $p$-adic context are to be found in \cite{vogan_local_langlands} and \cite{cliftonetal}.  Low rank comparisons between the two stable distributions are  made in \cite{cliftonetal}*{Part 2}.

The second and third authors would like to thank the developers of the Atlas of Lie groups software.  It was a pleasure to see our early conjectures borne out by low rank computations.

\section{The local Langlands correspondence}
\label{llc}

This section begins with a review of the local Langlands correspondence as conceived in \cite{ABV}. An important feature of this version of the correspondence is the notion of  \emph{strong real forms} and their representations.  More recently, strong real forms have been supplanted by the equivalent notion of \emph{strong involutions} (\cite{Adams-Fokko}).  We have chosen to use the language of strong involutions in our review.

Another difference in our review is in limiting ourselves to only \emph{pure} strong involutions.  In doing so, we limit the scope of \cite{ABV} to fewer real forms of $G$.  This limitation is compensated for by not having to introduce covers of the dual group ${^\vee}G$.  We still capture all of the information needed for the quasisplit form of $G$, while preserving a sense of how the theory applies to other real forms.

The objects parameterizing irreducible representations in \cite{ABV} have also been supplanted by newer parameters in \cite{Adams-Fokko}  and \cite{AVParameters}.  We call these newer parameters \emph{Atlas parameters}.  The advantages to Atlas parameters are their amenability to Vogan duality and Hecke algebra computations.  These advantages are used in Section \ref{pairings}.  Another advantage to using Atlas parameters is in defining canonical extensions of representations of $\mathrm{GL}_{N}(\mathbb{R})$ to representations of $\mathrm{GL}_{N}(\mathbb{R}) \rtimes \langle \vartheta \rangle$.  We call these canonical extensions \emph{Atlas extensions}.

We conclude this section with a discussion of the Grothendieck groups of representations for connected groups $G$ and for the disconnected group $\mathrm{GL}_{N} \rtimes \langle \vartheta \rangle$.  In the connected case the Grothendieck group is isomorphic to the $\mathbb{Z}$-span of distribution characters.  In the disconnected case we construct a quotient of the Grothendieck group which will be seen to be isomorphic to the $\mathbb{Z}$-span of twisted distribution characters.

\subsection{Extended groups}
\label{extgroups}

In this section $G$ can be an arbitrary  connected reductive
complex group.
We give a version of the local Langlands correspondence suitable to
our application. We largely follow \cite{ABV}, with modifications from
the more recent papers \cite{Adams-Fokko}, \cite{AVParameters} and \cite{AvLTV}.

Our starting point is the connected reductive complex group $G$,
together with a pinning
\begin{equation}
  \label{pinning}
  \left(B,H, \{ X_{\alpha}\}\right)
\nomenclature{$\left(B,H, \{ X_{\alpha}\}\right)$}{pinning}
 \end{equation}
in which $B$ is a Borel subgroup, $H \subset B$ is a maximal torus and
$\{X_{\alpha}\}$ is a set of simple root vectors relative to the positive root system
$R^{+}(G,H) = R(B,H)\nomenclature{$R^{+}(G,H)$}{system of positive roots}$ of $R(G,H)
\nomenclature{$R(G,H)$}{root system}$.  Let ${^\vee}\rho = \frac{1}{2} \sum_{\alpha \in R^{+}(G,H)} \ch \alpha
\nomenclature{$\ch\rho$}{half-sum of positive co-roots}
$.

We fix an  inner class of real forms for $G$.  The inner class
is determined by a unique algebraic involution $\delta_{0}\nomenclature{$\delta_{0}$}{algebraic involution fixing the pinning}$ of $G$
fixing the pinning (\cite{Adams-Fokko}*{Section 2}). The involution defines the
extended group
$$G^\Gamma=G\rtimes\langle\delta_{0}\rangle.
\nomenclature{$G^{\Gamma}$}{extended group}$$

A \emph{strong involution} of
$G^{\Gamma}$ is an element $\delta \in G^{\Gamma} - G\nomenclature{$\delta$}{strong involution}$ such that
$\delta^{2}$ is central in $G$ and has finite order
(\cite{AvLTV}*{Definition 12.3}).  Two strong
involutions are equivalent if they are $G$-conjugate.
There is a surjective map from (equivalence classes of) strong
involutions to (isomorphism classes of) real forms.  This map takes a strong
involution $\delta$ to the real form
$G(\mathbb{R},\delta)\nomenclature{$G(\mathbb{R},\delta)$}{real form of a strong involution}$ in the inner class with
Cartan involution
$$\theta_\delta=\mathrm{Int}(\delta).$$
This map is bijective
if $G$ is adjoint, but is not injective in general.

There is also a well-known bijection between real forms in a given inner class and
the cohomology set $H^{1}(\mathbb{R}, G/Z(G))$ (\cite{springer}*{12.3.7}).
The domain of the quotient map
\begin{equation}
\label{puremap}
H^{1}(\mathbb{R}, G) \rightarrow H^{1}(\mathbb{R}, G/Z(G))
\end{equation}
defines the set of \emph{pure inner forms} (\cite{vogan_local_langlands}*{Section 2}).   Let $\sigma \in \Gamma$ be the nontrivial element of the Galois group.  For any 1-cocycle $z \in Z^{1}(\mathbb{R}, G)$ one may define a strong involution by
$$z(\sigma) \exp(\uppi i \, \ch\rho) \delta_{0} \in G^{\Gamma}$$
 ($\exp(\uppi i \, \ch\rho)\delta_{0}$ is the \emph{large} involution
in \cite{AVParameters}*{(11f)-(11h)}).  This sends classes in
$H^{1}(\mathbb{R}, G)$ to
$G$-conjugacy classes in $G^{\Gamma}- G$.
The (equivalence classes of) \emph{pure strong involutions} are
defined as the image of this map.
The quasisplit real form is pure in the sense that the trivial
cocycle of $H^{1}(\mathbb{R}, G/Z(G))$ lies in the image of
(\ref{puremap}).
 The pure strong involution canonically
corresponding to the quasisplit pure real form is
$$\delta_{q} = \exp(\uppi i \, \ch\rho)\delta_{0}.
\nomenclature{$\delta_{q}$}{strong involution
corresponding to the quasisplit real form}
$$
Given a strong involution $\delta$ we set $K \nomenclature{$K$}{fixed-point set of a strong involution}$ to be the fixed-point
subgroup $G^{\theta_\delta}$.  The real form $G(\R, \delta)$
contains
\begin{equation}
\label{maxcompact}
K(\R)=G(\R, \delta)\cap K
\end{equation}
as a maximal compact subgroup and is determined by $K$ (\cite{AVParameters}*{(5f)-(5g)}).
By a representation of $G(\R, \delta)$ we usually mean an admissible
$(\g,K)$-module, although we will need admissible group
representations in Section \ref{whitsec}.
A \emph{representation of a strong involution} is a pair
$(\pi,\delta)\nomenclature{$(\pi,\delta)$}{representation of a strong involution}$ in which $\delta$ is a strong involution and $\pi$ is
an admissible $(\g,K)$-module. There is a natural notion of
equivalence  of strong involutions (\cite{Adams-Fokko}*{Definition 6.1}),
and we let $\Pi(G(\mathbb{R}, \delta))$ be the set of
equivalence classes of irreducible representations $(\pi',\delta')$ of strong
involutions in which $\delta'$ is equivalent to $\delta$.
Let
$$\Pi(G/\R) =  \coprod_{\delta} \Pi(G(\mathbb{R}, \delta))
\nomenclature{$\Pi(G/\R)$}{set of equivalence classes of irreducible representations}
$$
be the disjoint union over the (equivalence classes of) pure strong involutions $\delta$.

Let $\ch G$ be the Langlands dual group of $G$ together with a pinning
$$\left({^{\vee} B, {^\vee}H, \{ X_{{^\vee}\alpha}} \}\right).$$
The pinning and the involution $\delta_{0}$ fix an
involution $\ch\delta_{0}$ of $\ch G$ as prescribed in
\cite{AVParameters}*{(12)}. Following this prescription, $\ch\delta_{0}$ is
trivial if and only if $\delta_{0}$ defines the inner class of the split form
of $G$. The group
$$\LG=\ch G\rtimes\langle\ch\delta_{0}\rangle
\nomenclature{$\ch G^{\Gamma}$}{L-group of $G$}
$$
is the L-group of our inner class.

Suppose $\lambda$ is a semisimple element of the Lie algebra $\ch\g$.  After
conjugating by $\ch G$ we may assume $\lambda\in \ch\h$.  Using the
canonical isomorphism $\ch\h\simeq\h^*$ we identify $\lambda$ with an
element of $\h^*$, and hence via the Harish-Chandra homomorphism, with
an infinitesimal character for $G$. This construction only depends on
the $\ch G$-orbit of $\lambda$. We refer to a semisimple element
$\lambda\in\ch \g$,
or a $\ch G$-orbit $\O \subset  \ch\mathfrak{g}$ of semisimple elements, as an \emph{infinitesimal
character} for $G$.  Let $$\Pi(\O,G/\R)\subset \Pi(G/\R)
\nomenclature{$\Pi(\O,G/\R)$}{set of equivalence classes of representations with infinitesimal
character $\O$}
$$ be the
representations (of pure strong involutions) with infinitesimal
character $\O$.
\nomenclature{$\O$}{infinitesimal character}

Let $P\left(\LG\right)\nomenclature{$P\left(\LG\right)$}{set of quasiadmissible homomorphisms}{}$ be the set of \emph{quasiadmissible} homomorphisms
$\phi:W_\R\rightarrow\LG$ (\cite{ABV}*{Definition 5.2}).
Associated to $\phi\in P(\LG)$ is an infinitesimal character
(\cite{ABV}*{Proposition 5.6}).
Let $$P\left(\O,\LG\right)
\nomenclature{$P\left(\O,\LG\right)$}{set of quasiadmissible homomorphisms
with
infinitesimal character $\O$}
$$ be the set of quasiadmissible homomorphisms with
infinitesimal character $\O$.  The group ${^\vee}G$ acts on
$P\left(\O,\LG\right)$ by conjugation.

\subsection{The space $X(\O,\LG)$}
\label{XO}

We make frequent use of the complex variety $X\left(\O,\LG\right)\nomenclature{$X\left(\O,\LG\right)$}{variety of geometric parameters}$ of \emph{geometric parameters} (\cite{ABV}*{Definition 6.9}, \cite{vogan_local_langlands}*{Definition 6.9}).
Here we sketch a  definition based on \cite{ABV}*{Proposition 6.17} and state its main properties.

Write the Weil group $W_{\mathbb{R}}\nomenclature{$W_{\mathbb{R}}=\mathbb{C}^{\times} \coprod j\mathbb{C}^{\times}$}{real Weil group}$ as $\mathbb{C}^{\times} \coprod j\mathbb{C}^{\times}$.  Suppose $\phi\in P\left(\LG\right)$. Define $\lambda,\gamma\in \ch \mathfrak{g}$ by
\begin{subequations}
\renewcommand{\theequation}{\theparentequation)(\alph{equation}}
  \begin{equation}\label{lambday}
  \phi(z)=z^\lambda \overline z^\gamma, \quad  z\in\C^{\times}.
\nomenclature{$\phi$}{Langlands parameter}
\end{equation}
Let $\ch\mathfrak n(\lambda)$ be the sum of the positive integer eigenspaces of $\mathrm{ad}(\lambda)$ on $\ch\g$,
and let $\ch N(\lambda)$ be the connected unipotent subgroup of $\ch G$ with Lie algebra $\ch\mathfrak n(\lambda)$.
Set \nomenclature{$\chG (\lambda)$}{}
\renewcommand{\theequation}{\theparentequation)(\alph{equation}}
\begin{equation}\label{Glambda}
\begin{aligned}
  \chG (\lambda)&=\Cent_{\ch G}(\exp(2\uppi i\lambda))\\
  \ch L(\lambda)&=\Cent_{\ch G}(\lambda)\subset \ch G(\lambda)\\
\end{aligned}
\end{equation}
and let
\begin{equation}
\ch P(\lambda)=\ch L(\lambda)\ch N(\lambda)
\end{equation}
a parabolic subgroup of $\ch G(\lambda)$.
Finally, write \nomenclature{$\ch K_y$}{}
\begin{equation}
  \label{y}
  \begin{aligned}
    y&=\exp(\uppi i\lambda)\phi(j)\\
\ch K_y&=\Cent_{\ch G}(y)
  \end{aligned}
\end{equation}
and let $\ch N_{y}(\lambda)$ be the group generated by
\begin{equation*}
\ch N(\lambda)\cap \mathrm{Int}(y)\big(\ch P(\lambda)\big)\quad\text{ and }\quad \ch P(\lambda)\cap \mathrm{Int}(y)\big(\ch N(\lambda)\big).
\end{equation*}
\end{subequations}
Define an equivalence relation on $P(\LG)$ by
$$
\phi(y,\lambda)\sim\phi'(y',\lambda')\ \text{ if }\
y'=y\ \text{ and }\ \lambda'=n\cdot\lambda\ \text{ for some }\ n\in \ch N_{y}(\lambda)\cap \ch K_y
$$
where the action is by the conjugation action of $\ch G$.
This equivalence relation 
preserves each subset $P(\O,\LG)\subset P(\LG)$.  We let $X(\O,\LG)$ be the set of
equivalence classes
\begin{equation}\label{eq:ABV-variety}
X\left(\O,\LG\right)=P\left(\O,\LG\right)/\sim
\end{equation}
with the quotient topology.
The element $y$ of \eqref{y} is constant on equivalence classes, so for
$p\in X\left(\O,\LG\right)$ we define $y(p)$ accordingly.

The quotient map
\begin{equation}\label{eq:orbitbijection}
  P\left(\O,\LG\right)\rightarrow X\left(\O,\LG\right)
\end{equation}
is ${^\vee}G$-equivariant and  passes to a bijection on the level of $\ch G$-orbits (\cite{ABV}*{Proposition 6.17}).
The space $X(\O,\LG)$ has a richer geometry: while the $\ch G$-orbits on $P(\O,\LG)$ are closed,
there are finitely many $\ch G$-orbits on $X(\O,\LG)$, which have interesting closure relations.
Here is some information about its structure
(see \cite{ABV}*{Section 6}).

Use the notation of \eqref{lambday}-\eqref{y},
and suppose $p\in X(\O, \ch G^{\Gamma})$. Let $y=y(p)$.  Note that
$\ch K_y$ is the fixed-point subgroup of the involution
$\mathrm{Int}(y)$ in $\ch G(\lambda)$.  There is an open and closed,
connected, smooth subvariety $$X_y(\O,\LG)\subset X(\O,\LG)$$ such
that the $\ch G$-orbits on $X_y(\O,\LG)$ are in bijection with the
$\ch K_y$-orbits on the partial flag variety $\ch G(\lambda)/\ch
P(\lambda)$.  Furthermore, this bijection respects the closure
relations between $\ch G$ and $\ch K_y$-orbits
(\cite{ABV}*{Proposition 6.16}).  Thus the geometry of $X(\O,\LG)$ is
closely related to that of a partial flag variety, which is a
classical object in algebraic geometry.

Now suppose $S\subset X(\O,\LG)$ is a $\ch G$-orbit.
If $p\in S$ let $\ch G_p=\Stab_{\ch G}(p)$.
A \emph{pure complete geometric parameter} for $X(\O,\LG)$ is a pair
$(S,\tau_S)\nomenclature{$(S,\tau_S)$}{pure complete geometric parameter}$ where $S$ is a $\ch G$-orbit on $X(\O,\LG)$ and $\tau_S
\nomenclature{$\tau_S$}{irreducible representation of component group}
$
is (an equivalence class of) an irreducible representation of the component group $\ch G_p/(\ch G_p)^0$ (\cite{ABV}*{Definitions 7.1 and 7.6}).
We denote the set of pure complete geometric
parameters for $X(\O, {^\vee}G^{\Gamma})\nomenclature{$\Xi(\O, {^\vee}G^{\Gamma})$}{set of pure complete geometric
parameters}$ by
$\Xi(\O,{^\vee}G^{\Gamma})$.

A special case of the local Langlands correspondence as stated in \cite{ABV}*{Theorem 10.11} is a bijection
\begin{equation}
\label{localLanglandspure}
\Pi\left(\O,G/\R\right)\longleftrightarrow \Xi\left(\O,\LG\right)
\end{equation}
between representations of pure strong involutions and pure complete geometric parameters.
Recall from the previous section that the left-hand side of (\ref{localLanglandspure})
contains the subset $\Pi(\O,G(\R,\delta_q))$.

\subsection{Extended groups for  $G$ and $\ch G$}
\label{extended}

We specialize the results of the previous section to the groups
$\mathrm{GL}_{N}$, $\mathrm{Sp}_{N}$ and $\mathrm{SO}_{N}$, providing
further details.

For the group $\mathrm{GL}_N$ we fix the usual pinning (\ref{pinning}) in which
$B\nomenclature{$B$}{upper-triangular Borel subgroup}$ is the upper-triangular subgroup, $H\nomenclature{$H$}{diagonal subgroup}$ is the diagonal subgroup,
and $X_{\alpha}$ is a matrix with $1$ in the entry corresponding to
$\alpha$ and zeroes elsewhere.  We fix the \emph{split} inner class
for $\mathrm{GL}_{N}$.  The split inner class consists of the split
group $\GL_N(\R)$, and, if $N$ is even, also the quaternionic form
$\GL_{N/2}({\mathbb H})$.

There are two algebraic involutions of $\mathrm{GL}_{N}$ which fix
the pinning: the identity, and $\vartheta$ (\ref{varthetadef}).  It is
a coincidence that the strong involution corresponding to the split
inner class is $\vartheta$.  Indeed,
$$
G^{\vartheta}=
\begin{cases}
  \mathrm{O}_{N},& N \mbox{ odd}\\
  \mathrm{Sp}_{N},&N \mbox{ even}
\end{cases}
$$
which match the respective maximal compact subgroups of
$\mathrm{GL}_{N}(\mathbb{R})$ and $\mathrm{GL}_{N/2}(\mathbb{H})$
((\ref{maxcompact}), \cite{beyond}*{(1.123)}).
Thus, we define
\begin{equation}
\label{GLnGamma}
\GL_N^\Gamma=\GL_N\rtimes\langle\delta_{0}\rangle
\end{equation}
where  $\delta_{0}^2=1$ and $\delta_{0}$ acts (by chance!) as $\vartheta$ on $\mathrm{GL}_{N}$.

\begin{lem}
\label{GLn}
\begin{enumerate}
\item There is a unique conjugacy class of strong involutions in $\mathrm{GL}_{N} \rtimes \langle \delta_{0} \rangle$ which maps to the (isomorphism class of the) split group $\GL_N(\R)$.

\item The strong involutions $\delta$ in this conjugacy class are characterized by
$$\delta^2=(-1)^{N+1} = \exp(2\uppi i\, \ch\rho).$$

\item If $N$ is odd this is the unique conjugacy class of strong involutions.

\item If $N$ is even there is exactly one other class, whose elements square to $1$.  This other class maps to the real form $\GL_{N/2}(\H)$.

\item In both cases
there is only  only one class of pure strong involutions, and it corresponds to the split form $\mathrm{GL}_{N}(\mathbb{R})$ so that
\begin{equation}
\label{glnpure}
\Pi\left(\O,\GL_N/\R\right)=\Pi\left(\O,\GL_N(\R)\right).
\end{equation}
\end{enumerate}

\end{lem}

\begin{proof}

  By \cite{Adams-Fokko}*{Proposition 12.19 (2)}  the conjugacy classes of strong involutions are parameterized by the $H$-conjugacy classes of the elements in
  $$
  \left\{t\delta_{0} \in H \rtimes \langle \delta_{0} \rangle : (t\delta_{0})^2\in Z(G)\right\}$$
  modulo the action of a Weyl group.
  For $\GL_N$ it is straightforward to compute that  up to conjugacy the representatives in this set are $\delta_{0}$, with $\delta_{0}^2=1$, and, when $N$ is even, $\delta=\exp( \uppi i \ch \rho)\delta_{0}$, with $\delta^2=-1$. The real forms associated to these strong involutions are as stated  (\cite{snowbird}*{Examples 7.6 and 7.8}).
We leave the identity  $\exp(2\uppi i \,\ch\rho)=(-1)^{N+1}$ as an exercise.  All that remains to be proven is that when $N$ is even, the real form $\mathrm{GL}_{N}(\mathbb{H})$ is not pure. Looking back to (\ref{puremap}), this follows from the fact that $H^{1}(\mathbb{R}, \mathrm{GL}_{N}) = \{1\}$ corresponds to a single real form, which is (quasi)split.
\end{proof}
If $\delta$ is a strong involution of any group $G$, we say $\delta$ has \emph{infinitesimal cocharacter} $g \in \h\nomenclature{$g$}{infinitesimal cocharacter}$  if $$\delta^2=\exp(2\uppi ig).$$  Lemma \ref{GLn} tells us that the pure strong involutions of $\mathrm{GL}_{N}$ are exactly those with infinitesimal cocharacter ${^\vee}\rho$.
Let $\delta_{q} = \delta_{0}$ when $N$ is odd and $\delta_{q} = \exp(\uppi i\, \ch \rho) \delta_{0}$ when $N$ is even.  According to Lemma \ref{GLn}, the strong involutions $\delta$ in the $\mathrm{GL}_{N}$-conjugacy class of $\delta_{q}$ form the set of pure involutions and these are the only strong involutions for which $\mathrm{GL}(\mathbb{R}, \delta) \cong \mathrm{GL}_{N}(\mathbb{R})$.

Since $\GL_N(\R)$ is split, the L-group is

\begin{equation}
\label{deltagln}
  \LGL=\ch\GL_N \times\langle\ch\delta_{0}\rangle\simeq \chGL_N \times\Z/2\Z.
\end{equation}
We write $\ch\GL_{N}$ instead of $\mathrm{GL}_{N}$ just to emphasize that the group is on the ``dual side".

We also need the extended group
$$
\GL_N \rtimes\langle\vartheta\rangle.
$$
Although this is isomorphic to
$\GL_N^\Gamma=\GL_N\rtimes\langle\delta_{0}\rangle$, it plays a very
different role.  The group $\GL_N^\Gamma$ plays a role in the ordinary
(untwisted) Langlands correspondence by  carrying  strong involutions and thereby information about real forms.  By contrast, the group $\GL\rtimes\langle\vartheta\rangle$ is the central object in the theory of twisted characters.  We use this notation to distinguish the two roles.  With this in mind, we have the group
\begin{equation}\label{eq:GLDelta}
\GL_N^\Gamma\rtimes\langle \vartheta\rangle=\langle \GL_N, \delta_{0},\vartheta\rangle
$$
in which $\vartheta$ and $\delta_{0}$ commute.  Similarly, we define
$$
{^\vee}\GL_N^\Gamma\rtimes\langle \vartheta\rangle=\langle {^\vee}\GL_N,\ch\delta_{0},\vartheta\rangle
\end{equation}
in which $\vartheta$ and ${^\vee} \delta_{0}$ commute.
See Section \ref{twistendsec} for a discussion of the twisted endoscopic groups for $(\ch\GL_N^\Gamma,\vartheta)$.

Next we consider extended groups for $\mathrm{Sp}_{N}$ and
$\mathrm{SO}_{N}$.   As in \cite{Arthur}*{Section 1.2}, we adopt the
convention of expressing these groups in such a way that their
upper-triangular subgroups are Borel subgroups.  With this convention,
their diagonal subgroups are maximal tori.  When there is no confusion
with the setting of $\mathrm{GL}_{N}$, we will also denote these Borel
subgroups by $B$ and maximal tori by $H$.  In fact, we will abusively
imitate the notation for $\mathrm{GL}_{N}$ when the setting is clear.
We arbitrarily fix a set of simple root vectors $\{X_{\alpha}\}$ for each simple root in $R(B,H)$.

Suppose first that $G = \mathrm{Sp}_{2n}$ or  $G =
\mathrm{SO}_{2n+1}$.  Each of these groups has one inner class, which
is the inner class of the split form.  This allows us to choose
$\delta_{0}$ to act trivially on these groups and set
$$G^{\Gamma} = G \times \langle \delta_{0} \rangle$$
where $\delta_{0}^{2} = 1$.
Define $\delta_{q} = \delta_{0}$, so that $G(\mathbb{R}, \delta_{q})$
is the split real form.
The dual groups ${^\vee}G = {^\vee}\mathrm{Sp}_{2n} =
\mathrm{SO}_{2n+1}$ and ${^\vee}G = {^\vee}\mathrm{SO}_{2n+1} =
\mathrm{Sp}_{2n}$ have Borel subgroups and maximal tori as earlier.
The L-group of $G$ corresponding to the split inner class is
$$
{^\vee}G^{\Gamma} = {^\vee}G \times \langle {^\vee}\delta_{0}
\rangle  \cong {^\vee}G \times \mathbb{Z}/2 \mathbb{Z}.
$$

Finally, take $G=\mathrm{SO}_{2n}$.  This group has two inner classes:
one for the split form and the other for the quasisplit form
$\mathrm{SO}(n+1,n-1)$, which is not split.  The following definitions follow from \cite{AVParameters}*{(12c)} and \cite{Bourbaki}*{ Chapter VI \S 4.8 XI}.

If the inner class contains the split form then $\ch \delta_{0}$ acts trivially and
$$
{^\vee}\mathrm{SO}_{2n}^{\Gamma} = \mathrm{SO}_{2n} \times \langle
{^\vee}\delta_{0} \rangle.$$
If, in addition, $n$ is even then $\delta_{0}$ acts trivially and
$\mathrm{SO}_{2n}^{\Gamma} = \mathrm{SO}_{2n} \times \langle \delta_{0} \rangle.$
On the on the other hand if $n$ is odd then $\delta_{0}$ acts by conjugation by  an element in $\mathrm{O}_{2n}- \mathrm{SO}_{2n}$ which preserves the pinning.  In this case $\mathrm{SO}_{2n}^{\Gamma}$ is a nontrivial semidirect product $\mathrm{SO}_{2n} \rtimes \langle \delta_{0} \rangle$.

If we fix the inner class to be
that of $\mathrm{SO}(n+1,n-1)$ then  the L-group is the semidirect product
$$
{^\vee}\mathrm{SO}_{2n}^{\Gamma} = \mathrm{SO}_{2n} \rtimes \langle \ch\delta_{0}\rangle
$$
in which ${^\vee}\delta_{0}$ acts by conjugation by  an element in $\mathrm{O}_{2n}- \mathrm{SO}_{2n}$ which preserves the pinning.  If $n$ here is even then $\mathrm{SO}_{2n}^{\Gamma} = \mathrm{SO}_{2n} \rtimes \langle \delta_{0} \rangle$ where $\delta_{0}$ acts in the same way as $\ch \delta_{0}$.  On the other hand, if $n$ is odd then $\mathrm{SO}_{2n}^{\Gamma} = \mathrm{SO}_{2n} \times \langle \delta_{0} \rangle.$

For either of the two inner classes of $\mathrm{SO}_{2n}$ the strong
involution $\delta_{q} = \exp(\uppi i {^\vee}\rho) \delta_{0}$
corresponds to a quasisplit real form $G(\mathbb{R}, \delta_{q})$
\cite{AVParameters}*{(11f)}.

\subsection{Atlas parameters for $\mathrm{GL}_{N}$}
\label{atlasparam}

For our application we use a formulation of the
local Langlands correspondence for $\GL_N(\R)$ which is well suited to Vogan duality
(see Section \ref{bbvd}).
The main references for this section are \cite{Adams-Fokko} and
\cite{AVParameters}*{Section 3}.

We start by working in the context of the extended group \eqref{GLnGamma}:
$\GL_N^\Gamma=\GL_N \rtimes \langle\delta_{0}\rangle$.
Let $\ch\rho$ be the half-sum of the positive coroots for $\GL_N$.
Following \cite{AVParameters}*{Section 3} we set
$$
\Xcal_{\ch\rho}=\left\{\delta \in \Norm_{\GL_N\delta_{0}}(H)\mid
\delta^2=\exp(2\uppi i\, \ch\rho)\right\}/H
\nomenclature{$\Xcal_{\ch\rho}$}{
set of conjugacy classes of strong involutions
 with infinitesimal
cocharacter ${^\vee}\rho$}
$$
where the quotient is by the conjugation action of $H$.  This is a set
of $H$-conjugacy classes of strong involutions with infinitesimal
cocharacter ${^\vee}\rho$.  By Lemma \ref{GLn}, these strong
involutions are all pure and correspond to the split form
$\mathrm{GL}_{N}(\mathbb{R})$.

Now we fix a $\vartheta$-fixed, regular, integrally dominant
element $\lambda \in \ch \h$ for $\GL_N$. This means
\begin{equation}\label{regintdom}
\begin{aligned}
     \vartheta(\lambda) &= \lambda\\
   \langle \lambda, {^\vee}\alpha
  \rangle &\neq  0, \quad \alpha\in R(\GL_N,H)\\
    \langle \lambda, {^\vee} \alpha \rangle &\notin \{ -1, -2, -3,
   \ldots \}, \ \alpha \in R^{+}(\mathrm{GL}_{N},H).
\end{aligned}
\end{equation}
This will be the infinitesimal character of our representations of $\GL_N(\R)$.
The assumption of integral dominance is
harmless
(\cite{AVParameters}*{Lemma 4.1}). We
shall remove the regularity assumption at the beginning of Section
\ref{equalapacketsing}.

The action of $\delta_{0}$ induces an action on the Weyl group
$W(\mathrm{GL}_{N},H)
\nomenclature{$W(\mathrm{GL}_{N},H)$}{Weyl group}
$.
Consider the set
\begin{equation}
\label{twistinv}
\left\{w\in W(\GL_N,H) : w\, \delta_{0}(w)=1\right\}.
\end{equation}
If $x\in\Xcal_{\ch \rho}$ then the action (by conjugation) of $x$ on
$H$ is equal to $w\delta_0$ for some $w$ in the set \eqref{twistinv}.
Define $p(x)=w$ accordingly. The map $p$ is surjective. Let
$\Xcal_{\ch\rho}^w$ be the fibre of $p$ over $w$ so that
\begin{equation}
\label{e:p}
\Xcal_{\ch\rho}^w=\left\{x\in\Xcal_{\ch\rho} : xhx\inv =w\delta_0 \cdot h, \ \mbox{ for all }  h\in H \right\}.
\nomenclature{$\Xcal_{\ch\rho}^w$}{}
\end{equation}

On the dual side we have an analogous set
in which the infinitesimal cocharacter ${^\vee}\rho$
is replaced by an infinitesimal character $\lambda$, namely
$$
\ch\Xcal_{\lambda}=\left\{\ch\delta \in \Norm_{\ch\GL_N \, \ch\delta_{0}}(\ch
H)\mid \ch\delta^2=\exp(2\uppi i\lambda)\right\}/\, \ch H.
\nomenclature{$\ch\Xcal_{\lambda}$}{}
$$
Recall that $\GL_N(\R)$ is split, so $\ch\delta_0$ acts trivially on $\ch \mathrm{GL}_N$,
and we can safely identify this set with
$$
\left\{\ch\delta \in \Norm_{\ch\GL_N}(\ch H)\mid \ch\delta^2=\exp(2\uppi
i\lambda)\right\}/\, \ch H.
$$
Since $\ch\delta_0$ acts trivially,
the analogue of \eqref{twistinv} is $$\left\{w\in W\mid w^2=1\right\}.$$
Let $\ch\Xcal_\lambda^w
\nomenclature{$\ch\Xcal_\lambda^w$}{}
$ be the analogue of \eqref{e:p}.

It is easily verified that
\begin{equation}
\label{weylsquared}
w\, \delta_{0}(w) = ww_{0}ww_{0}^{-1} = (ww_{0})^{2}, \quad w \in
W(\mathrm{GL}_{N},H)
\end{equation}
where $w_{0} \in W(\mathrm{GL}_{N},H)$ is the long Weyl group element.
It follows that
$$w \mapsto ww_{0}$$
defines a bijection from (\ref{twistinv}) to $\{ w \in
W(\mathrm{GL}_{N},H): w^{2} =1\}$.
This map allows us to pair any set $\mathcal{X}_{{^\vee}\rho}^{w}$ with the set ${^\vee}\mathcal{X}_{\lambda}^{ww_{0}}$.

The next result follows from \cite{Adams-Fokko}, \cite{ABV} and
\cite{AVParameters}*{Theorem 3.11}. We give the proof, which is much
simpler in the case of $\GL_N$ than for other reductive groups.

\begin{lem}
  \label{XXXi}
  There is a canonical bijection
  $$\coprod_{\{w: w\delta_0(w)=1\}} \mathcal{X}_{{^\vee}\rho}^{w} \times
  {^\vee}\mathcal{X}_{\lambda}^{ww_{0}}\longleftrightarrow\Xi\left(\O,\ch\GL_N^{\Gamma}\right).
  $$
\end{lem}

\begin{proof}
  First of all $|\Xcal_{\ch\rho}^w|=1$  for all $w$. This follows from
  \cite{Adams-Fokko}*{Proposition 12.19(5)}: the dual inner class
  is the  equal rank inner class, consisting of (products of) unitary groups $\mathrm{U}(p,q)$,
  and it is well known that the Cartan subgroups of $\mathrm{U}(p,q)$ are all connected.
  This is equivalent to the fact that all $L$-packets for $\GL_N(\R)$ are singletons.

  So the lemma comes down to the statement that there is a bijection
  $$
  \coprod_{w\delta_0(w)=1} \ch\Xcal^{ww_{0}}_\lambda\longleftrightarrow \Xi\left(\O,\ch\GL_N^{\Gamma}\right).
  $$
  Recall the right-hand side is the set of pure complete geometric parameters $(S,\tau)$
  where $S$ is a $\ch \mathrm{GL}_N$-orbit in  $X\left(\O, \ch \mathrm{GL}_{N}^{\Gamma} \right)$ and $\tau$ is an irreducible representation of the component
  group of the centralizer of a point in $X\left(\O,\ch \mathrm{GL}_N^{\Gamma}\right)$. Since $\ch\delta_0$ acts trivially on $\ch \mathrm{GL}_N$,
  these centralizers are products of general linear groups, and are hence connected. Therefore we are further reduced to showing
  $$
  \coprod_{w\delta_0(w)=1} \ch\Xcal^{ww_{0}}_\lambda\longleftrightarrow
  X\left(\O,\ch \mathrm{GL}_N^{\Gamma}\right)/\ch \mathrm{GL}_N.
  $$

  Suppose $y\in \ch\Xcal^{ww_{0}}_\lambda$. This means that $y\in\Norm_{\ch \mathrm{GL}_N}(\ch H)$ (we can ignore the extension),
  $y$ maps to $ww_{0}\in W(\mathrm{GL}_N,H)$, and $y^2=\exp(2\uppi i\lambda)$. Define
  $\phi: W_\R\rightarrow \ch \mathrm{GL}_N^{\Gamma}$ by
  $$
  \begin{aligned}
    \phi(z)&=z^\lambda \overline z^{\mathrm{Ad}(y)(\lambda)}, \quad z \in \mathbb{C}^{\times}\\
    \phi(j)&=\exp(-\uppi i\lambda)y
  \end{aligned}
  $$
  (compare \eqref{lambday} and (d)).
  It is straightforward to see that $\phi$ is a quasiadmissible homomorphism
  (see the end of Section \ref{extended}), and only a little more work to show
  that it induces the bijection indicated. See \cite{Adams-Fokko}*{Proposition 9.4}.

\end{proof}
Together with  \eqref{localLanglandspure} this gives
\begin{thm}
  \label{paramsGLn}
Let $\O$ be the $\ch \mathrm{GL}_N$-orbit of $\lambda$.
  There are canonical bijections:
  $$
\coprod_{\{ w : w \delta_{0}(w) = 1\}}
\mathcal{X}^{w}_{{^\vee}\rho} \times {^\vee}\mathcal{X}^{ww_{0}}_{\lambda} \longleftrightarrow
\Xi\left(\O,\ch \mathrm{GL}_N^{\Gamma}\right) \longleftrightarrow \Pi\left(\O,\GL_N(\R)\right).
$$
\end{thm}
As in \cite{AVParameters}*{Theorem 3.11}
 the bijection of  Theorem \ref{paramsGLn} is written as
\begin{equation}
\label{avbij}
\Xcal_{\ch\rho}^w\times \ch\Xcal_{\lambda}^{w w_0}\ni(x,y) \mapsto J(x,y,\lambda)
\end{equation}
We call the pair $(x,y)\nomenclature{$(x,y)$}{Atlas parameter}$ on the left the \emph{Atlas parameter} of the
irreducible representation $J(x,y, \lambda)\nomenclature{$J(x,y,\lambda)$}{Langlands quotient}$.
By Lemma \ref{XXXi}, the Atlas parameter $(x,y)$ is equivalent to a complete geometric parameter $\xi \in \Xi(\O, \mathrm{GL}_{N}^{\Gamma})$, and accordingly we define
$$\pi(\xi) = J(x,y, \lambda). \nomenclature{$\pi(\xi)$}{(equivalence class of the) irreducible Harish-Chandra module (representation) associated to $\xi$}$$
The representation $\pi(\xi)$ is the Langlands quotient of a standard representation which we denote by $M(\xi)$ or $M(x,y)$.  \nomenclature{$M(\xi), \ M(x,y)$}{(equivalence class of the) standard module (representation) associated to $\xi$}

\subsection{Twisted Atlas parameters for $\GL_N$}
\label{extrepsec}

Our next task is to describe the generalization of
Theorem \ref{paramsGLn}
to the $\vartheta$-twisted setting. This involves certain
irreducible representations of the extended group
$\GL_N(\R)\rtimes\langle\vartheta\rangle$.
We specialize the results of \cite{AVParameters}*{Sections 3-5}
to this case. We are fortunate that  some of the more
complicated issues that arise in \cite{AVParameters}
do not occur for $\mathrm{GL}_{N}$.

We continue with the hypotheses of (\ref{regintdom}). Recall that both
${^\vee}\rho$
and $\lambda$ are fixed by $\vartheta$.
By Clifford theory, an irreducible representation of
$\GL_N(\R)\rtimes\langle\vartheta\rangle$ restricted to $\GL_N(\R)$
is either an irreducible $\vartheta$-fixed representation, or the direct
sum of two irreducible representations which are exchanged  by the action of $\vartheta$.
We only need representations of the first type.

It is a lengthy but straightforward task to show that the map
(\ref{avbij}) is $\vartheta$-equivariant (\emph{cf.}
\cite{Christie-Mezo}*{Theorem 4.1}).  Therefore $J(x,y, \lambda)$ is
$\vartheta$-stable if and only if
$(x,y)\in \mathcal{X}^{w}_{{^\vee}\rho} \times
{^\vee}\mathcal{X}^{{ww_{0}}}_{\lambda} $ is fixed by $\vartheta$.
Let
$$\Pi(\O, \mathrm{GL}_{N}(\mathbb{R}))^\vartheta \subset \Pi(\O, \mathrm{GL}_{N}(\mathbb{R}))
\nomenclature{$\Pi(\O, \mathrm{GL}_{N}(\mathbb{R}))^\vartheta$}{
$\vartheta$-fixed equivalence classes of representations
}
$$
be the subset of $\vartheta$-fixed irreducible representations and set
\begin{equation*}
  \label{W2}
W(\delta_{0},\vartheta)=\{w\in W(\mathrm{GL}_{N},H) \mid w\delta_{0}(w)=1, w=\vartheta(w)\}
\end{equation*}
(\emph{cf.} (\ref{twistinv})).  By the $\vartheta$-equivariance,
Theorem \ref{paramsGLn} restricts to these sets and we obtain
\begin{cor}
\label{twistclass}
Suppose $\lambda$ satisfies the hypotheses of (\ref{regintdom}) and let $\O$ be its $\ch \mathrm{GL}_{N}$-orbit.  Then there is a canonical bijection
$$
\coprod_{\{w\in W(\delta_{0},\vartheta)\}}\Xcal_{\ch\rho}^w\times \ch\Xcal^{ww_0}_\lambda
\longleftrightarrow
\Pi\left(\O,\GL_N(\R)\right)^\vartheta
$$
written $(x,y)\mapsto J(x,y,\lambda)$.
\end{cor}

We now introduce the \emph{extended parameters} of
\cite{AVParameters}*{Sections 3-5},
and summarize the facts that we need.
Fix $w\in W(\delta_{0},\vartheta)$.
An \emph{extended parameter} for $w$ is a set
\begin{equation}
\label{quad}
E  = (\uplambda,\uptau,\ell,t), \quad \uplambda,\uptau\in X^*(H),
\ \ell,t \in X_*(H)
\nomenclature{$(\uplambda,\uptau,\ell,t)$}{extended parameter}
\end{equation}
satisfying certain conditions depending on $w$ (see
\cite{AVParameters}*{Definition 5.4}).\footnote{Warning!  The symbols
  $\uplambda$ and $\uptau$ here are not to be confused with symbols
  $\lambda$ and $\tau$ appearing elsewhere.  Note the slight
  difference in font.  We have chosen to use $\uplambda$ and $\uptau$
  for ease of comparison with \cite{AVParameters}.}  There is a
surjective map
\begin{equation}
\label{extsurj}
E \mapsto (x(E), y(E))
\end{equation}
 taking extended parameters for $w$ to $\XX$.  This map only
depends on $\uplambda$ and $\ell$.
In addition,
$$
J(x(E),y(E),\lambda)\in\Pi(\O,\GL_N(\R))^\vartheta,
$$
and every $\vartheta$-fixed irreducible representation arises this way.
The remaining parameters $\uptau$ and $t$ in $E$ define an irreducible representation $J(E,\lambda)$ of
$\GL_N(\R)\rtimes\langle\vartheta\rangle$ satisfying
$$
J(E,\lambda)|_{\GL_N(\R)}=J(x(E),y(E),\lambda).
$$
The representation $J(x(E),y(E), \lambda)$ is determined by a
quasicharacter of a  Cartan subgroup of $\mathrm{GL}_{N}(\mathbb{R})$.
The representation $J(E,\lambda)$ is determined by the semidirect
product of this Cartan subgroup with an element $h \vartheta \in
\mathrm{GL}_{N} \rtimes \vartheta$ (\cite{AVParameters}*{(24e)}) and a
  choice of extension of the
quasicharacter to the semidirect product.  The value of the extended
quasicharacter on the element $h\vartheta$ depends on a choice of sign
\cite{AVParameters}*{Definition 5.2}, and the square root of this sign is
given by
\begin{equation}
\label{z}
z(E)=i^{\langle \uptau,(1+w)t\rangle}(-1)^{\langle \uplambda,t\rangle}.
\end{equation}
The preceding discussion is a specialization of a general framework to $\mathrm{GL}_{N}(\mathbb{R}) \rtimes \langle \vartheta \rangle$.  One of the special properties of $\GL_N(\R)$ is that the preimage of any $(x,y) \in \mathcal{X}_{{^\vee}\rho}^{w} \times {^\vee}\mathcal{X}_{\lambda}^{ww_{0}}$ under (\ref{extsurj}) has a \emph{preferred extended parameter} of the form
$$(\uplambda,\uptau,0,0).$$
This comes down to the fact that $X_{{^\vee}\rho}^{w}$ is a singleton (see the proof of Lemma \ref{XXXi}).
By taking $t=0$ we see $z(\uplambda, \uptau,0,0)=1$, and this amounts to taking the aforementioned semidirect product of the Cartan subgroup with $h\vartheta = \vartheta$, and setting the value of the extended quasicharacter at $\vartheta$ equal to $1$.  In this way, the preferred extended parameter defines a canonical extension
\begin{equation}
\label{canext}
J(x,y, \lambda)^{+} = J((\uplambda, \uptau,0,0), \lambda)
\nomenclature{$J(x,y, \lambda)^{+}$}{Atlas extension}
\end{equation}
of $J(x,y,\lambda)$ to $\mathrm{GL}_{N}(\mathbb{R}) \rtimes \langle \vartheta \rangle$.
We call this extension the \emph{Atlas extension} of $J(x,y, \lambda)$.

Going back to Theorem \ref{paramsGLn} and Corollary \ref{twistclass},
we may formulate the result as follows.
\begin{cor}
  \label{twistclass2}
  There is a natural bijection of $\vartheta$-fixed sets
  $$
  \coprod_{\{w\in W(\delta_{0},\vartheta)\}}\Xcal_{\ch\rho}^w\times \ch\Xcal^{ww_0}_\lambda
 \longleftrightarrow \Xi(\O,\LGL)^\vartheta \longleftrightarrow \Pi(\O,\GL_N(\R))^\vartheta
 \nomenclature{$\Xi(\O,\LGL)^\vartheta $}{$\vartheta$-fixed complete geometric parameters}
  $$
Furthermore, if $\xi \in \Xi(\O,\LGL)^\vartheta$  is identified with $(x,y)$ under the first bijection then there is a canonical representation
$$
\pi(\xi)^+ = J(x,y,\lambda)^{+}
\nomenclature{$\pi(\xi)^+$}{(equivalence class of) the Atlas extension of $\pi(\xi)$}
$$
extending $\pi(\xi)$ to $\GL_N(\R)\rtimes\langle\vartheta\rangle$.
\end{cor}
The irreducible representation $\pi(\xi)^{+}$ is defined as the unique (Langlands) quotient of a representation
$M(\xi)^+
\nomenclature{$M(\xi)^+$}{(equivalence class of) the Atlas extension of $M(\xi)$}$ such that $M(\xi)^{+}_{|\mathrm{GL}_{N}(\mathbb{R})} = M(\xi)$.  We call $\pi(\xi)^{+}$ and $M(\xi)^{+}$ the \emph{Atlas extensions} of $\pi(\xi)$ and $M(\xi)$ respectively.

\subsection{Grothendieck groups of characters}
\label{grothchar}

The setting for studying characters of reductive groups is the
Grothendieck group of representations with a given infinitesimal character.
There is  a corresponding notion in the twisted setting.
In this section we establish notation for the objects that we need.

Fix a semisimple orbit $\O\subset\ch\g$, which we view as an
infinitesimal character for $G$ (cf. Section \ref{extgroups}).
Recall $\Pi(\O,G/\R)$ is the set of equivalence classes of
representations $(\pi, \delta)$ of pure strong involutions.  We
define $K\Pi(\O,G/\R)
\nomenclature{$K\Pi(\O,G/\R)$}{Grothendieck group of
representations of pure strong involutions}
$ to be the Grothendieck group of
representations of pure strong involutions with infinitesimal
character $\O$ (see \cite{ABV}*{(15.5)-(15.6)}). We identify this with
the $\Z$-span of distribution characters of the irreducible
representations in $\Pi(\O,G/\R)$. We refer to elements of this space as virtual characters.

When $G$ is $\Sp_N$ or $\SO_N$ we only need the subspace of stable
characters, and only for the quasisplit form. So we define
$$
K\Pi(\O,G(\R,\delta_q))^\stable\subset K\Pi(\O,G(\R,\delta_q))
\nomenclature{$K\Pi(\O,G(\R,\delta_q))^\stable$}{stable virtual characters}
$$
to be the subspace spanned by the (strongly) stable  virtual characters.
If we identify virtual characters with functions on $G(\mathbb{R},\delta_{q})$
these are the virtual characters $\eta$ which satisfy $\eta(g)=\eta(g')$
whenever  strongly regular semisimple elements $g,g' \in G(\R,\delta_q)$ are $G$-conjugate.
See \cite{shelstad}*{Section 5} or \cite{ABV}*{Definition 18.2}.

\subsection{Grothendieck groups of twisted characters}
\label{grtwisted}

Here, we consider the split inner class of $\GL_N$, equipped with the involution $\vartheta$.
Recall   in this case $\Pi(\O,\GL_N/\R)=\Pi(\O,\GL_N(\R))$
(\emph{cf.} \eqref{glnpure}), and so
$$
K\Pi(\mathrm{GL}_{N}/\mathbb{R}) = K \Pi(\mathrm{GL}_{N}(\mathbb{R})).
$$
We define
$$
K\Pi(\O,\GL_N(\R))^\vartheta\subset K\Pi(\O,\GL_N(\R))
\nomenclature{$K\Pi(\O,\GL_N(\R))^\vartheta$}{}
$$
to be the submodule spanned by $\Pi(\O, \mathrm{GL}_{N}(\mathbb{R}))^\vartheta$.  This is not the Grothendieck group of $\vartheta$-stable representations of $\mathrm{GL}_{N}(\mathbb{R})$, but we retain the ``$K$" to help align the object with its ambient Grothendieck group.
On the other hand we let
\begin{equation}
\label{KPiOGLext}
K\Pi(\O,\GLext)
\end{equation}
be the Grothendieck group of admissible representations of $\mathrm{GL}_{N}(\mathbb{R}) \rtimes \langle \vartheta \rangle$ with infinitesimal character $\O$.

We now discuss the $\mathbb{Z}$-module of twisted characters of $\GL_N(\mathbb{R})$.
An irreducible character in $K\Pi(\O,\mathrm{GL}_{N}(\mathbb{R}) \rtimes \langle \vartheta \rangle)$
is a distribution
$$f \mapsto \mathrm{Tr} \int_{\mathrm{GL}_{N}(\mathbb{R})} f(x) \pi(x)
\, dx + \mathrm{Tr} \int_{\mathrm{GL}_{N}(\mathbb{R})} f(x\vartheta)
\pi(x) \pi(\vartheta)  \, dx,$$
where $f \in C_{c}^{\infty}(\mathrm{GL}_{N}
  (\mathbb{R}) \rtimes \langle \vartheta \rangle)$ and $\pi$ is an
irreducible representation of $\mathrm{GL}_{N}(\mathbb{R}) \rtimes
\langle \vartheta \rangle$.  The restriction of such a distribution
character to the non-identity component $\mathrm{GL}_{N}(\mathbb{R}) \rtimes
\vartheta$ has the form
\begin{equation}
  \label{twistchar}
  f \mapsto  \mathrm{Tr} \int_{\mathrm{GL}_{N}(\mathbb{R})} f(x\vartheta))
\pi(x) \pi(\vartheta)  \, dx, \quad f \in  C_{c}^{\infty}(\mathrm{GL}_{N}
  (\mathbb{R})\rtimes \vartheta).
\end{equation}
When the resulting restricted distribution is non-zero, we define it
to be an irreducible \emph{twisted character} of
$\mathrm{GL}_{N}(\mathbb{R}) \rtimes \vartheta$.  We define
$$K\Pi(\O, \mathrm{GL}_{N}(\mathbb{R}), \vartheta)$$
to be the
$\mathbb{Z}$-module generated by the irreducible twisted
characters of $\mathrm{GL}_{N}(\mathbb{R}) \rtimes \vartheta$ of
infinitesimal character $\O$.
\nomenclature{$K\Pi(\O, \mathrm{GL}_{N}(\mathbb{R}), \vartheta)$}{virtual twisted characters}

As noted in Section \ref{extrepsec}, an irreducible representation of $\GLext$
restricts either to an irreducible $\vartheta$-fixed representation of $\GL_N(\R)$,
or to a direct sum $\pi\oplus (\pi \circ \vartheta)$ of inequivalent irreducible representations. In the second case the twisted character is $0$, so we only need to consider the first case.  The first case describes the irreducible representations in $K\Pi(\O,\GL_N(\R))^\vartheta$.
If $\pi\in\Pi(\O,\GL_N(\R))^\vartheta$ then
it has two extensions $\pi^{\pm}$ to $\GLext$, satisfying

\begin{equation}
  \label{minuspi}
\pi^-(\vartheta)=-\pi^+(\vartheta).
\nomenclature{$\pi^{-}$}{}
\end{equation}
Consequently the twisted characters of $\pi^{\pm}$ agree up to sign.  If we set $U_{2} = \{ \pm 1\}\nomenclature{$U_2=\{\pm 1\}$}{}$ then it follows that the homomorphism
$$K \Pi(\O, \mathrm{GL}_{N}(\mathbb{R}))^\vartheta
\otimes _{\mathbb{Z}} \mathbb{Z}[U_{2}] \rightarrow K
\Pi(\O,\mathrm{GL}_{N}(\mathbb{R}), \vartheta),$$
which restricts the distribution character of $\pi(\xi)^{+}$ to the
non-identity component, is surjective.  By (\ref{minuspi}), the homomorphism
passes to an isomorphism
\begin{subequations}
\renewcommand{\theequation}{\theparentequation)(\alph{equation}}
\label{twistgroth1}
\begin{equation}
K\Pi(\O, \mathrm{GL}_{N}(\mathbb{R}),
\vartheta)\cong K \Pi(\O, \mathrm{GL}_{N}(\mathbb{R}))^\vartheta
\otimes _{\mathbb{Z}} \mathbb{Z}[U_{2}]/ \langle (\pi \otimes 1 ) +
  (\pi\otimes -1)\rangle
\end{equation}
where the quotient runs over $\pi\in \Pi(\O,\GL_N(\R))^\vartheta$.
The map taking
$\pi(\xi) \in \Pi(\O,
\mathrm{GL}_{N}(\mathbb{R}))^\vartheta$ to the twisted character
$$  f \mapsto  \mathrm{Tr} \int_{\mathrm{GL}_{N}(\mathbb{R})}
f(x\vartheta) \,
\pi(\xi)(x) \,\pi(\xi)^{+}(\vartheta)  \, dx, \quad f \in
C_{c}^{\infty}(\mathrm{GL}_{N}
(\mathbb{R})\rtimes \vartheta)
$$
extends to an isomorphism
\begin{equation}
  K\Pi(\O,\GL_N(\R),\vartheta)\simeq K\Pi(\O,\GL_N(\R))^\vartheta.
\end{equation}
\end{subequations}
We should once again
remind the reader  that the $\mathbb{Z}$-modules appearing
in (\ref{twistgroth1}) are not Grothendieck groups in any natural fashion, notwithstanding the
appearance of the ``$K$''.  Nevertheless it is helpful to use this
notation, to help remind the reader of the origins of these modules.

\section{Sheaves and Characteristic Cycles}
\label{sheaves}

Suppose $\psi_G$ is an Arthur parameter for $G$ as in
(\ref{aparameter}).  In this section we give more details on the
definition of the ABV-packet $\Pi_{\psi_{G}}^{\mathrm{ABV}}$ and its stable
virtual character $\eta_{\psi_{G}}^{\mathrm{ABV}}$ (\ref{eq:etaABVIntro}).  The results apply in the more general context of
complex connected reductive groups $G$ (\cite{ABV}*{Sections 19, 22}).
However, for this section $G$ will be $\mathrm{Sp}_{N}$,
$\mathrm{SO}_{N}$ or $\mathrm{GL}_{N}$, with the setup of Section
\ref{llc}.  The definitions depend on a pairing between characters and
sheaves.

We also define a pairing between \emph{twisted} characters and \emph{twisted}
sheaves for $\mathrm{GL}_{N}$
\cite{Christie-Mezo}*{Sections 5-6}. The key properties of this
twisted pairing are listed in this section and shall be proved in
Section \ref{pairings}.

\subsection{The pairing and the ABV-packets in the non-twisted case}
\label{sec:ABV-packetsdef}
Let $\phi_{\psi_G}$ be the Langlands parameter associated to $\psi_{G}$
\eqref{phipsi}, $\O$ be the infinitesimal character of $\phi_{\psi_G}$, and
$S_{\psi_{G}} \subset X(\O,\LG)$ (\ref{eq:ABV-variety}) be the corresponding orbit
(\cite{ABV}*{Proposition 6.17}, (\ref{eq:orbitbijection})).
Recall that  $\Xi(\O,\LG)$ is the set of
pure complete geometric parameters (see the end of Section \ref{XO}).
There is a bijection \eqref{localLanglandspure} between $\Xi(\O,\LG)$
and  $\Pi(\O,G/\R)$, the (equivalence classes of) irreducible
representations of pure strong involutions of $G$.

Let $\mathcal{C}(X(\O,{^\vee}G^{\Gamma}))$ be
\nomenclature{$\mathcal{C}(X(\O,{^\vee}G^{\Gamma}))$}{category
  of constructible sheaves}
the category of ${^\vee}G$-equivariant constructible sheaves of
complex vector spaces on $X(\O,{^\vee}G^{\Gamma})$. This is
an abelian category and its simple objects are parameterized by the
set of complete geometric parameters  $\xi = (S, \tau_{S}) \in
\Xi(\O, {^\vee}G^{\Gamma})$ as follows.
Choose $p\in S$, let  $\chG_p=\Stab_{\chG}(p)$, and choose
a character $\tau_\xi$  of the component group of $\chG_p$ so that
$(p,\tau_\xi)$ is a representative of $\xi$.
Then $\tau_\xi$ pulled back to $\chG_p$ defines
an algebraic vector bundle
\begin{equation}
  \label{bundle}
        {^\vee}G \times_{{^\vee}G_{p} }V \rightarrow S.
\end{equation}
The sheaf of
sections of this vector bundle is, by definition, a
${^\vee}G$-equivariant local system on $S$ (\cite{ABV}*{Section 7, Lemma 7.3}). Extend this local
system to the closure $\bar{S}$ by zero and then take the direct image
into $X(\O,{^\vee}G^{\Gamma})$ to obtain an irreducible
(\emph{i.e.} simple) ${^\vee}G$-equivariant constructible sheaf
denoted by $\mu(\xi)$ (\cite{ABV}*{(7.10)(c)}).
\nomenclature{$\mu(\xi)$}{irreducible constructible sheaf}

Now let
$\mathcal{P}(X(\O,{^\vee}G^{\Gamma}))$ be the abelian category of
\nomenclature{$\mathcal{P}(X(\O,{^\vee}G^{\Gamma}))$}{category of perverse sheaves}
${^\vee}G$-equivariant perverse sheaves of complex vector spaces on
$X(\O,{^\vee}G^{\Gamma})$  \cite{Lunts}*{Section 5}.  The simple
objects of $\mathcal{P}(X(\O,{^\vee}G^{\Gamma}))$ are defined
from $\xi = (S, \tau_{S}) \in \Xi({^\vee}G^{\Gamma}, \O)$
and the algebraic vector bundle (\ref{bundle}) by taking the
\emph{intermediate} extension \cite{bbd}*{Section 2} to the closure $\bar{S}$
instead of the extension by zero.  This is denoted $P(\xi)$
 (\cite{ABV}*{(7.10)(d)}).
It is an irreducible ${^\vee}G$-equivariant
perverse sheaf on $X(\O,{^\vee}G^{\Gamma})$.
\nomenclature{$P(\xi)$}{irreducible perverse sheaf}

The Grothendieck groups of the two categories
$\mathcal{C}(X(\O,{^\vee}G^{\Gamma}))$ and
$\mathcal{P}(X(\O,{^\vee}G^{\Gamma}))$ are canonically isomorphic
(\cite{bbd}, \cite{ABV}*{Lemma 7.8}).  We identify the two Grothendieck
groups via this
isomorphism and denote them by $KX(\O,{^\vee}G^{\Gamma})$.
\nomenclature{$KX(\O,{^\vee}G^{\Gamma})$}{Grothendieck group}
This Grothendieck group has two natural bases
$$\{\mu(\xi)\mid \xi \in
\Xi(\O,\LG)\}
\quad \mbox{ and }\quad \{P(\xi)\mid \xi \in \Xi(\O,\LG)\}.$$

Suppose $\xi=(S,\tau)\in \Xi(\O,\LG)$.
We define two invariants associated to $\xi$.
First, let $d(\xi)\nomenclature{$d(\xi)$}{dimension of the orbit attached to a complete geometric parameter}$ be the dimension of $S_\xi$.
Second, associated to $\xi$
is the representation $\pi(\xi)$ of a pure strong involution of
$G$ (\ref{localLanglandspure}). Let $e(\xi) = \pm1 \nomenclature{$e(\xi)$}{Kottwitz invariant attached to a complete geometric parameter}$ be the Kottwitz invariant of the underlying real
form of this strong involution (\cite{ABV}*{Definition 15.8}).

As discussed in the introduction, we define a perfect pairing
\begin{equation}
\label{pairing}
\langle\,\cdot,\cdot\rangle:K\Pi(\O,G/\R)\times KX(\O,\LG)\rightarrow\Z
\end{equation}
by
\begin{equation*}
\langle M(\xi), \mu(\xi') \rangle = e(\xi)\, \delta_{\xi, \xi'}.
\end{equation*}
The pairing also takes a simple form relative to the bases given by
$\pi(\xi)$ and $P(\xi')$ (\cite{ABV}*{Theorem 1.24, Sections
  15-17}).  We state it as a theorem.
\begin{thm}
  \label{ordpairing}
The pairing \eqref{pairing} satisfies
$$
\langle \pi(\xi), P(\xi') \rangle = (-1)^{d(\xi)} \, e(\xi) \,
\delta_{\xi, \xi'}, \quad \xi,\xi'\in \Xi(\O,\LG).
$$
\end{thm}
This pairing allows us to regard elements of $K\Pi(\O,G/\R)$ as
$\mathbb{Z}$-linear functionals of $ KX(\O,\LG)$.
The microlocal multiplicity maps
$\chimic_S$  discussed in (\ref{mmm}) are $\mathbb{Z}$-linear
functionals on  $KX(\O,\LG)$.  Before making the obvious connection with the
pairing \eqref{pairing}, we review some facts needed to define $\chi_{S}^{\mathrm{mic}}$.
To begin, we consider the category of
$\ch G$-equivariant coherent $\mathcal{D}$-modules on
$X(\O,\LG)$. We
denote this category by
$\D(X(\O,\LG))$.
Here,
$\mathcal{D}$ is the sheaf of algebraic differential operators on
$X(\O,\LG)$
(\cite{Boreletal}*{VIII.14.4}, \cite{ABV}*{Section 7}).
Convenient references for equivariant $\mathcal{D}$-modules are \cite{Hotta}
and \cite{SmithiesTaylor}.
\nomenclature{$\mathcal{D} (X(\O,
  {^\vee}G^{\Gamma}))$}{category of $\mathcal{D}$-modules}

The equivariant Riemann-Hilbert correspondence (\cite{Boreletal}*
{Theorem VIII.14.4}) induces an isomorphism
\begin{equation}
\label{dr}
DR : K\D(X(\O,\LG))\rightarrow K X(\O,\LG).
\nomenclature{$DR$}{Riemann-Hilbert correspondence}
\end{equation}
For simplicity we write $X=X(\O,\LG)$, and $\D X=\D(X(\O,\LG))$.

The sheaf $\D$ is filtered by the order of the differential operators, and
the associated graded ring is canonically isomorphic to
$\mathcal{O}_{T^{*}(X)}$, the coordinate ring of the cotangent bundle
of $X$ (\cite{Hotta}*{Section 1.1}).
Suppose $\mathcal M\in\D X$. Then $\mathcal M$ has a filtration
such that the resulting graded sheaf
$\mathrm{gr} \mathcal{M} \nomenclature{$\mathrm{gr} \mathcal{M}$}{graded sheaf of a $\mathcal D$-module $\mathcal M$}$ is a
coherent $\mathcal{O}_{T^{*}(X)}$-module (\cite{Hotta}*{Section 2.1}).

The support of $\mathrm{gr} \mathcal{M}$ is a
closed subvariety of $T^{*}(X)$ (\cite{ABV}*{Definition 19.7}).  Each minimal ${^\vee}G$-invariant
component
of this closed subvariety is the closure of a conormal
bundle $T^{*}_{S}(X)$, where
\nomenclature{$T^{*}_{S}(X)$}{conormal bundle}
 $S \subset X$ is a ${^\vee}G$-orbit (\cite{ABV}*{Proposition 19.12(c)}).
Therefore to each conormal bundle $T^{*}_{S}(X)$ we may
attach a non-negative integer, denoted by
$\chimic_{S}(\mathcal{M})$, which (when nonzero) is the length of
the module $\mathrm{gr}\mathcal{M}$ localized at $T^{*}_{S}(X)$
\cite{Hotta}*{Section 2.2}.

The
characteristic cycle of $\mathcal{M}$ is defined as
$$\mathrm{Ch}(\mathcal{M}) = \sum_{S\in X/\ch G} \chi^{\mathrm{mic}}_{S}( \mathcal{M})
\ \overline{T_{S}^{*}(X)}.
\nomenclature{$\mathrm{Ch}(\mathcal{M})$}{characteristic cycle
of a $\mathcal D$-module $\mathcal M$}
$$
For a given $\ch G$-orbit $S$  we may regard $\chi_{S}^{\mathrm{mic}}$ as a
function on $\mathcal{D}$-modules which is additive for
short exact sequences (\cite{ABV}*{Proposition 19.12(e)}).  It
 therefore defines a homomorphism
$K \mathcal{D}(X(\O, {^\vee}G^{\Gamma}))
\rightarrow \Z$,
called the \emph{microlocal multiplicity} along $S$.
\nomenclature{$\chi^{\mathrm{mic}}_{S}$}{microlocal multiplicity map}
Using the isomorphism (\ref{dr}), we interpret this as a homomorphism
\begin{equation*}  \label{micromult}
\chi^{\mathrm{mic}}_{S} : K X(\O, {^\vee}G^{\Gamma})
\rightarrow \mathbb{Z}.
\end{equation*}

We now return to the pairing (\ref{pairing}) and its relationship to
$\chi^{\mathrm{mic}}_{S}$.  This relationship defines $\etaABV_{\psi_G}$.
We first define $\eta^{\mathrm{mic}}_{\psi_G}\in K\Pi(\O,G/\R)$ to be the element
of $K\Pi(\O,G/\R)$ corresponding via the pairing to
the element $\chimic_S$ in the dual of
$KX(\O,\LG)$.
Explicitly working through the identifications in the definition we see
\begin{equation}
  \label{etapsi}
  \etamic_{\psi_G}=
 \sum_{\xi \in \Xi(\O,\LG)}
(-1)^{d(S_{\xi}) - d(S_{\psi_{G}})} \ \chi^{\mathrm{mic}}_{S_{\psi_{G}}}(P(\xi)) \, \pi(\xi).
\nomenclature{$\eta^{\mathrm{mic}}_{\psi_{G}}$}{stable virtual character}
\nomenclature{$d(S)$}{dimension of $S$}
\end{equation}

An important result of Kashiwara and Adams-Barbasch-Vogan is
\begin{prop}[\cite{ABV}*{Theorem 1.31, Corollary 19.16}]
$\etamic_{\psi_G}$ is a stable virtual character.
\end{prop}

The \emph{microlocal packet} $\Pi^{\mathrm{mic}}_{\psi_{G}}$ of $\psi_{G}$  is
defined to be the irreducible representations
in the support of $\etamic_{\psi_{G}}$. In other words
$$
\Pi^{\mathrm{mic}}_{\psi_{G}}  = \{ \pi(\xi) : \xi \in \Xi(\O\LG)\mid
\chimic_{S_{\psi_{G}}}(P(\xi)) \neq 0\}.
\nomenclature{$\Pi_{\psi_{G}}^{\mathrm{mic}}$}{microlocal packet}
$$
This is a set of irreducible representations of pure strong involutions of $G$.
We are primarily interested in the packet for the quasisplit strong
involutions.
We therefore define
\begin{equation}
\label{etapsiabv}
  \eta^{\mathrm{ABV}}_{\psi_{G}} = \eta_{\psi_{G}}^{\mathrm{mic}}(\delta_{q})
\nomenclature{$\eta^{\mathrm{ABV}}_{\psi_{G}}$}{
stable virtual character defining the ABV-packet}
\nomenclature{$\eta_{\psi_{G}}^{\mathrm{mic}}(\delta_{q})$}{}
\end{equation}
to be the restriction of $\etamic_{\psi_{G}}$ to the submodule of $K
\Pi(\O, G/\mathbb{R})$ generated by the representations in
$\Pi(\O, G(\mathbb{R}, \delta_{q}))$.  The $\mathrm{ABV}$-packet
$\Pi_{\psi_{G}}^{\mathrm{ABV}}$
is defined as the support of $\eta^{\mathrm{ABV}}_{\psi_{G}}$, that is
\begin{equation}
\label{abvdef}
\PiABV_{\psi_{G}} = \{ \pi(\xi) : \xi \in \Xi(\O, \LG),
\chimic_{S_{\psi_{G}}}(P(\xi)) \neq 0, \pi(\xi)\in \Pi(G(\R,\delta_q))
\}.
\nomenclature{$\Pi_{\psi_{G}}^{\mathrm{ABV}}$}{ABV-packet}
\end{equation}

We conclude this section with a restatement of Theorem \ref{ordpairing}.
Define the representation-theoretic transition
matrix $m_r$ by
\begin{equation}
\label{orddecomp}
M(\xi)  = \sum_{\xi' \in \Xi(\O, {^\vee}G^{\Gamma})}
m_r(\xi',\xi) \, \pi(\xi').
\nomenclature{$m_r(\xi',\xi)$}{representation-theoretic multiplicity}
\end{equation}
Define the geometric ``transition matrix'' $c_{g}$ by
\begin{equation}
  \label{geomat}
P(\xi) = \sum_{\xi' \in \Xi(\O,
  {^\vee}G^{\Gamma})} (-1)^{d(\xi)}\, c_{g}(\xi', \xi) \,
\mu(\xi').
\end{equation}
(see \cite{ABV}*{(7.11)(c)}).
\nomenclature{$c_{g}(\xi', \xi)$}{signed sheaf-theoretic multiplicity}
Then \cite{ABV}*{Corollary 15.13} says
\begin{prop}
  \label{ordpairingequiv}
Theorem \ref{ordpairing} is
equivalent to the identity
\begin{equation}
\label{pairingmc}
m_r(\xi', \xi) = (-1)^{d(\xi) - d(\xi')} \, c_g(\xi, \xi').
\end{equation}
\end{prop}
This equation relates the decomposition of characters with the
decomposition of sheaves.

\subsection{The pairing in the twisted case}
\label{conperv}

As discussed in the previous section, the pairing \eqref{pairing}
plays a fundamental role in the
definition of ABV-packets. We now discuss a twisted version of this
pairing  for $\GL_N$.

We replace $K\Pi(\O,\GL_N(\R))$ with the $\mathbb{Z}$-module
$K\Pi(\O,\GL_N(\R),\vartheta)$ of twisted characters
(\ref{twistgroth1}).
Associated to
$\xi\in\Xi(\O, \ch\GL_{N}^{\Gamma})^{\vartheta}$ are an irreducible
representation $\pi(\xi)\in \Pi(\O ,\GL_N(\R))^\vartheta$ as well as a
canonical extension
$\pi(\xi)^+$ to $\GLext$ (Corollary \ref{twistclass2}). The twisted character of $\pi(\xi)^+$ is an
element of
the space $K\Pi(\O,\GL_N(\R),\vartheta)$ of twisted characters,
and this gives a basis of $K\Pi(\O,\GL_N(\R),\vartheta)$ parameterized
by $\Xi(\O, \ch \GL_N^{\Gamma})^{\vartheta}$.
See \eqref{KPiOGLext} and the end of Section \ref{extrepsec}.

The twisted characters are to be paired with twisted sheaves which are
elements in a $\Z$-module generalizing $KX(\O,\LG)$.
The twisted objects for this pairing are given
in \cite{ABV}*{(25.7)} (see also \cite{Christie-Mezo}*{Section 5.4}).  We
provide a short summary.

Let $s \in
\mathrm{GL}_{N}$ be an element such that
\begin{equation}
  \label{sigmaaut}
\upsigma
= \mathrm{Int}(s) \circ \vartheta
\end{equation}
is an automorphism of
$\mathrm{GL}_{N}$ of finite order.  Then $\upsigma$ acts on
$X(\O, {^\vee}\mathrm{GL}_{N}^{\Gamma})$ in a manner which
is compatible with the ${^\vee}\mathrm{GL}_{N}$-action (\cite{ABV}*{(25.1)}),
and so also acts on its $^{\vee}\mathrm{GL}_{N}$-equivariant sheaves.
\nomenclature{$\upsigma$}{automorphism of $\mathrm{GL}_{N}$}

Let $\mathcal{P}(X(\O, {^\vee}\mathrm{GL}_{N}^{\Gamma});
\upsigma)$
\nomenclature{$\mathcal{P}(X(\O, {^\vee}\mathrm{GL}_{N}^{\Gamma});
\upsigma)$}{category of perverse sheaves with a compatible $\upsigma$-action}
 be the category of ${^\vee}\mathrm{GL}_{N}$-equivariant perverse
sheaves with a compatible $\upsigma$-action.  An object in this category
is a pair $(P,\upsigma_{P})$ in which $P$ is an equivariant perverse sheaf and
$\upsigma_{P}$ is an automorphism of $P$ which is compatible with
$\upsigma$ (\cite{Christie-Mezo}*{Section 5.4}).    Similarly, we define
$\mathcal{C}(X(\O, {^\vee}\mathrm{GL}_{N}^{\Gamma});
\upsigma)$
\nomenclature{$\mathcal{C}(X(\O, {^\vee}\mathrm{GL}_{N}^{\Gamma});
\upsigma)$}{category of constructible sheaves  with a compatible $\upsigma$-action}
to be the category of ${^\vee}\mathrm{GL}_{N}$-equivariant
constructible
sheaves with a compatible $\upsigma$-action.  An object in this category
is a pair $(\mu,\upsigma_{\mu})$ in which $\mu$ is an equivariant
constructible sheaf and
$\upsigma_{\mu}$ is an automorphism of $\mu$ which is compatible with
$\upsigma$.

\nomenclature{$K(X(\O,
  {^\vee}\mathrm{GL}_{N}^{\Gamma});\upsigma)$}{Grothendieck group}
The Grothendieck groups of these two categories are isomorphic
\cite{Christie-Mezo}*{(35)}.  We identify them and denote  their
Grothendieck groups by
$K(X(\O, {^\vee}\mathrm{GL}_{N}^{\Gamma});\upsigma)$.
This is the sheaf-theoretic analogue
of $K \Pi(\mathrm{GL}_{N}(\mathbb{R}) \rtimes \langle \vartheta
\rangle)$.

As with the representations (see (\ref{canext})), we seek a canonical choice of
extension of $P(\xi)$, \emph{i.e.} a canonical automorphism
$\upsigma_{P(\xi)}$ of $P(\xi)$.

\begin{lem}
\label{cansheaf}
Let ${^\vee}G = {^\vee}\mathrm{GL}_{N}$,  $\xi = (S, \tau_{S}) \in
\Xi(\O,
{^\vee}\mathrm{GL}_{N}^{\Gamma})^{\vartheta}$, $p \in S$, and
(\ref{bundle}) be the equivariant vector bundle representing
$\mu(\xi)$.
\begin{enumerate}[label={(\alph*)}]
\item Suppose $p' \in S$ and $p'= a\cdot p$ for some $a \in {^\vee}
  \mathrm{GL}_{N}$. Then the maps
\begin{align}
\label{bundmap}
(g,v) &\mapsto (ga^{-1},v)\\
\nonumber g\cdot p &\mapsto (ga^{-1}) \cdot p'
\end{align}
define an isomorphism of equivariant vector bundles
\begin{equation}
\label{bundiso1}
{^\vee}G \times_{{^\vee}G_{p}} V \rightarrow {^\vee}G \times_{{^\vee}G_{p'}} V.
\end{equation}
which is  independent of the choice of $a$.

\item There exist canonical choices of pairs
$$
  \mu(\xi)^{+} = (\mu(\xi), \upsigma_{\mu(\xi)}^{+}) \in
  \mathcal{C}(X(\O, {^\vee}\mathrm{GL}_{N}^{\Gamma});
  \upsigma),$$
$$P(\xi)^{+} = (P(\xi), \upsigma_{P(\xi)}^{+}) \in
  \mathcal{P}(X(\O, {^\vee}\mathrm{GL}_{N}^{\Gamma}); \upsigma)$$

such that
if $p \in S$ is fixed by $\upsigma$ then  $\upsigma_{\mu(\xi)}^{+}$
(and $\upsigma_{P(\xi)}^{+}$) acts trivially on the stalk  of
$\mu(\xi)$ (and $P(\xi) \in KX(\O,
{^\vee}\mathrm{GL}_{N}^{\Gamma})$) at $p$.
\end{enumerate}
\nomenclature{$\mu(\xi)^{+}$}{twisted constructible sheaf}
 \nomenclature{$P(\xi)^{+}$}{twisted perverse sheaf}
\end{lem}
\begin{proof}
Let $(p,\tau)$ and $(p',\tau')$ be representatives of $\xi$.  It is
well-known that the component group ${^\vee}G_{p}/ ({^\vee}G_{p})^{0}$
is trivial for the general linear group (\cite{ABV}*{Lemma 7.5}), and so
$\tau$ is its trivial
quasicharacter.  For the same reason, $\tau'$ is the trivial quasicharacter of
the trivial group. Both $\tau$ and $\tau'$ lift to the trivial
representations of ${^\vee}G_{p}$ and ${^\vee}G_{p'} = a\, {^\vee}G
a^{-1}$ respectively.
By definition
$$(gh,v)  = (g,\tau(h)v) = (g,v), \ (g,v) \in {^\vee}G
\times_{{^\vee}G_{p}} V, \quad  h \in {^\vee}G_{p}.$$
Applying (\ref{bundmap}) to the left-most element, we obtain
$$(gha^{-1},v) = (ga^{-1} aha^{-1},v) = (ga^{-1}, \tau'(aha^{-1})v) = (ga^{-1},v)$$
in ${^\vee}G \times_{{^\vee}G_{p'}} V$. This proves that the map
(\ref{bundmap}) is well-defined.  The map is clearly a
${^\vee}\mathrm{GL}_{N}$-equivariant isomorphism.
The element $a \in {^\vee}G$ is unique up to right-multiplication by
an element in $a_{1} \in {^\vee}G_{p}$.  Since
$$(g(aa_{1})^{-1}, v ) = (ga^{-1} a a_{1}^{-1} a^{-1}, v) =
(ga^{-1},\tau'(aa_{1}^{-1}a^{-1})v) = (ga^{-1},v)$$
in ${^\vee}G \times_{{^\vee}G_{p'}} V$, the isomorphism
(\ref{bundmap}) is independent of the choice of $a$. This proves the
first assertion.

Suppose $p'  = \upsigma(p)  = a \cdot p \in S$.  Then $\upsigma$
induces a bundle isomorphism
$${^\vee}G \times_{{^\vee}G_{p}} V \rightarrow {^\vee}G
\times_{{^\vee}G_{p'}} V,$$
which when composed with the inverse of (\ref{bundiso1}) yields a
canonical automorphism which we set equal to
$\upsigma_{\mu(\xi)}^{+}$.
To be explicit
\begin{equation}
\label{expaut}
\upsigma_{\mu(\xi)}^{+} (g,v) = (\upsigma(g)a, v), \quad (g,v) \in
        {^\vee}G \times_{{^\vee}G_{p}} V
\end{equation}
We identify $\upsigma_{\mu(\xi)}^{+}$ with the unique automorphism of
$\mu(\xi)$ which it determines.

This choice of $\upsigma_{\mu(\xi)}^{+}$ determines a canonical choice
$\upsigma_{P(\xi)}^{+}$ by virtue of the fact that $\mu(\xi)$ occurs
in the decomposition of $P(\xi)$ in  $KX(\O,
{^\vee}\mathrm{GL}_{N}^{\Gamma})$ with multiplicity one ((7.11)(b)
\cite{ABV}, \cite{Christie-Mezo}*{pp. 154-155}).

Finally, suppose $\upsigma$ preserves $p$.  Then $g = a=1$ in
(\ref{expaut}) and the last assertion is proved.
\end{proof}

We now imitate the definition of $K\Pi(\O,\GL_N(\R),\vartheta)$
\eqref{twistgroth1} for the sheaves appearing in Lemma \ref{cansheaf}.
Attached to $\xi\in \Xi(\O,\ch\GL_N^\Gamma)^{\vartheta}$ are
perverse sheaves $P(\xi)^{\pm}$, where $P(\xi)^+$ is defined in
Lemma \ref{cansheaf}, and $P(\xi)^-$ is the unique other choice of extension.
Furthermore, the {\it microlocal traces} of $P(\xi)^\pm$ differ by
sign (\cite{ABV}*{(25.1)(j)}).
Similar comments apply to $\mu(\xi)^{\pm}$.

We are interested only in irreducible sheaves with non-vanishing microlocal
trace.  We consequently follow the definition of
(\ref{twistgroth1}) in defining the quotient
\begin{equation}
  \label{twistsheafgroth}
  K X(\O,\LGL,\upsigma)=
  K(X(\O,\LGL))^\upsigma\otimes
    \mathbb{Z}[U_{2}]/ \langle (P(\xi) \otimes 1) + (P(\xi)\otimes -1)\rangle
\nomenclature{$K (X(\O, {^\vee}\mathrm{GL}^{\Gamma}),
    \upsigma)$}{$\mathbb{Z}$-module of twisted sheaves}
  \end{equation}
  where the quotient runs over $\xi\in\Xi(\O,\LGL)^\vartheta$.

  This is the $\mathbb{Z}$-module which we shall pair with
  $$K\Pi(\O, \mathrm{GL}_{N}(\mathbb{R}))^{\vartheta} \cong
  K\Pi(\O, \mathrm{GL}_{N}(\mathbb{R}), \vartheta)$$
  in Section \ref{pairings}.  We call the elements of this module \emph{twisted
  sheaves}, and remind the reader that these modules are not naturally
Grothendieck groups, even though we have kept the ``$K$'' in the notation.

For reasons that will only become clear in Section \ref{whitsec}, the
definition of our twisted pairing involves some additional signs. The
signs depend on the integral lengths of parameters, which may be
described as follows.

From now on we assume  $\lambda \in \O$ satisfies the regularity condition (\ref{regintdom}).
Let $\xi\in\Xi(\O,\ch\GL_N^\Gamma)^{\vartheta}$.
 Lemma \ref{XXXi}  tells us that associated to $\xi$ is an element $x\in \mathcal{X}_{\ch\rho}$. Set
  $\theta_x=\mathrm{Int}(x)\in\Norm_{\mathrm{GL}_{N}}(H)$.
Let
\begin{subequations}
\renewcommand{\theequation}{\theparentequation)(\alph{equation}}
\label{R}
\begin{equation}
R(\lambda)=\left\{\alpha\in R(\GL_N,H)\mid \langle\lambda,\ch\alpha\rangle\in\Z\right\}
\nomenclature{$R(\lambda)$}{$\lambda$-integral root system}
\end{equation}
be the $\lambda$-integral roots, with positive $\lambda$-integral roots
\begin{equation}
R^{+}(\lambda) = \left\{ \alpha \in R(\lambda)\mid \langle\lambda,\ch\alpha\rangle>0\right\}.
\nomenclature{$R^{+}(\lambda)$}{positive $\lambda$-integral roots}
\end{equation}
\end{subequations}
Define the
\emph{integral   length}, following \cite{ABV}*{(16.16)}, as
\begin{equation}
  \label{intlength}
l^{I}(\xi) = -\frac{1}{2} \left( | \{\alpha \in R^{+}(\lambda) :
\theta_x(\alpha) \in R^{+}(\lambda) \}| + \dim(H^{\theta_x}) \right).
\nomenclature{$l^{I}(\xi)$}{integral length}
\end{equation}
The integral length takes values in the non-positive integers.

Furthermore define
$$
R_{\vartheta}^{+}(\lambda) = \{ \alpha \in
R((\mathrm{GL}_{N}^{\vartheta})^{0}, (H^{\vartheta})^{0})\mid
\langle\lambda,\ch\alpha\rangle\in \Z_{>0}\}.
$$
We define the $\vartheta$\emph{-integral length} by
\begin{equation}
  \label{thetalength}
l^{I}_{\vartheta}(\xi) = -\frac{1}{2} \left( | \{\alpha \in
R^{+}_{\vartheta}(\lambda)  \mid
\theta_x(\alpha) \in R^{+}_{\vartheta}(\lambda) \} |+
\dim((H^{\vartheta})^{\theta_x}) \right).
\end{equation}
This is the integral length for the principal endoscopic group whose dual is
$^{\vee}\GL_N^\vartheta$.  Let us say a little more about this. The $\vartheta$-length was introduced in \cite{AMR}*{Definition 6.2} in a slightly different form. It appears  there in the coefficients of the twisted character formula of a  $\vartheta$-stable Speh representation  (\cite{AMR}*{Equation 6.3.2}).
This mirrors the coefficients of the non-twisted character formula of a Speh representation of the principal endoscopic group.
In Section \ref{pairings} we will see that the $\vartheta$-length is the correct notion of length in the combinatorics of the twisted Hecke module actions. This mirrors the combinatorics of the Hecke module actions for the principal twisted endoscopic group.
\nomenclature{$l^{I}_{\vartheta}(\xi) $}{$\vartheta$-integral length}

Now we define a perfect pairing (under the assumption \eqref{regintdom}):
\begin{equation}
  \label{pair2}
\langle \cdot, \cdot \rangle:  K \Pi(\O,
\mathrm{GL}_{N}(\mathbb{R}), \vartheta)
\times K X(\O, {^\vee}\mathrm{GL}_{N}^{\Gamma}, \upsigma)
\rightarrow \mathbb{Z}
\end{equation}
by setting
\begin{equation}
\label{pairdef2}
\langle M(\xi)^{+}, \mu(\xi')^{+} \rangle = (-1)^{l^{I}(\xi) -
  l^{I}_{\vartheta}(\xi)} \, \delta_{\xi, \xi'}
\end{equation}
for $\xi, \xi' \in \Xi(\O,
{^\vee}\mathrm{GL}_{N}^{\Gamma})^{\vartheta}$.
 The analogue of Theorem \ref{ordpairing} is
\begin{thm}
  \label{twistpairing}
Suppose $\lambda \in \O$ satisfies (\ref{regintdom}).  Define the pairing (\ref{pair2}) by (\ref{pairdef2}).  Then
  $$\langle \pi(\xi)^{+}, P(\xi')^{+} \rangle = (-1)^{d(\xi)} \,
  (-1)^{l^{I}(\xi)-l^{I}_{\vartheta}(\xi)} \, \delta_{\xi, \xi'}$$
where $\xi, \xi' \in
\Xi(\O, {^\vee}\mathrm{GL}_{N}^{\Gamma})^{\vartheta}$.
\end{thm}
The proof of this theorem is the primary purpose of Section
\ref{pairings}. Its proof is modelled on the proof of Theorem \ref{ordpairing}
in \cite{ABV}*{Sections 15-17}.

The signs $(-1)^{l^{I}(\xi)-l^{I}_{\vartheta}(\xi)}$ appear in the
pairing to account for the comparison of extensions given in Section
\ref{whitsec}.  Note that if $\vartheta$ were taken to be the
 identity automorphism then the signs would disappear, and one would
 recover the ordinary pairing (\ref{pairing}) for $\mathrm{GL}_{N}$.

We conclude this section by giving a twisted analogue of Proposition \ref{ordpairingequiv}.
This analogue will only be needed in Sections 7 and 9, so the reader may wish to skip this discussion and return to it
later.

For $\xi, \xi' \in \Xi(\O,
{^\vee}\mathrm{GL}_{N}^{\Gamma})^{\vartheta}$, define $m_{r}(\xi'_{\pm},
\xi_{+})$ to be the multiplicity of the
representation $\pi(\xi')^{\pm}$ in $M(\xi)^{+}$ in the Grothendieck group
$K\Pi(\O, \mathrm{GL}_{N}(\mathbb{R}) \rtimes
\langle \vartheta \rangle)$ (Section \ref{grothchar}).  In other words
\begin{equation*}
\label{extmult}
M(\xi)^{+}  = \sum_{\xi' \in \Xi(\O, {^\vee}G^{\Gamma})^{\vartheta}}
m_{r}(\xi'_{+},\xi_{+}) \, \pi(\xi')^{+}+ m_{r}(\xi'_{-},\xi_{+}) \,
\pi(\xi')^{-}+ \cdots
\end{equation*}
where the omitted summands are irreducible representations of
$\mathrm{GL}_{N}(\mathbb{R}) \rtimes \langle \vartheta \rangle$ of the
second type (Section \ref{grothchar}).  Define
\begin{equation}
  \label{twistmult}
  m^{\vartheta}_{r}(\xi',\xi) = m_{r}(\xi'_{+}, \xi_{+})
  -m_{r}(\xi'_{-}, \xi_{+})
\end{equation}
for $\xi, \xi' \in \Xi(\O,
{^\vee}\mathrm{GL}_{N}^{\Gamma})^{\vartheta}$ (\emph{cf.}
\cite{AvLTV}*{(19.3d)}).
\nomenclature{$m^{\vartheta}_{r}(\xi',\xi)$}{}
By construction, the image of $M(\xi)^{+}$ in $K
\Pi(\O, \mathrm{GL}_{N}(\mathbb{R}), \vartheta)$
(\ref{twistgroth1}) decomposes as
\begin{equation}
M(\xi)^{+} = \sum_{\xi' \in \Xi(\O,
{^\vee}\mathrm{GL}_{N}^{\Gamma})^{\vartheta}} m^{\vartheta}_{r}(\xi',\xi) \, \pi(\xi')^{+}.
\label{twistmult1}
\end{equation}
\begin{lem}
\label{invertmultmat}
The matrix given by
$$m^{\vartheta}_{r}(\xi',\xi), \quad \xi, \xi' \in \Xi(\O,
{^\vee}\mathrm{GL}_{N}^{\Gamma})^{\vartheta}$$
(\ref{twistmult}) is invertible.
\end{lem}

\begin{proof}

The invertibility of the matrix given by $m_{r}(\xi', \xi)$ in
(\ref{orddecomp}) follows since it is uni-upper-triangular with
respect to the Bruhat order (\cite{AvLTV}*{Definition 18.4}, \cite{greenbook}*{Lemma 6.6.6}). We show that $m_{r}^{\vartheta}(\xi'\xi)$ inherits the same two
properties.
By restricting to $\GL_N(\R)$ we see
$$
m_{r}(\xi', \xi) = m_{r}(\xi_{+}', \xi_{+}) + m_{r}(\xi_{-}', \xi_{+})
$$
(see the equation preceding \cite{AvLTV}*{19.3c}).
Furthermore $m_{r}(\xi, \xi) = 1$ and the definition of Atlas extensions imply
$m_{r}(\xi_{+}, \xi_{+})= 1$ and $m_{r}(\xi_{-}, \xi_{+})= 0 $.
Therefore $m^{\vartheta}_{r}(\xi,\xi)= 1$ and by (\ref{twistmult}) we conclude
$m^{\vartheta}_{r}(\xi',\xi)$ is uni-upper-triangular. In particular, it is invertible.
\end{proof}

In a parallel fashion, we define  $c_{g}(\xi'_{\pm},\xi_{+})$ for $\xi, \xi' \in
(\O, {^\vee}G^{\Gamma})^{\vartheta}$ by
\begin{equation}
\label{extsheafmult}
P(\xi)^{+} = \sum_{\xi' \in \Xi(\O,
  {^\vee}G^{\Gamma})^{\vartheta}} (-1)^{d(\xi')}\, c_{g}(\xi'_{+},
\xi_{+}) \, \mu(\xi')^{+} + (-1)^{d(\xi')}\, c_{g}(\xi'_{-}, \xi_{+}) \,
\mu(\xi')^{-}+ \cdots
\end{equation}
in the Grothendieck group $K X(\O, {^\vee}\mathrm{GL}_{N};
\upsigma)$ of Section \ref{conperv}.  Setting
\begin{equation}
  \label{twistgmult}
c_{g}^{\vartheta}(\xi',\xi) = c_{g}(\xi'_{+},\xi_{+}) -
c_{g}(\xi'_{-}, \xi_{+}).
\nomenclature{$c_{g}^{\vartheta}(\xi',\xi)$}{}
\end{equation}
we see that the image of $P(\xi)^{+}$ in $K X(\O,
{^\vee}\mathrm{GL}_{N}, \upsigma)$ is
\begin{equation}
  \label{imP}
\sum_{\xi' \in \Xi(\O,
  {^\vee}G^{\Gamma})^{\vartheta}} (-1)^{d(\xi')}\, c_{g}^{\vartheta}(\xi',
\xi) \, \mu(\xi')^{+}.
\end{equation}
Just as Theorem \ref{ordpairing} is equivalent to Proposition \ref{ordpairingequiv}. We
have the following equivalence.
\begin{prop}
\label{p:twist}
Theorem \ref{twistpairing} is equivalent to the identity
\begin{equation}
  \label{e:twist}
m_{r}^{\vartheta}(\xi', \xi) = (-1)^{l^{I}_{\vartheta}(\xi) -
  l^{I}_{\vartheta}(\xi')}\ c_{g}^{\vartheta}(\xi, \xi')
\end{equation}
for all $\xi, \xi' \in \Xi(\O,
{^\vee}\mathrm{GL}_{N}^{\Gamma})^{\vartheta}$.
\end{prop}
\begin{proof}
Using Lemma \ref{invertmultmat}
 we compute
\begin{align*}
\langle \pi(\xi_{1})^{+}, P(\xi_{2})^{+} \rangle & = \sum_{\xi_{1}',
  \xi_{2}'} (m_{r}^{\vartheta})^{-1}(\xi_{1}', \xi_{1})
\ c_{g}^{\vartheta}(\xi_{2}' ,\xi_{2}) \,(-1)^{d(\xi_{2}')} \langle
M(\xi_{1}')^{+},  \mu(\xi_{2}') \rangle    \\
& = \sum_{\xi_{1}'}  (m_{r}^{\vartheta})^{-1}(\xi_{1}', \xi_{1})
\ c_{g}^{\vartheta}(\xi_{1}' ,\xi_{2}) \,(-1)^{d(\xi_{1}')}
\,(-1)^{l^{I}(\xi_{1}') - l^{I}_{\vartheta}(\xi_{1}') }
\end{align*}
for $\xi_{1}, \xi_{2} \in \Xi(\O,
{^\vee}\mathrm{GL}_{N}^{\Gamma})^{\vartheta}$.
If Theorem \ref{twistpairing} holds then this sum is equal to
$$(-1)^{d(\xi_{1})}  \,(-1)^{l^{I}(\xi_{1}) -
  l^{I}_{\vartheta}(\xi_{1}) } \delta_{\xi_{1}, \xi_{2}}$$
and so
$$m_{r}^{\vartheta}(\xi_{1}', \xi_{1})  = (-1)^{ l^{I}(\xi_{1})
  -d(\xi_{1})} \, (-1)^{  l^{I}(\xi_{1}')-d(\xi_{1}' )} (-1)^{
l^{I}_{\vartheta}(\xi_{1}) -
l^{I}_{\vartheta}(\xi_{1}')}\ c_{g}^{\vartheta}(\xi_{1}, \xi_{1}').$$
By \cite{AMR1}*{Proposition B.1},
$$(-1)^{ l^{I}(\xi_{1})-d(\xi_{1}) } =(-1)^{ l^{I}(\xi_{1}') - d(\xi_{1}' )}$$
is a constant independent of any parameters $\xi_{1}$ and $\xi_{1}'$.  Thus,
$$m_{r}^{\vartheta}(\xi_{1}', \xi_{1})  = (-1)^{
l^{I}_{\vartheta}(\xi_{1}) - l^{I}_{\vartheta}(\xi_{1}')}\ c_{g}^{\vartheta}(\xi, \xi').$$
The process we have given may easily be reversed to prove the converse
statement.
\end{proof}

\section{The proof of Theorem \ref{twistpairing}}
\label{pairings}

\subsection{The Beilinson-Bernstein correspondence
  in the proof of Theorem \ref{twistpairing}}
  \label{bbvd}

  Our proof of  Theorem \ref{twistpairing}
  will follow the same strategy as
the proof of Theorem \ref{ordpairing} in \cite{ABV}*{Sections 15-17}.
We recall some of the theory of KLV-polynomials in the non-twisted
context first.

The basic tool in this theory is the Hecke algebra for
$\mathrm{GL}_{N}(\mathbb{R})$ (\cite{ABV}*{(16.10)}).
For Harish-Chandra modules of $\GL_N(\R)$ of infinitesimal character $\O$,
this is a free $\mathbb{Z}[q^{1/2}, q^{-1/2}]$-algebra
$\mathcal{H}(\O)$, which comes equipped with a representation on
the Hecke module
$$
\mathcal{K} \Pi(\O, \mathrm{GL}_{N}( \mathbb{R})) = K\Pi(\O,
 \mathrm{GL}_{N}(\mathbb{R})) \otimes_{\mathbb{Z}} \mathbb{Z} [q^{1/2}, q^{-1/2}].
\nomenclature{$\mathcal{K} \Pi(\O, \mathrm{GL}_{N}( \mathbb{R}))$}{Hecke module}
$$
This representation is
actually transported from a Hecke algebra action on a module generated
by constructible sheaves (\cite{ICIV}*{Proposition 12.5}, \cite{LV}),
using  the Riemann-Hilbert
\eqref{dr} and  Beilinson-Bernstein \cite{BB} correspondences.

It is the latter kind of Hecke algebra action which gives us a
representation of $\mathcal H(\O)$
on
$$
\mathcal{K} X(\O, {^\vee}\mathrm{GL}_{N}^{\Gamma}) = K X(\O,
{^\vee}\mathrm{GL}_{N}^{\Gamma}) \otimes_{\mathbb{Z}}
\mathbb{Z}[q^{1/2}, q^{-1/2}]
\nomenclature{$\mathcal{K} X(\O, {^\vee}\mathrm{GL}_{N}^{\Gamma})$}{Hecke module}
$$
\cite{ABV}*{Proposition 16.13}. In order to describe the details of the Hecke action in the twisted
case (Section \ref{twisthmodule}), it is convenient to replace the space
$KX(\O,{^\vee}\mathrm{GL}_{N}^{\Gamma})$ with a space of
characters of representations of certain inner forms of  ${^\vee}\mathrm{GL}_{N}$.
To be more specific, we define
$$\ch \Pi(\O,\GL_N(\R))
\nomenclature{$\ch\Pi(\O,\GL_N(\R))$}{}
$$ to be the set
of irreducible
characters obtained by applying the Riemann-Hilbert and
Beilinson-Bernstein correspondences to the irreducible equivariant
perverse sheaves on $X(\O, {^\vee}\mathrm{GL}_{N}^{\Gamma})$.

Here is some detail about $\ch\Pi(\O,\GL_N(\R))$.
Suppose $\xi=(S,\tau_S)\in \Xi(\O, \ch\mathrm{GL}_{N}^{\Gamma})$
and write $\phi$ for the Langlands parameter with orbit $S$ (\cite{ABV}*{Proposition 6.17}, (\ref{eq:orbitbijection})).
Define $\lambda$
and $y$ by
\eqref{lambday}, $\ch \mathrm{GL}_{N}(\lambda)
\nomenclature{$\ch \mathrm{GL}_{N}(\lambda)$}{}
$ by (\ref{Glambda}), and $\ch K_y$ as in Equation (\ref{y}).
It is easy to see that $\ch \mathrm{GL}_{N}(\lambda)$ is a product of groups $\GL_{n_i}$,
and that the real group corresponding to $\ch K_y$ is a product of
indefinite unitary groups $\mathrm{U}(p_i,q_i)$ with $p_{i} + q_{i} = n_{i}$.
Let $
\nomenclature{$\ch\rho_\lambda$}{}
\ch \rho_\lambda=\frac12\sum_{\alpha\in R^+(\lambda)} \ch\alpha$ (see (\ref{R})).
Then $\ch\rho- \ch\rho_\lambda$ defines a two-fold cover of $\ch K_y
$
which we denote by $\ch\wt K_y
\nomenclature{$\ch\wt K_y$}{two-fold cover of $\ch K_y$}$ (\cite{AV}*{Definition 8.11}).
The set $\ch\Pi(\O,\GL_N(\R))$ consists of $(\ch\gl(\lambda),\ch\wt
K_y)\nomenclature{$\ch\gl(\lambda)$}{complex Lie algebra of $\ch \mathrm{GL}_{N}(\lambda)$}$-modules.

To summarize:
\begin{prop}[\cite{ICIII}*{Proposition 1.2}, \cite{ABV}*{Theorem 8.5}]
  \label{p:rhbb}
The Riemann-Hilbert and Beilinson-Bernstein correspondences define a bijection
    $$
    \Xi(\O,\LGL)\longleftrightarrow \ch\Pi(\O,\GL_N(\R)).
    $$
    In this correspondence $\xi \in \Xi(\O,\LGL)$ is  sent to an
    irreducible $(\ch\gl(\lambda),\ch\wt K_y)$-module of infinitesimal
    character $\ch\rho$. This correspondence induces an isomorphism
   of $\mathbb{Z}$-modules
  \begin{equation}
    \label{rhbb}
K X\left(\O, {^\vee}\mathrm{GL}_{N}^{\Gamma}\right) \cong K
{^\vee}\Pi\left(\O, \mathrm{GL}_{N}(\mathbb{R})\right).
\nomenclature{$K\ch\Pi\left(\O,\mathrm{GL}_{N}(\mathbb{R})\right)$}{dual Grothendieck group}
  \end{equation}
\end{prop}

\subsection{Vogan Duality for $\GL_N$}
\label{duality}

We want to understand the $(\ch\gl(\lambda),\ch\wt
K_y)$-modules of Section \ref{bbvd} in terms of
our parameters.
Suppose $\xi\in \Xi(\O,\LGL)$, and let $(x,y)$ be the corresponding
Atlas parameter in $\mathcal{X}_{\ch\rho}\times
\ch\mathcal{X}_\lambda$ given by Lemma \ref{XXXi}.
As we shall see, the reversed pair $(y,x)$ then defines an Atlas
parameter for $\ch
\mathrm{GL}_{N}(\lambda)$ (\cite{AVParameters}*{Section 6.1}).
In the case of integral infinitesimal character this is an example of
Vogan duality in the version of \cite{Adams-Fokko}*{Corollary 10.8}.

Here are some details in our setting.
Let  $\sigma_w \in \mathrm{GL}_{N}$ be the Tits representative of an
element  $w \in
W(\mathrm{GL}_{N}, H)$ \cite{AVParameters}*{Section 12}, and $w_0 \in
W(\mathrm{GL}_{N}, H)$ and $w_0'
\in W(\ch\mathrm{GL}_{N}(\lambda), \ch H)$ be the
long elements in their respective Weyl groups.
Set
  $$
  \delta'_0=\sigma_{w'_0}\sigma_{w_0}\inv\delta_{0}\in
  \GL_N\rtimes\langle\delta_{0}\rangle
  $$
  (see (\ref{GLnGamma})).

  \begin{lem}
    \label{dualparam}
    \begin{enumerate}[label={(\alph*)}]

\item  $(\delta'_0)^2=\exp(2\uppi i(\ch\rho-\ch\rho_\lambda)) \in
  Z(\mathrm{GL}_{N}(\lambda)).$

\item $\mathrm{GL}_{N}(\lambda) \rtimes \langle \delta_{0}' \rangle$ is an
  E-group for $\ch \mathrm{GL}_{N}(\lambda)$  in the sense of
  \cite{ABV}*{Definition 4.6},
  with second invariant
  $\exp(2\uppi i(\ch\rho-\ch\rho_\lambda))$.

\item  The  pair
 $(y,x)\in\ch\Xcal_\lambda\times \Xcal_{\ch\rho}$
is naturally  an Atlas parameter for an irreducible
$(\ch\gl(\lambda),\ch\wt K_y)$-module.
\end{enumerate}
\end{lem}

  \begin{proof}
    For part (a) we compute
    \begin{align*}
(\delta_{0}')^{2} &= \left(\delta_{0}' \sigma_{w_{0}'} (\delta_{0}')^{-1}\right) \,
      \left(\delta_{0}' \sigma_{w_{0}}^{-1} \delta_{0}\right)\\
      & = \sigma_{\delta_{0}'(w_{0})} \, \left(\sigma_{w_{0}'}
      \sigma_{w_{0}'}^{-1} \delta_{0} \sigma_{w_{0}}^{-1} \delta_{0}\right)\\
      & = \sigma_{w_{0}'} \left( \sigma_{w_{0}'} \sigma_{w_{0}}^{-1}
      \sigma_{\delta_{0}(w_{0})}^{-1}\right) \\
      & = \sigma_{w_{0}'}^{2} \sigma_{w_{0}}^{-2}.
    \end{align*}
using property \cite{AVParameters}*{(53g)} twice.  The final equality
is a consequence of \cite{AVParameters}*{Proposition 12.1}.

It is straightforward to show that conjugation by $\delta_{0}'$
preserves the pinning of $\mathrm{GL}_{N}(\lambda)$ obtained by
restricting the usual pinning of $\mathrm{GL}_{N}$.  This is all that
needs to be verified for part (b), once the definition of an E-group
is recalled.

For part (c), suppose $(x,y) \in \mathcal{X}_{\ch \rho}^{w} \times \ch
\mathcal{X}_{\lambda}^{ww_{0}}$ (Lemma \ref{XXXi}).  We must prove that
$$(y,x) \in \ch\mathcal{X}_{\lambda}^{ww_{0}} \times \mathcal{X}_{\ch
  \rho}^{ww_{0}w_{0}'}$$
relative to the extended groups
\begin{equation}
  \label{vogandualgroup}
\ch\mathrm{GL}_{N}(\lambda) \rtimes \langle \ch\delta_{0} \rangle
\quad \mbox{ and }\quad \mathrm{GL}_{N}(\lambda) \rtimes \langle \delta_{0}'
\rangle.
\end{equation}
It is a tautology that $y \in \ch\mathcal{X}_{\lambda}^{ww_{0}}$.  For
the class $x$ the corresponding statement follows from the fact that
$x$ acts on $H$ as
$$w\delta_{0} = ww_{0}w_{0}'\delta_{0}'.$$
The pair $(y,x)$ now determines a $(\ch \mathfrak{h}, \ch
H^{y})$-module of infinitesimal character $\ch\rho_{\lambda}$
\cite{AVParameters}*{Corollary 3.9}.  This is equivalent to a
$(\ch \mathfrak{h}, \widetilde{\ch  H^{y}})$-module of infinitesimal
character $\ch\rho$ (\cite{Knapp-Vogan}*{\emph{p.} 719}).   The latter module
then leads to a $(\ch\gl(\lambda),\ch\wt
K_y)$-module following the prescription of \cite{AVParameters}*{(20)}.
\end{proof}

Suppose $\xi\in \Xi(\O,\LGL)$ corresponds to
$(x,y)\in \Xcal_{\ch\rho}^{w} \times  \ch\Xcal^{ww_{0}}_\lambda$ as in
Lemma \ref{XXXi}.
We define
\begin{equation}
  \label{vogandual1}
  \ch\xi=(y,x)\in   \ch\Xcal^{ww_0}_\lambda \times \Xcal_{\ch\rho}^{ww_{0}w_{0}'}.
\nomenclature{$\ch\xi$}{dual parameter of $\xi$}
\end{equation}
By Lemma \ref{dualparam} (c), the Atlas parameter $(y,x)$
defines an irreducible
$(\ch\gl(\lambda),\ch\wt K_y)$-module, which we denote by
$\pi(\ch\xi)
\nomenclature{$\pi(\ch\xi)$}{Vogan dual of $\pi(\xi)$}
$.  The  $(\ch\gl(\lambda),\ch\wt K_y)$-module
$\pi(\ch\xi)$ is the Langlands quotient of a standard
$(\ch\gl(\lambda),\ch\wt K_y)$-module (\cite{AVParameters}*{(20)}),
which we denote by  $M(\ch\xi)
\nomenclature{$M(\ch\xi)$}{Vogan dual of $M(\xi)$}
$.
\begin{prop}
\label{PpimuMord}
 Under the bijection \eqref{rhbb} we have:

  \begin{itemize}
  \item[(a)]
$P(\xi)\mapsto \pi(\ch\xi)$
\item[(b)]
$(-1)^{d(\xi)} \mu(\xi)\mapsto  M(\ch\xi)$
\end{itemize}
\end{prop}

\begin{proof}  This proposition holds in greater generality, but is
  simpler for $\mathrm{GL}_{N}(\mathbb{R})$.
 Suppose  $\xi = (S, \tau_{S})$ corresponds to $(x,y)$ as in Lemma
 \ref{XXXi}. The
   $(\ch\gl(\lambda),\ch\wt K_y)$-module corresponding to $P(\xi)$
 under the Riemann-Hilbert
  and Beilinson-Bernstein correspondences is described by
  \cite{ABV}*{Proposition 6.16} and \cite{ICIII}*{Corollary 2.2,
    Proposition 2.7}.  These results tell us that the
  $(\ch\gl(\lambda),\ch\wt K_y)$-module
 is determined by an
  $(\ch\h,\wt{\ch H^y})$-module.  The character of $\wt{\ch H^y}$ in
 this  $(\ch\h,\wt{\ch H^y})$-module is completely determined by $\ch
 \rho$ and  $\tau_{S}$.  In our case the matter is
 simplified in that $\tau_{S}$
 is the trivial representation of a trivial component group.  This is
 also equivalent to the group
  $\ch H^{y}$ being connected, or to the fact
 that all Cartan subgroups of $\U(p,q)$ are connected.  In consequence
 the $(\ch\h,\wt{\ch H^y})$-module is determined entirely by the
infinitesimal character $\ch \rho$ specified on $\ch \h$.

  On the other hand, according to the proof of Lemma \ref{dualparam} (c), the Atlas
  parameter $(y,x)$
  determines an  irreducible
  $(\ch\gl(\lambda),\ch\wt K_y)$-module
  in terms of a   $(\ch\h,\wt{\ch H^y})$-module with infinitesimal
  character $\ch \rho$.  Since $\ch
  H^y$ is connected this  $(\ch\h,\wt{\ch H^y})$-module is determined
  by $\ch\rho$ alone, and is equal to the $(\ch\h,\wt{\ch H^y})$-module obtained
  from $P(\xi)$ above.  This proves (a).

For (b) we recall
(\ref{geomat}) and apply \cite{ICIII}*{Theorem 1.6} to obtain
\begin{equation}
  \label{dualmat}
\pi({^\vee}\xi) = \sum_{\xi' \in \Xi(\O,
  {^\vee}G^{\Gamma})}  c_{g}(\xi', \xi) \,
M({^\vee}\xi').
\end{equation}
(The absence of signs in (\ref{dualmat}) is due to the fact that the
sheaf on the
left-hand side of  \cite{ICIII}*{1.5} is equal to
$(-1)^{d(\delta)} P(\delta)$ according to the definitions of
\cite{ICIII}*{5.13} and \cite{ABV}*{(7.10)(e)}, see also
the proof of \cite{ABV}*{Proposition 16.13}).
The matrix $c_{g}$ is invertible and so (\ref{dualmat}) implies
\begin{equation}
  \label{dualmat1}
M({^\vee}\xi) =  \sum_{\xi' \in \Xi(\O,
  {^\vee}G^{\Gamma})}  c_{g}^{-1}(\xi', \xi) \, \pi({^\vee} \xi').
 \end{equation}
Similarly, by inverting the matrix $c_{g}$ in (\ref{geomat}), we obtain
\begin{equation}
  \label{geomat1}
(-1)^{d(\xi)} \mu(\xi) = \sum_{\xi' \in \Xi(\O,
  {^\vee}G^{\Gamma})}  c_{g}^{-1}(\xi', \xi) \,
P(\xi').
\end{equation}
By part (a) the Riemann-Hilbert and
Beilinson-Bernstein correspondences carry the right-hand side of
(\ref{geomat1}) to the right-hand side of (\ref{dualmat1}).
Therefore the left-hand sides correspond, which gives (b).
\end{proof}

\begin{cor}
\label{twistpairingun}
The pairing
\begin{equation}
  \label{pair3ord}
  \langle \cdot , \cdot \rangle:
  K\Pi(\O,\GL_{N}(\R))\times
  K\ch\Pi(\O,\GL_N(\R))\rightarrow\Z
\end{equation}
 defined by
  $$
  \langle M(\xi),M(\ch\xi')\rangle=(-1)^{l^{I}(\xi)}\delta_{\xi,\xi'}
  $$
  satisfies
  $$
  \langle \pi(\xi), \pi({^\vee}\xi') \rangle =
  (-1)^{l^{I}(\xi)}\delta_{\xi, \xi'}
  $$
\end{cor}
\begin{proof} By Proposition \ref{PpimuMord}, Theorem \ref{ordpairing}
  is equivalent to the assertion
that if a pairing
\begin{equation*}
\langle\,\cdot,\cdot\rangle':K\Pi(\O,G/\R)\times
K{}^{\vee}\Pi(\O,\LG)\rightarrow\Z
\end{equation*}
is defined by
\begin{equation*}
\langle M(\xi), M({}^{\vee}\xi') \rangle' = (-1)^{d(\xi)} \delta_{\xi, \xi'}
\end{equation*}
then
\begin{equation}\label{eq:ordpairingalternativeversion}
\langle \pi(\xi), \pi({^\vee}\xi') \rangle' =
   (-1)^{d(\xi)}  \delta_{\xi, \xi'}.
\end{equation}
By \cite{AMR1}*{Proposition B.1}
\begin{equation}
  \label{constl}
  (-1)^{d(\xi)}(-1)^{l^{I}(\xi)} = (-1)^{d(\xi) + l^{I}(\xi)} = (-1)^{c}
  \end{equation}
does not depend on $\xi$.
Therefore, the pairing in (\ref{pair3ord}) satisfies
$$\langle  \cdot, \cdot \rangle = (-1)^{c}\, \langle \cdot, \cdot
\rangle'.$$
The assertion of the corollary follows from
Equation (\ref{eq:ordpairingalternativeversion}) and Equation (\ref{constl}).
\end{proof}

  \subsection{Vogan Duality for twisted $\GL_N$}
  \label{twistedduality}

  In the previous section we replaced the sheaf-theoretic module
  $KX(\O, \ch \mathrm{GL}_{N}^{\Gamma})$ with the isomorphic
  representation-theoretic module $K\ch \Pi(\O,
  \mathrm{GL}_{N}(\mathbb{R}))$.
  We now wish to replace the twisted sheaf-theoretic module
  $KX(\O, \ch\mathrm{GL}_{N}^{\Gamma}, \upsigma)$ (\ref{twistsheafgroth})
  with a space of twisted characters,
  and hence restate Theorem \ref{twistpairing}
  with a statement about twisted representations analogous to
  Corollary \ref{twistpairingun}.
  The main tool is Vogan duality for
  the  disconnected group
  $\mathrm{GL}_{N}(\mathbb{R}) \rtimes \langle \vartheta \rangle$, as
  discussed in    \cite{AVParameters}*{Sections 6.1 and 11}.
  The arguments closely follow those of \cite{AVParameters}.
  Nevertheless it is worth giving some details because we need
  the case of non-integral infinitesimal character which
  is not covered by \cite{AVParameters}*{Section 11}. See
  Section \ref{heckesection}.

    By analogy with (\ref{twistsheafgroth}) we define
    \begin{equation*}
    K\ch\Pi(\O,\GL_N(\R),\vartheta)=
    K \ch\Pi(\O,\GL_N(\R))^{\vartheta}
    \otimes \Z[U_2]/\langle \pi({^\vee}\xi) \otimes 1 ) +
(\pi(^{\vee}\xi)\otimes -1)\rangle
\nomenclature{$K\ch\Pi(\O,\GL_N(\R),\vartheta)$}{}
\nomenclature{$K\ch\Pi(\O,\GL_N(\R))^{\vartheta}$}{}
\end{equation*}
where the complete geometric parameters $\xi$ run over $\Xi(\O,\LGL)^\vartheta$.

Using Propositions \ref{p:rhbb} and \ref{PpimuMord}, we define a bijection
$$P(\xi)^+\mapsto \pi(\ch\xi)^+, \quad \xi \in
\Xi(\O,\LGL)^\vartheta.$$
The extended representation $\pi(\ch\xi)^+$ on the right is obtained
by Vogan duality
from $\pi(\xi)^{+}$ as in \cite{AVParameters}*{Corollary 6.4}.  The bijection yields an isomorphism
\begin{equation}
\label{sheaftorep}
K X(\O, {^\vee}\mathrm{GL}_{N}^{\Gamma}, \upsigma) \cong K
       {^\vee}\Pi(\O,
       \mathrm{GL}_{N}(\mathbb{R}), \vartheta).
     \end{equation}

\begin{prop}
\label{PpimuM}
 Under the isomorphism \eqref{sheaftorep}
$$(-1)^{d(\xi)} \mu(\xi)^+\mapsto  M(\ch\xi)^+, \quad \xi \in
 \Xi(\O,\LGL)^\vartheta.$$
\end{prop}

\begin{proof}
  Define a $\mathbb{Z}$-linear map
$$B:K X(\O, \ch \mathrm{GL}_{N}^{\Gamma})^{\vartheta} \otimes
  \mathbb{Z}[U_2] \rightarrow   K \ch\Pi(\O,\GL_N(\R))^{\vartheta}
    \otimes \Z[U_2]$$
by setting
$$B(P(\xi)^{+})  = \pi(\ch \xi)^{+} \mbox{ and } B(P(\xi)^{-})  =
 \pi(\ch \xi)^{-}, \quad \xi \in \Xi(\O, \ch
 \mathrm{GL}_{N}^{\Gamma})^{\vartheta}.$$
 Recall Equation (\ref{extsheafmult}).
The matrix $c_{g}$ given by this equation is invertible.  We may therefore
invert Equation  (\ref{extsheafmult}) by
writing
$$
(-1)^{d(\xi)} \mu(\xi)^{+} = \sum_{\xi' \in \Xi(\O,
  {^\vee}\mathrm{GL}_{N}^{\Gamma})^{\vartheta}}  c_{g}^{-1}(\xi'_{+},
\xi_{+}) \, P(\xi')^{+} +  c_{g}^{-1}(\xi'_{-}, \xi_{+}) \,
P(\xi')^{-} + \cdots
$$
The projection of this equation to $K X(\O, \ch
\mathrm{GL}_{N}^{\Gamma})^{\vartheta} \otimes \mathbb{Z}[U_2]$ is
\begin{equation}
  \label{invextsheafmult}
\sum_{\xi' \in \Xi(\O,
  {^\vee}\mathrm{GL}_{N}^{\Gamma})^{\vartheta}}  c_{g}^{-1}(\xi'_{+},
\xi_{+}) \, P(\xi')^{+} +  c_{g}^{-1}(\xi'_{-}, \xi_{+}) \,
P(\xi')^{-}.
\end{equation}
Applying $B$ to (\ref{invextsheafmult}), we obtain
\begin{align}
 \nonumber
& B((-1)^{d(\xi)} \mu(\xi)^{+})\\
\nonumber &= B\left(\sum_{\xi' \in \Xi(\O,
  {^\vee}\mathrm{GL}_{N}^{\Gamma})^{\vartheta}}  c_{g}^{-1}(\xi'_{+},
\xi_{+}) \, P(\xi')^{+} +  c_{g}^{-1}(\xi'_{-}, \xi_{+}) \,
P(\xi')^{-} \right)\\
\label{Bdecomp}
& = \sum_{\xi' \in \Xi(\O,
  {^\vee}\mathrm{GL}_{N}^{\Gamma})^{\vartheta}}  c_{g}^{-1}(\xi'_{+},
\xi_{+}) \, \pi(\ch \xi')^{+} +  c_{g}^{-1}(\xi'_{-}, \xi_{+}) \,
\pi(\ch \xi')^{-}
\end{align}
The sum on the right is a formal sum of extensions of
$(\ch\gl(\lambda),\ch\wt K_y)$-modules (Lemma \ref{dualparam} (c)) to
$(\ch\gl(\lambda),\ch\wt K_y \rtimes \langle \vartheta
\rangle)$-modules.   Since both $\pi(\ch \xi)^{+}$ and $\pi(\ch
\xi)^{-}$ restrict to the same $(\ch\gl(\lambda),\ch\wt K_y)$-module
$\pi(\ch \xi)$, we write the restriction of this sum as
\begin{equation}
  \label{restB}
 B((-1)^{d(\xi)} \mu(\xi)^{+})_{|K\ch\Pi(\O,
    \mathrm{GL}_{N}(\mathbb{R}))^{\vartheta}} = \sum_{\xi' \in \Xi(\O,
  {^\vee}\mathrm{GL}_{N}^{\Gamma})^{\vartheta}}  (c_{g}^{-1}(\xi'_{+},
\xi_{+}) +  c_{g}^{-1}(\xi'_{-}, \xi_{+})) \,
\pi(\ch \xi').
\end{equation}
In a similar manner we apply to equation
(\ref{invextsheafmult}) the forgetful functor which takes
$(\ch\mathrm{GL}_{N} \rtimes \langle \upsigma \rangle)$-equivariant
sheaves to $\ch \mathrm{GL}_{N}$-equivariant sheaves.  The result is
$$(-1)^{d(\xi)}  \mu(\xi) = \sum_{\xi' \in \Xi(\O,
  {^\vee}\mathrm{GL}_{N}^{\Gamma})^{\vartheta}}  (c_{g}^{-1}(\xi'_{+},
\xi_{+}) + c_{g}^{-1}(\xi'_{-}, \xi_{+})) \,
P(\xi').$$
Comparing this equation with (\ref{geomat1}), we see that
$$c_{g}^{-1}(\xi'_{+},
\xi_{+}) + c_{g}^{-1}(\xi'_{-}, \xi_{+}) = c_{g}^{-1} (\xi', \xi).$$
Consequently, equation (\ref{restB}) takes the form
$$B((-1)^{d(\xi)} \mu(\xi)^{+})_{|K\ch\Pi(\O,
  \mathrm{GL}_{N}(\mathbb{R}))^{\vartheta}} = \sum_{\xi' \in
  \Xi(\O,
  {^\vee}\mathrm{GL}_{N}^{\Gamma})^{\vartheta}}  c_{g}^{-1}(\xi',
\xi) \,\pi(\ch \xi'),$$
and by (\ref{dualmat1})
$$B((-1)^{d(\xi)} \mu(\xi)^{+})_{|K\ch\Pi(\O, \mathrm{GL}_{N}(\mathbb{R}))^{\vartheta}} = M(\ch \xi).$$
The standard module $M(\ch \xi)$ has exactly two extensions $M(\ch
\xi){^\pm}$.  We need to show
 $$B((-1)^{d(\xi)} \mu(\xi)^{+}) = M(\ch \xi)^{+}$$
and for this it suffices to prove that $\pi(\ch \xi)^{+}$ occurs in
$M(\ch \xi)^{+}$ as a (sub)quotient.  Looking back to (\ref{Bdecomp}),
the latter is equivalent to proving that $c_{g}^{-1}(\xi_{+}, \xi_{+})
\neq 0$.  Looking a bit further back to (\ref{invextsheafmult}) we see
that this amounts to $P(\xi)^{+}$ appearing in the decomposition of
$\mu(\xi)^{+}$, and this is true by definition (see the proof of Lemma
\ref{cansheaf}).
\end{proof}

Using Proposition \ref{PpimuM} we can restate Theorem \ref{twistpairing}.

\begin{lem}
\label{twistpairing2}
Theorem \ref{twistpairing} is equivalent to the following assertion.
The pairing
\begin{equation}
  \label{pair3}
\langle \cdot , \cdot \rangle: K\Pi(\O,
  \mathrm{GL}_{N}(\mathbb{R}), \vartheta) \times K
{^\vee}\Pi(\O, \mathrm{GL}_{N}(\mathbb{R}),
\vartheta) \rightarrow \mathbb{Z}
\end{equation}
  defined by
$$\langle M(\xi)^{+}, M({^\vee}\xi')^{+} \rangle  =  (-1)^{
    l^{I}_{\vartheta}(\xi)} \, \delta_{\xi, \xi'}$$
satisfies
$$\langle \pi(\xi)^{+}, \pi({^\vee}\xi')^{+} \rangle =
  (-1)^{l^{I}_{\vartheta}(\xi)} \, \delta_{\xi, \xi'}$$
where $\xi, \xi' \in
 \Xi(\O, {^\vee}\mathrm{GL}_{N}^{\Gamma})^\vartheta$.
\end{lem}
\begin{proof}
We just need to notice that
by Proposition \ref{PpimuM}, Theorem \ref{twistpairing} is equivalent to the assertion
that if a pairing
\begin{equation*}
\langle \cdot , \cdot \rangle: K\Pi(\O,
  \mathrm{GL}_{N}(\mathbb{R}), \vartheta) \times K
{^\vee}\Pi(\O, \mathrm{GL}_{N}(\mathbb{R}),
\vartheta) \rightarrow \mathbb{Z}
\end{equation*}
is defined by
\begin{equation*}
\langle M(\xi)^{+}, M({^\vee}\xi')^{+} \rangle  =   (-1)^{d(\xi)}(-1)^{l^{I}(\xi)-l^{I}_{\vartheta}(\xi)} \, \delta_{\xi, \xi'}
\end{equation*}
then
$$\langle \pi(\xi)^{+}, \pi({^\vee}\xi')^{+} \rangle' =
   (-1)^{d(\xi)}(-1)^{l^{I}(\xi)-l^{I}_{\vartheta}(\xi)} \, \delta_{\xi, \xi'}.$$
The proof then follows exactly like that of Corollary \ref{twistpairingun},
we leave the details to the reader.
\end{proof}

\subsection{Twisted Hecke modules}
\label{twisthmodule}

The proof of Theorem \ref{twistpairing} relies on a Hecke algebra
and Hecke modules as in the ordinary, non-twisted setting of Sections
\ref{bbvd}-\ref{duality}.  In the twisted setting, Lusztig and Vogan
define a Hecke algebra which we denote by $\mathcal{H}(\lambda)$
\cite{LV2014}*{Section 3.1}.  This Hecke algebra
acts on the Hecke modules
$$
\mathcal{K} \Pi(\O,
\mathrm{GL}_{N}(\mathbb{R}), \vartheta) = K \Pi(\O,
\mathrm{GL}_{N}(\mathbb{R}), \vartheta) \otimes_{\mathbb{Z}}
\mathbb{Z}[q^{1/2}, q^{-1/2}]
\nomenclature{$\mathcal{K} \Pi(\O,
\mathrm{GL}_{N}(\mathbb{R}), \vartheta)$}{twisted Hecke module}
$$
and
$$
\mathcal{K} {^\vee} \Pi(\O, \mathrm{GL}_{N}(\mathbb{R}), \vartheta) =
K {^\vee} \Pi(\O, \mathrm{GL}_{N}(\mathbb{R}), \vartheta)
\otimes_{\mathbb{Z}} \mathbb{Z}[q^{1/2}, q^{-1/2}]
\nomenclature{$\mathcal{K} {^\vee} \Pi(\O, \mathrm{GL}_{N}(\mathbb{R}))$}{dual twisted Hecke module}
$$
as in \cite{LV2014}*{Section 7}.  We shall extend the pairing
(\ref{pair3}) to these Hecke modules.
Once the Hecke algebra action is supplemented with \emph{Verdier duality}
\cite{LV2014}*{Section 2.4}, we present special bases of the Hecke
modules, essentially eigenvectors of Verdier duality.  Theorem
\ref{twistpairing} will be seen to follow from a theorem expressing
the values of the pairing on the special bases (Theorem \ref{pairingC}).

We continue with a closer look at the Hecke algebra $\mathcal{H}(\lambda)$.
Let $\kappa$ be a $\vartheta$-orbit on the set of simple
\nomenclature{$\kappa$}{$\vartheta$-orbit of simple roots}
roots  of $R^{+}(\lambda)$.
The orbit $\kappa$ is equal to one of the following:
\begin{align}
\nonumber\text{one root }&\{\alpha=\vartheta(\alpha)\} &(\text{type }
1)~ & \\
 \text{two roots
}&\{\alpha, \beta=\vartheta(\alpha)\},~\quad\left<\alpha, {^\vee}\beta\right>=0
&(\text{type } 2)~ \label{simplekappa}\\
\nonumber \text{two roots }&\{\alpha,\beta=\vartheta(\alpha)\},~\quad\left<
\alpha, {^\vee}\beta\right>=-1 & (\text{type } 3).
\end{align}
Write $W(\lambda)$ for the Weyl group of the integral roots
$R(\lambda)$, and let
$$W(\lambda)^{\vartheta}=\{w\in W(\lambda):\vartheta(w)=w \}.$$
\nomenclature{$W(\lambda)^{\vartheta}$}{$\vartheta$-fixed integral Weyl group}
The group $W(\lambda)^{\vartheta}$
is a Coxeter group (\cite{LV2014}*{Section 4.3}) with generators
\begin{equation}
\label{simplekappa1}
w_{\kappa}=\left\{\begin{array}{ll}
s_{\alpha}& \kappa \text{ type }1\\
s_{\alpha}s_{\beta}& \kappa \text{ type }2\\
s_{\alpha}s_{\beta}s_{\alpha}& \kappa \text{ type }3.
\end{array}
\right.
\end{equation}
The Hecke algebra $\mathcal{H}(\lambda)$ (\cite{AVParameters}*{Section 10},
\cite{LV2014}*{Section 4.7})
\nomenclature{$\mathcal{H}(\lambda)$}{twisted Hecke algebra}
is a free  $\mathbb{Z}[q^{1/2},q^{-1/2}]$-algebra with basis
\begin{equation*}
  \label{heckeop}
  \{T_{w}:w\in W(\lambda)^{\vartheta} \}.
  \end{equation*}
It is a consequence of \cite{LV2014}*{Equation 4.7 (a)}
that $\mathcal{H}(\lambda)$ is generated by
the operators  $T_{\kappa}:=T_{w_{\kappa}}
\nomenclature{$T_{\kappa}$}{Hecke operator}
$, where  $\kappa$  is a
$\vartheta$-orbit as in (\ref{simplekappa}).

Before we move to a discussion of $\mathcal{H}(\lambda)$-modules,
we digress on how the $\vartheta$-orbits $\kappa$ are further categorized
relative to a fixed parameter $\xi \in \Xi(\O,
{^\vee}\mathrm{GL}_{N}^{\Gamma})^\vartheta$. The parameter $\xi \in
\Xi(\O, {^\vee}\mathrm{GL}_{N}^{\Gamma})^\vartheta$
is equivalent to an Atlas parameter $(x,y)$ as in Lemma \ref{XXXi}.  The
adjoint action of $x$ acts as an involution on $R(\lambda)$.
This action separates the
$\vartheta$-orbits of roots into various types, \emph{e.g.} real,
imaginary, etc.  Lusztig and Vogan
combine this information with the types of (\ref{simplekappa}) and
also with the types defined by Vogan in \cite{greenbook}*{Section 8.3}.  The
interested reader must be vigilant
in distinguishing between these three kinds of types!  The list of
combined types may be found in \cite{LV2014}*{Section 7} or
\cite{AVParameters}*{Table 1}.

Not all of the types that appear in this list
are relevant for
$\mathrm{GL}_{N}(\mathbb{R})$.  For example the classification of
roots in  \cite{greenbook}*{Section 8.3} labels the roots
as either  type I or  type II, and it is well-known that roots of
$\mathrm{GL}_{N}(\mathbb{R})$ are all of type II.
Another well-known
fact is that $\mathrm{GL}_{N}(\mathbb{R})$ has no compact roots
relative to $x$ in the sense of \cite{Knapp}*{Section VI.3}.
Using these two facts, it is
tedious, but simple, to verify
that the only relevant types for
$\mathrm{GL}_{N}(\mathbb{R})$ in \cite{AVParameters}*{Table 1} are labelled as
\begin{align}
  \label{glntypes}
  &\mathtt{1C+, 1C-, 1i2f, 1i2s, 1r2, 1rn, 2C+, 2C-,2Ci,}\\
\nonumber   &\mathtt{  2Cr, 2i22,
    2r22, 2rn, 3C+, 3C-, 3Ci, 3r, 3rn.}
\end{align}
Any $\vartheta$-orbit $\kappa$ also has a type relative to the dual
parameter ${^\vee}\xi$ (\ref{vogandual1}).  The dual parameter is
equivalent to the Atlas parameter $(y,x)$ and the adjoint action of
$y$ is essentially the negative of the adjoint action of $x$
(\cite{AVParameters}*{Definition 3.10}). In
consequence it is easy to convert the types of (\ref{glntypes}) into
types for the Vogan dual group  $\ch \mathrm{GL}_{N}(\lambda) \rtimes
\langle \ch \delta_{0} \rangle$ ((\ref{vogandualgroup}),
\cite{AVParameters}*{Section 11 and Table 5}). They are
\begin{align}
  \label{dualglntypes}
  &\mathtt{1C-, 1C+, 1r1f, 1r1s, 1i1, 1ic, 2C-, 2C+,2Cr,}\\
\nonumber   &\mathtt{  2Ci, 2r11,
    2i11, 2ic, 3C-, 3C+, 3Cr, 3i, 3ic.}
\end{align}

Let us return to the subject of Hecke modules.
In \cite{LV2014}*{Section 4} and \cite{AVParameters}*{Section7} it is explained how
$\mathcal{K}\Pi(\O, \mathrm{GL}_{N}(\mathbb{R}),\vartheta)$
can be made
into a Hecke module by defining the action of the operators $T_{\kappa}$
on the generating set $\{M(\xi)^{+}:\xi\in
\Xi(\O,{^\vee}\mathrm{GL}_{N}^{\Gamma})^\vartheta \}$.  The
actions are computed explicitly in a geometric setting in  \cite{LV2014}*{Section 7}, and
are presented in terms of extended Atlas parameters in \cite{AVParameters}*{Proposition 10.4}.  A case-by-case summary of the
actions is given in \cite{AVParameters}*{Table 5},
according to the categorization of (\ref{glntypes}).

The construction defining the Hecke algebra
$\mathcal{H}(\lambda)$ and the Hecke module structure for the module $\mathcal{K}
\Pi(\O, \mathrm{GL}_{N}(\mathbb{R}), \vartheta)$ in
\cite{LV2014}, also defines a Hecke algebra
${^\vee}\mathcal{H}(\lambda)$ and a Hecke module structure for
$\mathcal{K} {^\vee}\Pi(\O,
\mathrm{GL}_{N}(\mathbb{R}), \vartheta)$.  The Hecke module actions
in this case are again given in \cite{AVParameters}*{Table 5} in terms
of (\ref{dualglntypes}).
The Hecke algebra
${^\vee}\mathcal{H}(\lambda)$ for the Vogan dual group (\ref{vogandualgroup}) is
generated by Hecke operators
$T_{{^\vee}\kappa}$,  where ${^\vee}\kappa$ runs over the simple coroots
corresponding to $\kappa$.
\nomenclature{$T_{{^\vee}\kappa}$}{Hecke operator}
The bijection between the two sets of operators
$$\{T_{\kappa}: \kappa \in R^{+}(\lambda) \mbox{ simple}\}
\longleftrightarrow \{T_{{^\vee}\kappa} : \kappa \in R^{+}(\lambda) \mbox{
  simple}\}$$
extends to an isomorphism $\mathcal{H}(\lambda) \cong
{^\vee}\mathcal{H}(\lambda)$.
In this manner, we also regard
$\mathcal{K} {^\vee}\Pi(\O,
{^\vee}\mathrm{GL}_{N}(\mathbb{R}), \vartheta)$ as an
$\mathcal{H}(\lambda)$-module.

There is a partial order on $\Xi(\O,
{^\vee}\mathrm{GL}_{N}^{\Gamma} )^\vartheta$, the \emph{Bruhat
  order} which is defined
geometrically (\cite{LV2014}*{Section 5.1}, \cite{ABV}*{(7.11)(f)}).
The Bruhat order for the dual
parameters ${^\vee}\Xi(\O,
{^\vee}\mathrm{GL}_{N}^{\Gamma})^\vartheta$
is defined by the inverse order
\begin{equation}
\label{dualBruhat}
{^\vee}\xi \geq {^\vee}\xi' \Longleftrightarrow \xi \leq \xi', \ \xi, \xi'
\in \Xi(\O,
{^\vee}\mathrm{GL}_{N}^{\Gamma} )^\vartheta.
\end{equation}

We now return to  the pairing (\ref{pair3}) and extend it to a Hecke
module pairing
\begin{equation}
  \label{pair4}
\langle \cdot  , \cdot \rangle:\mathcal{K}\Pi(\O,
\mathrm{GL}_{N}(\mathbb{R}),\vartheta) \times \mathcal{K} {^\vee}\Pi(
\O,{^\vee} \mathrm{GL}_{N}(\mathbb{R}), \vartheta) \rightarrow
\mathbb{Z}[q^{1/2},q^{-1/2}],
\end{equation}
by setting
$$\langle M(\xi)^{+}, M({^\vee}\xi')^{+}\rangle  = (-1)^{l^{I}_{\vartheta}(\xi)}
\, q^{\left(l^{I}(\xi)+l^{I}({^\vee}\xi')\right)/2}\, \delta_{\xi, \xi'}$$
for all $\xi, \xi' \in \Xi(\O,
{^\vee}\mathrm{GL}_{N}^{\Gamma})^\vartheta$.  In view of the
Kronecker delta, the term
$q^{1/2\left(l^{I}(\xi)+l^{I}({^\vee}\xi')\right)}$ in the pairing
could be replaced by
$q^{1/2\left(l^{I}(\xi)+l^{I}({^\vee}\xi)\right)}$ or
$q^{1/2\left(l^{I}(\xi')+l^{I}({^\vee}\xi')\right)}$.  In fact, both
of the latter terms are independent of $\xi$ or $\xi'$, as may be seen
by the following lemma.
\begin{lem}
\label{indlength}
Suppose $\xi \in \Xi(\O,
{^\vee}\mathrm{GL}_{N}^{\Gamma})^\vartheta$.  Then
$$l^{I}(\xi)+l^{I}({^\vee}\xi)= -\frac{1}{2} \left( |R^{+}(\lambda)| +
\dim(H) \right)$$
\end{lem}
\begin{proof}
For simplicity let us identify the dual group
${^\vee}\mathrm{GL}_{N}(\lambda)$ with
$\mathrm{GL}_{N}(\lambda)$.   Similarly, we
identify ${^\vee}H$ with $H$.
Let $(x,y)$ be the Atlas parameter of $\xi$ as in (\ref{avbij}).
Then $(y,x)$ is the Atlas parameter for ${^\vee}\xi$
(\ref{vogandual1}), where the adjoint action of $y$ on
$\mathfrak{h}$ is the negative of the adjoint action of $x$ on
$\mathfrak{h}$ (\cite{AVParameters}*{Definition 3.10}).  From definition
(\ref{intlength}) we compute that
\begin{align*}
&l^{I}(\xi) + l^{I}({^\vee}\xi)\\
 &= -\frac{1}{2} \left( | \{\alpha \in R^{+}(\lambda) :
x\cdot \alpha \in R^{+}(\lambda) \}| + \dim(H^{x}) \right) -\frac{1}{2} \left( | \{\alpha \in R^{+}(\lambda) :
y\cdot \alpha \in R^{+}(\lambda) \}| + \dim(H^{y}) \right)  \\
& = -\frac{1}{2} \left( | \{\alpha \in R^{+}(\lambda) :
x\cdot \alpha \in R^{+}(\lambda) \}|  + | \{\alpha \in R^{+}(\lambda) :
-(x\cdot \alpha) \in R^{+}(\lambda) \}| +  \dim(H^{x})+ \dim(H^{-x}) \right)  \\
& = -\frac{1}{2} \left( |R^{+}(\lambda)| + \dim(H) \right).
\end{align*}
\end{proof}

\subsection{The Hecke module isomorphism}
\label{heckesection}

The extended pairing  (\ref{pair4}) induces a $\Z$-module isomorphism
\begin{align}
\label{eq:dualformalmap}
\mathcal{K} \Pi(\O, \mathrm{GL}_{N}(\mathbb{R}),
\vartheta) &\rightarrow
\mathcal{K} {^\vee}\Pi(\O, \mathrm{GL}_{N}(\mathbb{R}),\vartheta)^{*}\\
M(\xi)^{+} &\mapsto \langle M(\xi)^{+}, \cdot \,  \rangle\nonumber
\end{align}
We endow $\mathcal{K}
{^\vee}\Pi(\O, \mathrm{GL}_{N}(\mathbb{R}),\vartheta)^{*}$
with the Hecke module structure given in \cite{AVParameters}*{Section 11}.
The main goal of this section is
\begin{prop}
  \label{Hisomorphism}
The map (\ref{eq:dualformalmap}) is an isomorphism of
$\mathcal{H}(\lambda)$-modules.
\end{prop}

This is a generalization of \cite{AVParameters}*{Proposition 11.2},
which is stated only in the case of integral infinitesimal character.
In the course of the proof we correct a sign in
\cite{AVParameters}*{Proposition 11.2}.

We first describe
the $\mathcal{H}(\lambda)$-action in more detail.
Since $\mathcal{H}(\lambda)$ is not commutative, one cannot
define a Hecke action on
$\mathcal{K} {^\vee}\Pi(\O,
\mathrm{GL}_{N}(\mathbb{R}),\vartheta)^{*}$
merely by transposing the action on  $\mathcal{K}
{^\vee}\Pi(\O, \mathrm{GL}_{N}(\mathbb{R}),\vartheta)$.  One
must include  an anti-automorphism of
$\mathcal{H}(\lambda)$ defined by
$$T_{w}\longmapsto
(-1)^{l_{\vartheta}(w)}q^{l(w)}T_{w}^{-1},\quad w\in W(\lambda)^{\vartheta},$$
(\emph{cf.} \cite{AVParameters}*{(50)} (removing $q$ on the left),
\cite{ABV}*{(17.15)(c)}).  Here, $l(w)$ is the length of $w$ with
respect to the simple reflections  in $W(\mathrm{GL}_{N}, H)$, and
\nomenclature{$l(w)$}{length in Weyl group}
$l_{\vartheta}(w)$ is the length of $w$ with respect to the generators
of (\ref{simplekappa1}).
The $\mathcal{H}(\lambda)$-action on $\mathcal{K} {^\vee}\Pi(\O,
\mathrm{GL}_{N}(\mathbb{R}),\vartheta)^{*}$ is defined by
\begin{align}\label{eq:dualheckeaction}
  {T}_{w}^{\ast}: \mathcal{K}
{^\vee}\Pi(\O, \mathrm{GL}_{N}(\mathbb{R}),\vartheta)^{*}
  &\rightarrow \mathcal{K}
{^\vee}\Pi(\O, \mathrm{GL}_{N}(\mathbb{R}),\vartheta)^{*} \\
T_{w}^{*}\cdot
\mu&= (-1)^{l_{\vartheta}(w)}q^{l(w)}(T_{w}^{-1})^{t}\cdot \mu,
\nonumber
\end{align}
where $(T_{w}^{-1})^{t}$ is the transpose of $T_{w}^{-1}$.
According to \cite{LV2014}*{Equation 7.2(a)}, for any $\vartheta$-orbit
$\kappa$ of a simple root in $R(\lambda)$ we have
\begin{equation}
  \label{quadrel}
(T_{w_{\kappa}}+1)(T_{w_{\kappa}}-q^{l(w_{\kappa})})=0.
\end{equation}
From this, the inverse of $T_{w_{\kappa}}$ may be computed to be
$$T_{w_{\kappa}}^{-1}=q^{-l(w_{\kappa})}T_{w_{\kappa}}+(q^{-l(w_{\kappa})}-1).$$
The Hecke action of (\ref{eq:dualheckeaction}) for  a generator
therefore takes  the  form
\begin{equation}
\label{17.15e}
T^{*}_{w_{\kappa}}\cdot \mu = -(T_{w_{\kappa}})^{t}\cdot
\mu+(q^{l(w_{\kappa})}-1)\mu.
\end{equation}

The Hecke module structures of both the domain and codomain in the
proposed isomorphism (\ref{eq:dualformalmap}) are now established, and both
are presented in \cite{AVParameters}*{Table 5} in terms of cross
actions and Cayley transforms (\cite{AVParameters}*{Tables 2-4}).  Our
next goal is to show that cross actions and Cayley transforms commute
with Vogan duality (\ref{vogandual1}).  In
the following two propositions we distinguish in notation between
$\alpha \in R(\lambda) \subset R(\mathrm{GL}_{N}, H)$ and
${^\vee}\alpha \in R({^\vee}\mathrm{GL}_{N}(\lambda), \, {^\vee}H)$
although the reader may prefer to identify these two roots.
\begin{prop}
\label{crossvogandual}
Fix a complete geometric parameter $\xi \in \Xi(\O,
{^\vee}\mathrm{GL}_{N}^{\Gamma})^\vartheta$ and $w_{\kappa} \in
W(\lambda)^{\vartheta}$ as in (\ref{simplekappa1}).   Then
\begin{enumerate}[label={(\alph*)}]
\item ${^\vee}(w_{\kappa} \times M(\xi)) = w_{{^\vee}\kappa}  \times M({^\vee}\xi)$
\item ${^\vee}(w_{\kappa} \times M(\xi)^{+}) = w_{{^\vee}\kappa}
  \times M({^\vee}\xi)^{+}$
\end{enumerate}
where on the left of (a) and (b), we use the convention ${^\vee}M(\xi') = M({^\vee}\xi')$ for any $\xi' \in \Xi(\O,
{^\vee}\mathrm{GL}_{N}^{\Gamma})^\vartheta$.
\end{prop}
\begin{proof}
We identify $\xi$ with its equivalent
Atlas parameter $(x,y)$ of Lemma \ref{XXXi}. The dual
parameter ${^\vee} \xi$ is then identified with the Atlas parameter
$(y,x)$ (\ref{vogandual1}).  By the definition of cross action in
\cite{Adams-Fokko}*{(9.11 f)},
\begin{equation}
\label{crossaction}
w_{\kappa} \times M(x,y) = M(\dot{w} x\dot{w}^{-1}, \dot{w}y\dot{w}^{-1})
\nomenclature{$w_{\kappa} \times$}{cross action}
\end{equation}
where $\dot{w} \in \mathrm{GL}_{N}(\lambda)$ is any representative for
$w_{\kappa}$.  By (\ref{vogandual1})
$${^\vee} (w_{\kappa} \times M(x,y)) = \, {^\vee} M(
\dot{w}x\dot{w}^{-1}, \dot{w}y\dot{w}^{-1}) = M(
\dot{w}y\dot{w}^{-1},\dot{w}x\dot{w}^{-1}) = w_{{^\vee} \kappa}
\times M(y,x)$$
This proves the first assertion of the proposition.

For part (b), we first claim that
\begin{equation}
\label{crossatlas}
w_{\kappa} \times M(\xi)^{+} = (w_{\kappa} \times M(\xi))^{+}.
\end{equation}
In other words the cross action carries an Atlas extension to an Atlas
extension (\ref{canext}).  When the $\vartheta$-orbit $\kappa$ of a
simple root in $R(\lambda)$ is also comprised of simple
roots in $R(\mathrm{GL}_{N}, H)$, this may be seen by noting that an
extended parameter $(\uplambda, \uptau, 0 ,0 )$ for $M(\xi)^{+}$
(\ref{quad}) is carried to an extended parameter of the form
$(\uplambda', \uptau', 0 ,0)$ under all instances of cross action in
\cite{AVParameters}*{Tables 2-4}.   For general $\kappa$ this fact
follows from \cite{AVParameters}*{Definition 10.4 and Tables 2-4}.
(This is a special property of the group $\mathrm{GL}_{N} \rtimes
\langle \delta_{0} \rangle$ which avoids ``bad" roots such as those of
type \texttt{2i12} in the tables.)

Taking the Vogan dual of (\ref{crossatlas}), we obtain
$${^\vee}(w_{\kappa} \times M(\xi)^{+}) = \, {^\vee}((w_{\kappa} \times M(\xi))^{+}) =
 (\, {^\vee}(w_{\kappa} \times M(\xi)))^{+} = (w_{{^\vee}\kappa}
\times M({^\vee}\xi))^{+}.$$
Here, the second equality follows from the definition of Vogan
dual for an Atlas extension  (\cite{AVParameters}*{Corollary 6.4}),  and the final equality
follows from the first assertion of the proposition.

To complete the second assertion, we must prove
$$(w_{{^\vee}\kappa} \times M({^\vee}\xi))^{+} = w_{{^\vee}\kappa}
\times M({^\vee} \xi)^{+},$$
which is analogous to (\ref{crossatlas}).  However, unlike
(\ref{crossatlas}) this identity is to be proved using
\cite{AVParameters}*{Tables 2-4} for the dual group rather than
for $\mathrm{GL}_{N} \rtimes  \langle \delta_{0} \rangle$.  Once again
we turn to extended parameters.  If $(\uplambda, \uptau, 0, 0)$
(\ref{quad}) is an
extended parameter corresponding to $M(\xi)^{+} = M(x,y)^{+}$ then  an
extended parameter corresponding to $M({^\vee}\xi )^{+} = M(y,x)^{+}$
may be chosen to have the form $(0,0,\ell',t')$.  Indeed, the zeros in
the first two entries satisfy the requisite equations of \cite{AVParameters}*{Propositions 3.8 and 4.5} when regarding $M(y,x)$ as a
$(\mathfrak{gl}_{N}(\lambda), \ch \widetilde{K}_{y_{\xi}})$-module of infinitesimal
character $\rho_{\lambda}$ (Section \ref{bbvd}).  As
\cite{AVParameters}*{Tables 2-4} indicate, any cross action from
(\ref{dualglntypes})  applied to
$(0,0, \ell',t')$ yields an extended parameter with zeros in the first
two entries. This
means that $w_{{^\vee} \kappa} \times M({^\vee}\xi)^{+}$ has an
extended parameter with zero in its first entry.  According to \cite{AVParameters}*{Lemma 5.3.1} this extended parameter corresponds to the
Atlas extension $(w_{{^\vee}\kappa} \times M({^\vee}\xi))^{+}$, and
the proposition is complete.
\end{proof}

\begin{prop}
\label{cayleyvogandual}
Fix a complete geometric parameter $\xi \in \Xi(\O,
{^\vee}\mathrm{GL}_{N}^{\Gamma})^\vartheta$ and suppose that
$\kappa$ as in (\ref{simplekappa}) allows for a Cayley transform
$c_{\kappa}$ (\cite{AVParameters}*{(42e)})\footnote{Contrary to custom,
  we  leave $\kappa$ in subscript regardless of whether the
  Cayley transform is made relative to
   real, complex or imaginary roots.}.  Then
\begin{enumerate}[label={(\alph*)}]
\item ${^\vee}\left(c_{\kappa} (M(\xi))\right) = c_{{^\vee}\kappa} (M({^\vee}\xi))$
\item ${^\vee}\left(c_{\kappa} (M(\xi)^{+}) \right) =
  c_{{^\vee}\kappa}  (M({^\vee}\xi)^{+})$
\end{enumerate}
where on the left of (a) and (b), we use the convention ${^\vee}M(\xi') = M({^\vee}\xi')$ for any $\xi' \in \Xi(\O,
{^\vee}\mathrm{GL}_{N}^{\Gamma})^\vartheta$.
\end{prop}
\begin{proof}
We see no means for avoiding a case-by-case proof, the cases being
according to the type of $\kappa$ as given in (\ref{glntypes}).  We
will prove an illustrative case in detail,
leaving the others to the reader.  As in the previous proposition we
identify $\xi$ with its equivalent
Atlas parameter $(x,y)$ of (\ref{avbij}), and identify the dual
parameter ${^\vee} \xi$ with $(y,x)$.

Suppose $\kappa = \{ \alpha, \beta = \vartheta \cdot \alpha \}$ is of
type \texttt{2r22}, \emph{i.e.} $\alpha, \beta \in R^{+}(\lambda)$ are
orthogonal roots which are real with respect to a(ny) strong
involution in the class $x$ (\cite{AVParameters}*{Proposition 3.4}.
Suppose further that $\alpha$ and $\beta$ are
simple in $R^{+}(\mathrm{GL}_{N}, H)$.  Then  $c_{\kappa}$ is defined
as the composition of the Cayley transforms $c_{\alpha}$ of $\alpha$,
and $c_{\beta}$ of $\beta$ (\emph{cf.} \cite{LV2014}*{7.6 (h)}).
The Cayley transform $c_{\alpha}$ is defined in terms of Atlas
parameters in \cite{Adams-Fokko}*{Definitions 14.1 and 14.8} as
follows.   Let $G_{\alpha}$ be the derived group of the centralizer of
$\ker(\alpha)$ in $G = \mathrm{GL}_{N}$, and  let $H_{\alpha} \subset
G_{\alpha} \cap H$ be
the one-parameter subgroup corresponding to $\alpha$. Then $G_{\alpha}$
is isomorphic to $\mathrm{SL}_{2}$ and $H_{\alpha}$ is a
Cartan subgroup of $G_{\alpha}$.  Let $\sigma_{\alpha} \in G_{\alpha}$
be the Tits representative (\cite{AVParameters}*{(53)}) of the
non-trivial Weyl group element  in $W(G_{\alpha},H_{\alpha})$ and
write $m_{\alpha}= \sigma_{\alpha}^{2}$.   The same formalism applies
with $G = \,{^\vee}\mathrm{GL}_{N}$ and ${^\vee}\alpha$, so that we
have a Tits representative $\sigma_{{^\vee}\alpha}$ and
$m_{{^\vee}\alpha} = (\sigma_{{^\vee}\alpha})^{2}$.  Let $\delta \in x$
and ${^\vee}\delta \in y$ be representative strong involutions.   Then the Atlas
parameters of the representations in the image of $c_{\alpha}
(M(x,y))$ are the classes of $(\sigma_{\alpha} \delta,
\sigma_{{^\vee}\alpha} \, {^\vee} \delta)$ and $(m_{\alpha}
\sigma_{\alpha} \delta, \sigma_{{^\vee}\alpha} \, {^\vee} \delta)$.
In this case the two classes coincide (\emph{cf.} type \texttt{1r2}
\cite{AVParameters}*{Table 1}) and therefore $c_{\alpha} (M(x,y))$ is
single-valued.

The same reasoning applied to $c_{\beta}$ leads to a single
representation in the image of $c_{\kappa} (M(x,y)) =
c_{\beta}(c_{\alpha}(M (x,y) ))$
\nomenclature{$c_{\kappa}$}{Cayley transform}
and the Atlas parameter of this
representation is the class of $(\sigma_{\beta} \sigma_{\alpha}
\delta, \sigma_{{^\vee}\beta} \sigma_{{^\vee} \alpha} \, {^\vee}
\delta)$.  The Vogan dual of this Atlas parameter (\ref{vogandual1})
is the class of
$(\sigma_{{^\vee}\beta} \sigma_{{^\vee} \alpha} \, {^\vee} \delta,
\sigma_{\beta} \sigma_{\alpha} \delta)$.  Following the path
delineated above, it is a straightforward exercise to compute that
this is the Atlas parameter of $c_{{^\vee}\kappa}(M(y,x))$, where now
${^\vee}\kappa = \{ {^\vee} \alpha, {^\vee} \beta\}$ is of type
\texttt{2i11} with respect to the representative ${^\vee}\delta$  of
$y$.  This proves the first assertion of the proposition for the first
example when $\alpha$ and $\beta$ are simple in $R^{+}(\mathrm{GL}_{N}, H)$.

When $\alpha$ and $\beta$ are merely simple in $R^{+}(\lambda)$ and
not necessarily simple in $R^{+}(\mathrm{GL}_{N}, H)$ then the Cayley
transform is defined by
\begin{equation}
\label{cayleydef}
c_{\kappa}( M(x,y))  = w^{-1} \times c_{w\kappa} ( w \times M(x,y))
\end{equation}
where $w \in W(\mathrm{GL}_{N},H)^{\vartheta}$  and $w\kappa$ is comprised of simple roots in $R^{+}(\mathrm{GL}_{N}, H)$ (\cite{AVParameters}*{Proposition 10.4}).
  In this case the Atlas parameter of $c_{\kappa}( M(x,y))$ is the class of
\begin{align}
\label{cayleyatlas}
&(\dot{w}^{-1} \sigma_{w\beta} \sigma_{w\alpha} \dot{w} \delta
\dot{w}^{-1} \dot{w}, \dot{w}^{-1}\sigma_{w\,{^\vee}\beta} \sigma_{w\,
  {^\vee} \alpha} \dot{w}\, {^\vee} \delta \dot{w}^{-1} \dot{w}) \\
\nonumber
&= (\dot{w}^{-1} \sigma_{w\beta} \sigma_{w\alpha} \dot{w} \delta ,
\dot{w}^{-1}\sigma_{w\, {^\vee}\beta} \sigma_{w\, {^\vee} \alpha}
\dot{w}\, {^\vee} \delta ),
\end{align}
where $\dot{w}$ is any representative of $w$ (\emph{cf.} the proof of
Proposition \ref{crossvogandual}) and $\sigma_{w\alpha}$ etc. are the
Tits representatives in $G_{w \alpha}$ etc.   These Tits
representatives, as well as the Tits representatives of
$\sigma_{\alpha} \in G_{\alpha}$, etc. for the possibly non-simple
roots, all have the form
\small$$\begin{bmatrix} 0 & 1 \\ -1 & 0 \end{bmatrix}$$ \normalsize
regarded as elements in $\mathrm{SL}_2$.  From this it is clear that
one may choose $\dot{w}$ so that $\dot{w}^{-1} \sigma_{w \alpha }
\dot{w} = \sigma_{\alpha}$, $\dot{w}^{-1} \sigma_{w \beta } \dot{w} =
\sigma_{\beta}$, and then (\ref{cayleyatlas}) reduces to
$$(\sigma_{\beta} \sigma_{\alpha}  \delta , \sigma_{ {^\vee}\beta} \sigma_{{^\vee} \alpha}  {^\vee} \delta ).$$
The class of this pair has the same form as the class in the case that $\kappa$ is simple earlier on.  Thus, the first assertion of the proposition follows for non-simple $\kappa$ as in the simple case.

For the second assertion of the proposition we choose an extended
parameter $(\uplambda, \uptau, 0, 0 )$ for $M(\xi)^{+} = M(x,y)^{+}$
and return to (\ref{cayleydef}) in this extended setting.  As noted in
the proof of Proposition \ref{crossvogandual},  $w \times M(\xi)^{+}$
has an extended parameter equivalent to $(\uplambda_{1}, \uptau_{1},
0,0)$.  According to \cite{AVParameters}*{Table 3}, $c_{w\kappa}(w
\times M(\xi)^{+})$ then also corresponds to an extended parameter of
the form $(\uplambda_{2}, \uptau_{2}, 0,0)$.  Finally applying a cross
action by $w^{-1}$ yields an extended parameter of the form
$(\uplambda_{3}, \uptau_{3}, 0,0)$ for $c_{\kappa}(M(\xi)^{+})$.  The
zeroes appearing in the two finally entries of this extended parameter
imply that $c_{\kappa}(M(\xi)^{+}) = (c_{\kappa}(M(\xi))^{+}$
(\emph{cf.} (\ref{canext})).  The sequence of equalities
$${^\vee}\left( c_{\kappa}(M(\xi)^{+}) \right) = \, {^\vee}\left(
(c_{\kappa}(M(\xi))^{+} \right) = \left( \,{^\vee}(c_{\kappa}(M(\xi))
\right)^{+} = (c_{{^\vee}\kappa}(M({^\vee}\xi)))^{+} =
c_{{^\vee}\kappa}(M({^\vee}\xi)^{+})$$
follows using the same reasoning as given at the end of Proposition
\ref{crossvogandual}.
\end{proof}

We are now ready to prove the main result of this section.
\begin{proof}[Proof of Proposition \ref{Hisomorphism}:]

Recalling the dual $\mathcal{H}(\lambda)$-action (\ref{17.15e}), the
proposition amounts to proving
\begin{equation}
\label{17.17}
\langle T_{w_{\kappa}}M(\xi_{1}) ^{+},  M({^\vee} \xi_{2})^{+} \rangle
= \langle M(\xi_{1})^{+},  -T_{w_{\kappa}} M({^\vee}\xi_{2})^{+}
+(q^{l(w_{\kappa})}-1) M({^\vee}\xi_{2}) ^{+}\rangle
\end{equation}
for all $\xi_{1}, \xi_{2} \in \Xi(\O,
{^\vee}\mathrm{GL}_{N}^{\Gamma})^\vartheta$ and $w_{\kappa}$ as
in (\ref{simplekappa1}).  Looking back to the definition of
(\ref{pair4}), the left-hand side of (\ref{17.17})  may be expressed as
\begin{align*}
\langle T_{w_{\kappa}}M(\xi_{1})^{+},  M({^\vee} \xi_{2})^{+} \rangle
 =
(-1)^{l^{I}_{\vartheta}(\xi_{2})} q^{(l^{I}(\xi_{2}) +
    l^{I}({^\vee}\xi_{2}))/2} \cdot ( \mbox{the coefficient of }
  M(\xi_{2})^{+}  \mbox{ in } T_{w_{\kappa}} M(\xi_{1})^{+}).
\end{align*}
Similarly, the right-hand side of (\ref{17.17}) may be expressed as
the product of $(-1)^{l^{I}_{\vartheta}(\xi_{1})} q^{(l^{I}(\xi_{1}) +
  l^{I}({^\vee}\xi_{1}))/2}$ with
$$ \mbox{the coefficient of }
M({^\vee}\xi_{1})^{+}  \mbox{ in }   -T_{w_{\kappa}}
M({^\vee}\xi_{2})^{+} +(q^{l(w_{\kappa})}-1) M({^\vee}\xi_{2})^{+}. $$
By Lemma \ref{indlength}, Equation (\ref{17.17}) is equivalent to proving that\footnote{$l^{I}_{\vartheta}(\xi_{2}) - l^{I}_{\vartheta}(\xi_{1})$ in Equation (\ref{17.17a}) appears incorrectly as $l^{I}(\xi_{2}) - l^{I}(\xi_{1})$ in \cite{AVParameters}*{(51b)}.  This leads to the map
 $[E]'\mapsto (-1)^{l^{I}(E)}[{^\vee}E]$ in \cite{AVParameters}*{Proposition 11.2}. The correct isomorphism is $[E]'\mapsto (-1)^{l_{\vartheta}^{I}(E)}[{^\vee}E]$. }
\begin{align}
\label{17.17a}
 &(-1)^{l^{I}_{\vartheta}(\xi_{2}) - l^{I}_{\vartheta}(\xi_{1})} \cdot
( \mbox{the coefficient of } M(\xi_{2})^{+}  \mbox{ in }
T_{w_{\kappa}} M(\xi_{1})^{+})\\
 \nonumber
 & = \mbox{ the coefficient of } M({^\vee}\xi_{1})^{+}  \mbox{ in }
 -T_{w_{\kappa}} M({^\vee}\xi_{2})^{+} +(q^{l(w_{\kappa})}-1)
 M({^\vee}\xi_{2})^{+}.
\end{align}
The proof of the proposition is  a case-by-case verification of
(\ref{17.17a}) according to the type of $\kappa$ relative to  $x$ for
$\xi_{1}  = (x,y)$ ((\ref{avbij}),  (\ref{glntypes})).
We prove a typical case in detail here, leaving the remaining cases to
the reader.

Suppose that $\kappa$ is a root of type \texttt{2i22} relative
to $\xi_{1}$.  Then $$c_{\kappa}(M(\xi_{1})^{+}) = \{M(\xi)^{+},
M(\xi')^{+}\}$$ is double-valued \cite{AVParameters}*{Table 1}.
According to
\cite{AVParameters}*{Table 5}, the coefficient of $M(\xi_{2})^{+}$ in
$T_{w_{\kappa}} M(\xi_{1})^{+}$ is $1$ when $\xi_{2}  =\xi_{1}, \xi,
\xi'$ and $0$ otherwise.  As for the $\vartheta$-integral lengths
(\ref{thetalength}), we compute
\begin{align*}
l^{I}_{\vartheta}(\xi) &= -\frac{1}{2} \left( | \{\alpha \in
R^{+}_{\vartheta}(\lambda)  :
w_{\kappa} x\cdot \alpha \in R^{+}_{\vartheta}(\lambda) \} |+
\dim((H^{\vartheta})^{w_{\kappa} x}) \right)\\
& = -\frac{1}{2} \left( | \{\alpha \in R^{+}_{\vartheta}(\lambda)  :
 x\cdot \alpha \in R^{+}_{\vartheta}(\lambda) \} | -2 +
 \dim((H^{\vartheta})^{ x}) \right)\\
 &=  l^{I}_{\vartheta}(\xi_{1}) + 1
\end{align*}
and so $(-1)^{l_{\vartheta}^{I}(\xi_{1}) - l_{\vartheta}^{I}(\xi)} =
-1$.  Similarly $(-1)^{l_{\vartheta}^{I}(\xi_{1}) -
  l_{\vartheta}^{I}(\xi')} = -1$.  The left-hand side of
(\ref{17.17a}) is therefore equal to $1$ if $\xi_{2} =\xi_{1}$, is
equal to $-1$ if $\xi_{2} =  \xi, \xi'$, and  is equal to  $0$
otherwise.

Let us consider the right-hand side of (\ref{17.17a}), in which
${^\vee} \kappa$ is of type \texttt{2r11} relative to ${^\vee}
\xi_{1}$.  According to \cite{AVParameters}*{Table 5},
$M({^\vee}\xi_{1})^{+}$ occurs in $T_{w_{\kappa}}
M({^\vee}\xi_{2})^{+}$ only if one of the following holds
\begin{enumerate}
\item $M({^\vee}\xi_{1})^{+} = M({^\vee}\xi_{2})^{+}$
\item $M({^\vee}\xi_{1})^{+}$ belongs to $c_{{^\vee}\kappa}(M({^\vee}\xi_{2})^{+})$
\item $M({^\vee}\xi_{1})^{+} = w_{{^\vee}\kappa} \times M({^\vee}\xi_{2})^{+}$.
\end{enumerate}
The third possibility reduces to the first.  Indeed, the  third
possibility holds if and only if $M({^\vee}\xi_{2})^{+} =
w_{{^\vee}\kappa} \times M({^\vee}\xi_{1})^{+}$, and for ${^\vee}
\kappa$ of type \texttt{2r11} relative to ${^\vee} \xi_{1}$ one may
compute that $w_{{^\vee}\kappa} \times M({^\vee}\xi_{1})^{+}=
M({^\vee}\xi_{1})^{+}$  using (\ref{crossaction}).  Hence, we need to
compute the right-hand side of (\ref{17.17a}) only for the first two
possibilities.

When $M({^\vee}\xi_{1})^{+} = M({^\vee}\xi_{2})^{+}$, \emph{i.e.}
$\xi_{1} = \xi_{2}$,  the right-hand
side of  (\ref{17.17a}) is
$$-(q^{2} - 2) + q^{l(w_{\kappa})}-1 = -(q^{2} - 2) + q^{2}-1 = 1$$
\cite{AVParameters}*{Table 5}, and this equals the left-hand side of
(\ref{17.17a}).

In the second possibility, $M({^\vee}\xi_{1})^{+}$ is a Cayley
transform of $M({^\vee}\xi_{2})^{+}$ and this is true if and only if
$M({^\vee}\xi_{2})$ is an inverse Cayley transform of
$M({^\vee\xi_{1}})^{+}$.  The latter condition is equivalent to
${^\vee}\xi_{2}$ being equal to either ${^\vee} \xi$ or ${^\vee}\xi'$,
and ${^\vee}\kappa$ is of type \texttt{2i11} relative to ${^\vee}
\xi_{2}$.   As \cite{AVParameters}*{Table 5} indicates, the right-hand
side of (\ref{17.17a}) equals $-1$, which is also equal to the
left-hand side of (\ref{17.17a}).
\end{proof}

\subsection{Verdier duality}

In proving Theorem \ref{twistpairing} we have extended
(\ref{pair2}) to the pairing (\ref{pair4}) of Hecke modules and discussed the related Hecke
algebra actions.  Ultimately, we must evaluate the pairing on special elements
which recover the basis elements $\pi(\xi)$ and $\pi(\ch \xi)$ of
Lemma \ref{twistpairing2}.  As already mentioned, the desired elements are essentially
eigenvectors of \emph{Verdier duality} (\cite{LV2014}*{Section 2.4}).
We introduce the key properties of Verdier
  duality, define the
``eigenvectors,'' and finally show that
the Hecke algebra isomorphism of Proposition \ref{Hisomorphism} is
equivariant with respect to Verdier duality.

Verdier duality on $\mathcal{K}
\Pi(\O, \mathrm{GL}_{N}(\mathbb{R}), \vartheta)$  is a
$\mathbb{Z}$-linear involution $D$
 satisfying
\nomenclature{$D$}{Verdier duality map}
\begin{equation}
  \label{verdual}
  \begin{aligned}
  D(q^{1/2} M(\xi)^{+}) &= q^{-1/2} D(
  (\xi)^{+})\\
  D((T_{\kappa} + 1) M(\xi)^{+} ) &= q^{-l(w_{\kappa})}(T_{\kappa} +
  1) \, D(M(\xi)^{+})
    \end{aligned}
\end{equation}
See \cite{LV2014}*{4.8(e)-(f)}.

\begin{thm}[\cite{LV2014}*{Section 8.1}]
  \label{theo:verdierProp}
Define elements
$R({\xi'},{\xi})\in \mathbb{Z}[q^{1/2},q^{-1/2}]$ by
$${D}(M(\xi)^{+}) = q^{-l^{I}({\xi})} \sum_{\xi' \in
  \Xi(\O, {^\vee}\mathrm{GL}_{N}^{\Gamma})^\vartheta}
(-1)^{l^{I}({\xi})-l^{I}({\xi'})} R({\xi'},{\xi}) \, M (\xi')^{+}$$
in $\mathcal{K} \Pi(\O, \mathrm{GL}_{N}(\mathbb{R}), \vartheta)$.
Then
\begin{enumerate}[label={(\alph*)}]
\item $R({\xi'},{\xi'})=1$,
\item $R({\xi'},\xi)\neq 0$ only if $\xi' \leq{\xi}$.
\end{enumerate}
In addition, if $D'$ is any $\mathbb{Z}$-linear involution of
$\mathcal{K} \Pi(\O, \mathrm{GL}_{N}(\mathbb{R}), \vartheta)$
which satisfies the properties of this theorem and (\ref{verdual}),
then $D = D'$.
\end{thm}
The constructions of \cite{LV2014} apply equally well to the dual
module, yielding a Verdier Duality
${^\vee}D$ on
\nomenclature{$\ch D$}{Verdier duality map}
$\mathcal{K} {^\vee}\Pi(\O,
\mathrm{GL}_{N}(\mathbb{R}),\vartheta)$.  The Verdier duality
${^\vee}D$ satisfies the obvious analogue of Theorem \ref{theo:verdierProp}.

We also a need a dual version of Verdier duality.
Let $\xi \in \Xi(\O,
{^\vee}\mathrm{GL}_{N}^{\Gamma})^\vartheta$ so that $\langle
M(\xi)^{+},  \cdot  \rangle$ belongs to $\mathcal{K}
{^\vee}\Pi(\O, \mathrm{GL}_{N}(\mathbb{R}),\vartheta)^{*}$.
As in  \cite{ABV}*{(17.15)(f)}, we define
$${^\vee}D^{*} :  \mathcal{K} {^\vee}\Pi(\O,
\mathrm{GL}_{N}(\mathbb{R}),\vartheta)^{*}\rightarrow \mathcal{K}
       {^\vee}\Pi(\O,
       \mathrm{GL}_{N}(\mathbb{R}),\vartheta)^{*}$$
by
$${^\vee}D^{*}  \langle M(\xi)^{+},  \cdot  \rangle = \overline{
  \langle M(\xi)^{+}, \, {^\vee} D (\cdot) \rangle},$$
where ${^\vee}D$ is the Verdier duality on $\mathcal{K}
{^\vee}\Pi(\O, \mathrm{GL}_{N}(\mathbb{R}),\vartheta)$ and
$$ \bar{\vspace{2mm}} \,: \mathbb{Z} [q^{1/2}, q^{-1/2}] \rightarrow
\mathbb{Z} [q^{1/2}, q^{-1/2}] $$
is the unique automorphism sending $q^{1/2}$ to $q^{-1/2}$.
Imitating the proof of \cite{ICIV}*{Lemma 13.4}, it is straightforward
to verify that ${^\vee}D^{*}$ satisfies the analogues of
(\ref{verdual}) and Theorem \ref{theo:verdierProp}.  It is in this proof that Proposition \ref{Hisomorphism} is used.  We are thus
justified in calling ${^\vee}D^{*}$ a Verdier dual.
\begin{prop}
\label{HeckeVerdier}
The Hecke module isomorphism (\ref{eq:dualformalmap}) is equivariant
with respect to $D$ and ${^\vee}D^{*}$, that is
$$\langle D M(\xi_{1})^{+} , M({^\vee}\xi_{2})^{+} \rangle =
\overline{ \langle M(\xi_1)^{+}, \, {^\vee} D M({^\vee}\xi_{2})^{+}
  \rangle},$$
for all $\xi_{1}, \xi_{2} \in \Xi(\O,
{^\vee}\mathrm{GL}_{N}^{\Gamma})^\vartheta$.
\end{prop}
\begin{proof}
As already remarked, both $D$ and ${^\vee}D^{*}$ satisfy (analogues
of) (\ref{verdual}) and Theorem \ref{theo:verdierProp}.  The resulting
properties of $D$ imply that the $\mathbb{Z}$-linear involution
$$\langle M(\xi_{1})^{+}, \cdot \rangle  \mapsto \langle D
M(\xi_{1})^{+}, \cdot \rangle $$
on $\mathcal{K} {^\vee}\Pi(\O,
\mathrm{GL}_{N}(\mathbb{R}),\vartheta)^{*}$ also satisfies the
analogues of (\ref{verdual}) and Theorem \ref{theo:verdierProp}.  The
proposition therefore follows from the uniqueness statement in the
dual analogue of Theorem \ref{theo:verdierProp}.
\end{proof}

We now define the special basis elements alluded to at the beginning
of this section.  The
special bases are defined in terms of the twisted KLV-polynomials
$P^{\vartheta}({\xi'},{\xi})\in \mathbb{Z}[q^{1/2}, q^{-1/2}]$  defined
in \cite{LV2014}*{Section 0.1}.
\nomenclature{$P^{\vartheta}({\xi'},{\xi})$}{twisted KLV-polynomials}
\begin{thm}[\cite{LV2014}*{Theorem 5.2}]
\label{theo:defpoly}
For every $\xi \in \Xi(\O,
{^\vee}\mathrm{GL}_{N}^{\Gamma})^\vartheta$, define
\begin{equation}
  \label{basis}
C^{\vartheta}(\xi)=\sum_{\xi'\in \Xi(\O,
{^\vee}\mathrm{GL}_{N}^{\Gamma})^\vartheta}
(-1)^{l^{I}(\xi)-l^{I}(\xi')} \,P^{\vartheta}(\xi',\xi)
\ M(\xi')^{+},
\nomenclature{$C^{\vartheta}(\xi)$}{}
\end{equation}
an element in $\mathcal{K}
\Pi(\O, \mathrm{GL}_{N}(\mathbb{R}), \vartheta)$.
Then
\begin{enumerate}
\item ${D}({C}^{\vartheta}({\xi}))=q^{-l^{I}(\xi)}\, {C}^{\vartheta}({\xi})$
\item $P^{\vartheta}(\xi,\xi)=1$
\item $P^{\vartheta}(\xi', \xi) = 0$ if $\xi' \nleq {\xi}$
\item $\deg P^{\vartheta}(\xi',{\xi}) \leq
  (l^{I}({\xi})-l^{I}({\xi'})-1)/2$ if ${\xi'}\leq{\xi}$.
\end{enumerate}
Conversely suppose  $\{\underline C(\xi',\xi)\}$ and
$\{\underline{P}(\xi', \xi)\}$
satisfy \eqref{basis} and (1-4).
Then $\underline P(\xi',\xi)=P^\vartheta(\xi',\xi)$ and $\underline
C(\xi',\xi)=C^\vartheta(\xi',\xi)$ for all $\xi',\xi \in
\Xi(\O \ch \mathrm{GL}_{N}^{\Gamma})^{\vartheta}$.

\end{thm}
Properties 2-3 of Theorem \ref{theo:defpoly} ensure that
$$\{ C^{\vartheta}(\xi): \xi \in \Xi(\O,
{^\vee}\mathrm{GL}_{N}^{\Gamma})^\vartheta \}$$
is a basis for $\mathcal{K}
\Pi(\O, \mathrm{GL}_{N}(\mathbb{R}), \vartheta)$.

\cite{LV2014}*{Theorem 5.2}  also applies to  $\mathcal{K}
{^\vee}\Pi(\O, \mathrm{GL}_{N}(\mathbb{R}), \vartheta)$, so
there is an obvious variant of Theorem \ref{theo:defpoly}
characterizing the dual KLV-polynomials ${^\vee}P^{\vartheta}(
\,{^\vee}\xi', {^\vee}\xi)
\nomenclature{$\ch P^{\vartheta}(
\,{^\vee}\xi', {^\vee}\xi)$}{dual twisted KLV-polynomials}
$ and the
basis elements
\begin{equation}
  \label{dualbasis}
C^{\vartheta}({^\vee}\xi) =\sum_{\xi' \in \Xi(\O,
{^\vee}\mathrm{GL}_{N}^{\Gamma})^\vartheta
  }(-1)^{l^{I}({^\vee}\xi)-l^{I}({^\vee}\xi')}
\ {^\vee}P^{\vartheta}(\, ^{\vee}\xi', {^\vee}\xi) \,
M({^\vee}\xi')^{+}.
\nomenclature{$C^{\vartheta}({^\vee}\xi)$}{}
\end{equation}

\begin{prop}
  \label{specialpi}
By specializing to $q=1$, we obtain
\begin{align*}
C^{\vartheta}(\xi)(1) & = \pi(\xi)^{+},\\
C^{\vartheta}({^\vee}\xi)(1) & = \pi({^\vee}\xi)^{+}
\end{align*}
for all $\xi \in \Xi(\O,
{^\vee}\mathrm{GL}_{N}^{\Gamma})^\vartheta$.
\end{prop}
\begin{proof}
The assertions of the proposition are given in purely
representation-theoretic terms.  However, in the second equality the
representations $M({^\vee}\xi')$ occurring in (\ref{dualbasis}) are
representations of possibly different strong involutions of
$\ch \mathrm{GL}_{N}(\lambda)$, which complicates matters.  It is
therefore clearer if we transport the assertion back to the original
context of constructible sheaves using (\ref{sheaftorep}).  The
assertion equivalent to
$$C^{\vartheta}({^\vee}\xi)(1) = \pi({^\vee}\xi)^{+}$$
in this context is that  $P({^\vee}\xi)^{+}$ equals
\begin{equation}
   \label{transsheaf}
\sum_{\xi' \in \Xi(\O,
{^\vee}\mathrm{GL}_{N}^{\Gamma})^\vartheta
  }(-1)^{d(\xi) } \ {^\vee}P^{\vartheta}( {^\vee}\xi', {^\vee}\xi)(1) \,
\mu({^\vee}\xi')^{+}
\end{equation}
(recall  Proposition \ref{PpimuM} and (\ref{constl})).
It follows from the definition of the KLV-polynomials (\cite{LV2014}*{Section 0.1}), (\ref{twistgmult}) and (\ref{constl}) that
\begin{equation}
\label{cgP}
{^\vee}P^{\vartheta}(\, {^\vee}\xi', {^\vee}\xi)(1) = (-1)^{d(\xi)-d(\xi')} \,
c_{g}^{\vartheta}(\xi',\xi) = (-1)^{l^{I}(\xi)-l^{I}(\xi')} \,
c_{g}^{\vartheta}(\xi',\xi).
\end{equation}
(Note that the definition of $P(\xi)$ in \cite{LV2014} differs from
ours by a shift in degree $d(\xi)$ \emph{cf}. \cite{ABV}*{(7.10)(d)}.)
This implies that (\ref{transsheaf}), as an element in $K X(\O,
{^\vee}\mathrm{GL}_{N}^{\Gamma}, \upsigma)$, is equal to
\begin{align*}
& \sum_{\xi' \in  \Xi(\O,
{^\vee}\mathrm{GL}_{N}^{\Gamma})^{\vartheta}
  }  (-1)^{d(\xi')} \left( c_{g}(\xi'_{+},\xi_{+}) -
c_{g}(\xi'_{-}, \xi_{+}) \right) \,
\mu({^\vee}\xi')^{+}\\
& = \sum_{\xi' \in \Xi(\O,
{^\vee}\mathrm{GL}_{N}^{\Gamma})^\vartheta
  }  c_{g}(\xi'_{+},\xi_{+})\,   (-1)^{d(\xi')} \mu({^\vee}\xi)^{+} -
c_{g}(\xi'_{-}, \xi_{+})   \,  (-1)^{d(\xi')}
\mu({^\vee}\xi')^{+}\\
& = \sum_{\xi' \in \Xi(\O,
{^\vee}\mathrm{GL}_{N}^{\Gamma})^\vartheta
  }  c_{g}(\xi'_{+},\xi_{+})\,   (-1)^{d(\xi')} \mu({^\vee}\xi)^{+} +
c_{g}(\xi'_{-}, \xi_{+})   \,  (-1)^{d(\xi')}
\mu({^\vee}\xi')^{-},\\
& = P(\xi)^{+}
\end{align*}
where the final equation is (\ref{imP}).
This proves the second assertion of the proposition.  The first may be
proved in the same manner.  However, a purely representation-theoretic
proof is also possible following \cite{AvLTV}*{Theorem 19.4}
\end{proof}
\subsection{The proof of Theorem \ref{twistpairing}}
\label{prooftwistpairing}

The main theorem of this section is
\begin{thm}\label{pairingC}
Pairing (\ref{pair4}) satisfies
\begin{align}\label{e:pairingC}
\langle C^{\vartheta}(\xi), C^{\vartheta}({^\vee}\xi')\rangle =
(-1)^{l^{I}_{\vartheta}(\xi)} \, q^{\left(l^{I}(\xi)+l^{I}({^\vee}\xi')\right)/2} \,
\delta_{\xi,\xi'}.
\end{align}
for all $\xi, \xi' \in \Xi(\O,
{^\vee}\mathrm{GL}_{N}^{\Gamma})^\vartheta$.
\end{thm}

To prove Theorem \ref{pairingC}
we need the following lemma.
\begin{lem}
\label{invklvpoly}
There are unique elements
$\underline{C}^{\vartheta}(\xi) \in\mathcal{K}\Pi(\O,
\mathrm{GL}_{N}(\mathbb{R}),\vartheta)$, $\xi \in \Xi(\O, {^\vee}\mathrm{GL}_{N}^{\Gamma})^\vartheta$, satisfying
\nomenclature{$\underline{C}^{\vartheta}(\xi)$}{}
\begin{equation*}
  \label{eq:orthogonalbasis}
\langle \underline{C}^{\vartheta}(\xi), C^{\vartheta}({^\vee}\xi') \rangle =
(-1)^{l^{I}_{\vartheta}(\xi')}\,
q^{\left(l^{I}(\xi)+l^{I}({^\vee}\xi')\right)/2}\, \delta_{\xi, \xi'}.
\end{equation*}
More explicitly, let $\underline{P}(\xi', \xi) \in \mathbb{Z}[q^{1/2}, q^{-1/2}]$ be the entries of the matrix inverse and transpose to the matrix formed by the polynomials
${^\vee}P^{\vartheta}( \, {^\vee}\xi', {^\vee}\xi)$ given in (\ref{dualbasis}), \emph{i.e.}
$$\sum_{ \xi' \in \Xi(\O,
{^\vee}\mathrm{GL}_{N}^{\Gamma})^\vartheta}  \underline{P}(\xi', \xi) \, {^\vee}P^{\vartheta}
({^\vee}\xi', {^\vee} \xi'') =  \delta_{\xi, \xi''}.$$
Then
$$\underline{C}^{\vartheta}(\xi) =\sum_{\xi'\in \Xi(\O,
{^\vee}\mathrm{GL}_{N}^{\Gamma})^\vartheta}
(-1)^{l^{I}(\xi)-l^{I}(\xi')}(-1)^{l_{\vartheta}^{I}(\xi)-l_{\vartheta}^{I}(\xi')} \, \underline{P}(\xi',\xi) \ M(\xi')^{+}.$$
\end{lem}
\begin{proof}
We just need to verify, for all $\xi, \xi'' \in \Xi(\O,
{^\vee}\mathrm{GL}_{N}^{\Gamma})^\vartheta$, the equality
\begin{align*}
\left\langle \sum_{\xi'}
(-1)^{l^{I}(\xi)-l^{I}(\xi')}(-1)^{l_{\vartheta}^{I}(\xi)-l_{\vartheta}^{I}(\xi')}
\, \underline{P}(\xi',\xi) \ M(\xi')^{+},
C^{\vartheta}({^\vee}\xi'') \right\rangle
 = (-1)^{l^{I}_{\vartheta}(\xi)}\,
q^{\left(l^{I}(\xi)+l^{I}({^\vee}\xi'')\right)/2}\, \delta_{\xi, \xi''}.
\end{align*}
Let $k = -\frac{1}{2} \left( |R^{+}(\lambda)| +
\dim(H) \right)$ as in Lemma \ref{indlength}. Applying (\ref{pair4}), we compute
\begin{align*}
&\left\langle \sum_{\xi'}
(-1)^{l^{I}(\xi)-l^{I}(\xi')}(-1)^{l_{\vartheta}^{I}(\xi)-l_{\vartheta}^{I}(\xi')} \, \underline{P}(\xi',\xi) \ M(\xi')^{+},
  C^{\vartheta}({^\vee}\xi'') \right\rangle \\
& = \left\langle \sum_{\xi'}
(-1)^{l^{I}(\xi)-l^{I}(\xi')}(-1)^{l_{\vartheta}^{I}(\xi)-l_{\vartheta}^{I}(\xi')} \, \underline{P}(\xi',\xi) \ M(\xi')^{+},   \sum_{\xi_{1}
  }(-1)^{l^{I}({^\vee}\xi'')-l^{I}({^\vee}\xi_{1})}
  \ {^\vee}P^{\vartheta}(\,{^\vee}\xi_{1}, {^\vee}\xi'') \,
M({^\vee}\xi_{1})^{+} \right\rangle\\
& =  (-1)^{l^{I}_{\vartheta}(\xi)} \, (-1)^{k}(-1)^{l^{I}(\xi) +
   l^{I}({^\vee}\xi'')} \,q^{k/2}
 \sum_{ \xi' \in \Xi(\O,
{^\vee}\mathrm{GL}_{N}^{\Gamma})^\vartheta}   \underline{P}(\xi', \xi) \, {^\vee}P^{\vartheta}
 (\,{^\vee}\xi', {^\vee}\xi'') \\
& = (-1)^{l^{I}_{\vartheta}(\xi)} \,q^{k/2} \delta_{\xi, \xi''}\\
& = (-1)^{l^{I}_{\vartheta}(\xi)}
 q^{\left(l^{I}(\xi)+l^{I}({^\vee}\xi'')\right)/2}\, \delta_{\xi,
   \xi''}.
\end{align*}
\end{proof}

\begin{proof}[Proof of Theorem \ref{pairingC}]

If one proves that  the coefficient polynomials
$\underline{P}(\xi',\xi)$ of $\underline{C}^{\vartheta}(\xi)$
 satisfy properties 1-4 of Theorem \ref{theo:defpoly}, then the
 uniqueness statement of that theorem implies $\underline{C}^{\vartheta}(\xi) =
 C^{\vartheta}(\xi)$ and the theorem
 follows from
 Lemma \ref{invklvpoly}.
To show properties 1-4
we follow the proof of \cite{ICIV}*{Lemma 13.7}.  For the first property
we apply Proposition \ref{HeckeVerdier}
\begin{align*}
\langle D \underline{C}^{\vartheta}(\xi), C^{\vartheta}({^\vee}\xi')
\rangle & = \overline{ \langle \underline{C}^{\vartheta}(\xi),
  \,{^\vee}D C^{\vartheta}({^\vee}\xi') \rangle}\\
& = q^{l^{I}({^\vee}\xi')} \, \overline{  \langle
  \underline{C}^{\vartheta}(\xi),  C^{\vartheta}({^\vee}\xi') \rangle}
\\
& = (-1)^{l^{I}_{\vartheta}(\xi)} q^{l^{I}(^{\vee}\xi')}
q^{-1/2(l^{I}(\xi) + l^{I}({^\vee}\xi'))}  \delta_{\xi,\xi'} \\
&= \langle q^{-l^{I}(\xi)}\, \underline{C}^{\vartheta}(\xi),
C^{\vartheta}({^\vee}\xi') \rangle.
\end{align*}
Since the elements $C^{\vartheta}({^\vee}\xi')$ form a basis we conclude that
$$ D \underline{C}^{\vartheta}(\xi) = q^{-l^{I}(\xi)}\, \underline{C}^{\vartheta}(\xi)$$
and the first property of Theorem \ref{theo:defpoly} is proved.

The second and third properties of Theorem  \ref{theo:defpoly}  follow
for $\underline{P}(\xi',\xi)$  since it is defined in terms of the transpose and
inverse of a unipotent matrix, and the map $\xi \mapsto \ch\xi$ (\ref{vogandual1}) is
order-reversing (\ref{dualBruhat}).

The fourth property of Theorem \ref{theo:defpoly} is proven by
induction on the integral length of a parameter.  This uses a
straightforward reformulation of \cite{ICIV}*{(13.9)} which is left to the
reader.

The uniqueness statement
in Theorem \ref{theo:defpoly} now implies $\underline P^\vartheta(\xi,\xi')=(-1)^{l_{\vartheta}^I(\xi)-l_{\vartheta}^I(\xi')}P^\vartheta(\xi,\xi')$ and
$\underline{C}^{\vartheta}(\xi) = C^{\vartheta}(\xi)$.
Finally by Lemma \ref{invklvpoly}, Equation \eqref{e:pairingC} holds, completing
the proof of the theorem.

\end{proof}

The proof of Theorem \ref{twistpairing} is now immediate.

\begin{proof}[Proof of Theorem \ref{twistpairing}]
  It is enough to prove Lemma \ref{twistpairing2}.
  Let $\langle \cdot \,,\cdot \rangle$ be the pairing in (\ref{pair4}).
  By Theorem \ref{pairingC} we have

$$
\langle C^{\vartheta}(\xi), C^{\vartheta}({^\vee}\xi')\rangle =
(-1)^{l^{I}_{\vartheta}(\xi)} \, q^{\left(l^{I}(\xi)+l^{I}({^\vee}\xi')\right)/2} \,
\delta_{\xi,\xi'}.
$$
Setting $q=1$ and applying Proposition \ref{specialpi}, we conclude
$$
\langle \pi(\xi)^+,\pi(\ch\xi')^+\rangle=
(-1)^{l^{I}_{\vartheta}(\xi)}\delta_{\xi,\xi'}
$$
as required.
\end{proof}

\subsection{Twisted KLV-polynomials for the dual of
  $\mathrm{GL}_{N}(\mathbb{R})$}
  \label{tKLV}

This section is thematically related to the others in Section
\ref{pairings}, although it admittedly has nothing to do with the
pairings.  The
sole purpose of this section is to determine the polynomials
${^\vee}P^{\vartheta}(\, ^{\vee}\xi', {^\vee}\xi) \in
\mathbb{Z}[q^{1/2}, q^{-1/2}]$ appearing in the
definition of $C^{\vartheta}({^\vee}\xi)$ (\ref{dualbasis}) under
certain circumstances. In our application $\xi$ will be the parameter of a
generic representation, and the value of
${^\vee}P^{\vartheta}(\, ^{\vee}\xi', {^\vee}\xi)$
will be used in Section \ref{whitsec} (Proposition \ref{conjq} and Proposition \ref{wasign1}) to compare
the Atlas extensions with the so-called \emph{Whittaker extensions}.

  A block of parameters  is defined in \cite{ICIV}*{Definition 1.14}.
  In particular $P(\xi,\xi')\ne 0$ implies $\xi,\xi'$ are in the same block.
\begin{prop}
  \label{dualpone}
  Suppose $\ch\xi_0$ is the unique maximal element of a block $\ch\mathcal B$
  with respect to  the Bruhat order (\ref{dualBruhat}).
Then ${^\vee}P^{\vartheta}({^\vee}\xi,{^\vee}\xi_{0})=1$ for all $\ch\xi\in\ch\mathcal B$.
\end{prop}

The proof of the Proposition \ref{dualpone} is algorithmic in nature
and relies on computations with Hecke operators.  A broad examination
of the algorithms is presented in \cite{adamsnotes} and \cite{LV2014}.
We assemble a few facts from these references here.  The facts are
centred upon the characterization of eigenspaces of Hecke operators.

For each $\xi\in \Xi(\O,{^\vee}\mathrm{GL}_{N}^{\Gamma})^\vartheta$,
Lusztig and Vogan define what it means for (the Weyl group element of)
a $\vartheta$-orbit
$\kappa$ of a simple root in $R(\lambda)$ to be a \emph{descent} for
${^\vee}\xi$ (\cite{LV2014}*{Section 7.2}).  We leave the definition (which is equivalent to
(\ref{l:descent})) to the interested reader, being content merely to
record the relevant properties.  By
definition, a $\vartheta$-orbit $\kappa$ is an \emph{ascent} for
${^\vee}\xi$ if it is not a descent for ${^\vee}\xi$.  To assist in
indexing these $\vartheta$-orbits we define $\boldsymbol{\tau}
({^\vee}\xi)$
\nomenclature{$\boldsymbol{\tau}({^\vee}\xi)$}{descent roots}
to be the set of $\vartheta$-orbits of
simple roots which are descent  for ${^\vee}\xi$.  Thus, $\kappa$ is
\begin{itemize}
\item  a \emph{descent} for ${^\vee}\xi$, if $\kappa\in
  \boldsymbol{\tau} ({^\vee}\xi)$, and
\item  an \emph{ascent} for ${^\vee}\xi$, if $\kappa\notin
  \boldsymbol{\tau} ({^\vee}\xi)$.
\end{itemize}
Recall that the $\mathcal{H}(\lambda)$-module
$$\mathcal{K} {^\vee}
\Pi(\O, \mathrm{GL}_{N}(\mathbb{R}), \vartheta) =K {^\vee}
\Pi(\O, \mathrm{GL}_{N}(\mathbb{R}), \vartheta) \otimes
\mathbb{Z}[q^{1/2}, q^{-1/2}]$$
((\ref{sheaftorep}), Section \ref{twisthmodule})
has a basis $\{M({^\vee}\xi)^{+} : \xi \in \Xi(\O,
       {^\vee}\mathrm{GL}_{N}^{\Gamma})^\vartheta\}$.
For each $\vartheta$-orbit $\kappa$ let
$$\widehat{T}_{\kappa} = q^{-l(w_{\kappa})/2}(T_{\kappa} + 1).
\nomenclature{$\widehat{T}_{\kappa}$}{}
$$
Suppose $\kappa\not\in \boldsymbol{\tau}({^\vee}\xi)$,
\emph{i.e.} $\kappa$
is an ascent for ${^\vee}\xi$.  Then we define
$$B_\kappa({^\vee}\xi)
\subset \Xi(\O,
       {^\vee}\mathrm{GL}_{N}^{\Gamma})^\vartheta
        \nomenclature{$B_\kappa({^\vee}\xi)$}{}$$
to be the set of parameters which indexes the non-zero summands in
$\widehat{T}_{\kappa} M({^\vee}\xi)^{+}$, that is
$$
\widehat{T}_{\kappa} M({^\vee}\xi)^{+} = \sum_{\xi' \in B_\kappa({^\vee}\xi)}
a({^\vee}\xi')\, M({^\vee} \xi')^{+}, \quad 0 \neq a({^\vee}
\xi')\in\mathbb{Z}[q^{1/2}, q^{-1/2}].
$$
A case-by-case inspection of \cite{AVParameters}*{Table 5}, or the
formulas of \cite{LV2014}*{7.5-7.7},  confirm that
$\xi \in B_\kappa({^\vee}\xi)$, and $|B_\kappa({^\vee}\xi)|\le 3$.
Let
$$\mathcal{M}_\kappa({^\vee}\xi) =\mathbb{Z}[q^{1/2},
  q^{-1/2}]\text{-span of }\{M({^\vee}\xi')^{+} : \xi' \in
B_\kappa({^\vee}\xi)\}.$$
Keeping in mind that $\kappa \notin \boldsymbol{\tau} ({^\vee}\xi)$,  we write
$${^\vee} \xi' \overset\kappa\to {^\vee}\xi
\nomenclature{$\ch \xi' \overset\kappa\to {^\vee}\xi$}{}$$
if $\kappa\in\boldsymbol{\tau}({^\vee}\xi')$ and $\xi' \in
B_\kappa({^\vee}\xi)$.  In this definition $\kappa$ is an ascent for
${^\vee}\xi$ and a descent for ${^\vee}\xi'$.

The quadratic relation
(\ref{quadrel}) gives
$$\widehat{T}_{\kappa}^2 = (q^{l(w_{\kappa})/2}+q^{-l(w_{\kappa})/2})\widehat{T}_{\kappa}.$$
An important consequence of this equation is that
the image of $\widehat{T}_{\kappa}$ is contained in the
$(q^{l(w_{\kappa})/2}+q^{-l(w_{\kappa})/2})$-eigenspace of
$\widehat{T}_{\kappa}$ (\cite{LV2014}*{7.2 (c)}).

From a case-by-case inspection of \cite{AVParameters}*{Table 5} or the
formulas of  \cite{LV2014}*{7.5-7.7},
it follows that if
$\kappa\not\in\boldsymbol{\tau}({^\vee}\xi)$ then the space
$\mathcal{M}_\kappa({^{\vee}}\xi)$ is
  $\widehat{T}_{\kappa}$-invariant, and  the
$(q^{l(w_{\kappa})/2}+q^{-\ell(\kappa)/2})$-eigenspace of
  $\widehat{T}_{\kappa}$ on $\mathcal{M}_\kappa({^\vee}\xi)$
is spanned by
\begin{equation}
  \label{eigenspace}
\left\{\widehat{T}_{\kappa} M({^\vee}\xi')^{+}~:~{^\vee} \xi'
\overset\kappa\to {^\vee}\xi\right\}=
\left\{\widehat{T}_{\kappa} M({^\vee}\xi')^{+}~:~{^\vee}\xi'\in
B_{\kappa}({^\vee}\xi) \text{ and }
\kappa\in\boldsymbol{\tau}({^\vee}\xi')\right\}.
\end{equation}

Now suppose $\kappa \in \boldsymbol{\tau}({^\vee}\xi)$, \emph{i.e.}
is a descent for ${^\vee}\xi$.  Then  \cite{AVParameters}*{Table 5} tells
us that
if $\kappa$ is not of type  \texttt{1r2}, \texttt{2r22} and
\texttt{2r21} with
respect to ${^\vee}\xi$, then
\begin{equation}
  \label{e:a}
\widehat{T}_{\kappa} M({^\vee}\xi)^{+} = a({^\vee}\xi) \left( M({^\vee}\xi)^{+}
  +\sum_{\{ \xi' :
    {^\vee}\xi\overset\kappa\to {^\vee} \xi'\}} M({^\vee}\xi')^{+} \right)
\end{equation}
for some $a({^\vee}\xi)\in \mathbb{Z}[q^{1/2}, q^{-1/2}]$.  A quick
glance at  (\ref{dualglntypes}) affirms that the
types \texttt{1r2, 2r22} and \texttt{2r21}  do not occur, so
(\ref{e:a}) holds for every $\kappa$
which is a descent.  Another
fact we need for $\kappa\in\boldsymbol{\tau}({^\vee}\xi)$ is that
    $C^{\vartheta}({^\vee}\xi) \in \mathcal{K} {^\vee}
    \Pi(\O, \mathrm{GL}_{N}(\mathbb{R}), \vartheta)$ defined in
    (\ref{dualbasis}) satisfies
\begin{equation}
 \label{l:descent}
    \widehat{T}_{\kappa} C^{\vartheta}(^{\vee}\xi)
=(q^{l(w_{\kappa})/2}+q^{-l(w_{\kappa})/2}) \,
C^{\vartheta}({^\vee}\xi).
\end{equation}
This is stated in the last paragraph of \cite{LV2014}*{Section 7.2}, where
$C^{\vartheta}({^\vee}\xi)$ is denoted as
$\mathfrak{A}_{\mathcal{L}}$ (see \cite{LV2014}*{Section 5}).\footnote{
  There is a misprint in this paragraph of \cite{LV2014}. In our notation
  $C^{\vartheta}({^\vee}\xi)$ belongs to the
  $q^{l(w_{\kappa})}$-eigenspace of $T_\kappa$ (not $T_\kappa+1$),
  and   $C^{\vartheta}({^\vee}\xi)$ is in the
  $(q^{l(w_{\kappa})/2}+q^{-l(w_{\kappa})/2})$-eigenspace of
  $\widehat{T}_{\kappa}$.}

Finally, fix an arbitrary $\vartheta$-orbit $\kappa$. If $\kappa$ is a
descent for ${^\vee}\xi'$ and is not of type \texttt{1ic,2ic,3ic},
then \cite{AVParameters}*{Table 5}
indicates that $\xi' \in
B_{\kappa}({^\vee}\xi)$, where $\kappa$ is an ascent for some ${^\vee}\xi$.  From
this it is easy to see that
\begin{align*}
\mathcal{K} {^\vee} \Pi(\O, \mathrm{GL}_{N}(\mathbb{R}), \vartheta)
=\sum_{\{{^\vee} \xi~:~\kappa\not\in \boldsymbol{\tau}({^\vee} \xi)\}}
\mathcal{M}_\kappa({^\vee}\xi)\qquad \oplus\
\bigoplus_{\substack{\kappa \text{ 1ic,~2ic,~3ic} \\ \text{for }
    {^\vee}\xi'' }} \mathbb{Z}[q^{1/2}, q^{-1/2}]\,  M({^\vee}\xi'').
\end{align*}
This decomposition and (\ref{eigenspace}) imply that the
$(q^{l(w_{\kappa})/2}+q^{-l(w_{\kappa})/2})$-eigenspace of
$\widehat{T}_{\kappa}$ is spanned by
\begin{equation}\label{e:spanned}
\text{ }\{\widehat{T}_{\kappa}
M({^\vee}\xi')^{+} :
\kappa \in \boldsymbol{\tau}({^\vee}\xi')\}.
\end{equation}

\begin{lem}
  \label{p:recursion}
  Suppose $\kappa$ is a $\vartheta$-orbit of a simple root in
  $R^{+}(\lambda)$, and $\xi, \xi' \in   \Xi(\O,
       {^\vee}\mathrm{GL}_{N}^{\Gamma})^\vartheta$ satisfy
       $\kappa \in \boldsymbol{\tau}({^\vee}\xi)$ and
  $\kappa \not\in
  \boldsymbol{\tau}({^\vee}\xi')$.
Then
\begin{equation*}
{^\vee}P^{\vartheta}({^\vee}\xi',{^\vee}\xi) = \sum_{\{ \xi'' :
  {^\vee}\xi''\overset\kappa\to{^\vee}\xi'\}}
{^\vee}P^{\vartheta}({^\vee}\xi'',{^\vee}\xi).
\end{equation*}
\end{lem}
\begin{proof}
Write $C^{\vartheta}({^\vee}\xi)$ as
  \begin{equation}
\label{eqa}
    C^{\vartheta}({^\vee}\xi)= \sum_{\{\xi'' :
      \kappa \in \boldsymbol{\tau}({^\vee}\xi'')\}}
    {^\vee}P^{\vartheta}({^\vee}\xi'',{^\vee}\xi)\,
    M({^\vee}\xi'')^{+} +
    \sum_{\{\xi':
      \kappa\not\in\boldsymbol{\tau}({^\vee}\xi')
      \}}{^\vee}P^{\vartheta}({^\vee}\xi',{^\vee}\xi)
     \,   M({^\vee}\xi')^{+}
  \end{equation}
 On the other hand by (\ref{l:descent}) and \eqref{e:spanned} we have
$$  C^{\vartheta}({^\vee}\xi)= \sum_{\{\xi'':
    \kappa\in\boldsymbol{\tau}({^\vee}\xi'') \}}
  b({^\vee}\xi'',{^\vee}\xi)\ \widehat{T}_{\kappa}
  M({^\vee}\xi'')^{+}$$
for some $b({^\vee}\xi'',{^\vee}\xi)\in \mathbb{Z}[q^{1/2}, q^{-1/2}]$.
Inserting \eqref{e:a} into this equation yields
\begin{align}
  \label{eqc}
  C^{\vartheta}({^\vee}\xi) = &\sum_{\{\xi'' :
    \kappa\in\boldsymbol{\tau}({^\vee}\xi'') \}}
  b({^\vee}\xi'',{^\vee}\xi)
  a({^\vee}\xi'') \, M({^\vee}\xi'')^{+}\\
\nonumber  &+
\sum_{\{\xi'' : \kappa\in\boldsymbol{\tau}({^\vee}\xi'') \}}
b({^\vee}\xi'',{^\vee}\xi)a({^\vee}\xi'')
  \sum_{\{\xi' : {^\vee}\xi'' \overset\kappa\to{^\vee}\xi'\}} M({^\vee}\xi')^{+}
\end{align}
Since $\kappa\not\in\boldsymbol{\tau}({^\vee}\xi')$ for
all ${^\vee}\xi'$ appearing
in the final sum, we may
compare the coefficients of $M({^\vee}\xi'')^{+}$ with
$\kappa\in\boldsymbol{\tau}({^\vee}\xi'')$
in  (\ref{eqa}) and (\ref{eqc}) to
conclude
$$b({^\vee}\xi'',{^\vee}\xi)a({^\vee}\xi'') =
{^\vee}P^{\vartheta}({^\vee}\xi'',{^\vee}\xi).$$
Therefore
\begin{align*}
C^{\vartheta}({^\vee}\xi) = &\sum_{\{\xi'' :
  \kappa\in\boldsymbol{\tau}({^\vee}\xi'') \}}
{^\vee}P^{\vartheta}({^\vee}\xi'',{^\vee}\xi) \,
M({^\vee}\xi'')^{+} \\
\nonumber&+ \sum_{\{\xi'' :
  \kappa\in\boldsymbol{\tau}({^\vee}\xi'') \}}
{^\vee}P^{\vartheta}({^\vee}\xi'',{^\vee}\xi) \sum_{\{\xi'
  :   {^\vee}\xi''\overset\kappa\to{^\vee}\xi'\}} M({^\vee}\xi')^{+}
\end{align*}
Changing the order of summation in the second sum, it becomes
$$\sum_{\{\xi':\kappa\not\in\boldsymbol{\tau}({^\vee}\xi')\}}
\sum_{\{\xi'': {^\vee}\xi''\overset\kappa\to{^\vee}\xi'\}}
{^\vee}P^{\vartheta}({^\vee}\xi'',{^\vee}\xi) \, M({^\vee}\xi')^{+}$$
Comparing with the second sum in  (\ref{eqa}) gives
$${^\vee}P^{\vartheta}({^\vee}\xi',{^\vee}\xi) = \sum_{\{\xi'':
  {^\vee}\xi''\overset\kappa\to{^\vee}\xi'\}}
{^\vee}P^{\vartheta}({^\vee}\xi'',{^\vee}\xi).$$
This completes the proof.
\end{proof}
Lemma \ref{p:recursion} simplifies even further in the present setting.
\begin{cor}
  \label{onebetween}
    Suppose $\kappa$ is a $\vartheta$-orbit of a simple root in
    $R^{+}(\lambda)$, and $\xi,\xi'\in  \Xi(\O,
  {^\vee}\mathrm{GL}_{N}^{\Gamma})^\vartheta$ satisfy
  $\kappa\in\boldsymbol{\tau}({^\vee}\xi)$ and
  $\kappa\not\in\boldsymbol{\tau}({^\vee}\xi')$. Then
  there exists exactly one $\xi''  \in  \Xi(\O,
  {^\vee}\mathrm{GL}_{N}^{\Gamma})^\vartheta$ such that
  ${^\vee}\xi''\overset\kappa\to{^\vee}\xi'$. Furthermore,
$$
{^\vee}P^{\vartheta}({^\vee}\xi',{^\vee}\xi)={^\vee}P^{\vartheta}({^\vee}\xi'',
{^\vee}\xi).$$
  \end{cor}
\begin{proof}
Here are the formulas, derived from \cite{LV2014}*{Section 7} or
\cite{AVParameters}*{Table 5}, for
$\widehat{T}_\kappa M({^\vee}\xi')^{+}$ for the various types
$\kappa\not\in\boldsymbol{\tau}({^\vee}\xi')$
which arise in (\ref{dualglntypes}).
\begin{center}
\begin{tabular}{|l|l|}
  \hline
  Type&$\widehat{T}_\kappa M({^\vee}\xi')^{+}$ \\
  \hline \hline
  \texttt{1C+} & $q^{-1/2} \ (M({^\vee}\xi')^{+}+w_{\kappa}\times
  M({^\vee}\xi')^{+})$\\
  \hline
  \texttt{1i1} & $q^{-1/2} \
  (M({^\vee}\xi')^{+}+w_{\kappa}\times M({^\vee}\xi')^{+}+ c_{\kappa}
  M({^\vee}\xi')^{+})$\\
  \hline
  \texttt{2C+} & $q^{-1} \ (M({^\vee}\xi')^{+}+ w_{\kappa}\times
  M({^\vee}\xi')^{+})$\\
  \hline
  \texttt{2Ci} & $(q^{-1}+1) \  (M({^\vee}\xi')^{+}+ c_{\kappa}
  M({^\vee}\xi')^{+})$\\
  \hline
  \texttt{2i11} &
  $q^{-1} \ (M({^\vee}\xi')^{+}+ w_{\kappa}\times M({^\vee}\xi')^{+}+
  c_{\kappa} M({^\vee}\xi')^{+})$\\
  \hline
  \texttt{3C+} & $q^{-3/2} \ (M({^\vee}\xi')^{+}+ w_{\kappa}\times
  M({^\vee}\xi')^{+})$\\
  \hline
  \texttt{3Ci} & $q^{-3/2}(q +1) \ (M({^\vee}\xi')^{+}+
  c_{\kappa} M({^\vee}\xi')^{+})$\\
  \hline
   \texttt{3i}& $q^{-3/2}(q +1) \  (M({^\vee}\xi')^{+}+
  c_{\kappa} M({^\vee}\xi')^{+})$\\
  \hline
  \end{tabular}
\end{center}
\noindent The ``Cayley transforms'' $c_{\kappa}$ appearing here follow the
notation of the references.  They may include cross actions in their
definition  (see  types \texttt{2ci} and \texttt{3ci}  which are 7.6
(c$^{\prime}$) and 7.7 (c$^{\prime}$) in \cite{LV2014}).  In
any event, the $c_{\kappa}$ appearing here are all single-valued.
Once more a case-by-case inspection of the types in  \cite{LV2014}*{7.5-7.7}
shows that in each entry of the second column there is exactly one summand
whose parameter makes
$\kappa$ a descent.  Consequently, the table indicates that
    $\{\xi'' : {^\vee}\xi'' \overset\kappa\to {^\vee}\xi'\}$ has
exactly one element.  The corollary now follows from Lemma \ref{p:recursion}.
\end{proof}

We are ready to provide the proof for  Proposition \ref{dualpone}
\begin{proof}
It follows  from the maximality of ${^\vee}\xi_{0}$ that
$l^{I}({^\vee}\xi_{0})$ is maximal among all the integral lengths
appearing from the representations in the block (\cite{ICIV}*{Lemma 12.10}).  If $\kappa$ is a
$\vartheta$-orbit of a simple root in $R^{+}(\lambda)$ with $\kappa
\notin \boldsymbol{\tau}({^\vee}\xi_{0})$ then
$l^{I}({^\vee}\xi'') > l^{I}({^\vee}\xi_{0})$ for some $\xi'' \in
B_{\kappa}({^\vee}\xi_{0})$ (\cite{LV2014}*{7.5-7.7})--a contradiction to the maximality of $l^{I}({^\vee}\xi_{0})$.
Therefore
$\kappa\in\boldsymbol{\tau}({^\vee}\xi_{0})$ for all
$\vartheta$-orbits $\kappa$.

If $\xi' \neq \xi_{0}$ then by the uniqueness hypothesis ${^\vee}\xi'$ is not
maximal in the Bruhat order.  By \cite{ICIV}*{Theorem 8.8},
$M({^\vee}\xi')^{+}$ is equal to the cross action or Cayley transform
of some representation in its block with higher integral length.
Looking to the formulas in \cite{LV2014}*{7.5-7.7}, we see that this
implies the existence of some
$\kappa\not\in\boldsymbol{\tau}({^\vee}\xi')$.

We now prove ${^\vee}P^{\vartheta}({^\vee}\xi',{^\vee}\xi_{0})= 1$ by
induction on $l^{I}({^\vee}\xi_{0}) -
l^{I}({^\vee}\xi')$.  If  $l^{I}({^\vee}\xi_{0}) = l^{I}({^\vee}\xi')$
then the uniqueness hypothesis implies ${^\vee}\xi' = {^\vee}\xi_{0}$
and we are done by \cite{LV2014}*{Theorem 5.2 (c)}.  Otherwise, $l^{I}({^\vee}\xi_{0}) > l^{I}({^\vee}\xi')$
and we have shown above that the hypotheses of
Corollary \ref{onebetween} are satisfied for some $\kappa$. Corollary
\ref{onebetween}
tells us that
${^\vee}P^{\vartheta}({^\vee}\xi',{^\vee}\xi_{0})={^\vee}P^{\vartheta}({^\vee}\xi'',
{^\vee}\xi_{0})$
for some
${^\vee}\xi'' \overset\kappa\to {^\vee}\xi'$. The condition
${^\vee}\xi'' \overset\kappa\to {^\vee}\xi'$ necessitates
$l^{I}({^\vee}\xi'') > l^{I}({^\vee}\xi')$
and so
$${^\vee}P^{\vartheta}({^\vee}\xi',{^\vee}\xi_{0})={^\vee}P^{\vartheta}({^\vee}\xi'',
{^\vee}\xi_{0})=1$$
by induction.
\end{proof}

\section{Endoscopic lifting for general linear groups following
  Adams-Barbasch-Vogan}
  \label{endosec}

In this section we review standard endoscopy and twisted endoscopy
from the perspective of \cite{ABV}, but restricted only to the
particular case of
the group $\mathrm{GL}_{N}$.  We shall be using all of the previously
defined objects and work under the assumption of (\ref{regintdom}) for
the infinitesimal character.  The references for this review are
\cite{ABV}*{Section 26} and   \cite{Christie-Mezo}*{Section 5}.

\subsection{Standard endoscopy}
\label{standend}

Let  ${^\vee}\mathrm{GL}_{N}^{\Gamma} =
{^\vee}\mathrm{GL}_{N} \rtimes \langle {^\vee}\delta_{0} \rangle$ be
as in (\ref{deltagln}).
An \emph{endoscopic datum} for ${^\vee}\mathrm{GL}_{N}^{\Gamma}$ is a pair
$$(s, \, {^\vee}G^{\Gamma})$$
which satisfies
\begin{enumerate}
\item $s \in {^\vee}\mathrm{GL}_{N}$ is semisimple

\item ${^\vee}G^{\Gamma} \subset {^\vee} \mathrm{GL}_{N}^{\Gamma}$ is
  open in the centralizer of $s$ in ${^\vee}\mathrm{GL}_{N}^{\Gamma}$

\item ${^\vee}G^{\Gamma}$ is an E-group for a group $G$ (\cite{ABV}*{Definition 4.6}).
\end{enumerate}
\nomenclature{$s$}{semisimple element in endoscopic datum}
This is a specialization of \cite{ABV}*{Definition 26.15} to
${^\vee}\mathrm{GL}_{N}^{\Gamma}$.  The groups ${^\vee}G$ and $G$ here are
isomorphic to products of smaller general linear groups. Consequently,
${^\vee}G$ and ${^\vee}\mathrm{GL}_{N}$ share the diagonal maximal torus
${^\vee}H$, and $G$ and $\mathrm{GL}_{N}$ share the diagonal maximal
torus $H$.  We shall abusively denote by $\delta_{q}$ the
strong involution on both $G$ and $\mathrm{GL}_{N}$ which correspond to the split real forms.  The group $G$ in this definition is called the
\emph{endoscopic group}.

We do not require the general concept of an E-group in this section.
From now on
we assume that ${^\vee}G^{\Gamma} =
{^\vee}G \times \langle
{^\vee}\delta_{0} \rangle$ where ${^\vee}\delta_{0}^{2} = 1$.   In
other words, ${^\vee}G^{\Gamma}$ is an L-group for $G$.

There is a notion of equivalence for endoscopic data, and using this
equivalence we may assume without loss of generality that $s \in
{^\vee}H$.  We fix $\lambda \in {^\vee}\mathfrak{h}$ satisfying the
hypotheses of  (\ref{regintdom}) so that $\lambda$ is
regular with respect to ${^\vee}\mathrm{GL}_{N}$. Let $\O_{G}$ be the
${^\vee}G$-orbit of $\lambda$ and $\O$ be the
${^\vee}\mathrm{GL}_{N}$-orbit of $\lambda$.
The second property of the endoscopic datum above allows us to define
the inclusion
\begin{equation}
\epsilon : \, {^\vee}G^{\Gamma} \hookrightarrow \, {^\vee}\mathrm{GL}_{N}^{\Gamma}.
\label{epinclusion}
\end{equation}
This inclusion induces another map (\cite{ABV}*{Corollary 6.21}),  which we abusively also denote as
\begin{equation}
\label{epinclusion1}
\epsilon: \, X\left(\O_{G}, {^\vee}G^{\Gamma}\right) \rightarrow
X\left(\O, {^\vee}\mathrm{GL}_{N}^{\Gamma}\right).
\nomenclature{$\epsilon$}{endoscopic maps}
\end{equation}
It is easily verified that the ${^\vee}G$-action on
$X(\O_{G}, {^\vee}G^{\Gamma})$ is compatibly carried under
$\epsilon$ to the ${^\vee}\mathrm{GL}_{N}$-action on $X(\O,
{^\vee}\mathrm{GL}_{N}^{\Gamma})$ (\cite{ABV}*{(7.17)}).  As a result,
 $\epsilon$ induces maps from the orbits of
$X(\O_{G}, {^\vee}G^{\Gamma})$ to the orbits of $X(\O,
{^\vee}\mathrm{GL}_{N}^{\Gamma})$. It also induces a homomorphism
\begin{equation}
\label{epinclusion2}
A^{\mathrm{loc}}(\epsilon) : {^\vee}G_{p}/({^\vee}G_{p})^{0} \rightarrow \,
({^\vee}\mathrm{GL}_{N})_{\epsilon(p)} / (
(\,{^\vee}\mathrm{GL}_{N})_{\epsilon(p)})^{0}
\nomenclature{$A^{\mathrm{loc}}(\epsilon)$}{endoscopic map of component groups}
\end{equation}
on the component groups.  As we have seen,
the component groups for $\mathrm{GL}_{N}$, and therefore also
for $G$, are trivial.

The inverse image functor of $\epsilon$ on equivariant constructible
sheaves induces a homomorphism
\begin{equation}
\label{inverseeps}
\epsilon^{*}: K(\O, {^\vee} \mathrm{GL}_{N}^{\Gamma})
\rightarrow K(\O_{G}, {^\vee}G^{\Gamma})
\nomenclature{$\epsilon^{*}$}{inverse image functor}
\end{equation}
\cite{ABV}*{Proposition 7.18}.
One may describe its values on irreducible constructible sheaves
$\mu(\xi)$, $\xi  = (S, \tau_{\xi}) \in
\Xi(\O, {^\vee}\mathrm{GL}_{N}^{\Gamma})$ as follows.
If the orbit $S$ is not the image of an orbit of $X(\O_{G},
{^\vee}G^{\Gamma})$ under $\epsilon$ then $\epsilon^{*} \mu(\xi) = 0$.
Otherwise, $S$
is the ${^\vee}\mathrm{GL}_{N}$-orbit of $\epsilon(p) \in S$ for some
$p \in X(\O_{G}, {^\vee}G^{\Gamma})$.  In this case
$\mu(\xi)$ may be identified with (\ref{bundle}) where
$\tau_{\xi} = 1$ the trivial
quasicharacter of the trivial group
$({^\vee}\mathrm{GL}_{N})_{\epsilon(p)} / (
(\,{^\vee}\mathrm{GL}_{N})_{\epsilon(p)})^{0} $.  The stalk of the
constructible sheaf $\epsilon^{*} \mu(\xi)$ at $p$ is then the stalk
$V$ of (\ref{bundle}). The representation on $V$ is given by the
quasicharacter $\tau_{\xi} \circ A^{\mathrm{loc}}(\epsilon)$, which is again
the trivial quasicharacter (on the trivial component group).   In
summary,
$$\epsilon^{*} \mu(\xi)  = \epsilon^{*} \mu(S, 1) = \sum_{\{
  S_{1}:  {^\vee}\mathrm{GL}_{N} \cdot \epsilon (S_{1}) = S  \}}
\mu(S_{1}, 1) $$
in which $S_{1}$ is a ${^\vee}G$-orbit in $X(\O_{G},
{^\vee}G^{\Gamma})$.  This sum will be seen to reduce to a single term
in Proposition \ref{injliftord}.

When $\epsilon^{*}$ is combined with the pairings of Theorem
\ref{ordpairing}, we obtain a map
$$\epsilon_{*} : K_{\mathbb{C}} \Pi(\O_{G}, G/\mathbb{R})
\rightarrow K_{\mathbb{C}} \Pi (\O,
\mathrm{GL}_{N}(\mathbb{R}))
\nomenclature{$K_{\mathbb{C}} \Pi(\O_{G}, G/\mathbb{R})$}{}
\nomenclature{$\epsilon_{*}$}{}
$$
defined on $\eta_{G} \in K_{\mathbb{C}} \Pi(\O_{G}, G/\mathbb{R})$ by
\begin{equation}
\label{loweps}
\left\langle \epsilon_{*} \eta_{G}, \mu(\xi) \right\rangle = \left\langle \eta_{G},
\epsilon^{*} \mu(\xi)\right\rangle_{G}, \quad \xi \in
\Xi(\O,{^\vee}\mathrm{GL}_{N}^{\Gamma}).
\end{equation}
Here, $K_{\mathbb{C} }= \mathbb{C} \otimes_{\mathbb{Z}} K$ and we have
placed a subscript $G$ beside the pairing on the right to distinguish
it from the pairing for $\mathrm{GL}_{N}$ on the left.

The endoscopic lifting map is a restriction of $\epsilon_{*}$ to a
subspace of $K_{\mathbb{C}} \Pi(\O_{G}, G/\mathbb{R})$
which is perhaps best described in two steps.  The first subspace is
generated by the (equivalence classes of) representations of the
quasisplit strong involution $\delta_{q}$ (Section \ref{extgroups}). We
denote this subspace by $K_{\mathbb{C}} \Pi(\O_{G}, G(
\mathbb{R}, \delta_{q}))$.  Lemma  \ref{GLn} tells
that
$$K_{\mathbb{C}} \Pi(\O_{G}, G( \mathbb{R}, \delta_{q})) =
K_{\mathbb{C}} \Pi(\O_{G}, G/\mathbb{R}),$$
but this will not be true when we look at twisted endoscopic groups in
Section \ref{twistendsec}.
Inside of  $K_{\mathbb{C}} \Pi(\O_{G}, G( \mathbb{R},
\delta_{q}))$ is the subspace generated by stable virtual characters
of $G(\mathbb{R}, \delta_{q})$ (Section \ref{grothchar},
\cite{ABV}*{18.2}).  We denote this subspace by $K_{\mathbb{C}}
\Pi(\O_{G}, G( \mathbb{R}, \delta_{q}))^{\mathrm{st}}$.  Again, since
$G$ is a product of general linear groups, stability is not an issue
and  we have
\begin{equation}
\label{samespace}
K_{\mathbb{C}} \Pi(\O_{G}, G( \mathbb{R}, \delta_{q}))^{\mathrm{st}} =
K_{\mathbb{C}} \Pi(\O_{G}, G/\mathbb{R}).
\end{equation}
This equality will not hold for twisted endoscopic groups in Section
\ref{twistendsec}.
We define the endoscopic lifting
\begin{equation}
\label{endlift}
\mathrm{Lift}_{0}: K_{\mathbb{C}} \Pi(\O_{G}, G(\mathbb{R},
\delta_{q}))^{\mathrm{st}} \rightarrow K_{\mathbb{C}} \Pi(\O,
\mathrm{GL}_{N}(\mathbb{R}))
\nomenclature{$\mathrm{Lift}_{0}$}{endoscopic lifting map}
\end{equation}
as the restriction of $\epsilon_{*}$ to $K_{\mathbb{C}}
\Pi(\O_{G}, G(\mathbb{R}, \delta_{q}))^{\mathrm{st}}$.

An argument of Shelstad (\cite{shelstad}, \cite{ABV}*{Lemma 18.11}) provides a
basis for $K_{\mathbb{C}} \Pi(\O_{G}, G(\mathbb{R},
\delta_{q}))^{\mathrm{st}}$.  The basis elements are of the form
\begin{equation}
\label{etaloc}
\eta^{\mathrm{loc}}_{S_{1}}(\delta_{q})  = \sum_{\tau_{S_{1}}} M(S_{1}, \tau_{S_{1}}),
\end{equation}
where $(S_{1}, \tau_{S_{1}}) \in
\Xi(\O_{G},{^\vee}G^{\Gamma})$ runs over those complete
geometric parameters which correspond to the strong involution
$\delta_{q}$  under the local Langlands correspondence.
As mentioned earlier, $\tau_{S_{1}}$ is trivial for $G$ and so
(\ref{etaloc}) reduces to
$$\eta^{\mathrm{loc}}_{S_{1}}(\delta_{q}) = M(S_{1}, 1),
\nomenclature{$\eta^{\mathrm{loc}}_{S_{1}}$}{virtual character of a pseudopacket}$$
a single standard representation.
\begin{prop}\label{injliftord}
In the setting of (\ref{regintdom}):
\begin{enumerate}[label={(\alph*)}]
\item Suppose $S_{1}, S_{2} \subset X(\O_{G},
  {^\vee}G^{\Gamma})$ are ${^\vee}G$-orbits which are carried to the
  same ${^\vee}\mathrm{GL}_{N}$-orbit $S$ under $\epsilon$.  Then
  $S_{1} = S_{2}$.
\item The endoscopic lifting map $\mathrm{Lift}_{0}$ (\ref{endlift}) is
  injective and
  sends $\eta^{\mathrm{loc}}_{S_{1}}(\delta_{q}) = M(S_{1}, 1)$ to
  $\eta^{\mathrm{loc}}_{S} = M(S, 1)$.

\item The endoscopic lifting map $\mathrm{Lift}_{0}$ is equal to the
  parabolic induction functor $\mathrm{ind}_{G(\mathbb{R},
    \delta_{q})}^{\mathrm{GL}_{N}(\mathbb{R})}$ on
  $K_{\mathbb{C}} \Pi(\O_{G}, G(\mathbb{R},
  \delta_{q}))^{\mathrm{st}}$.
\end{enumerate}
\end{prop}
\begin{proof}
By \cite{ABV}*{Definition 6.9} and  \cite{ABV}*{(6.4)},  we may take
${^\vee}G$-orbits
$$S_{1}= {^\vee}G \cdot (y_{1}, \mathcal{F}(\lambda)), \ S_{2} = {^\vee}G \cdot (y_{2}, \mathcal{F}(\mathrm{Ad}(g_{1})\lambda)),$$
where $g_{1} \in {^\vee}G$, $y_{1},y_{2} \in {^\vee}G^{\Gamma}$, and
$$\mathcal{F}(\lambda) = \mathrm{Ad}(P(\lambda)) \lambda, \
\mathcal{F}(\mathrm{Ad}(g_{1})\lambda) = \mathrm{Ad}(g_{1}) (\mathrm{Ad}(P(\lambda))\lambda)$$
for a solvable subgroup $P(\lambda) \subset {^\vee}G$ (\cite{ABV}*{Lemma 6.3}).
By the hypothesis of the first assertion, there exists $g \in {^\vee}\mathrm{GL}_{N}$ such that $gy_{2}g^{-1} = y_{1}$ and
$$\mathrm{Ad}(gg_{1}) (\mathrm{Ad}(P(\lambda))\lambda) = \mathrm{Ad}(P(\lambda))\lambda.$$
In particular, $\mathrm{Ad}(gg_{1}) \lambda = \mathrm{Ad}(g_{2}) \lambda$ for some $g_{2} \in P(\lambda)$.  As $\lambda \in {^\vee}\mathfrak{h}$ is regular the resulting equation
$\mathrm{Ad}(g_{2}^{-1}gg_{1}) \lambda =  \lambda$ implies that
$g_{2}^{-1}gg_{1} \in {^\vee}H \subset {^\vee}G$.  Since $g_{1}, g_{2}
\in {^\vee}G$, this also implies $g \in {^\vee}G$ and the first
assertion is proven.

The second part of the second assertion is  equivalent to
$\epsilon_{*} \eta^{\mathrm{loc}}_{S_{1}}(\delta_{q}) = \eta^{\mathrm{loc}}_{S}$ and
this is proved in \cite{ABV}*{Proposition 26.7}.
Finally, to prove the claim of injectivity, we observe that by the
first assertion $\mathrm{Lift}_{0}$ sends the basis
$$\left\{\eta^{\mathrm{loc}}_{S_{G}}(\delta_{q}) : S_{G} \mbox{ a }
{^\vee}G\mbox{-orbit of } X(\O_{G}, {^\vee}G^{\Gamma})\right\}$$
of $K_{\mathbb{C}} \Pi(\O_{G}, G(\mathbb{R},
\delta_{q}))^{\mathrm{st}}$ bijectively onto the  linearly independent subset
$$\left\{ \eta^{\mathrm{loc}}_{{^\vee}\mathrm{GL}_{N} \cdot S_{G}}: S_{G} \mbox{ a } {^\vee}G\mbox{-orbit of } X(\O_{G}, {^\vee}G^{\Gamma}) \right\}$$
of $K_{\mathbb{C}} \Pi(\O, \mathrm{GL}_{N}(\mathbb{R}))$.

We now prove the third assertion. Since the standard characters form a
    basis for  (\ref{samespace}) and $\mathrm{Lift}_{0}$ is additive,  it suffices to prove
    $\mathrm{Lift}_{0}\left(M(S_{1}, 1)\right) =
    \mathrm{ind}_{G(\mathbb{R},
      \delta_{q})}^{\mathrm{GL}_{N}(\mathbb{R})} M(S_{1},
    1)$.   By the second assertion, this is equivalent to
    proving
 \begin{equation}
\label{indstages}
 \mathrm{ind}_{G(\mathbb{R}, \delta_{q})}^{\mathrm{GL_{N}(\mathbb{R})}} M(S_{1}, 1) =
    M(S, 1).
    \end{equation}
Let us recall the definition of $M(S_{1}, 1)$ using the
Langlands classification \cite{Langlands}.  The ${^\vee}G$-orbit
$S_{1}$ corresponds to a unique ${^\vee}G$-orbit of an L-parameter
$\phi_{G}$ for $G$ (\cite{ABV}*{Proposition 6.17}).  The  image of
$\phi_{G}$ is contained in a  Levi subgroup ${^\vee}G_{0} \subset
{^\vee}G$ minimally (\cite{borel}*{Section 11.3}).  It follows that the
L-parameter $\phi_{G}$ factors through an L-parameter $\phi_{G_{0}}$
for $G_{0}$, and  $\phi_{G_{0}}$ corresponds to a unique
${^\vee}G_{0}$-orbit $S_{0}$ of geometric parameters for $G_{0}$.  The
standard characters are defined so that
$M(S_{1}, 1) = \mathrm{ind}_{G_{0}(\mathbb{R},
  \delta_{q})}^{G(\mathbb{R}, \delta_{q})} M(S_{0}, 1)$
and $M(S, 1) = \mathrm{ind}_{G_{0}(\mathbb{R},
  \delta_{q})}^{\mathrm{GL}_{N}(\mathbb{R})}
M(S_{0}, 1)$ (\cite{ABV}*{(11.2)}, \cite{AV92}*{(8.22)}).
Identity (\ref{indstages}) is therefore a consequence of induction by
stages.
\end{proof}
The proof that $\mathrm{Lift}_{0}\left(\eta^{\mathrm{loc}}_{S_{1}}\right) (\delta_{q}) =
\eta^{\mathrm{loc}}_{S}$ in this proposition follows from an elementary
computation of $\epsilon^{*} \eta^{\mathrm{loc}}_{S}$ (\cite{ABV}*{Proposition 23.7}).  It is much more difficult to compute the value of
$\mathrm{Lift}_{0}$ on the stable virtual character $\eta^{\mathrm{mic}}_{\psi_{G}}$
given in (\ref{etapsi}).   Let $\psi = \epsilon \circ \psi_{G}$.
According to \cite{ABV}*{Theorem 26.25}
\begin{equation}
\label{ordmiclift}
\mathrm{Lift}_{0}\left(\eta^{\mathrm{mic}}_{\psi_{G}}\right)= \sum_{ \xi \in
  \Xi(\O,{^\vee}\mathrm{GL}_{N}^{\Gamma})}
(-1)^{d(S_{\xi}) - d(S_{\psi})}
\,\chi_{S_{\psi}}^{\mathrm{mic}} (P(\xi) ) \,\pi(\xi) =
\eta^{\mathrm{mic}}_{\psi}.
\end{equation}
Recall from (\ref{abvdef}) that the ABV-packets $\Pi^{\mathrm{ABV}}_{\psi_{G}}$ and
$\Pi^{\mathrm{ABV}}_{\psi}$ are defined from $\eta^{\mathrm{mic}}_{\psi_{G}}$
and $\eta^{\mathrm{mic}}_{\psi}$ respectively.  We shall see in Section \ref{glnpacket}
that these ABV-packets are singletons.  Equation (\ref{ordmiclift})
implies that the endoscopic lift of $\Pi^{\mathrm{ABV}}_{\psi_{G}}$ is
$\Pi^{\mathrm{ABV}}_{\psi}$.

\subsection{Twisted endoscopy}
\label{twistendsec}

Following the path of the previous section, we  define  twisted
endoscopic data, the twisted endoscopic version of $\mathrm{Lift}_{0}$ (\ref{endlift}),
compute  twisted variants of $\mathrm{Lift}_{0}
(\eta^{\mathrm{loc}}_{S})$  for $S\in  X(\O_{G}, {^\vee}G^{\Gamma})$, and compute twisted variants of $\mathrm{Lift}_{0} (\eta^{\mathrm{mic}}_{\psi_{G}})$.

An \emph{endoscopic datum} for $({^\vee}\mathrm{GL}_{N}^{\Gamma},
\vartheta)$ is a pair
$$(s, \, {^\vee}G^{\Gamma})$$
which satisfies
\begin{enumerate}
\item $s \in {^\vee}\mathrm{GL}_{N}$ is $\vartheta$-semisimple (see
\cite{KS}*{(2.1.3)})

\item ${^\vee}G^{\Gamma} \subset {^\vee} \mathrm{GL}_{N}^{\Gamma}$ is
  open in the fixed-point set of $\upsigma = \mathrm{Int}(s) \circ
  \vartheta$ in $\GL_N^\Gamma\rtimes\langle \vartheta\rangle$  (\ref{eq:GLDelta})
\item ${^\vee}G^{\Gamma}$ is an E-group for a group $G$ (\cite{ABV}*{Definition 4.6}).
\end{enumerate}
This is a special case of  \cite{Christie-Mezo}*{Definition 5.1} to
${^\vee}\mathrm{GL}_{N}^{\Gamma}$.  There is a notion of equivalence for these
endoscopic data (\cite{Christie-Mezo}*{Definition 5.6},
\cite{KS}*{(2.1.5)-(2.1.6)}). Up to this equivalence the relevant elements $s \in  {^\vee}\mathrm{GL}_{N}$ in Arthur's work are drawn from
     \cite{Arthur}*{Section 1.2} (see \cite{Christie-Mezo}*{Example 5.2} for the
     details in the present setting).
We shall restrict our attention to at most two elements $s$ from
\cite{Arthur}*{Section 1.2}, namely the diagonal matrix
\begin{equation}
\label{simple1}
s =  \begin{bmatrix} 0 & & 1\\
 & \iddots &\\
 1 & &0 \end{bmatrix}
\tilde{J}^{-1} = \small \begin{bmatrix}   1  &  &  &0 \\
 & -1& &  \\
 &  & \ddots  & \\
0  & & & (-1)^{N-1} \end{bmatrix}, \normalsize
\end{equation}
and, if $N$ is even, the diagonal matrix
\begin{equation}
\label{simple2}
s =  \tiny \begin{bmatrix} 0 & & & & & -1\\
& & &  & \iddots &\\
& & & -1& & \\
& & 1& & &\\
& \iddots & & & &\\
 1 & & & & &0 \end{bmatrix} \normalsize
\tilde{J}^{-1}.
\end{equation}
The resulting automorphisms $\upsigma = \mathrm{Int}(s) \circ
\vartheta$ are involutions, and the endoscopic groups $G$ are simple
to describe.  For $s$ as in (\ref{simple1}) and odd $N$, the
endoscopic group is $\mathrm{Sp}_{N-1}$.  For $s$ as in
(\ref{simple1}) and even $N$, the endoscopic group is
$\mathrm{SO}_{N}$.  For $s$ as in (\ref{simple2}), the endoscopic
group is $\mathrm{SO}_{N+1}$. In the first and last cases, the
embeddings of ${^\vee}G^{\Gamma}$ in ${^\vee}\mathrm{GL}_{N}^{\Gamma}$
determine a real structure on $G$ as a split group.   In the second
case there are two distinct embeddings of ${^\vee}G^{\Gamma}$ in
${^\vee}\mathrm{GL}_{N}^{\Gamma}$ which determine either a split real
structure on $G = \mathrm{SO}_{N}$, or a non-split and quasisplit
structure on $\mathrm{SO}_{N}$ (\cite{Arthur}*{Section 1.2}).  We henceforth
work only with $s$ as in
(\ref{simple1})-(\ref{simple2}), the so-called \emph{simple}
endoscopic data.  As in the previous
section, we will not require the concept of an E-group for these
classical endoscopic groups.  From now on we assume that ${^\vee}G^{\Gamma}$
 is an
L-group as in the end of Section \ref{extended}.

Unlike the previous section, we must distinguish between maximal tori
in ${^\vee}G$ and ${^\vee}\mathrm{GL}_{N}$.   We let ${^\vee}H$ be the
diagonal maximal torus in ${^\vee}\mathrm{GL}_{N}$, and ${^\vee}H_{G}$
be  a maximal torus in ${^\vee}G$.
In this case ${^\vee}H$ is preserved by $\mathrm{Int}(s)$ as well as
$\vartheta$.  We may and shall take
$${^\vee}H_{G} =
({^\vee}H)^{\upsigma} = ({^\vee}H)^{\vartheta}.$$

We fix $\lambda \in
{^\vee}\mathfrak{h}$ satisfying (\ref{regintdom}) so that $\lambda$ is
regular with respect to
${^\vee}\mathrm{GL}_{N}$. Let $\O_{G}$ be the
${^\vee}G$-orbit of
$\lambda$ and $\O$ be the ${^\vee}\mathrm{GL}_{N}$-orbit of
$\lambda$.

The $\epsilon$ maps of (\ref{epinclusion})-(\ref{epinclusion1}) have
obvious analogues and are equally valid in the twisted setting.  The
proof of the injectivity of $\epsilon$ on the level of
${^\vee}G$-orbits is more delicate in the twisted setting.
\begin{prop}
  \label{twistinj}
In the setting of (\ref{regintdom}),
 suppose $S_{1}, S_{2} \subset X(\O_{G}, {^\vee}G^{\Gamma})$
 are ${^\vee}G$-orbits which are carried to the same
 ${^\vee}\mathrm{GL}_{N}$-orbit under $\epsilon$.  Then $S_{1} =
 S_{2}$.
\end{prop}
\begin{proof}
By \cite{ABV}*{Proposition 6.17}, the ${^\vee}G$-orbits of
$X(\O_{G}, {^\vee}G^{\Gamma})$ are in one-to-one
correspondence with ${^\vee}G$-orbits of L-parameters with
infinitesimal character in $\O_{G}$.  It is simpler to prove
the proposition in the more familiar territory of L-parameters.
To this end let $\phi_{1} : W_{\mathbb{R}} \rightarrow
{^\vee}G^{\Gamma}$ and $\phi_{2}: W_{\mathbb{R}} \rightarrow
{^\vee}G^{\Gamma}$ be L-parameters with infinitesimal characters
$\lambda_{1} \in \O_{G}$ and $\lambda_{2} \in
\O_{G}$ respectively.    Suppose further that
\begin{equation}
  \label{gconj}
 \phi_{1}=  \mathrm{Int}(g) \circ \phi_{2}
  \end{equation}
for some $g \in {^\vee}\mathrm{GL}_{N}$.
We shall prove that there exists $g' \in {^\vee}G$ such that $\phi_{1}
= \mathrm{Int}(g') \circ \phi_{2}$.
Without loss of generality we may assume that $\lambda_{1} = \lambda
\in {^\vee}\mathfrak{h}$ and $\lambda_{2} = \mathrm{Ad}(g_{1}) \lambda
\in {^\vee}\mathfrak{h}$ for some $g_{1} \in {^\vee}G$ normalizing
${^\vee}H$.  Moreover, we may assume that $\phi_{1}(\mathsf{j})$ and
$\phi_{2}(\mathsf{j})$ normalize ${^\vee}H$.  As in \cite{ABV}*{Proposition 5.6},
$$\phi_{1} (z) = z^{\lambda}
\bar{z}^{\mathrm{Ad}(\phi_{1}(\mathsf{j}))\lambda}, \quad \phi_{2} (z) =
z^{\mathrm{Ad}(g_{1}) \lambda}
\bar{z}^{\mathrm{Ad}(\phi_{2}(\mathsf{j})g_{1})\lambda}, \quad z \in
\mathbb{C}^{\times}.$$
By hypothesis
$$z^{\mathrm{Ad}(gg_{1}) \lambda}
\bar{z}^{\mathrm{Ad}(g\phi_{2}(\mathsf{j})g_{1})\lambda} = \mathrm{Int}(g)
\circ \phi_{2}(z)= \phi_{1}(z) = z^{\lambda}
\bar{z}^{\mathrm{Ad}(\phi_{1}(\mathsf{j}))\lambda}.$$
It follows that $\mathrm{Ad}(gg_{1}) \lambda = \lambda$.  Since
$\lambda$ is regular in ${^\vee}\mathrm{GL}_{N}$,  the element
$gg_{1}$ belongs to ${^\vee}H$ and $g = hg_{1}^{-1}$ for some $h \in
{^\vee}H$.
Returning to the identity (\ref{gconj}), we obtain
$$h g_{1}^{-1} \phi_{2}(\mathsf{j}) g_{1}h^{-1}= \phi_{1}(\mathsf{j}).$$
Set $n_{1}  = \phi_{1}(\mathsf{j})$ and $n_{2} = g_{1}^{-1} \phi_{2}(\mathsf{j}) g_{1}$,
so that the previous equation becomes
$$h n_{2} h^{-1} = n_{1}.$$
As $\epsilon$ maps into the fixed-point set of $\upsigma = \mathrm{Int}(s) \circ
\vartheta$, both $n_{1}, n_{2}$ are fixed by $\upsigma$.
Since $n_{1}$ and $n_{2}$ normalize ${^\vee}H$, they represent
involutive elements in the
fixed-point subgroup $W({^\vee}\mathrm{GL}_{N}, {^\vee}H)^{s\vartheta}
= W({^\vee}\mathrm{GL}_{N}, {^\vee}H)^{\vartheta}$.  We seek an
element $h' \in {^\vee}H^{\vartheta} \subset {^\vee}G$ such that
\begin{equation}
\label{seek}
h' n_{2} \, (h')^{-1} = h n_{2} h^{-1} = n_{1}.
\end{equation}
With this element in hand we may set $g' = h' g_{1}^{-1}$ and
the proposition is proved.
Since ${^\vee}H = {^\vee}H^{\vartheta} \, {^\vee}H^{-\vartheta}$, we
may decompose $h = h_{1}h_{2}$, where $h_{1} \in {^\vee}H^{\vartheta}$
and $h_{2} \in{^\vee}H^{-\vartheta}$. We compute that
\begin{equation}
  \label{hdecomp}
  (h_{1}n_{2}h_{1}^{-1}n_{2}^{-1}) (h_{2}n_{2}h_{2}^{-1}n_{2}^{-1})
 = h_{1} (h_{2}n_{2}h_{2}^{-1}n_{2}^{-1})(n_{2}h_{1}^{-1}n_{2}^{-1}) =
 n_{1}n_{2}^{-1} \in {^\vee}H^{s\vartheta} =
{^\vee}H^{\vartheta}.
\end{equation}
In addition, since $n_{1}, n_{2} \in W({^\vee}\mathrm{GL}_{N},
{^\vee}H)^{\vartheta}$, we have
$$\vartheta(h_{2}n_{2} h_{2}^{-1} n_{2}^{-1})  =
\vartheta(h_{2}) n_{2} \vartheta(h_{2}^{-1})n_{2}^{-1} = h_{2}^{-1} n_{2} h_{2}
n_{2}^{-1} = (h_{2} n_{2} h_{2}^{-1} n_{2}^{-1})^{-1}.$$
Therefore $ h_{2} n_{2} h_{2}^{-1}n_{2}^{-1} \in
{^\vee}H^{-\vartheta}$.   Similarly, $ h_{1} n_{2}
h_{1}^{-1} n_{2}^{-1} \in {^\vee}H^{\vartheta}$.   By (\ref{hdecomp}) we have
$ h_{2} n_{2} h_{2}^{-1}n_{2}^{-1} \in {^\vee}H^{-\vartheta} \cap
{^\vee}H^{\vartheta}$.   It is now an elementary exercise to show that
there exists $h_{2}'$ in the finite $2$-group ${^\vee}H^{-\vartheta}
\cap {^\vee}H^{\vartheta}$ such that
$$ h_{2}' n_{2} (h_{2}')^{-1}n_{2}^{-1} =  h_{2} n_{2} h_{2}^{-1}n_{2}^{-1}.$$
Indeed, conjugation by $n_{2}$ is represented by a
$\vartheta$-invariant product of $2$-cycles when
$W({^\vee}\mathrm{GL}_{N}, {^\vee}H)^{\vartheta}$ is regarded as a
subgroup of the symmetric group, and the elements in
${^\vee}H^{-\vartheta} \cap {^\vee}H^{\vartheta}$ are represented by
diagonal elements with $\pm1$ as entries.
Leaving the details of the exercise to the reader, we set $h' =
h_{1}h_{2}'$ and (\ref{seek}) holds.

\end{proof}

At this point the picture of twisted endoscopy is more or less the
same as the picture of standard endoscopy.  The new idea in the
twisted setting is to include the action of
$\upsigma = \mathrm{Int}(s) \circ \vartheta$ into the objects
pertinent to endoscopy.  In particular we
wish to extend the sheaf theory of \cite{ABV} for
${^\vee}\mathrm{GL}_{N}$ to the disconnected group
${^\vee}\mathrm{GL}_{N} \rtimes \langle \upsigma \rangle$, where we
identify the automorphism $\upsigma$ of
(\ref{sigmaaut}) with the automorphism in the endoscopic datum. This
mimics the extension of the representation theory of $\mathrm{GL}_{N}$
to the disconnected group $\mathrm{GL}_{N} \rtimes \langle \vartheta
\rangle$ in Section \ref{extrepsec}.  Rather than viewing the sheaves
in $\mathcal{C}(\O, {^\vee}\mathrm{GL}_{N}^{\Gamma}
;\upsigma)$ as ${^\vee}\mathrm{GL}_{N}$-equivariant with compatible
$\upsigma$-action (Section \ref{conperv}), we view them simply as $({^\vee}\mathrm{GL}_{N}
\rtimes \langle \upsigma \rangle)$-equivariant sheaves and apply the
theory of \cite{ABV} which is valid in this generality
\cite{Christie-Mezo}*{Section 5.4}.

Let $\xi = (S, 1)  \in \Xi(\O, {^\vee}
\mathrm{GL}_{N}^{\Gamma})^\vartheta$ and $(p,1)$ be a
representative for the class $\xi$.   Here, $p \in S$ and
$1$ is the trivial representation  of the trivial group
${^\vee} (\mathrm{GL}_{N})_{p}/ ({^\vee} (\mathrm{GL}_{N})_{p} )^{0}$
with representation space $V \cong \mathbb{C}$ as in (\ref{bundle}).
We define $1^{+}$ on
\begin{equation}
\label{twistglncompgp}
{^\vee}(\mathrm{GL}_{N})_{p} /({^\vee} (\mathrm{GL}_{N})_{p} )^{0}
\times \langle \upsigma \rangle
\end{equation}
 by
\begin{equation}
\label{oneextend}
1^{+} (\upsigma)  = \upsigma_{\mu(\xi)^{+} }= 1
\end{equation}
(\emph{cf.} (\ref{expaut})).
In this way, $1^{+}$ defines the local system underlying the
irreducible $({^\vee}\mathrm{GL}_{N} \rtimes \langle \upsigma
\rangle)$-equivariant constructible sheaf $\mu(\xi)^{+}$ (Lemma
\ref{cansheaf},   \cite{ABV}*{\emph{p.} 83}).

In a similar, but completely vacuous, fashion we may include  the
trivial action of $\upsigma$ on $\mu(\xi_{1}) \in
\mathcal{C}(\O_{G}, {^\vee}G)$ with $\xi_{1} = (S_{1},
\tau_{1})$ and $p_{1} \in S_{1}$.  In other words, we may regard
$\mu(\xi_{1})$ as a $({^\vee}G \times \langle \upsigma
\rangle)$-equivariant sheaf whose underlying local system is defined
by a quasicharacter $\tau_{1}^{+}$ on
\begin{equation}
\label{twistcompgroup}
{^\vee}G_{p_{1}}/ ({^\vee} G_{p_{1}})^{0} \times \langle \upsigma \rangle
\end{equation}
 by $\tau_{1}^{+}(\upsigma) = 1$.
 This artifice allows us to define a homomorphism,  as in (\ref{epinclusion2}), by
  \begin{equation*}
 \label{epinclusion3}
 A^{\mathrm{loc}}(\epsilon):  {^\vee}G_{p_{1}}/({^\vee}G_{p_{1}})^{0} \times
 \langle \upsigma \rangle \rightarrow \,
 ({^\vee}\mathrm{GL}_{N})_{\epsilon(p_{1})}  / (
 (\,{^\vee}\mathrm{GL}_{N})_{\epsilon(p_{1})})^{0} \times \langle
 \upsigma \rangle.
\end{equation*}

The inverse image functor (\ref{inverseeps}) in the twisted setting is
defined on  $({^\vee}\mathrm{GL}_{N} \rtimes \langle \upsigma
\rangle)$-equivariant sheaves (Section \ref{conperv})
$$\epsilon^{*}:  KX(\O, {^\vee}\mathrm{GL}_{N}^{\Gamma},
\upsigma) \rightarrow KX(\O_{G}, {^\vee}G^{\Gamma}) $$
in the following manner.  If the orbit $S$  in $\xi = (S, 1)
\in \Xi(\O, {^\vee} \mathrm{GL}_{N}^{\Gamma})^\vartheta$ is
not the image of an orbit of $X(\O_{G}, {^\vee}G^{\Gamma})$
under $\epsilon$ then $\epsilon^{*} \mu(\xi) = 0$.  Otherwise $S$  is
the ${^\vee}\mathrm{GL}_{N}$-orbit of $\epsilon(p_{1}) \in S$ for some
$p_{1} \in S_{1} \subset X(\O_{G}, {^\vee}G^{\Gamma})$.
By Proposition \ref{twistinj} the orbit $S_{1}$ is unique.  The stalk of the
constructible sheaf $\epsilon^{*} \mu(\xi)^{+}$ at $p_{1}$ is  the
stalk $V \cong \mathbb{C}$ of (\ref{bundle}) at $p = \epsilon(p_{1})$. The
representation on $V$ is given by the quasicharacter $1^{+}
\circ A^{\mathrm{loc}}(\epsilon)$, which is again the (abusively denoted)
trivial quasicharacter $1^{+}$ on the group
(\ref{twistcompgroup}).   In summary,
\begin{equation}
\label{epsrest1}
\epsilon^{*} \mu(\xi)^{+}  = \epsilon^{*} \mu(S, 1^{+})=
\mu(S_{1}, 1^{+}).
\end{equation}
By contrast $\mu(\xi)^{-}$ is characterized by the quasicharacter
$1^{-}$ of (\ref{twistglncompgp}) with
$1^{-}(\upsigma) = -1$.
With some
obvious substitutions we obtain
$$\epsilon^{*} \mu(\xi)^{-}  = \epsilon^{*} \mu(S, 1^{-})=
\mu(S_{1}, 1^{-}).$$

As in standard endoscopy, we combine $\epsilon^{*}$ with a pairing,
namely the pairing of Theorem \ref{twistpairing}, to define
$$\epsilon_{*}:  K_{\mathbb{C}} \Pi ( \O_{G} ,G/ \mathbb{R}
) \rightarrow K_{\mathbb{C}} \Pi(\O,
\mathrm{GL}_{N}(\mathbb{R}), \vartheta).$$
To be precise, the image of  any $\eta \in K_{\mathbb{C}}
\Pi(\O_{G}, G/\mathbb{R})$ under  $\epsilon_{*}$ is
determined by
\begin{equation}
\label{twistloweps}
\left\langle \epsilon_{*} \eta, \mu(\xi)^{+}\right \rangle = \left\langle
\eta, \epsilon^{*} \mu(\xi)^{+}\right\rangle_{G}, \quad \xi \in
\Xi(\O,{^\vee}\mathrm{GL}_{N}^{\Gamma})^\vartheta.
\end{equation}
(\emph{cf.} (\ref{loweps})).  The twisted endoscopic lifting map
\begin{equation}
\label{twistendlift}
\mathrm{Lift}_{0}: K_{\mathbb{C}} \Pi(\O_{G}, G(\mathbb{R},
\delta_{q}))^{\mathrm{st}} \rightarrow K_{\mathbb{C}} \Pi(\O,
\mathrm{GL}_{N}(\mathbb{R}), \vartheta)
\nomenclature{$\mathrm{Lift}_{0}$}{twisted endoscopic lifting map}
\end{equation}
is the restriction of $\epsilon_{*}$ to $K_{\mathbb{C}}
\Pi(\O_{G}, G(\mathbb{R}, \delta_{q}))^{\mathrm{st}}$, a proper
subspace of $K_{\mathbb{C}}(\O_{G},G/ \mathbb{R})$
(\emph{cf.} (\ref{samespace})).  The pairing on
the  right-hand side of (\ref{twistloweps}) is determined by pairing
representations of  $G(\mathbb{R}, \delta_{q})$  with elements of
the form $\mu(S_{1}, \tau_{1}^{\pm})$.  This is defined by
$$\left\langle M(S_{1}, \tau_{1}), \mu(S_{1}, \tau_{1}^{\pm}) \right\rangle_{G} =
\pm  1,\quad \mbox{  and  }\quad \left\langle M(S_{1}', \tau_{1}'), \mu(S_{1},
\tau_{1}^{\pm}) \right\rangle_{G} = 0 $$
when $\tau_{1}' \neq \tau_{1}$.  (For a more conceptual explanation of
these pairings see \cite{Christie-Mezo}*{ \emph{p.} 151}.)

Now, we wish to compute $\mathrm{Lift}_{0}$ on the basis elements
(\ref{etaloc}) of $K_{\mathbb{C}} \Pi(\O_{G}, G(\mathbb{R},
\delta_{q}))^{\mathrm{st}}$.  To maintain ease of comparison with \cite{ABV} we
compute $\mathrm{Lift}_{0}$ on the virtual representations
$\eta_{S_{1}}^{\mathrm{loc}}(\upsigma)(\delta_{q})$ (\cite{ABV}*{\emph{p.} 279}). These virtual characters are defined by
\begin{equation*}
\label{etaloc1}
\eta_{S_{1}}^{\mathrm{loc}}(\upsigma)(\delta_{q}) = \sum_{\tau_{1} }
\mathrm{Tr}( \tau_{1}^{+}(\upsigma)) \, M(S_{1}, \tau_{1}) =
\sum_{\tau_{1} } M(S_{1}, \tau_{1}),
\nomenclature{$\eta_{S_{1}}^{\mathrm{loc}}(\upsigma)(\delta_{q})$}{}
\end{equation*}
where $\tau_{1}$ runs over all quasicharacters of ${^\vee}G_{p_{1}}/
({^\vee}G_{p_{1}})^{0}$ as in (\ref{twistcompgroup}) which correspond
to the strong involution $\delta_{q}$ (\cite{ABV}*{Definition 18.9}).
It is immediate from the definitions that
\begin{equation*}
\label{nosigmaloc}
\eta_{S_{1}}^{\mathrm{loc}}(\upsigma)(\delta_{q}) =
\eta_{S_{1}}^{\mathrm{loc}}(\delta_{q})
\end{equation*}
and so this virtual character is stable (\cite{ABV}*{Lemma 18.10}).
(Although not needed for our purposes, one could
adhere to the framework of \cite{ABV} further by extending \cite{ABV}*{Definitions 26.10 and 26.13} through taking products with $\langle
\upsigma \rangle$, and then speak of $\upsigma$-stability.)
\begin{prop}
\label{twistimlift}
Suppose $S_{1} \subset X(\O_{G}, {^\vee}G^{\Gamma})$ is a
${^\vee}G$-orbit which is carried to a ${^\vee}\mathrm{GL}_{N}$-orbit
$S$ under $\epsilon$. Then the endoscopic lifting map
(\ref{twistendlift})  sends $\eta^{\mathrm{loc}}_{S_{1}}(\upsigma)(\delta_{q})
$ to $(-1)^{l^{I}(S, 1) - l^{I}_{\vartheta} (S, 1)}
\, M(S, 1)^{+}$.
\end{prop}
\begin{proof}
We prove the proposition without the injectivity of orbits given in
Lemma \ref{twistinj}.  Not assuming injectivity, instead of
equation (\ref{epsrest1}), we see that
$$\epsilon^{*} \mu(S, 1)^{+} = \sum_{S_{1}'} \mu(S_{1}',
1) = \mu(S_{1}, 1) + \sum_{S_{1}' \neq S_{1}}
\mu(S_{1}', 1),$$
where $S_{1}'$ runs over the $^{\vee}G$-orbits in $X(\O_{G},
{^\vee}G^{\Gamma})$ carried to $S'$ under $\epsilon$ (\emph{cf.}
\cite{ABV}*{Proposition 23.7 (b)}).   Therefore, according to
(\ref{pairdef2}),  when $\xi = (S,1)$
\begin{align*}
\left\langle \mathrm{Lift}_{0}\left(\eta^{\mathrm{loc}}_{S_{1}}
(\upsigma)(\delta_{q})\right), \mu(\xi)^{+} \right\rangle & =
\left\langle \eta^{\mathrm{loc}}_{S_{1}} (\upsigma)(\delta_{q}), \, \epsilon^{*}
\mu(S,1)^{+} \right\rangle_{G}\\
& = \left\langle \sum_{\tau_{1} } M(S_{1}, \tau_{1}), \mu(S_{1}, 1)\right \rangle_{G} \\
& = 1
\end{align*}
Now suppose $\xi = (S', 1)$ where $S'$ is a
${^\vee}\mathrm{GL}_{N}$-orbit not equal to $S$.  Then
$$\epsilon^{*} \mu(S', 1)^{+} = \sum_{S_{1}'} \mu(S_{1}', 1) $$
where $S_{1}' \neq S_{1}.$
If the sum runs over the empty set then we interpret it to equal zero
and compute that
$$\left\langle \mathrm{Lift}_{0}\left(\eta^{\mathrm{loc}}_{S_{1}}
(\upsigma)(\delta_{q})\right), \mu(\xi)^{+} \right\rangle = \left\langle
\eta^{\mathrm{loc}}_{S_{1}} (\upsigma)(\delta_{q}), \epsilon^{*} \mu(\xi)^{+}
\right\rangle_{G} = \left\langle \eta^{\mathrm{loc}}_{S_{1}} (\upsigma)(\delta_{q}), 0
\right\rangle_{G} = 0.$$
Otherwise the index of the sum is not empty and by (\ref{pairdef2})
\begin{align*}
\left\langle \mathrm{Lift}_{0}\left(\eta^{\mathrm{loc}}_{S_{1}}
(\upsigma)(\delta_{q})\right), \mu(\xi)^{+} \right\rangle & =
\left\langle \eta^{\mathrm{loc}}_{S_{1}} (\upsigma)(\delta_{q}), \, \epsilon^{*}
\mu(\xi)^{+} \right\rangle_{G}\\
& = \left\langle \sum_{\tau_{1} } M(S_{1}, \tau_{1}),  \sum_{S_{1}'}
\mu(S_{1}', 1) \right\rangle_{G} \\
& = 0
\end{align*}
Looking back to the definition of the pairing (\ref{pair2}), we see that we have proven the proposition.
\end{proof}

\begin{prop}
\label{injlift2}
Under the hypothesis of (\ref{regintdom}), the twisted endoscopic lifting map (\ref{twistendlift}) is injective.
\end{prop}
\begin{proof}
The proof follows from Lemma \ref{twistinj}  as
in the proof of Proposition \ref{injliftord}.  We need only to observe that
according to Proposition \ref{twistimlift}, $\mathrm{Lift}_{0}$ sends the
basis
$$\left\{\eta^{\mathrm{loc}}_{S_{G}}(\delta_{q}) : S_{G} \mbox{ a }
{^\vee}G\mbox{-orbit of } X(\O_{G}, {^\vee}G^{\Gamma}) \right\}$$
of $K_{\mathbb{C}} \Pi(\O_{G}, G(\mathbb{R},
\delta_{q}))^{\mathrm{st}}$ bijectively onto the  linearly independent subset
$$\left\{ (-1)^{l^{I}(\epsilon(S_{G}), 1) - l^{I}_{\vartheta} (\epsilon(S_{G}), 1)}
\, M(\epsilon(S_{G}), 1)^{+}: S_{G} \mbox{ a }
      {^\vee}G\mbox{-orbit of } X(\O_{G}, {^\vee}G^{\Gamma})
      \right\}$$
of $K_{\mathbb{C}} \Pi(\O, \mathrm{GL}_{N}(\mathbb{R}),\vartheta)$.
\end{proof}

The next and final goal of this section is to provide the twisted
analogue of the endoscopic lifting of the virtual characters attached
to A-parameters as in (\ref{ordmiclift}).  As a guiding principle, it
helps to remember that in moving from $\eta^{\mathrm{loc}}_{S}$ to
$\eta^{\mathrm{loc}}_{S}(\upsigma)(\delta_{q})$ we extended the component
groups by $\langle \upsigma \rangle$ to obtain (\ref{twistcompgroup}),
and then extended the quasicharacters $\tau_{1}$ defined on the
original component groups.  We shall follow the same process with
$\eta^{\mathrm{mic}}_{\psi_{G}}$, doing our best to avoid the theory of microlocal
geometry.

The stable virtual character (\ref{etapsi}) for the endoscopic group $G$ is
$$\eta^{\mathrm{mic}}_{\psi_{G}} =  \sum_{\xi \in \Xi(\O_{G}, {^\vee}G^{\Gamma})}
(-1)^{d(S_{\xi}) - d(S_{\psi_{G}})} \ \chi^{\mathrm{mic}}_{S_{\psi_{G}}}(P(\xi)) \, \pi(\xi)
\in K\Pi(\O_{G}, G/\mathbb{R})^{\mathrm{st}}.$$
Here, $S_{\psi_{G}} \subset X(\O_{G}, G^{\Gamma})$  is the
${^\vee}G$-orbit determined by the L-parameter $\phi_{\psi_{G}}$, and $\xi
= (S_{\xi}, \tau_{S_{\xi}})$.  For each such $\xi$, there is a
representation $\tau^{\mathrm{mic}}_{S_{\psi_{G}}}(P(\xi))$ of ${^\vee}G_{\psi_{G}}/
({^\vee}G_{\psi_{G}})^{0}$,
\nomenclature{$\tau^{\mathrm{mic}}_{S_{\psi_{G}}}(P(\xi))$}{representation of component group}
the component group of the centralizer in
${^\vee}G$ of the image of $\psi_{G}$, which satisfies the following
properties
\begin{align}
\label{miclocalsys}
& \bullet \tau^{\mathrm{mic}}_{S_{\psi_{G}}}(P(\xi)) \mbox{ represents  a (possibly
  zero) } {^\vee}G\mbox{-equivariant local system }  Q^{\mathrm{mic}}(P(\xi))\text{ of }\\
& \nonumber  \ \ \mbox{ complex vector spaces.} \\
\label{degtaumic}
& \bullet \mbox{The degree of } \tau^{\mathrm{mic}}_{S_{\psi_{G}}}(P(\xi)) \mbox{
  is equal to } \chi_{S_{\psi_{G}}}^{\mathrm{mic}}(P(\xi)). \\
\label{spectaumic}
& \bullet  \mbox{If } \xi = (S_{\psi_{G}}, \tau_{S_{\psi_{G}}})\mbox{ then }
\tau^{\mathrm{mic}}_{S_{\psi_{G}}}(P(\xi)) = \tau_{S_{\psi_{G}}} \circ i_{S_{\psi_{G}}},
\mbox{ where }\\
& \nonumber \qquad\qquad\qquad\qquad\qquad\qquad\qquad i_{S_{\psi_{G}}}: {^\vee}G_{\psi_{G}}/ ({^\vee}G_{\psi_{G}})^{0}
\rightarrow {^\vee}G_{p}/ ({^\vee}G_{p})^{0}\\
&\nonumber \ \ \mbox{ is a surjective  homomorphism for } p \in S_{\psi_{G}}.
\nomenclature{$Q^{\mathrm{mic}}(P(\xi))$}{local system}
\nomenclature{$i_{S_{\psi_{G}}}$}{homomorphism of component groups}
\end{align}
(\cite{ABV}*{Theorem 24.8, Corollary 24.9, Definition 24.15}). By (\ref{degtaumic}), we may rewrite $\eta^{\mathrm{mic}}_{\psi_{G}}$ as
\begin{equation}
  \label{etapsitrace}
  \eta^{\mathrm{mic}}_{\psi_{G}} =  \sum_{\xi \in \Xi(\O_{G}, {^\vee}G^{\Gamma})}
(-1)^{d(S_{\xi}) - d(S_{\psi_{G}})}
  \ \mathrm{Tr}\left(\tau^{\mathrm{mic}}_{S_{\psi_{G}}}(P(\xi))(1)\right) \, \pi(\xi).
\end{equation}
Next, we extend ${^\vee}G_{\psi_{G}}/ ({^\vee}G_{\psi_{G}})^{0}$ trivially to
\begin{equation}
\label{artgroupext}
{^\vee}G_{\psi_{G}}/ ({^\vee}G_{\psi_{G}})^{0} \times \langle \upsigma \rangle,
\end{equation}
and extend $\tau^{\mathrm{mic}}_{S_{\psi_{G}}}(P(\xi))$ trivially to (\ref{artgroupext}) by defining $\tau^{\mathrm{mic}}_{S_{\psi_{G}}}(P(\xi))(\upsigma)$ to be the identity map.  We define
\begin{align*}
\eta^{\mathrm{mic}}_{\psi_{G}} (\upsigma) &=  \sum_{\xi \in \Xi(\O_{G}, {^\vee}G^{\Gamma})}
(-1)^{d(S_{\xi}) - d(S_{\psi_{G}})} \ \mathrm{Tr}\left(\tau^{\mathrm{mic}}_{S_{\psi_{G}}}(P(\xi))(\upsigma)\right) \, \pi(\xi) \\
&=  \sum_{\xi \in \Xi(\O_{G}, {^\vee}G^{\Gamma})}
(-1)^{d(S_{\xi}) - d(S_{\psi_{G}})} \ \dim\left(\tau^{\mathrm{mic}}_{S_{\psi_{G}}}(P(\xi)) \right) \, \pi(\xi).
\end{align*}
Clearly
\begin{equation}
\label{nosigma}
\eta^{\mathrm{mic}}_{\psi_{G}} (\upsigma) = \eta^{\mathrm{mic}}_{\psi_{G}}.
\end{equation}
Finally,  define
\begin{align}
\label{etapsi1}
\eta^{\mathrm{mic}}_{\psi_{G}} (\upsigma)(\delta_{q}) & =\sum_{(S_{\xi}, \tau_{S_{\xi}}) }
(-1)^{d(S_{\xi}) - d(S_{\psi_{G}})}
\ \mathrm{Tr}\left(\tau^{\mathrm{mic}}_{S_{\psi_{G}}}(P(\xi))(\upsigma)\right) \, \pi(\xi)\\
\nonumber &= \sum_{(S_{\xi}, \tau_{S_{\xi}}) }
(-1)^{d(S_{\xi}) - d(S_{\psi_{G}})}
\ \dim\left(\tau^{\mathrm{mic}}_{S_{\psi_{G}}}(P(\xi)) \right) \, \pi(\xi)
\end{align}
in which the sum runs over only those $\xi = (S_{\xi}, \tau_{S_{\xi}}) \in
\Xi(\O_{G}, {^\vee}G^{\Gamma})$ in which $\tau_{S_{\xi}}$
corresponds to the strong involution $\delta_{q}$.
Therefore
\begin{equation*}
\eta^{\mathrm{mic}}_{\psi_{G}} (\upsigma)(\delta_{q})= \eta_{S_{\psi_{G}}}^{\mathrm{mic}}(\delta_{q})=\eta_{S_{\psi_{G}}}^{\mathrm{ABV}}
\end{equation*}
(\ref{etapsiabv}). The virtual
character $\eta^{\mathrm{mic}}_{\psi_{G}}(\upsigma) (\delta_{q})$ is a summand of the
stable virtual character $\eta^{\mathrm{mic}}_{\psi_{G}}$ and is therefore also stable
(\cite{ABV}*{Theorem 18.7}).  Consequently, $\eta^{\mathrm{mic}}_{\psi_{G}}(\upsigma)
(\delta_{q})$ lies in the domain of $\mathrm{Lift}_{0}$.  In addition,
the ABV-packet $\Pi^{\mathrm{ABV}}_{\psi_{G}}$ consists of the irreducible
characters in the support of $\eta^{\mathrm{mic}}_{\psi_{G}}(\upsigma) (\delta_{q})$ (\ref{abvdef}).

What we have done for $\eta^{\mathrm{mic}}_{\psi_{G}}$ we begin to do for
$\eta_{\psi}^{\mathrm{mic}+}$, which we define as
\begin{equation}
\label{etatilde}
  \eta_{\psi}^{\mathrm{mic}+} = \sum_{\xi \in \Xi(\O,
  {^\vee}\mathrm{GL}_{N}^{\Gamma})^\vartheta}
(-1)^{d(S_{\xi}) - d(S_{\psi})}
\ \mathrm{Tr}(\chi^{\mathrm{mic}}_{S_{\psi}}(P(\xi))) \,
(-1)^{l^{I}(\xi) - l^{I}_{\vartheta}(\xi)} \pi(\xi)^{+}
\end{equation}
for
\begin{equation}
\label{glnpsi}
\psi  = \epsilon \circ \psi_{G}.
\end{equation}
The main difference now is that $\upsigma$ does not act trivially on
${^\vee}\mathrm{GL}_{N}$ and so the extensions require more attention.
Properties (\ref{miclocalsys})-(\ref{spectaumic}) hold for
$\psi$ and $\mathrm{GL}_{N}$ as they do for $\psi_{G}$ and $G$.

The first step is writing
$$\eta_{\psi}^{\mathrm{mic}+} =  \sum_{\xi \in \Xi(\O,
  {^\vee}\mathrm{GL}_{N}^{\Gamma})^\vartheta}
(-1)^{d(S_{\xi}) - d(S_{\psi})}
\ \mathrm{Tr}(\tau^{\mathrm{mic}}_{S_{\psi}}(P(\xi))(1)) \,
(-1)^{l^{I}(\xi) - l^{I}_{\vartheta}(\xi)}  \pi(\xi)^{+}.
\nomenclature{$\eta_{\psi}^{\mathrm{mic}+}$}{virtual twisted character}$$
This holds from (\ref{degtaumic})  as  (\ref{etapsitrace}) did for the endoscopic group $G$.  What is new and simpler here is that the component group $({^\vee}\mathrm{GL}_{N})_{\psi}/ (({^\vee}\mathrm{GL}_{N})_{\psi})^{0}$ is trivial (\cite{Arthur84}*{Section 2.3}). It follows that $\tau^{\mathrm{mic}}_{S_{\psi}}(P(\xi)) $ is either the trivial representation or zero.

Let us digress briefly to examine property (\ref{spectaumic}) for $\xi = (S_{\psi}, \tau_{S_{\psi}})$.  Since the component group $({^\vee}\mathrm{GL}_{N})_{p}/ (({^\vee}\mathrm{GL}_{N})_{p})^{0}$ is trivial, the quasicharacter $\tau_{S_{\psi}} = 1$ is trivial.  It follows that
\begin{equation}
\label{containlpacket}
\tau^{\mathrm{mic}}_{S_{\psi}}(P(S_{\psi}, \tau_{S_{\psi}})) = \tau_{S_{\psi}}\circ i_{S_{\psi}} = 1 \circ i_{S_{\psi} }= 1 \neq 0.
\end{equation}
In particular, $\pi(S_{\psi}, 1)$ is in the support of $\eta^{\mathrm{mic}}_{\psi}$ and belongs to $\Pi^{\mathrm{ABV}}_{\psi}$.  In the next section we will prove that this is the only representation in $\Pi^{\mathrm{ABV}}_{\psi}$.

Returning to the matter of extensions, there is an obvious extension
$$({^\vee}\mathrm{GL}_{N})_{\psi}/ (({^\vee}\mathrm{GL}_{N})_{\psi})^{0} \times \langle \upsigma \rangle$$
of the trivial component group, as $\upsigma$ fixes the image of $\psi$.  We wish to extend the representation $\tau_{S_{\psi}}^{\mathrm{mic}} (P(\xi))$ to this group for  $\xi \in  \Xi(\O, {^\vee}\mathrm{GL}_{N}^{\Gamma})^\vartheta$.
The action of $\upsigma$ on $P(\xi) \in \mathcal{P}(\O,
{^\vee}\mathrm{GL}_{N}^{\Gamma}; \upsigma)$ determines an action on
the stalks of the local system $Q^{\mathrm{mic}}(P(\xi))$ as in
(\ref{miclocalsys}) (\cite{ABV}*{(25.1)}).
\cite{ABV}*{Proposition 26.23 (b)} allows us to choose a stalk over a $\upsigma$-fixed point
$p'$ (related to $S_{\psi}$) in the topological space of
$Q^{\mathrm{mic}}(P(\xi))$.  This places us in the same setting as Lemma
\ref{cansheaf}, with $\tau_{S}$ replaced by
$\tau_{S_{\psi}}^{\mathrm{mic}}(P(\xi))$ and $S$ replaced by the
${^\vee}\mathrm{GL}_{N}$-orbit of $p'$.   As a result, $\upsigma$
determines a canonical isomorphism of the stalk at $p'$ equal to $1$.
In short, we define
\begin{equation}
\label{oneextend1}
\tau_{S_{\psi}}^{\mathrm{mic}}(P(\xi)^{+})(\upsigma) =1
\end{equation}
 and extend $\tau_{S_{\psi}}^{\mathrm{mic}}(P(\xi))$ to a quasicharacter
 $\tau_{S_{\psi}}^{\mathrm{mic}}(P(\xi)^{+})$.
 \nomenclature{$\tau_{S_{\psi}}^{\mathrm{mic}}(P(\xi)^{+})$}{}
  The quasicharacter
 $\tau_{S_{\psi}}^{\mathrm{mic}}(P(\xi)^{+})$ represents the
 $({^\vee}\mathrm{GL}_{N} \rtimes \langle \upsigma
 \rangle)$-equivariant local system of the restriction of
 $Q^{\mathrm{mic}}(P(\xi))$ to the orbit of $p'$.
We may extend  $i_{S_{\psi}}$ in (\ref{spectaumic}) to include
the products with $\langle \upsigma \rangle$.  Definitions
(\ref{oneextend}) and (\ref{oneextend1})  are compatible in that
$$\tau_{S_{\psi}}^{\mathrm{mic}}(P(\xi))^{+} = 1^{+} \circ i_{S_{\psi}}.$$

Finally, we define
\begin{equation}
\label{etaplus}
\eta_{\psi}^{\mathrm{mic}+}(\upsigma) = \sum_{\xi \in \Xi(\O,
  {^\vee}\mathrm{GL}_{N}^{\Gamma})^\vartheta}
(-1)^{d(S_{\xi}) - d(S_{\psi})}
\ \mathrm{Tr}(\tau^{\mathrm{mic}}_{S_{\psi}}(P(\xi)^{+})(\upsigma)) \,
(-1)^{l^{I}(\xi) - l^{I}_{\vartheta}(\xi)}  \pi(\xi)^{+}.
\nomenclature{$\eta_{\psi}^{\mathrm{mic}+}(\upsigma)$}{}
\end{equation}
It is clear from definition (\ref{etatilde}) that
$\eta_{\psi}^{\mathrm{mic}+}(\upsigma) = \eta_{\psi}^{\mathrm{mic}+}$.

The obvious definition of the quasicharacter
$\tau_{S_{\psi}}^{\mathrm{mic}}(P(\xi)^{-})$ is to take
$\tau_{S_{\psi}}^{\mathrm{mic}}(P(\xi)^{-})(\upsigma) = -1$.  With this
definition in place the following proposition is a consequence of
\cite{ABV}*{Corollary 24.9}.
\begin{prop}
\label{24.9c}
The functor $\tau_{S_{\psi}}^{\mathrm{mic}}(\cdot)$, from
$({^\vee}\mathrm{GL}_{N} \rtimes \langle \upsigma \rangle
)$-equivariant perverse sheaves to representations of
${^\vee}G_{\psi}/ ({^\vee}G_{\psi})^{0} \times \langle \upsigma
\rangle$, induces a map from the Grothendieck group
$K(X(\O, {^\vee}\mathrm{GL}_{N}^{\Gamma}); \upsigma)$ to the
space of virtual representations.    Furthermore the \emph{microlocal
  trace} map
$$\mathrm{Tr} \, \left(\tau_{S_{\psi}}^{\mathrm{mic}} (\cdot)(\upsigma) \right)$$
induces a homomorphism from $K(X(\O,
{^\vee}\mathrm{GL}_{N}^{\Gamma}), \upsigma)$ (as in
(\ref{twistsheafgroth})) to $\mathbb{C}$.
\end{prop}
A similar statement is true for $\tau_{S_{\psi_{G}}}^{\mathrm{mic}}$ and the
$({^\vee}G \times \langle \upsigma \rangle)$-equivariant sheaves
defined earlier.
\begin{thm}\label{thm:etaplus1}
\begin{enumerate}[label={(\alph*)}]
\item As a function on $K(X(\O,
  {^\vee}\mathrm{GL}_{N}^{\Gamma}), \upsigma)$ we have
$$\left\langle \eta_{\psi}^{\mathrm{mic}+}(\upsigma), \cdot \right\rangle = (-1)^{d(S_{\psi})} \,
    \mathrm{Tr} \, \left(\tau_{S_{\psi}}^{\mathrm{mic}} (\cdot)(\upsigma) \right).$$

\item The stable virtual character $\eta_{\psi}^{\mathrm{mic}+}(\upsigma)$ is equal to
  $$ (-1)^{d(S_{\psi})} \sum_{\xi \in \Xi(\O,
  {^\vee}\mathrm{GL}_{N}^{\Gamma})^\vartheta} \mathrm{Tr} \,
  \left(\tau_{S_{\psi}}^{\mathrm{mic}} (\mu(\xi)^{+})(\upsigma)
  \right)\,  (-1)^{l^{I}(\xi)-l^{I}_{\vartheta}(\xi)} \, M(\xi)^{+}.$$

\item $\mathrm{Lift}_{0}
  \left(\eta^{\mathrm{mic}}_{\psi_{G}}(\upsigma)(\delta_{q}) \right) =
  \eta_{\psi}^{\mathrm{mic}+}(\upsigma).$

\end{enumerate}
\end{thm}
\begin{proof}
The first two assertions  follow from Theorem  \ref{twistpairing} and
the computation
$$\left\langle \eta_{\psi}^{\mathrm{mic}+}(\upsigma)(\delta_{q}), P(\xi)^{+}
\right\rangle = (-1)^{d(S_{\psi})} \,
    \mathrm{Tr} \, \left(\tau_{S_{\psi}}^{\mathrm{mic}} (P(\xi)^{+})(\upsigma) \right)$$
 (\emph{cf.} \cite{ABV}*{Lemma 26.9}).

 For the final assertion, we compute
 \begin{align*}
 \left\langle \epsilon_{*} \eta^{\mathrm{mic}}_{\psi_{G}}(\upsigma)(\delta_{q}), \mu(\xi)^{+} \right\rangle & =
 \left\langle \eta^{\mathrm{mic}}_{\psi_{G}}(\upsigma)(\delta_{q}), \epsilon^{*} \mu(\xi)^{+} \right\rangle_{G} \\
 & =\left\langle \eta^{\mathrm{mic}}_{\psi_{G}}(\upsigma), \epsilon^{*} \mu(\xi)^{+} \right\rangle_{G} \\
 & = (-1)^{d(S_{\psi_{G}})} \,
    \mathrm{Tr} \, \left(\tau_{S_{\psi_{G}}}^{\mathrm{mic}} (\epsilon^{*}
    \mu(\xi)^{+})(\upsigma) \right)
 \end{align*}
 using (\ref{epsrest1}) and \cite{ABV}*{Lemma 26.9} for $\eta^{\mathrm{mic}}_{\psi_{G}}(\upsigma)$.
 By the deep result  \cite{ABV}*{Theorem 25.8}, and the first assertion
 of the theorem,  we may continue with
 \begin{align*}
     & = (-1)^{d(S_{\psi})}  \mathrm{Tr} \,
   \left(\tau_{S_{\psi}}^{\mathrm{mic}} (\mu(\xi)^{+})(\upsigma)
   \right)\\
 & = \left\langle \eta_{\psi}^{\mathrm{mic}+}(\upsigma), \mu(\xi)^{+} \right\rangle
  \end{align*}
 and the theorem is proven.

 \end{proof}

\section{ABV-packets for general linear groups}
\label{glnpacket}

In this section we prove that
any ABV-packet for $\mathrm{GL}_{N}(\mathbb{R})$ consists of a single
(equivalence class of an)
irreducible representation. This implies that such an
ABV-packet is equal to its corresponding L-packet (\cite{ABV}*{Theorem 22.7 (a)}). From the classification of the unitary dual of
$\mathrm{GL}_{N}(\mathbb{R})$ we may deduce that the single
representation in the packet is unitary.

In this section we let
$$\psi:W_{\mathbb{R}} \times \mathrm{SL}_{2} \rightarrow
\mathrm{GL}_{N}^{\Gamma}$$
be an arbitrary A-parameter for
$\mathrm{GL}_{N}(\mathbb{R})$. The description of the ABV-packet
$\Pi_{\psi}^{\mathrm{ABV}}$ will be achieved in three steps. First, we
treat the case of an \emph{irreducible} A-parameter.
Second, we compute the ABV-packet for a Levi subgroup of
$\mathrm{GL}_{N}$, whose dual group contains  the image of
$\psi$ minimally. The final result is obtained from the second
step by considering
the Levi subgroup as an endoscopic group of $\mathrm{GL}_{N}$ and
applying the endoscopic lifting (\ref{ordmiclift}).

Following the description of \cite{Arthur}*{Equation (1.4.1)}, all
A-parameters $\psi$ for $\mathrm{GL}_{N}(\mathbb{R})$ may be
represented as formal direct sums of irreducible representations of
$W_{\mathbb{R}}\times \mathrm{SL}_{2}$
\begin{align}\label{eq:Arthurparameterdecomposition}
  \psi &= \ell_{1} (\mu_{1}\boxtimes \nu_{n_{1}})  \boxplus \cdots
         \boxplus \ell_{r} (\mu_{r} \boxtimes \nu_{n_{r}}).
\nomenclature{$\mu_{r} \boxtimes \nu_{n_{r}}$}{}
\end{align}
Here, $\nu_{n_{j}}$ is the unique irreducible
representation of $\mathrm{SL}_{2}$ of dimension $n_{j}$,
and $\mu_{j}$ is an irreducible representation of $W_{\mathbb{R}}$
with bounded image.
The representations $\mu_{j}$ are of dimension one or two \cite{Knapp94}*{Section 3}.  The parameter $\psi$ in
(\ref{eq:Arthurparameterdecomposition}) is said to be
\emph{irreducible} if  $r = 1$ and  $\ell_{1} = 1$.
\begin{prop}\label{prop:irreducibleArthurparameter}
  Suppose $\psi$ is an irreducible A-parameter of
  $\mathrm{GL}_{N}$.
  Then $\Pi_{\psi}^{\mathrm{ABV}}$ consists of a single unitary representation.
\end{prop}

\begin{proof}
We begin with the case of  $\psi = \mu \boxtimes \nu_{N}$, in which $\mu$ is a one-dimensional representation of
$W_{\mathbb{R}}$. Since $\nu_{N}$ is irreducible, the image of $\mathrm{SL}_{2}$ under $\psi$ contains a principally unipotent (\emph{i.e.} regular and unipotent)
element of $\mathrm{GL}_{N}$. \cite{arancibia_characteristic}*{Theorem 4.11 (d)} (a generalization of \cite{ABV}*{Theorem 27.18}) therefore implies that $\Pi_{\psi}^{\mathrm{ABV}}$
consists of a single unitary character.

Let us now suppose that $\psi=\mu \boxtimes \nu_{n}$ is an
A-parameter in which $\mu$ is a two-dimensional irreducible
representation of $W_{\mathbb{R}}$ (\emph{i.e.} $N=2n$).
The restriction of $\psi$ to $\mathbb{C}^{\times}$ may be
taken to have the form
\begin{equation}\label{eq:AJparameterequation0}
\psi(z)=z^{\lambda_{1}}\bar{z}^{\lambda_{2}}, \quad z \in \mathbb{C}^{\times}
\end{equation}
where
$\lambda_{1}$, $\lambda_{2} \in {^\vee}\mathfrak{h}$  are semisimple
elements with $\exp(2\uppi i(\lambda_{1}-\lambda_{2}))=1$ (\emph{cf.}
\cite{ABV}*{Proposition 5.6}).
Let $\mathcal{L}$ be the centralizer of
$\psi(\mathbb{C}^{\times})$ in ${^\vee}\mathrm{GL}_{N}$.
Following \cite{Taibi}*{4.2.2}, it is straightforward to verify the
following technical conditions
\begin{enumerate}
\item[AJ1.] The identity component of the centre of the centralizer of
  $\psi(j)$ in $\mathcal{L}$
 is contained in the centre of $\mathrm{GL}_{N}$.
\item[AJ2.] $\psi(\mathrm{SL}_{2})$ contains a principally unipotent element of $\mathcal{L}$.
\item[AJ3.]
  $\langle \lambda_{1}+{}^{\vee}\rho_{\mathcal{L}},\alpha \rangle \neq 0$ for all roots
  $\alpha\in R\left(^{\vee}\mathrm{GL}_{N}, {^\vee}H\right)$.
\end{enumerate}
These three conditions place $\psi$ among the family of A-parameters
studied by Adams and Johnson in \cite{Adams-Johnson} \cite{Arthur89}*{Section 5}.  According to \cite{arancibia_characteristic}*{Corollary 4.18},
$$\Pi_{\psi}^{\mathrm{ABV}}= \Pi_{\psi}^{\mathrm{AJ}},$$
where $\Pi_{\psi}^{AJ}$ denotes the packet
 of cohomologically induced representations introduced in \cite{Adams-Johnson}*{Definition 2.11}. The set $\Pi_{\psi}^{AJ}$ is in
 bijection with a set of parabolic subgroups (\cite{AMR}*{Section 8.2}), which
 in this case reduces to a single parabolic subgroup (with Levi
 subgroup isomorphic to $\mathrm{Res}_{\mathbb{C}/\mathbb{R}}
 \mathrm{GL}_{N}$).
\end{proof}
Let us go back to the case of a general A-parameter
$\psi$ as in Equation (\ref{eq:Arthurparameterdecomposition}). Let
$$\psi=\boxplus _{i=1}^{r}\ell_{i} \psi_i, \quad
\psi_{i}=\mu_{i} \boxtimes \nu_{n_{i}}, $$
be its decomposition into irreducible A-parameters
$\psi_{i}$. Let
$N_{i}$ be the dimension of $\psi_{i}$ and define
\begin{equation}\label{eq:Hdecomposition}
  {}^{\vee}G =\prod_{i=1}^{r}\left({}^{\vee}G_{i}\right)^{\ell_{i}}
 \cong \prod_{i=1}^{r} (\mathrm{GL}_{N_{i}})^{\ell_{i}}\\
  \end{equation}
to be the obvious Levi subgroup of $\ch \mathrm{GL}_{N}$ containing
the image of $\psi$.
Let ${^\vee}G^{\Gamma}={}^{\vee}G\times \langle {^\vee} \delta_{0} \rangle$, a subgroup of ${^\vee}\mathrm{GL}_{N}^{\Gamma}$.
It is immediate that $\psi$ factors through an A-parameter
\begin{align}
\label{eq:ArthurfactorsH}
 \psi:W_{\mathbb{R}} \times \mathrm{SL}_{2}
 \xrightarrow{\psi_{G}} {}^{\vee}G^{\Gamma}\hookrightarrow
             {}^{\vee}\mathrm{GL}_{N}^{\Gamma},
\end{align}
where $\psi_{G} = \times_{i=1}^{r} \ell_{i} \psi_{G,i}$ and each  $\psi_{G,i}$ is an
irreducible A-parameter of ${}^{\vee}
G_{i}^{\Gamma}={}^{\vee}G_{i}\times \langle {^\vee}\delta_{0} \rangle$.
The description of the ABV-packet corresponding to $\psi_{G}$ is a fairly straightforward
consequence of
Proposition \ref{prop:irreducibleArthurparameter}.  We must only remind ourselves that the direct product of (\ref{eq:Hdecomposition}) translates into a tensor product of ABV-packets as it passes through the process defining the packets in Section \ref{sec:ABV-packetsdef}.
\begin{cor}
\label{cor:LeviABVpacket}
The ABV-packet $\Pi_{\psi_{G}}^{\mathrm{ABV}}$ consists of a single unitary representation
$\pi(S_{\psi_{G}}, 1)$.
\end{cor}
\begin{proof}
Let $\O_{G} \subset {^\vee}\mathfrak{g}$ be the
${^\vee}G$-orbit of the infinitesimal character determined by the
L-parameter $\phi_{\psi_{G}}$.  This orbit has an obvious decomposition
$\O=\prod_{i=1}^{r}(\O_{i})^{\ell_{i}}$ into
${^\vee}G_{i}$-orbits, and
the variety
$X(\O_{G}, {^\vee}G^{\Gamma})$ of geometric parameters decomposes as
$$X\left(\O_{G},
{^\vee}G^{\Gamma}\right)=\prod_{i=1}^{r}\left(X\left(\O_{i},
{}^\vee G_{i}^{\Gamma}\right)\right)^{\ell_{i}}.$$
As a consequence, all complete geometric parameters
$\xi=(S,\tau_{S}) = (S, 1)$ in
$\Xi\left(\O_{G},{}^{\vee}G^{\Gamma}\right)$
decompose as
$$\xi=\left(\prod_{i=1}^{r}(S_{i})^{\ell_{i}},\bigotimes_{i=1}^{r}
\tau_i^{\otimes \ell_{i}} \right),$$
for ${}^{\vee}G_{i}$-orbits $S_{i} \subset X\left(\O_{i},
{}^\vee G_{i}^{\Gamma}\right)$, and trivial quasicharacters $\tau_{i}
= 1$.
Furthermore, the algebra of differential operators
$\mathcal{D}_{X\left(\O_{G},{}^{\vee}G^{\Gamma}\right)}$,
its graded sheaf, and the conormal bundle
$T_{{}^{\vee} H}^{\ast}(X\left(\O_{G},{}^{\vee}G^{\Gamma}\right))$
all decompose as direct products.
Irreducible
$\mathcal{D}_{{X\left(\O_{G},{}^{\vee}G^{\Gamma}\right)}}$-modules
$D(\xi)$ (\cite{ABV}*{(7.10)(e)}, (\ref{dr})) are therefore tensor products
$$D(\xi)=\bigotimes_{i=1}^{r} D(\xi_{i})^{\otimes \ell_{i}},\quad
\xi_{i}=\left(S_{i},\tau_{i}\right),$$
of irreducible $\mathcal{D}_{{X\left(\O_i,{}^{\vee}G_{i}^{\Gamma}\right)}}$-modules,
and the same can be said about their corresponding irreducible graded modules.
Consequently, the \emph{singular support}
$\mathrm{SS}(D(\xi))$ of the
graded sheaf $\mathrm{gr} D(\xi)$ (\cite{ABV}*{Definition 19.7})
decomposes as
$$\mathrm{SS}(D(\xi))=\prod_{i=1}^{r}(\mathrm{SS}(D(\xi_{i})))^{\ell_{i}},$$
and in particular
$$T_{S_{\psi_{G}}}^{\ast}\left(X\left(\O_{G},
{^\vee}G^{\Gamma}\right)\right) \subset \mathrm{SS}(D(\xi))$$
if and only if
$$T_{S_{\psi_{G,i}}}^{\ast}\left(X\left(\O_{i},
{}^{\vee}G_i^{\Gamma}\right)\right)\subset\mathrm{SS}(D(\xi_i))$$
for all $1 \leq i \leq r$.  This is equivalent to saying
\begin{equation*}
\chi_{S_{\psi_{G}}}^{\mathrm{mic}}(P(\xi))\neq 0
\quad\Longleftrightarrow\quad \chi_{S_{\psi_{G,i}}}^{\mathrm{mic}}(P(\xi_{i}))\neq 0, \quad \forall\  1\leq i\leq r.
\end{equation*}
In other words, $\pi(\xi)\in
\Pi_{\psi_{G}}^{\mathrm{ABV}}$
if and only if
$\pi(\xi_i)\in \Pi_{\psi_{G,i}}^{\mathrm{ABV}}$ for all $1 \leq i \leq r$
(\ref{abvdef}).
By Proposition
\ref{prop:irreducibleArthurparameter},
each  ABV-packet  $\Pi_{\psi_{G,i}}^{\mathrm{ABV}}$ consists of a single unitary
representation.  It is a consequence of (\ref{containlpacket}) that
each such unitary representation is of the form $\pi(S_{\psi_{G,i}},
1)$.  A look back to the definition of ABV-packets reveals that
$$\Pi_{\psi_{G}}^{\mathrm{ABV} } = \{ \otimes_{i = 1}^{r} \pi(S_{\psi_{G,i}},1) \} = \{ \pi(S_{\psi_{G}},1) \}.$$
\end{proof}
We are ready for the final step of describing the ABV-packets for
$\mathrm{GL}_{N}(\mathbb{R})$.
\begin{prop}
\label{prop:singletonGLN}
  Let $\psi$ be an A-parameter for $\mathrm{GL}_{N}$ as in (\ref{eq:Arthurparameterdecomposition}).
  Then the ABV-packet $\Pi_{\psi}^{\mathrm{ABV}}$ consists of a single unitary representation  $\pi(S_{\psi}, 1)$.
\end{prop}
\begin{proof}
   Define ${}^{\vee}G$ as in (\ref{eq:Hdecomposition}).
    Take $s \in  Z({^\vee}G) \subset {^\vee}\mathrm{GL}_{N}$
    to be as regular as possible so that its centralizer in
    ${^\vee}\mathrm{GL}_{N}$ is equal to
    ${^\vee}G$.  Set ${^\vee}G^{\Gamma} = {^\vee}G \times \Gamma$, so that $(s,
    {^\vee}G^{\Gamma})$ is an endoscopic datum (Section \ref{standend}).
 Write
    $\psi_{G}:W_{\mathbb{R}} \times \mathrm{SL}_{2} \rightarrow {^\vee}G^{\Gamma}$
    for the A-parameter for $G$, satisfying
    $\psi = \epsilon \circ \psi_{G}$  ((\ref{epinclusion}),
    (\ref{eq:ArthurfactorsH})).
According to (\ref{ordmiclift}), Corollary \ref{cor:LeviABVpacket},
and Proposition \ref{injliftord} (c)
$$\eta^{\mathrm{mic}}_{\psi} = \mathrm{Lift}_{0}\left( \eta^{\mathrm{mic}}_{\psi_{G}}\right) =
\mathrm{Lift}_{0} \left(\pi(S_{\psi_{G}}, 1)\right) =
\mathrm{ind}_{G(\mathbb{R}, \delta_{q})}^{\mathrm{GL}_{N}(\mathbb{R})}
\pi(S_{\psi_{G}}, 1).$$
The proposition now follows from the fact that
parabolic induction for general linear groups takes irreducible
unitary representations to irreducible unitary representations
(\cite{Tadic}*{Proposition 2.1, Sections 4-5}).
\end{proof}

\begin{cor}
\label{etaplus2}
The stable virtual character $\eta_{\psi}^{\mathrm{mic}+}(\upsigma)$ defined in (\ref{etaplus}) is equal to
 $(-1)^{l^{I}(\xi)-l^{I}_{\vartheta}(\xi)} \pi(\xi)^{+}$,
where $\xi = (S_{\psi}, 1)$.  In particular,
$$\mathrm{Lift}_{0}\left(\eta^{\mathrm{mic}}_{\psi_{G}}(\upsigma)(\delta_{q})\right) = (-1)^{l^{I}(\xi)-l^{I}_{\vartheta}(\xi)} \pi(\xi)^{+}$$
\end{cor}
\begin{proof}
By Proposition \ref{prop:singletonGLN}, $\tau_{S_{\psi}}^{\mathrm{mic}}(P(\xi))$ is non-zero only for $\xi = (S_{\psi}, 1)$.  By definition (\ref{oneextend1}), $\tau_{S_{\psi}}^{\mathrm{mic}}(P(\xi)^{+})(\upsigma) = 1$  when $\xi = (S_{\psi}, 1)$ and is zero otherwise. The first assertion follows. The second assertion is a consequence of the first and Theorem \ref{thm:etaplus1}.
\end{proof}

\section{Whittaker extensions and their relationship to Atlas
  extensions}
\label{whitsec}

Thus far we have been working with preferred extensions of
irreducible representations, from $\mathrm{GL}_{N}(\mathbb{R})$ to
$\mathrm{GL}_{N}(\mathbb{R}) \rtimes \langle \vartheta \rangle$.
These are the Atlas extensions of (\ref{canext}).  Arthur uses a
different choice of canonical extension in \cite{Arthur}, which we
call the  \emph{Whittaker extension}.   After reviewing the definition of
Whittaker extensions, we will compute the sign giving the difference from the Atlas
extensions.  We conclude by rewriting  the pairing
of Theorem \ref{twistpairing} using Whittaker extensions. Written in
this manner, the pairing becomes simpler (Corollary \ref{twistpairingfinal}).

The review of Whittaker extensions which we are about to give may be
found in \cite{Arthur}*{Section 2.2}.
We fix a unitary character $\chi  \nomenclature{$\chi$}{unitary character of $U(\R)$}$ on the
upper-triangular unipotent subgroup $U(\mathbb{R}) \subset
B(\mathbb{R}) \nomenclature{$U$}{upper-triangular unipotent subgroup}$ which satisfies $\chi \circ \vartheta = \chi$.   In this manner   $(U, \chi)\nomenclature{$(U, \chi)$}{Whittaker datum}$ is a
$\vartheta$-fixed Whittaker datum.  We work under the hypothesis of
(\ref{regintdom}) on an infinitesimal character $\lambda \in
{^\vee}\mathfrak{h}$ and set $\O$ to be its
${^\vee}\mathrm{GL}_{N}$-orbit.
Let $\xi \in \Xi(\O,
{^\vee}\mathrm{GL}_{N}^{\Gamma})^\vartheta$ so that $\pi(\xi)$
is (an infinitesimal equivalence class of) an irreducible
representation of $\mathrm{GL}_{N}(\mathbb{R})$.   Here, and whenever
we define Whittaker extensions, we must work with a \emph{bona fide}
admissible group representation in this equivalence class which we
also denote $(\pi(\xi),V)$. If $\pi(\xi)$  is
tempered then up to a scalar there is a unique Whittaker functional $\omega: V
\rightarrow \mathbb{C}$ satisfying
\begin{equation}
\label{whittfunctional}
\omega(\pi(\xi)(u)v) = \chi(u)\, \omega(v), \quad u \in U(\mathbb{R}),
\end{equation}
for all smooth vectors $v \in V$.
It follows that there is a unique operator $\mathcal{I}^{\thicksim}$
which intertwines $\pi(\xi) \circ \vartheta$ with $\pi(\xi)$ and also satisfies
$\omega \circ \mathcal{I}^{\thicksim} = \omega$.  We extend
$\pi(\xi)$ to a representation $\pi(\xi)^{\thicksim}$ of
$\mathrm{GL}_{N}(\mathbb{R}) \rtimes \langle \vartheta \rangle$
by setting $\pi(\xi)^{\thicksim}(\vartheta) =
\mathcal{I}^{\thicksim}$.  We call this extension
$\pi(\xi)^{\thicksim}\nomenclature{$\pi(\xi)^{\thicksim}$}{(equivalence class of) the  Whittaker extension of $\pi(\xi)$}$
the \emph{Whittaker extension} of $\pi(\xi)$.

If $\pi(\xi)$ is not tempered then we express it as the Langlands
quotient of a representation induced from an essentially tempered
representation of a Levi subgroup. The $\vartheta$-stability of
$\pi(\xi)$  and the uniqueness statement in the Langlands
classification together imply the
$\vartheta$-stability of the essentially tempered representation.  The
earlier argument for tempered representations has an obvious analogue
for the essentially tempered representation of the Levi subgroup.  We
may argue as above to extend the essentially tempered representation
to the semi-direct product of the Levi subgroup with $\langle
\vartheta \rangle$.  One then induces this extended representation to
$\mathrm{GL}_{N}(\mathbb{R}) \rtimes \langle \vartheta  \rangle$.  The
unique irreducible quotient of this representation is the canonical
extension of $\pi(\xi)$, namely the \emph{Whittaker extension}
$\pi(\xi)^{\thicksim}$ of $\pi(\xi)$.   If one omits the Langlands
quotient in this argument then we obtain the Whittaker extension
$M(\xi)^{\thicksim}
\nomenclature{$M(\xi)^{\thicksim}$}{(equivalence class of) the Whittaker extension of
 $M(\xi)$}$ of the standard representation $M(\xi)$.

We now turn to the question of how $\pi(\xi)^{\thicksim}$ differs from
$\pi(\xi)^{+}$.  The operators $\pi(\xi)^{\thicksim}(\vartheta)$ and
$\pi(\xi)^{+}(\vartheta)$ are involutive, and  both intertwine
$\pi(\xi) \circ \vartheta$ with $\pi(\xi)$. Therefore they are equal up to a
sign, \emph{i.e.}
\begin{equation}
\label{normsigns}
\pi(\xi)^{\thicksim}(\vartheta) = \pm \, \pi(\xi)^{+}(\vartheta).
\end{equation}
\begin{lem}
Suppose $\xi \in \Xi(\O,
{^\vee}\mathrm{GL}_{N}^{\Gamma})^\vartheta$ and $M(\xi)$ is a
$\vartheta$-stable principal series representation
of $\mathrm{GL}_{N}(\mathbb{R})$.
Then the Whittaker and Atlas
extensions of $M(\xi)$ are equal, \emph{i.e.}  $M(\xi)^{\thicksim} = M(\xi)^{+}$.
\label{prinsame}
\end{lem}
\begin{proof}
We abusively identify the equivalence class $M(\xi)$ with one of its representatives
 $M(\xi) = \mathrm{ind}_{B(\mathbb{R})}^{\mathrm{GL}_{N}(\mathbb{R})}
  \pi_{0}$.  Define an operator $\mathcal{I}$ on functions $f$ in the
  space of $M(\xi)$ by
$$\mathcal{I}f(g) = f(\vartheta(g)), \quad g \in \mathrm{GL}_{N}(\mathbb{R}).$$
 It is easily verified that  $\mathcal{I}$ intertwines
  $(\mathrm{ind}_{B(\mathbb{R})}^{\mathrm{GL}_{N}(\mathbb{R})}  \pi_{0})
  \circ \vartheta$
  with $ \mathrm{ind}_{B(\mathbb{R})}^{\mathrm{GL}_{N}(\mathbb{R})}
  (\pi_{0} \circ \vartheta)$.  By the $\vartheta$-stability of $M(\xi)$
  we have
$$\mathrm{ind}_{B(\mathbb{R})}^{\mathrm{GL}_{N}(\mathbb{R})}
  \pi_{0}\cong  (\mathrm{ind}_{B(\mathbb{R})}^{\mathrm{GL}_{N}(\mathbb{R})}
  \pi_{0}) \circ \vartheta \cong  \mathrm{ind}_{B(\mathbb{R})}^{\mathrm{GL}_{N}(\mathbb{R})}
  (\pi_{0} \circ \vartheta).$$
The uniqueness statement in the Langlands quotient theorem may be
applied to the equivalence above to conclude that $\pi_{0} = \pi_{0}
\circ{\vartheta}$.  Therefore $\mathcal{I}$ intertwines $M(\xi) \circ
\vartheta$ with $M(\xi)$.  We will prove the lemma by showing that
$M(\xi)^{\thicksim}(\vartheta)
= \mathcal{I} = M(\xi)^{+}(\vartheta)$.

Given a(ny) Whittaker functional $\omega$ satisfying (\ref{whittfunctional}),
$M(\xi)^{\thicksim}(\vartheta)$ is the unique intertwining operator satisfying $\omega \circ M(\xi)^{\thicksim}(\vartheta)  = \omega$.  A convenient Whittaker
functional to work with is
\begin{equation}
  \label{whittfun}
\omega(f) = \int_{U(\mathbb{R})} f(\dot{w}_{0} u) \overline{\chi}(u) \,
du
\end{equation}
\cite{Sha80}*{Section 2}.
Here $f$ is a smooth function in the space of $M(\xi)$ and
$\dot{w}_{0} = \tilde{J}$ is a
representative of the long Weyl group element in
$W(\mathrm{GL}_{N},H)$.   We compute
\begin{align*}
\omega\circ \mathcal{I}(f) &= \int_{U(\mathbb{R})} \mathcal{I}f(\dot{w}_{0} u)
\overline{\chi}(u) \, du\\
&= \int_{U(\mathbb{R})} f(\vartheta(\dot{w}_{0} u))
\overline{\chi}(u) \, du\\
&= \int_{U(\mathbb{R})} f(\dot{w}_{0} u)
\overline{\chi}(\vartheta(u)) \, du\\
&= \int_{U(\mathbb{R})} f(\dot{w}_{0} u)
\overline{\chi}(u) \, du\\
&= \omega(f).
\end{align*}
This proves $\omega \circ \mathcal{I} = \omega$ so that
$M(\xi)^{\thicksim}(\vartheta) = \mathcal{I}$.

To prove $M(\xi)^{+}(\vartheta) = \mathcal{I}$ we  recall the
definition of $M(\xi)^{+}$ as in (\ref{canext}).  The complete
geometric parameter $\xi$
corresponds to an Atlas parameter $(x,y)$ (Lemma \ref{XXXi}), where $x$
is the equivalence class of a strong involution
$$\delta = \exp(\uppi i \ch \rho)
\dot{w}_{0} \delta_{0} = \exp(\uppi i \ch \rho) \tilde{J} \delta_{0}$$
as in
\cite{AVParameters}*{Proposition 3.2} (see the proof of Lemma
\ref{biglem}).  This strong involution negates every
positive root in $R(B,H)$.  It follows that the underlying
$(\mathfrak{gl}_{N},K)$-module of $M(\xi)^{+}$ is the
representation \cite{AVParameters}*{(20)}, in which the
Borel subalgebra $\mathfrak{b}$ is real.  This implies that $M(\xi)^{+} =
\mathrm{ind}_{B(\mathbb{R}) \rtimes \langle \vartheta \rangle
}^{\mathrm{GL}(N,\mathbb{R}) \rtimes \langle \vartheta \rangle} \pi_{0}^{+}$,
where $\pi_{0}^{+}(\vartheta) = 1$, since the value of the function $z$ given in (\ref{z}) is one.
Suppose $f$ is a function in the space of $\mathrm{ind}_{B(\mathbb{R})
  \rtimes \langle \vartheta \rangle
}^{\mathrm{GL}(N,\mathbb{R}) \rtimes \langle \vartheta \rangle}
\pi_{0}^{+}$.  We compute
\begin{align*}
  \left(M(\xi)^{+}(\vartheta)f \right) (g) &= f(g \vartheta)\\
  & = f(\vartheta \, \vartheta(g))\\
  &= \pi_{0}^{+}(\vartheta) \, f(\vartheta(g))\\
  & = f(\vartheta(g))\\
  & = \left(\mathcal{I}f \right) (g).
\end{align*}
This proves that $M(\xi)^{+}(\vartheta) = \mathcal{I}$.
\end{proof}
Our next goal is to make a link between the signs in (\ref{normsigns}) and the
twisted multiplicity polynomials $m_{r}^{\vartheta}(\xi', \xi)$
appearing in  (\ref{twistmult1}). By Proposition \ref{p:twist} and
(\ref{cgP}), we
have the alternative formulations
\begin{equation}\label{QP}
\begin{aligned}
m_{r}^{\vartheta}(\xi',\xi) &= (-1)^{l^{I}_{\vartheta}(\xi') - l^{I}_{\vartheta}(\xi)}
c_{g}^{\vartheta}(\xi, \xi') \\
&= (-1)^{l^{I}(\xi') -  l^{I}_{\vartheta}(\xi') + l^{I}(\xi)
  -  l^{I}_{\vartheta}(\xi)} \ {^\vee} P^{\vartheta}({^\vee}\xi, {^\vee}\xi')(1) .
\end{aligned}
\end{equation}
If $\pi(\xi')$ is a subquotient of
$M(\xi)$ occurring with multiplicity one then it is easily verified that
$m_{r}^{\vartheta}(\xi',\xi) = 1$ if and only if $\pi(\xi')^{+}$ is a
subquotient of $M(\xi)^{+}$.  Similarly, $m_{r}^{\vartheta}(\xi',\xi) =
-1$ if and only if $\pi(\xi')^{-}$ is a subquotient of $M(\xi)^{+}$.  In this sense $m_{r}^{\vartheta}(\xi',\xi)$ is a signed multiplicity.

There is a special irreducible subquotient of $M(\xi)$ which is \emph{generic},
\emph{i.e.} admits a non-zero Whittaker functional as in
(\ref{whittfunctional}).
\begin{lem}
\label{uniquegeneric}
Suppose $\xi \in \Xi(\O, {^\vee}
\mathrm{GL}_{N}^{\Gamma})^\vartheta$.  Then
\begin{enumerate}[label={(\alph*)}]
\item (up to infinitesimal equivalence)
there is a unique irreducible $\vartheta$-stable generic
representation $\pi(\xi_{0}) = M(\xi_{0})$ which occurs in $M(\xi)$ as
a subquotient with multiplicity one;

\item (any representative in the class of) $\pi(\xi_{0})$ embeds as a
  subrepresentation of (any representative in the class of)  $M(\xi)$;

\item (any representative in the class of) $\pi(\xi_{0})^{\thicksim}$
  embeds as a subrepresentation of (any representative in the class
  of)  $M(\xi)^{\thicksim}$.
\end{enumerate}
\end{lem}
\begin{proof}
A result due to Vogan and Kostant states that every standard
representation $M(\xi)$ contains a unique generic irreducible
subquotient occurring with multiplicity one (\cite{Kostant78}*{Theorems E and L}, \cite{Vogan78}*{Corollary 6.7}).  In the rest of the
proof we write $\pi(\xi_{0})$ for the actual generic representation
(not the equivalence class) for some $\xi_{0 } \in \Xi(\O,
{^\vee}\mathrm{GL}_{N}^{\Gamma})$.  It is straightforward to
verify that $\pi(\xi_{0}) \circ \vartheta$ satisfies
(\ref{whittfunctional}), just as $\pi(\xi_{0})$ does.  Therefore
$\pi(\xi_{0}) \circ \vartheta$ is the unique irreducible generic
subquotient of $M(\xi) \circ \vartheta \cong M(\xi)$.  By
uniqueness, $\pi(\xi_{0}) \circ \vartheta \cong \pi(\xi_{0})$ and so
$\xi_{0} \in \Xi(\O, {^\vee}
\mathrm{GL}_{N}^{\Gamma})^\vartheta$.

The statements about $\pi(\xi_{0})$ occurring as a subrepresentation
of $M(\xi)$ and $\pi(\xi_{0}) = M(\xi_{0})$ may be found in \cite{Vogan78}*{Theorem 6.2} and \cite{casshah}*{Theorem 6.2}.

For part (c) we consider the standard representation
$M(\xi)$, which has
a Whittaker functional $\omega$  (\cite{Sha80}*{Proposition 3.2}).  The
functional $\omega$
restricts to a non-zero Whittaker functional on $\pi(\xi_{0})$.
By definition, $M(\xi)^{\thicksim}(\vartheta)$ is the intertwining
operator which satisfies $\omega \circ M(\xi)^{\thicksim}(\vartheta) = \omega$.
Restricting this equation to the subrepresentation $\pi(\xi_{0})$
yields in turn that
\begin{equation}
  \label{subrep}
\pi(\xi_{0})^{\thicksim}(\vartheta) =
M(\xi)^{\thicksim}(\vartheta)_{|\pi(\xi_{0})} \mbox{   and    }\pi(\xi_{0})^{\thicksim}
\hookrightarrow M(\xi)^{\thicksim}.
\end{equation}
\end{proof}
Lemma \ref{uniquegeneric}  tells us that the
multiplicity of $\pi(\xi_{0})^{\thicksim}$ in $M(\xi)^{\thicksim}$ is
one.  On the other hand $m_{r}^{\vartheta}(\xi_{0}, \xi)$ tells us
about the ``signed multiplicity" of $\pi(\xi_{0})^{+}$ in
$M(\xi)^{+}$.  We investigate $m_{r}^{\vartheta}(\xi_{0}, \xi)$ further
before comparing the two multiplicities.
\begin{prop}
  \label{conjq}
Suppose $\xi \in \Xi(\O, {^\vee}
\mathrm{GL}_{N}^{\Gamma})^\vartheta$ and $\pi(\xi_{0})$ is the
generic subrepresentation of $M(\xi)$ (Lemma \ref{uniquegeneric}).
Then
 \begin{equation}
  \label{qequation}m_{r}^{\vartheta}(\xi_{0}, \xi) = (-1)^{l^{I}(\xi) -
  l^{I}_{\vartheta}(\xi) + l^{I}(\xi_{0}) -
  l^{I}_{\vartheta}(\xi_{0})}.
\end{equation}
\end{prop}
\begin{proof}
We see from (\ref{QP}) that the proposition is equivalent to
$${^\vee}P^{\vartheta}(\,{^\vee}\xi, {^\vee}\xi_{0}) (1) =1.$$
This equation follows from Proposition \ref{dualpone} once we
establish that  ${^\vee}\xi_{0}$ is the unique maximal parameter in the
block of $\pi({^\vee}\xi)$ in the (dual) Bruhat order.  This is equivalent
to establishing that $\xi_{0}$ is the unique minimal parameter in the
block of $\pi(\xi_{0})$ ((\ref{dualBruhat}),
\cite{ICIV}*{Theorem 1.15}). We use  \cite{ABV}*{Proposition 1.11} to convert the Bruhat order
for the representations of $\mathrm{GL}_{N}(\mathbb{R})$ into a
closure relation between ${^\vee}\mathrm{GL}_{N}$-orbits of
$X(\O, {^\vee}\mathrm{GL}_{N}^{\Gamma})$.   Moreover, this
proposition  implies that the minimality  of
$\xi_{0} = (S_{0}, 1  )$
is equivalent to the ${^\vee}\mathrm{GL}_{N}$-orbit  $S_{0} \subset
X(\O, {^\vee}\mathrm{GL}_{N}^{\Gamma})$ being maximal and
therefore open.  The uniqueness of the generic parameter follows from
the fact that there is a unique open orbit in each component of
$X(\O, {^\vee}\mathrm{GL}_{N}^{\Gamma})$ (\emph{cf.}
\cite{ABV}*{\emph{p.} 19}).
\end{proof}
The signed multiplicities $m_{r}^{\vartheta}(\xi',\xi)$ of Atlas
extensions may be compared to the multiplicities of Whittaker
extensions when there is a representation in the block for which the two extensions agree.
This is the case when the block contains a $\vartheta$-fixed principal
series (Lemma \ref{prinsame}).   Under these circumstances, we obtain
a formula for the signs in (\ref{normsigns}).
\begin{prop}
\label{wasign1}
Suppose $\xi \in \Xi(\O, {^\vee}
\mathrm{GL}_{N}^{\Gamma})^\vartheta$ and $\pi(\xi_{0})$ is the
generic representations of Lemma \ref{uniquegeneric}.
If $\pi(\xi_{0})^{+}$ occurs in the decomposition of  a principal
series representation $M(\xi_{p})^{+} \in K\Pi(\O,
\mathrm{GL}_{N}(\mathbb{R}),
\vartheta)$  then
  $$M(\xi)^{\thicksim}(\vartheta)  = (-1)^{l^{I}(\xi) -
  l^{I}_{\vartheta}(\xi)} \ M(\xi)^{+} (\vartheta)$$
and
$$\pi(\xi)^{\thicksim}(\vartheta)  = (-1)^{l^{I}(\xi) -
  l^{I}_{\vartheta}(\xi)} \ \pi(\xi)^{+} (\vartheta).$$
\end{prop}
\begin{proof}
 Suppose that $\xi_{p} \in
\Xi(\O,{^\vee}\mathrm{GL}_{N}^{\Gamma})^\vartheta$  is the complete geometric
parameter of a $\vartheta$-stable principal series representation
$M(\xi_{p})^{+}$ as in the hypothesis.  It is straightforward to show
$l^{I}(\xi_{p}) = l^{I}_{\vartheta}(\xi_{p}) = 0$.  By Lemma
\ref{prinsame}, (\ref{twistmult1}) and (\ref{qequation})
\begin{align*}
  \label{prindecomp}
  M(\xi_{p})^{\thicksim}&= m_{r}^{\vartheta}(\xi_{0},\xi_{p}) \,
\pi(\xi_{0})^{+} \  +  \sum_{\xi' \neq \xi_{0}} m_{r}^{\vartheta}(\xi',\xi_{p})\,
\pi(\xi')^{+}\\
\nonumber &= (-1)^{l^{I}(\xi_{0}) - l^{I}_{\vartheta}(\xi_{0})}
\, \pi(\xi_{0})^{+} \   + \sum_{\xi' \neq \xi_{0}} m_{r}^{\vartheta}(\xi',\xi_{p})\,
\pi(\xi')^{+}.
\end{align*}
According to (\ref{subrep}), with $\xi = \xi_{p}$, and the observations
immediately preceding  Lemma \ref{uniquegeneric}, this equation implies
\begin{equation}
  \label{conjgen}
(-1)^{l^{I}(\xi_{0}) - l^{I}_{\vartheta}(\xi_{0})}
  \pi(\xi_{0})^{+}(\vartheta) = \pi(\xi_{0})^{\thicksim}(\vartheta).
\end{equation}
Thus, the proposition holds for $\xi = \xi_{0}$.

It remains to prove that the proposition  holds when $\xi \neq \xi_{0}$.  We
compute, using (\ref{qequation}) and (\ref{conjgen}), that
\begin{align*}
M(\xi)^{+} &= \sum_{\xi' \neq \xi_{0}} m_{r}^{\vartheta}(\xi',\xi)
\, \pi(\xi')^{+}+ m_{r}^{\vartheta}(\xi_{0},\xi) \, \pi(\xi_{0})^{+}\\
&=\sum_{\xi' \neq \xi_{0}} m_{r}^{\vartheta}(\xi',\xi) \,
\pi(\xi')^{+} + (-1)^{l^{I}(\xi) -
  l^{I}_{\vartheta}(\xi) + l^{I}(\xi_{0}) -
  l^{I}_{\vartheta}(\xi_{0})}   \, \pi(\xi_{0})^{+}\\
& = \sum_{\xi' \neq \xi_{0}} m_{r}^{\vartheta}(\xi',\xi)
\, \pi(\xi')^{+}  + (-1)^{l^{I}(\xi) -
  l^{I}_{\vartheta}(\xi)}   \pi(\xi_{0})^{\thicksim}.
\end{align*}
This equation and the observations before Lemma \ref{uniquegeneric} imply
$$(-1)^{l^{I}(\xi) -
  l^{I}_{\vartheta}(\xi)} \pi(\xi_{0})^{\thicksim}(\vartheta) =
M(\xi)^{+} (\vartheta)_{|\pi(\xi_{0})}.$$
Combining this equation with (\ref{subrep}), we see in turn that
$$(-1)^{l^{I}(\xi) -
  l^{I}_{\vartheta}(\xi)}  M(\xi)^{+} (\vartheta)_{|\pi(\xi_{0})} =
\pi(\xi_{0})^{\thicksim}(\vartheta) = M(\xi)^{\thicksim}
(\vartheta)_{|\pi(\xi_{0})}$$
and $(-1)^{l^{I}(\xi) -
  l^{I}_{\vartheta}(\xi)}  M(\xi)^{+} (\vartheta) = M(\xi)^{\thicksim}
(\vartheta)$.  By taking Langlands quotients of the last equation we
obtain
$(-1)^{l^{I}(\xi) -
  l^{I}_{\vartheta}(\xi)}  \pi(\xi)^{+} (\vartheta) = \pi(\xi)^{\thicksim}
(\vartheta)$.
\end{proof}
Proposition \ref{wasign1} describes the sign appearing in
(\ref{normsigns}) explicitly.  Unfortunately, the
hypotheses of the Proposition  do not always hold.  It is not
true that every generic representation $\pi(\xi_{0})^{+}$ appears in
the decomposition of a $\vartheta$-stable principal series
representation.  This may already be seen for $\mathrm{GL}_{2}$.  It
is instructive to examine and remedy  this
special case.

For every positive integer $m$  let
$\mathrm{ind}^{\mathrm{GL}_{2}( \mathbb{R})}_{B(\mathbb{R})} \pi_{m}$
be the principal series representation with
\begin{equation}
\label{gl2ps}
\pi_{m} = \left\{ \begin{array}{ll} | \cdot
  |^{(m-1)/2} \otimes |\cdot|^{(-m-1)/2}, &m \mbox{ even}\\
   | \cdot
|^{(m-1)/2} \otimes sgn(\cdot)|\cdot|^{-(m-1)/2}, & m \mbox{ odd}
\end{array} \right.
\end{equation}
Let $D_{m}$ be the relative (limit of) discrete series representation
\nomenclature{$D_{m}$}{(limit) of discrete series representation}
which is the unique
subrepresentation of $\mathrm{ind}^{\mathrm{GL}_{2}(
  \mathbb{R})}_{B(\mathbb{R})} \pi_{m}$.

For even $m$ both $D_{m}$ and $\mathrm{ind}^{\mathrm{GL}_{2}(
  \mathbb{R})}_{B(\mathbb{R})} \pi_{m}$ are $\vartheta$-stable.
This may be seen by computing that the linear map on the space of $\mathrm{ind}^{\mathrm{GL}_{2}(  \mathbb{R})}_{B(\mathbb{R})}
\pi_{m}$ defined by
\begin{equation}
  \label{Ttheta}
  f(x) \mapsto f(\vartheta(x)), \quad x \in \mathrm{O}(2),
\end{equation}
intertwines $(\mathrm{ind}^{\mathrm{GL}_{2}(
  \mathbb{R})}_{B(\mathbb{R})} \pi_{m}) \circ \vartheta$ with
$\mathrm{ind}^{\mathrm{GL}_{2}(  \mathbb{R})}_{B(\mathbb{R})}
\pi_{m}$.  Thus, for even $m$ we have an obvious embedding of
$D_{m}$ into a $\vartheta$-stable principal series representation
exhibited by an explicit intertwining operator.

For odd $m$ the map (\ref{Ttheta}) intertwines $(\mathrm{ind}^{\mathrm{GL}_{2}(
  \mathbb{R})}_{B(\mathbb{R})} \pi_{m}) \circ \vartheta$ with
$\mathrm{ind}^{\mathrm{GL}_{2}(   \mathbb{R})}_{B(\mathbb{R})}
(\pi_{m}  \circ \vartheta)$, where
$$\pi_{m} \circ \vartheta=
sgn (\cdot) | \cdot |^{(m-1)/2} \otimes |\cdot|^{-(m-1)/2}.$$
Unfortunately, for odd $m >1$  the representation
$\mathrm{ind}^{\mathrm{GL}_{2}(   \mathbb{R})}_{B(\mathbb{R})}
(\pi_{m}  \circ \vartheta)$ is not equivalent to
$\mathrm{ind}^{\mathrm{GL}_{2}(   \mathbb{R})}_{B(\mathbb{R})}
\pi_{m}$.  This may be deduced from the uniqueness statement in the
Langlands Classification.    Consequently, instead of embedding
$D_{m}$ into a $\vartheta$-stable principal series representation, we
must seek another means of finding an operator which intertwines
$D_{m} \circ \vartheta$ with
$D_{m}$.   For this we look to facts about the discrete series of
$\mathrm{SL}_{2}(\mathbb{R})$.
\begin{lem}
\label{gl2wa}
  Suppose $m$ is a  positive integer.  Then $D_{m}$ is
  $\vartheta$-stable and $D_{m}^{\thicksim} \cong D_{m}^{+}$.
  \end{lem}
\begin{proof}
It is well-known that the restriction of $D_{m}$ to
$\mathrm{SL}_{2}(\mathbb{R})$ decomposes as a direct sum $D_{m+}
\oplus D_{m-}$ of irreducible (limits of) discrete series
representations (\cite{Knapp}*{\emph{pp.} 471-472}).  In
this decomposition $D_{m+}$ is the unique representation which is generic
with respect to $(U \cap \mathrm{SL}_{2}(\mathbb{R}), \chi)$.  Let
$\omega_{+}$ be its Whittaker functional.  The representation $D_{m}$
is equivalent
to the tensor product of the  trivial central character
$1_{Z(\mathrm{GL}_{2}(\mathbb{R}))}$ with
$\mathrm{ind}_{\mathrm{SL}_{2}(\mathbb{R})}^{\mathrm{SL}_{2}^{\pm}(\mathbb{R})}
D_{m+}$.   The direct sum $D_{m+} \oplus D_{m-}$ also occurs as the
unique subrepresentation of  the principal series representation
$\mathrm{ind}_{B
  \cap\mathrm{SL}_{2}(\mathbb{R})}^{\mathrm{SL}_{2}(\mathbb{R})}
\left( (\pi_{m})_{| H \cap \mathrm{SL}_{2}(\mathbb{R})} \right)$  (\cite{Knapp}*{(2.19)}).
It is easily verified that this principal series representation is
$\vartheta$-stable. Lemma  \ref{prinsame} and Proposition
\ref{wasign1} apply equally well for representations of
$\mathrm{SL}_{2}(\mathbb{R})$
and so we conclude that
$D_{m+}$ is $\vartheta$-stable with $D_{m+}^{\thicksim} \cong
D_{m+}^{+}$.   Let $\omega$ be the linear functional on the space of
$D_{m}$ obtained by
composing $\omega_{+}$ with the orthogonal projection $P_{+}$ onto the
space of $D_{m+}$.  We compute that for all $u \in
U(\mathbb{R})$  and $f = (1-P_{+})f + P_{+}f$ in the space of $D_{m}$
\begin{align*}
\omega( D_{m}(u) f) &= \omega_{+} (P_{+} (D_{m-}(u)(1-P_{+}) f + D_{m+}(u) P_{+}f))\\
 &= \omega_{+} (D_{m+}(u) P_{+} f) \\
 &= \chi(u) \, \omega_{+} (P_{+} f) \\
 &= \chi(u)\, \omega(f).
 \end{align*}
This proves that $\omega$ is a Whittaker functional for $D_{m}$ and implies
$$D_{m}^{\thicksim} \cong 1_{Z(\mathrm{GL}_{2}(\mathbb{R}))}
\otimes
\mathrm{ind}_{\mathrm{SL}_{2}(\mathbb{R})}^{\mathrm{SL}_{2}^{\pm}(\mathbb{R})
}D_{m+}^{\thicksim}.$$
The lemma now follows from
\begin{align*}
D_{m}^{\thicksim} &\cong 1_{Z(\mathrm{GL}_{2}(\mathbb{R}))} \otimes \mathrm{ind}_{\mathrm{SL}_{2}(\mathbb{R})}^{\mathrm{SL}_{2}^{\pm}(\mathbb{R})} D_{m+}^{\thicksim}\\
& \cong  1_{Z(\mathrm{GL}_{2}(\mathbb{R}))} \otimes
\mathrm{ind}_{\mathrm{SL}_{2}(\mathbb{R})}^{\mathrm{SL}_{2}^{\pm}(\mathbb{R})}
D_{m+}^{+}\\
& \cong D_{m}^{+}
\end{align*}
in which we appeal to \cite{AVParameters}*{(20)} for the construction of
$D_{m}^{+}$ and use \cite{Knapp-Vogan}*{Proposition 2.77} for the
induction of finite index.
\end{proof}

In the following two lemmas we see that although a $\vartheta$-stable
irreducible generic representation need not be a subrepresentation of
a $\vartheta$-stable principal series representation, it is at worst a
subrepresentation of a $\vartheta$-stable standard representation
which is essentially induced from  $D_{m}$.  We deal with the tempered
representations first.

\begin{lem}
  \label{biglem}
Suppose $\pi_{\mathrm{gen}}$ is a $\vartheta$-stable irreducible
tempered representation of $\mathrm{GL}_{N}(\mathbb{R})$ with
integral infinitesimal character $\lambda \in {^\vee}\mathfrak{h}$
chosen as in (\ref{regintdom})
\emph{i.e.}
$$\vartheta(\lambda) = \lambda\quad \mbox{ and }\quad \langle \lambda,
     {^\vee}\alpha \rangle \in \{ 1,2, \cdots\},
\ \alpha \in R^{+}(\mathrm{GL}_{N}, H).$$
\begin{enumerate}[label={(\alph*)}]

\item Relative to the standard basis of the diagonal Lie algebra
  $\mathfrak{h}$, $\lambda$ has
  coordinates of the form $$(\lambda_{1}, \ldots, \lambda_{ N/2} ,
  -\lambda_{N/2}, \ldots, -\lambda_{1})$$
  when $N$ is even, and of the form
  $$(\lambda_{1}, \ldots, \lambda_{ (N-1)/2} ,0, -\lambda_{(N-1)/2}, \ldots,
  -\lambda_{1})$$
  when $N$ is odd.  The coordinates appear in strictly
  decreasing order.

\item Suppose $N$ is odd.  Then the coordinates  $\lambda_{j}$ are
  all integers,
  and $\pi_{\mathrm{gen}}$ embeds into a $\vartheta$-stable principal
  series representation.

\item Suppose $N$ is even.  Then the coordinates $\lambda_{j}$
  are either all integers or all half-integers (elements in
$\mathbb{Z} + \frac{1}{2}$).  In the latter case, $\pi_{\mathrm{gen}}$
embeds into a $\vartheta$-stable principal series representation.  In
the former case, $\pi_{\mathrm{gen}}$ embeds into a
$\vartheta$-stable principal series representation when $N$ is
divisible by $4$.  When $N$ is not divisible by $4$ in the former case,
$\pi_{\mathrm{gen}}$ embeds into the
representation $\pi$ which is parabolically induced from the
representation
\begin{equation}
  \label{almprin}
|\cdot|^{\lambda_{1}} \otimes \cdots \otimes |
\cdot |^{\lambda_{\frac{N}{2}-1}} \otimes D_{2\lambda_{N/2}+1} \otimes
|\cdot|^{-\lambda_{\frac{N}{2}-1}} \otimes \cdots \otimes | \cdot
|^{-\lambda_{1}}.
\end{equation}
The representation $\pi$ is  $\vartheta$-stable and $\pi^{\thicksim} = \pi^{+}$.
\end{enumerate}
\end{lem}
\begin{proof}
The first assertion is an immediate consequence of
$$\lambda =\vartheta(\lambda) = -\mathrm{Ad}(\tilde{J})(\lambda).$$
Let
$\alpha_{1}, \ldots , \alpha_{N-1}$ be the simple roots of
$\mathrm{GL}_{N}$ relative to the Borel subgroup $B$.

Suppose first that $N$ is  odd.  Then $$\langle \lambda, {^\vee}\alpha_{(N-1)/2}
\rangle = \lambda_{(N-1)/2} - 0 $$ is a positive integer by the
integrality hypothesis.  The remaining $\lambda_{j}$ for  $j <(N-1)/2$
are then seen to be positive integers by applying the integrality hypothesis to
$\langle \lambda, {^\vee}\alpha_{j} \rangle$ successively.

Our strategy in showing that $\pi_{\mathrm{gen}}$ embeds into a
$\vartheta$-stable principal series representation is to express
$\pi_{\mathrm{gen}}$ in terms of its Atlas parameter (Corollary
\ref{twistclass}) and to apply Cayley transforms and cross actions to its Atlas
parameter  \cite{AVParameters}*{Section 7}.  The resulting
parameters correspond to $\vartheta$-stable standard representations in the
block of $\pi_{\mathrm{gen}}$ (\cite{ICIV}*{Definition 1.14},
\cite{ICIV}*{Theorem 8.8}, \cite{AVParameters}*{Section 7}).  For
$\mathrm{GL}_{N}(\mathbb{R})$, $\pi_{\mathrm{gen}}$ is the unique
irreducible generic representation in its block (Lemma \ref{uniquegeneric}).
Furthermore, every standard
representation in the block contains the irreducible generic representation
 $\pi_{\mathrm{gen}}$,  as its
unique irreducible subrepresentation (\cite{Vogan78}*{Theorem 6.2},
\cite{Vogan78}*{Corollary 6.7}, \cite{casshah}*{Theorem 6.2}).  In
consequence, it suffices to find Cayley transforms and cross actions
which carry the Atlas parameter of $\pi_{\mathrm{gen}}$ to a
$\vartheta$-stable principal series representation.

By Corollary \ref{twistclass}, the representation $\pi_{\mathrm{gen}}$ corresponds to an element $w \in W(\mathrm{GL}_{N}, H)$ satisfying $w \delta_{0}(w) = (ww_{0})^{2} = 1$ ((\ref{twistinv}), (\ref{weylsquared})).  Such Weyl group elements are called $\delta_{0}$\emph{-twisted involutions} \cite{AVParameters}*{Section 3}.
We begin by determining the $\delta_{0}$-twisted involution $w \in
W(\mathrm{GL}_{N}, H)$ attached
to $\pi_{\mathrm{gen}}$. The
$\delta_{0}$-twisted involution $w \in
W(\mathrm{GL}_{N}, H)$ determines an involutive automorphism $w
\delta_{0}$ on $H$  (\cite{AVParameters}*{(14e)}).  This automorphism
also acts on
${^\vee}\mathfrak{h}$, and the integral length (\ref{intlength}) of
$\pi_{\mathrm{gen}}$ is equal to
\begin{equation}
\label{genlen}
  -\frac{1}{2} \left( | \{\alpha \in R^{+}(\mathrm{GL}_{N}, H) :
w \delta_{0} \cdot \alpha \in R^{+}(\mathrm{GL}_{n},H) \}| + \dim(H^{w\delta_{0}})
\right).
\end{equation}
By  \cite{ICIV}*{Lemma 12.10} and the arguments of the proof of Proposition \ref{conjq},
the length of
$\pi_{\mathrm{gen}}$ is minimal among the lengths of all
representations in its block.   It is not difficult to see that
(\ref{genlen}) is minimized at $w = 1$.  Moreover, there exists a
representation in the block of $\pi_{\mathrm{gen}}$ corresponding to
$w=1$ (\cite{ICIV}*{Theorem 8.8} and \cite{Adams-Fokko}*{Section 14}).
It follows
that $\pi_{\mathrm{gen}}$ is in fact this representation and that the
$\delta_{0}$-twisted involution $w$ for $\pi_{\mathrm{gen}}$ is
trivial.

According to \cite{Carter}*{Lemma 5}, there are orthogonal
positive roots, $\beta_{1}, \ldots \beta_{m}$ such that
$$s_{\beta_{1}} \cdots \, s_{\beta_{m}} = ww_{0} = w_{0}.$$
From this we compute $\beta_{1} = \alpha_{1} + \cdots
 + \alpha_{N-1}$, $\beta_{2} = \alpha_{2} + \cdots +
  \alpha_{N-2}$, \ldots, $\beta_{(N-1)/2} = \alpha_{(N-1)/2} +
  \alpha_{(N+1)/2}$ and $m = (N-1)/2$. By Corollary \ref{twistclass}
  there is an element $y \in
  {^\vee}\mathcal{X}_{\lambda}^{s_{\beta_{1}} \cdots
    s_{\beta_{m}}} =
  {^\vee}\mathcal{X}_{\lambda}^{w_{0}}  $ and an element $x \in
  \mathcal{X}_{\rho^{\vee}}^{1}$ such that $J(x,y, \lambda)
  = \pi_{\mathrm{gen}}$.

Before listing Cayley transforms and cross actions to apply to the
parameter $(x,y)$, it is worthwhile to describe the $\delta_{0}$-twisted
involution $w' \in W(\mathrm{GL}_{N}, H)$ attached to a principal
series representation.  A principal series representation is
parabolically induced from a real Borel subgroup.  In the Atlas
parameterization this is equivalent to the automorphism
$w'\delta_{0}$ carrying  all positive roots of $R(\mathrm{GL}_{N}, H)$
to negative roots, making
$B$ a real Borel subgroup (\cite{Knapp-Vogan}*{Proposition 4.76}).
Since $\delta_{0}$ preserves the set of positive roots, this forces
$w' = w_{0}$, the long Weyl group element.

We wish to apply Cayley transforms and cross actions  to the
parameter $(x,y) \in  \mathcal{X}_{\rho^{\vee}}^{1} \times
{^\vee}\mathcal{X}_{\lambda}^{w_{0}}$ in order to arrive at a parameter
$(x',y') \in \mathcal{X}_{\rho^{\vee}}^{w_{0}} \times
{^\vee}\mathcal{X}_{\lambda}^{1}$ and a $\vartheta$-stable representation $\pi(x',y') = J(x',y',\lambda)$ (Corollary \ref{twistclass}).  Recall from Section \ref{heckesection} that Cayley transforms and
cross actions are made relative to $\vartheta$-orbits of simple roots $\kappa$
(\ref{simplekappa}).  A Cayley transform $c_{\kappa}$ applied to a parameter
in $\mathcal{X}_{\rho^{\vee}}^{w} \times
{^\vee}\mathcal{X}_{\lambda}^{ww_{0}}$ produces parameters in
$\mathcal{X}_{\rho^{\vee}}^{w_{\kappa}w} \times
{^\vee}\mathcal{X}_{\lambda}^{w_{\kappa}ww_{0}}$ (\cite{Adams-Fokko}*{Definition 14.1}), where
$w_{\kappa}$ is prescribed in (\ref{simplekappa1}).  A cross action
$\kappa \times $ applied to a parameter in $\mathcal{X}_{\rho^{\vee}}^{w} \times
{^\vee}\mathcal{X}_{\lambda}^{ww_{0}}$ results in a parameter in
$\mathcal{X}_{\rho^{\vee}}^{w_{\kappa}ww_{\kappa}^{-1}} \times
{^\vee}\mathcal{X}_{\lambda}^{w_{\kappa}ww^{-1}_{\kappa}w_{0}}$
(\cite{Adams-Fokko}*{(9.11f)}).  In short, Cayley transforms
left-multiply a $\delta_{0}$-twisted involution $w$ by $w_{\kappa}$, and
cross actions conjugate  $w$ by $w_{\kappa}$.  We wish to move from
the  $\delta_{0}$-twisted involution $1$ to the $\xi_{0}$-twisted
involution $w_{0}= s_{\beta_{1}} \cdots s_{\beta_{(N-1)/2}}$ using
these operations.

We first  describe how to move from $1$ to $s_{\beta_{1}}$.  We define the
symbol $\leadsto$ to mean ``takes a parameter in the set on the left
to  parameters in the set on the right''.
Let $\kappa_{j}$ be the $\vartheta$-orbit of $\alpha_{j}$.  It is
straightforward to verify (using the \emph{Atlas of Lie Groups and Representations} software)
\begin{equation}\label{caycross}
\begin{aligned}
  \mathcal{X}_{\rho^{\vee}}^{1} \times
{^\vee}\mathcal{X}_{\lambda}^{w_{0}}
\  &\stackrel{c_{\kappa_{(N-1)/2}}}{\leadsto} \mathcal{X}_{\rho^{\vee}}^{s_{\beta_{m}}} \times
{^\vee}\mathcal{X}_{\lambda}^{s_{\beta_{m}}w_{0}}\\
&\stackrel{\kappa_{(N-3)/2} \times }{\leadsto}
\mathcal{X}_{\rho^{\vee}}^{s_{\beta_{m-1}}} \times
{^\vee}\mathcal{X}_{\lambda}^{s_{\beta_{m-1}}w_{0}}\\
&\stackrel{\kappa_{(N-5)/2} \times }{\leadsto}
\mathcal{X}_{\rho^{\vee}}^{s_{\beta_{m-2}}} \times
{^\vee}\mathcal{X}_{\lambda}^{s_{\beta_{m-2}}w_{0}}\\
& \vdots\\
&\stackrel{\kappa_{1} \times }{\leadsto}
\mathcal{X}_{\rho^{\vee}}^{s_{\beta_{1}}} \times
{^\vee}\mathcal{X}_{\lambda}^{s_{\beta_{1}}w_{0}}.
\end{aligned}
\end{equation}
To move from $s_{\beta_{1}}$ to $s_{\beta_{1}}s_{\beta_{2}}$  we repeat this
procedure, but terminating with $\stackrel{\kappa_{2}
  \times}{\leadsto}$.  Repeating this procedure in the obvious
fashion, we arrive to $w_{0}= s_{\beta_{1}} \cdots
s_{\beta_{(N-1)/2}}$, as desired.  This proves part (b).

We now continue under the assumption that $N$ is even.  By the
integrality hypothesis
$$\langle \lambda,
{^\vee}\alpha_{N/2}
\rangle = \lambda_{N/2} - (-\lambda_{N/2}) = 2 \lambda_{N/2} $$ is a
positive integer.  Therefore $\lambda_{N/2} \in
\frac{1}{2}\mathbb{Z}$.  If $\lambda_{N/2} \in \mathbb{Z} +
\frac{1}{2}$  then the integrality of $\langle \lambda,
{^\vee}\alpha_{(N/2)-1}  \rangle = \lambda_{(N/2)-1} - \lambda_{N/2} $
implies $\lambda_{(N/2)-1} \in \mathbb{Z} + \frac{1}{2}$.  Similar
computations with the remaining simple roots $\alpha_{(N/2)-1}, \ldots ,
\alpha_{1}$ then imply that all $\lambda_{j}$ are half-integers.  If
$\lambda_{N/2}$ is an integer then the same argument proves that all
$\lambda_{j}$ are integers.

The approach to embedding $\pi_{\mathrm{gen}}$ into standard
representations for even $N$ is the same as in the odd case.  In
particular, $\pi_{\mathrm{gen}}  = J(x,y,\lambda)$ where $(x,y) \in
\mathcal{X}_{\rho^{\vee}}^{1} \times
{^\vee}\mathcal{X}_{\lambda}^{w_{0}}$.  We wish to apply
Cayley transforms and cross actions to $(x,y)$ in order to arrive to
standard representations of the desired form.  There are three cases
to consider.

In the case that all coordinates of $\lambda$ are half-integers the Cayley
transform $c_{\kappa_{N/2}} = c_{\alpha_{N/2}}$ may be used to replace
$c_{\kappa_{(N-1)/2}}$ in (\ref{caycross}) in order to obtain the same
conclusion.

However, in the case that all coefficients of $\lambda$ are integers
the Cayley transform $c_{\alpha_{N/2}}$ does not yield parameters
which are $\vartheta$-stable and must therefore be ignored
($\alpha_{N/2}$ is a root of type \texttt{1i2s} in
\cite{AVParameters}*{Tables 1-2}).  In the special case of $\mathrm{GL}_{4}(\mathbb{R})$ one
may circumvent this obstacle as follows
\begin{equation}\label{caycross1}
\begin{aligned}
\mathcal{X}_{\rho^{\vee}}^{1} \times
{^\vee}\mathcal{X}_{\lambda}^{w_{0}}
\  &\stackrel{c_{\kappa_{1}}}{\leadsto}
\mathcal{X}_{\rho^{\vee}}^{s_{\alpha_{1}}s_{\alpha_{3}}} \times
{^\vee}\mathcal{X}_{\lambda}^{s_{\alpha_{1}}s_{\alpha_{3}}w_{0}}\\
&\stackrel{\kappa_{2} \times }{\leadsto}
\mathcal{X}_{\rho^{\vee}}^{s_{\alpha_{2}}s_{\alpha_{1}}s_{\alpha_{3}}s_{\alpha_{2}}} \times
{^\vee}\mathcal{X}_{\lambda}^{s_{\alpha_{2}}s_{\alpha_{1}}s_{\alpha_{3}}s_{\alpha_{2}}w_{0}}\\
&\stackrel{c_{\kappa_{1}}}{\leadsto}
\mathcal{X}_{\rho^{\vee}}^{s_{\alpha_{1}}
  s_{\alpha_{3}}s_{\alpha_{2}}s_{\alpha_{1}}s_{\alpha_{3}}s_{\alpha_{2}}}
\times
{^\vee}\mathcal{X}_{\lambda}^{s_{\alpha_{1}}
  s_{\alpha_{3}}s_{\alpha_{2}}s_{\alpha_{1}}s_{\alpha_{3}}s_{\alpha_{2}}
  w_{0}}\\
&=
\mathcal{X}_{\rho^{\vee}}^{s_{\beta_{1}} s_{\beta_{2}}} \times
{^\vee}\mathcal{X}_{\lambda}^{1}.
\end{aligned}
\end{equation}
More generally, if $N$ is divisible by four one may replace the first
step of (\ref{caycross}) with the appropriate analogue of
(\ref{caycross1}) and then continue by performing cross actions as in
(\ref{caycross}) to arrive at a parameter in
$\mathcal{X}_{\rho^{\vee}}^{s_{\beta_{1}} s_{\beta_{2}}} \times
{^\vee}\mathcal{X}_{\lambda}^{s_{\beta_{1}} s_{\beta_{2}}w_{0}}$.
Iterating this process, one arrives at a parameter in
$\mathcal{X}_{\rho^{\vee}}^{w_{0}} \times
{^\vee}\mathcal{X}_{\lambda}^{1}$ which corresponds to a
$\vartheta$-stable principal
series representation.

In the last case where the coordinates of $\lambda$ are integers and $N$
has remainder two when divided by four one may only iterate the process
just described to arrive at a parameter
$(x',y') \in \mathcal{X}_{\rho^{\vee}}^{s_{\beta_{1}} \cdots s_{\beta_{(N/2)-1}}} \times
{^\vee}\mathcal{X}_{\lambda}^{s_{\beta_N/2}}$.  The
involution corresponding to this parameter acts on roots by
\begin{equation}
  \label{cartanaction}
s_{\beta_{1}} \cdots s_{\beta_{(N/2)-1}} \delta_{0} = -s_{\alpha_{N/2}}.
\end{equation}
In the Atlas parameterization this
implies that $\alpha_{N/2}$ is an imaginary root and all simple
roots orthogonal to $\alpha_{N/2}$ are real.  Let us again identify
the $\vartheta$-stable representation
$$\pi = \pi(x',y') = J(x',y', \lambda)$$
with its  underlying $(\mathfrak{gl}_{N}, K)$-module (\cite{AVParameters}*{(20)}).  In the language
of \cite{Knapp-Vogan}*{Section 11} this module is the unique irreducible
quotient of
$$M(x',y') = {^u}\mathcal{R}_{\mathfrak{b}, H \cap K}^{\mathfrak{gl}_{N}, K}
  Z(\mathfrak{b}).$$
Here, $\mathfrak{b} = \mathfrak{h} \oplus \mathfrak{n}$ is the
upper-triangular Borel subalgebra, $K =
K_{\delta}$,
$$\delta= \exp( \uppi i \ch\rho) \sigma_{\alpha_{1} + \cdots +
  \alpha_{N-1}} \cdots \sigma_{\alpha_{\frac{N}{2} -1} +
  \alpha_{\frac{N}{2}+1}} \delta_{0}$$
(\cite{AVParameters}*{Proposition 3.2}), and  $Z(\mathfrak{b}) =
\mathbb{C}(y',\lambda) \otimes
\wedge^{\mathrm{top}}(\mathfrak{n})$ as in \cite{AVParameters}*{(20)}.
By \cite{Knapp-Vogan}*{Corollary 11.86}, we may write
\begin{equation}
  \label{instages}
{^u}\mathcal{R}_{\mathfrak{b}, H \cap K}^{\mathfrak{gl}_{N}, K}
  Z(\mathfrak{b}) =  {^u}\mathcal{R}_{\mathfrak{p}, M \cap
    K}^{\mathfrak{gl}_{N}, K} {^u} \mathcal{R}_{\mathfrak{b} \cap
    \mathfrak{m}, H \cap K}^{\mathfrak{m}, K \cap M}
  Z(\mathfrak{b}) =  {^u}\mathcal{R}_{\mathfrak{p}, M \cap
    K}^{\mathfrak{gl}_{N}, K} \pi_{M}
  \end{equation}
where $\mathfrak{p}  = \mathfrak{m} \oplus \mathfrak{u} \supset
\mathfrak{b}$ is the parabolic subalgebra corresponding to
$\alpha_{N/2}$, and  $\pi_{M} = {^u} \mathcal{R}_{\mathfrak{b} \cap
    \mathfrak{m}, H \cap K}^{\mathfrak{m}, K \cap M}
Z(\mathfrak{b})$.
The induction functor ${^u}\mathcal{R}_{\mathfrak{p}, M \cap
    K}^{\mathfrak{gl}_{N}, K}$ on the right may be identified with
parabolic induction (\cite{Knapp-Vogan}*{Proposition 11.57}) and
$\pi_{M}$ is the underlying module of (\ref{almprin}) (\cite{Knapp-Vogan}*{Theorem 11.178}).

It remains to prove that the Whittaker extension of $\pi$ equals its
Atlas extension.  Using the formula for the Whittaker functional
of a parabolically induced representation (\cite{Sha81}*{Proposition 3.2}), we have
$$({^u}\mathcal{R}_{\mathfrak{b}, H \cap K  }^{\mathfrak{gl}_{N}, K}
  Z(\mathfrak{b}) )^{\thicksim} = ( {^u}\mathcal{R}_{\mathfrak{p}, (M \cap
    K) }^{\mathfrak{gl}_{N}, K } \pi_{M}  )^{\thicksim}  =
  {^u}\mathcal{R}_{\mathfrak{p}, (M \cap  K) \rtimes \langle \vartheta
    \rangle}^{\mathfrak{gl}_{N}, K\rtimes
    \langle \vartheta \rangle }  \pi_{M}^{\thicksim}.$$
  By definition of the Atlas extension (\ref{canext})
  $$ ({^u}\mathcal{R}_{\mathfrak{b}, H \cap K}^{\mathfrak{gl}_{N}, K}
  Z(\mathfrak{b}))^{+} = {^u}\mathcal{R}_{\mathfrak{b}, H \cap
    K}^{\mathfrak{gl}_{N}, K} (  Z(\mathfrak{b})^{+})$$
  and so using  induction by stages as in (\ref{instages}) we have
  \begin{equation}
    \label{nearps}
 ({^u}\mathcal{R}_{\mathfrak{b}, H \cap K}^{\mathfrak{gl}_{N}, K}
  Z(\mathfrak{b}))^{+} =  {^u}\mathcal{R}_{\mathfrak{p}, (M \cap
    K)\rtimes \langle \vartheta \rangle}^{\mathfrak{gl}_{N}, K\rtimes
    \langle \vartheta \rangle} {^u} \mathcal{R}_{\mathfrak{b} \cap
    \mathfrak{m}, (H \cap K)\rtimes \langle \vartheta
    \rangle}^{\mathfrak{m}, (K \cap M)\rtimes \langle \vartheta
    \rangle}
  Z(\mathfrak{b})^{+} = {^u}\mathcal{R}_{\mathfrak{p}, (M \cap
    K)\rtimes \langle \vartheta \rangle}^{\mathfrak{gl}_{N}, K\rtimes
    \langle \vartheta \rangle} \pi_{M}^{+}.
  \end{equation}
Therefore, if $\pi_{M}^{+} = \pi_{M}^{\thicksim}$ then it follows
that
$$ ({^u}\mathcal{R}_{\mathfrak{b}, H \cap K}^{\mathfrak{gl}_{N}, K}
  Z(\mathfrak{b}))^{+} = ( {^u}\mathcal{R}_{\mathfrak{b}, H \cap
    K}^{\mathfrak{gl}_{N}, K}
  Z(\mathfrak{b}))^{\thicksim},$$
  \emph{i.e.} the Atlas extension and Whittaker extensions are equal.

For the  proof of $\pi_{M}^{+} = \pi_{M}^{\thicksim}$ we may assume
without loss of generality that $M = \mathrm{GL}_{2}$ and   $\pi_{M} =
D_{m}$, and appeal to  Lemma \ref{gl2wa}.
\end{proof}

The next lemma is a generalization of the previous one to include
generic representations with
non-integral infinitesimal characters.

\begin{lem}
\label{biglem1}
Suppose $\pi_{\mathrm{gen}}$ is a $\vartheta$-stable irreducible
generic representation of $\mathrm{GL}_{N}(\mathbb{R})$ with
infinitesimal character satisfying (\ref{regintdom}).  Then
$\pi_{\mathrm{gen}}$ embeds into a $\vartheta$-stable
standard representation $M(\xi_{p})$, $\xi_{p} \in \Xi(\O, {^\vee}
\mathrm{GL}_{N}^{\Gamma})^\vartheta$ such that
$M(\xi_{p})^{\thicksim} = M(\xi_{p})^{+}$.
\end{lem}
\begin{proof}
According to \cite{Vogan78}*{Theorem 6.2}, $\pi_{\mathrm{gen}}$ is
infinitesimally equivalent
to a parabolically induced representation
$\mathrm{ind}_{P(\mathbb{R})}^{\mathrm{GL}_{N}(\mathbb{R})} (\pi'
\otimes e^{\nu})$. Here,  $P(\mathbb{R})$ is a cuspidal standard parabolic
subgroup whose Levi subgroup $M(\mathbb{R})$ has Langlands
decomposition $M(\mathbb{R}) = M^{1}A$ (\cite{Knapp}*{Section V.5}), $\pi'$ is a
(limit of) discrete series representation of $M^{1}$, and $\nu$ lies in
the complex Lie algebra $\mathfrak{a}$ of $A$.  Since $P(\mathbb{R})$
is standard and cuspidal the Levi subgroup $M(\mathbb{R})$ decomposes
diagonally into blocks
$$M(\mathbb{R}) = M_{1}(\mathbb{R}) \times \cdots \times
M_{\ell}(\mathbb{R})$$
in which each block $M_{j}(\mathbb{R})$ is isomorphic to either
$\mathrm{GL}_{2}(\mathbb{R})$ or $\mathrm{GL}_{1}(\mathbb{R})$.
Accordingly,
\begin{equation}
  \label{mfactor}
  M^{1} = M^{1}_{1} \times \cdots \times M^{1}_{\ell},
  \end{equation}
where $M_{j}'$ is isomorphic to $\mathrm{SL}_{2}^{\pm}(\mathbb{R})$ or $\{ \pm
1\}$; and
$$\pi' = \pi_{1}' \otimes \cdots \otimes \pi_{\ell}',$$
where $\pi'_{j}$ is equivalent to
$D_{m_{j}}' := (D_{m_{j}})_{|\mathrm{SL}_{2}^{\pm}(\mathbb{R})}$ (\emph{cf.}
(\ref{gl2ps})) when $M_{j}' \cong
\mathrm{SL}_{2}^{\pm}(\mathbb{R})$, and is equivalent to $1$
or $sgn$ when $M^{1}_{j} \cong \{\pm 1\}$.  In addition,
\begin{equation}
  \label{afactor}
\mathfrak{a} = \mathfrak{a}_{1} \oplus \cdots \oplus
\mathfrak{a}_{\ell}
\end{equation}
and $\nu = \nu_{1} + \cdots + \nu_{\ell}$.  One should expect that the
$\vartheta$-stability
of $\pi_{\mathrm{gen}} \cong \mathrm{ind}_{P(\mathbb{R})}^{\mathrm{GL}_{N}(\mathbb{R})} (\pi'
\otimes e^{\nu})$ would put constraints on the constituent
representations $\pi'_{j} \otimes e^{\nu_{j}}$ of
$M_{j}(\mathbb{R})$.  This is indeed so, and we now prove that the
$\vartheta$-stability yields
a partition of the set
\begin{equation}
  \label{factorset}
\{ \pi_{1} \otimes e^{\nu_{1}}, \ldots, \pi_{\ell} \otimes
e^{\nu_{\ell}} \}
\end{equation}
into either pairs of the form $\{
D_{m_{j}}' \otimes e^{\nu_{j}},  D_{m_{j}}' \otimes e^{-\nu_{j}}\}$ or
singletons of the form $\{ D_{m_{j}}' \otimes e^{0}\} = \{ D_{m_{j}}
\}$.

The $\vartheta$-stability implies that the distribution character
of $\pi_{\mathrm{gen}}$ is equal to
the distribution character of
$$\pi_{\mathrm{gen}} \circ \vartheta \cong
\left( \mathrm{ind}_{P(\mathbb{R})}^{\mathrm{GL}_{N}(\mathbb{R})} (\pi'
\otimes e^{\nu}) \right) \circ \vartheta \cong
\mathrm{ind}_{\vartheta(P(\mathbb{R}))}^{\mathrm{GL}_{N}(\mathbb{R})}
\left(  \pi' \circ \vartheta \otimes e^{\vartheta \cdot \nu}
\right).$$
By the Langlands Disjointness Theorem
(\cite{Langlands}*{\emph{pp.} 149-150}, \emph{cf.} \cite{Knapp}*{Theorem 14.90}), there
exists $g \in \mathrm{O}(N)$ such that
\begin{eqnarray}
  \label{ldt}
 \mathrm{Int}(g)
\circ \vartheta( M^{1}) = M^{1}&,&  \mathrm{Int}(g)
\circ \vartheta (A) = A,\\
\nonumber \pi' \circ  \vartheta \circ  \mathrm{Int}(g^{-1})
\cong \pi'&,& \mbox{and}\quad \vartheta \cdot (\mathrm{Ad}(g^{-1})\nu) = \nu.
\end{eqnarray}
Recall that $\vartheta$ is the composition of
$\mathrm{Int}(\tilde{J})$ and inverse-transpose.  The
inverse-transpose automorphism stabilizes $M^{1}$ and $A$.  Since
inverse-transpose acts on $\mathrm{SL}_{2}(\mathbb{R})$ as the
inner automorphism $\mathrm{Int}\left(\tiny \begin{bmatrix} 0 & 1 \\ -1 & 0
  \end{bmatrix} \right)\normalsize$, and acts trivially on $\{\pm
1\}$, it is easy to see that
inverse-transpose stabilizes $\pi'$.  The value of the differential of
inverse-transpose  at $\nu$ is $-\nu$.  Taking these facts into
consideration, we may read (\ref{ldt}) as
\begin{eqnarray*}
  \label{ldt1}
 \mathrm{Int}(g_{1})( M^{1}) = M^{1}&,&  \mathrm{Int}(g_{1}) (A) = A,\\
\nonumber \pi' \circ   \mathrm{Int}(g_{1}^{-1})
\cong \pi' &,& \mbox{and}\quad \mathrm{Ad}(g_{1}^{-1})\nu = -\nu,
\end{eqnarray*}
where $g_{1} = g\tilde{J} \in \mathrm{O}(N)$.  After possibly
multiplying by an element in $M^{1} \cap \mathrm{O}(N)$, we may assume
that $\mathrm{Int}(g_{1})$ fixes a representative $\lambda' \in
{^\vee}\mathfrak{m}$ of the infinitesimal character of $\pi'$.  The
infinitesimal character of $\pi'$ decomposes as
$$\lambda' = \lambda_{1}' + \cdots + \lambda_{\ell}',$$
where $\lambda'_{j} \in {^\vee}\mathfrak{m}_{j}$.  If
$M_{j} = \mathrm{GL}_{2}$ then $\lambda'_{j}$ determines $\pi'_{j} = D_{m_{j}}'$
up to equivalence.  Since
$g_{1}$ normalizes $M^{1}$ and fixes
$\lambda'$, $\mathrm{Int}(g_{1})(M^{1}_{j}) = M^{1}_{k}$ for some $1 \leq
k \leq \ell$, and $\mathrm{Ad}(g_{1})(\lambda_{j}') = \lambda'_{k}$.  In
particular, if $M_{j} \cong \mathrm{GL}_{2}$ then $\pi_{j}' \cong
\pi_{k}' \cong D'_{m_{j}}$.  If $M_{j} \cong \mathrm{GL}_{1}$ then
$\mathrm{Int}(g)_{|M_{j}^{1}}$ is the unique isomorphism of $M^{1}_{j}
\cong \{\pm 1\}$ onto $M_{k}^{1} \cong \{\pm 1\}$ and so  $\pi'_{j} =
\pi_{k}'$.  If $k = j$ in either of these two cases then
$-\nu_{j} = \mathrm{Ad}(g_{1}^{-1}) (\nu_{j}) = \nu_{j}$ and $\nu_{j}
= 0$.  In this manner, the element $g_{1}$ specifies the singletons
for the partition of (\ref{factorset}).  The pairs in
the partition of (\ref{factorset}) become evident once we
establish that $g_{1}$ acts involutively on the factors of
(\ref{mfactor}) and (\ref{afactor}).  For this, we observe that
$\mathrm{Ad}(g_{1}^{2})$
fixes both $\lambda'$ and $\nu$ and so fixes
a representative of the infinitesimal character of
$\pi_{\mathrm{gen}}$ (\cite{Knapp}*{Proposition 8.22}).  As we are
assuming that the infinitesimal character is regular, $g_{1}^{2}$ belongs
to the Cartan subgroup determined by $\lambda'$ and $\nu$, which is a subgroup
of $M(\mathbb{R})$.  Thus, $\mathrm{Int}(g_{1}^{2})(M^{1}_{j}) =
M^{1}_{j}$, $\mathrm{Ad}(g_{1}^{2})\mathfrak{a}_{j} =
\mathfrak{a}_{j}$, which proves the desired involutive action of
$g_{1}$.

The distribution character, and therefore the infinitesimal
equivalence class, of the irreducible representation
$\mathrm{ind}_{P(\mathbb{R})}^{\mathrm{GL}_{N}(\mathbb{R})} (\pi'
\otimes e^{\nu})$ is independent of the choice of parabolic subgroup
and is invariant under conjugation by elements in
$\mathrm{GL}_{N}(\mathbb{R})$.  This allows us to permute the factors
of $M$, $\pi'$ and $\nu$.  In view of our partition of
(\ref{factorset}) into pairs and singletons, we may choose a
permutation such that, without loss of generality,
\begin{equation}
  \label{newpigen}
\begin{aligned}
  \pi' \otimes e^{\nu} = & (\varpi'_{1} \otimes e^{\nu'_{1}}) \otimes
  \cdots \otimes (\varpi'_{k} \otimes e^{\nu_{k}'})\\
 &\otimes (\varpi'_{k+1} \otimes e^{0}) \otimes  \cdots
\otimes (\varpi'_{h}
\otimes e^{0}) \\
& \otimes (\varpi'_{h+1} \otimes e^{0}) \otimes  \cdots
\otimes (\varpi'_{i}
\otimes e^{0})\\
 & \otimes (\varpi'_{k} \otimes e^{-\nu'_{k}}) \otimes
  \cdots \otimes (\varpi'_{1} \otimes e^{-\nu_{1}'}).
\end{aligned}
\end{equation}
The factors $\varpi'_{j}$ here have the same form as the $\pi'_{j}$.  What
is different in
this decomposition of $\pi' \otimes e^{\nu}$ is that we are separating the
factors into three groups.  The first group is comprised of the first
and fourth lines of (\ref{newpigen}).  This group
corresponds to the pairs in the partition of (\ref{factorset}).  The
second and third groups encompass the  singletons, which are
essentially (limits of) discrete series.   The group
in the second line is taken to be those singletons in the partition
whose infinitesimal characters in ${^\vee}\mathfrak{h}$ all have
half-integral entries (elements in
$\mathbb{Z} + \frac{1}{2}$).  The group in the third
line is comprised of the  singletons whose infinitesimal characters
all have integral entries.

It is not difficult to realize
representations in the first and second groups as subquotients of a
 principal series  representation.
For example, letting $B_{1}$ be the upper-triangular
Borel subgroup in $\mathrm{GL}_{2}$, the representation
$$(D'_{m} \otimes e^{\nu_{1}'}) \otimes (D_{m}' \otimes
e^{-\nu_{1}'})$$
embeds into
\begin{equation}
  \label{Iform1}
\mathrm{ind}_{B_{1}(\mathbb{R}) \times
  B_{1}(\mathbb{R})}^{\mathrm{GL}_{2}(\mathbb{R}) \times \mathrm{GL}_{2}(\mathbb{R})}
\left( (|\cdot|^{(m-1)/2} \otimes sgn(\cdot) |\cdot|^{-(m-1)/2})
\otimes e^{\nu_{1}'}  \right) \otimes \left( (sgn(\cdot)
|\cdot|^{(m-1)/2} \otimes  |\cdot|^{-(m-1)/2})
\otimes e^{-\nu_{1}'}  \right)
\end{equation}
if $m$ is odd, and embeds into
\begin{equation}
  \label{Iform2}
\mathrm{ind}_{B_{1}(\mathbb{R}) \times
  B_{1}(\mathbb{R})}^{\mathrm{GL}_{2}(\mathbb{R}) \times
  \mathrm{GL}_{2}(\mathbb{R}) }
\left( (|\cdot|^{(m-1)/2} \otimes |\cdot|^{-(m-1)/2})
\otimes e^{\nu_{1}'}  \right) \otimes \left(
|\cdot|^{(m-1)/2} \otimes |\cdot|^{-(m-1)/2})
\otimes e^{-\nu_{1}'}  \right)
\end{equation}
if $m$ is even (\emph{cf.} (\ref{gl2ps})). More generally, one may
parabolically induce
$$ \left( (\varpi'_{1} \otimes e^{\nu_{1}'}) \otimes
  \cdots \otimes (\varpi'_{k} \otimes e^{\nu_{k}'})  \right) \otimes
  \left( (\varpi'_{k} \otimes e^{-\nu_{k}'}) \otimes
  \cdots \otimes (\varpi'_{ 1} \otimes e^{-\nu_{1}'})  \right)$$
to a representation $\pi_{I}$ of $\mathrm{GL}_{n_{1}}(\mathbb{R}) \times
\mathrm{GL}_{n_{1}}(\mathbb{R})$, where $n_{1}$ is the sum of the block
sizes of the first $k$ blocks. By induction in stages, one may show that
$\pi_{I}$  embeds into a  principal series representation
of  $\mathrm{GL}_{n_{1}}(\mathbb{R}) \times \mathrm{GL}_{n_{1}}(\mathbb{R})$.

  The representation
$$(\varpi'_{k+1} \otimes e^{0}) \otimes  \cdots
\otimes (\varpi'_{h}  \otimes e^{0})$$
  in the  second group of (\ref{newpigen}) is equal to
$$(D_{n_{k+1}}' \otimes e^{0}) \otimes \cdots \otimes (D'_{n_{h}}
  \otimes e^{0})$$ where $n_{k+1}, \ldots , n_{h}$ are all
even integers, and without loss of generality, $n_{k+1} \geq n_{k+2}
\geq \cdots \geq n_{h}$.  This places us in the context of Lemma
\ref{biglem} (\cite{Knapp}*{Theorem 14.91}), but it is easy to write
things out explicitly here
again.  By (\ref{gl2ps}), each factor embeds into a
principal series representation of $\mathrm{GL}_{2}(\mathbb{R})$.  One
may parabolically induce $(\varpi'_{k+1} \otimes e^{0}) \otimes  \cdots
\otimes (\varpi'_{h}  \otimes e^{0})$ to a representation $\pi_{II}$ of
$\mathrm{GL}_{n_{2}} (\mathbb{R})$, where $n_{2} = 2(h-(k+1))$ is the sum of the
block sizes from $k+1$ to $h$.  Using
induction in stages, one
may show that  $\pi_{II}$ embeds into a principal series
representation of $\mathrm{GL}_{n_{2}}(\mathbb{R})$.  After
conjugating by an element in $\mathrm{GL}_{n_{2}}(\mathbb{R})$, the
principal series representation may be taken to be induced from the
upper-triangular Borel subgroup with the
quasicharacter
\begin{equation}\label{IIform}
\begin{aligned}
| \cdot |^{(n_{k+1} -1)/2} &\otimes | \cdot |^{(n_{k+2} -1)/2} \otimes
\cdots \otimes | \cdot |^{(n_{h} -1)/2} \\
 & \otimes | \cdot |^{-(n_{h} -1)/2}
\otimes \cdots \otimes  | \cdot |^{-(n_{k+2} -1)/2} \otimes  | \cdot
|^{-(n_{k+1} -1)/2}.
\end{aligned}
\end{equation}

This brings us to the representation $(\varpi'_{h+1} \otimes e^{0})
\otimes  \cdots \otimes (\varpi'_{i} \otimes e^{0})$ in the third
group of (\ref{newpigen}).  It too may be parabolically induced to a
representation $\pi_{III}$ of $\mathrm{GL}_{n_{3}}(\mathbb{R})$ where $n_{3}$ is
the sum of the block sizes from $h+1$ to $i$.  This induced
representation must be irreducible, as otherwise, by induction in
stages,
$\mathrm{ind}_{P(\mathbb{R})}^{\mathrm{GL}_{N}(\mathbb{R})}(\pi'
\otimes e^{\nu}) \cong \pi_{\mathrm{gen}}$ would be reducible.  It
follows from \cite{Knapp}*{Theorem 14.91} that $\pi_{III}$ is tempered.
In addition, the  $\vartheta$-stability of
$\mathrm{ind}_{P(\mathbb{R})}^{\mathrm{GL}_{N}(\mathbb{R})}(\pi'
\otimes e^{\nu})$ combined with the Langlands Disjointness Theorem
imply that $\pi_{III}$ is a $\vartheta$-stable representation of
$\mathrm{GL}_{n_{3}}(\mathbb{R})$.  By Lemma \ref{biglem}, $\pi_{III}$
appears as a subquotient of either a $\vartheta$-stable principal
series representation of $\mathrm{GL}_{n_{3}}(\mathbb{R})$ or of a
representation parabolically induced from a representation of the form
(\ref{almprin}).

Taking the tensor product of $\pi_{I}$, $\pi_{II}$ and $\pi_{III}$, we
obtain a representation of a block-diagonal Levi subgroup $L$ of
$\mathrm{GL}_{N}(\mathbb{R})$ such that
$$L(\mathbb{R}) \cong \mathrm{GL}_{n_{1}}(\mathbb{R}) \times
\mathrm{GL}_{n_{2}}(\mathbb{R}) \times \mathrm{GL}_{n_{3}}(\mathbb{R})
\times \mathrm{GL}_{n_{1}}(\mathbb{R}).$$
By construction
$$\pi_{I} \otimes \pi_{II} \otimes \pi_{III} =
\mathrm{ind}_{(P \cap L)(\mathbb{R})}^{L(\mathbb{R})} (\pi' \otimes
e^{\nu}),$$
and appears as a
subquotient of either a principal series representation of
$L(\mathbb{R})$, or  a representation
described by (\ref{almprin}) which is nearly in the principal series.

Suppose first that $\pi_{I} \otimes \pi_{II} \otimes \pi_{III}$ is a
subquotient of a principal series representation of $L(\mathbb{R})$
and let $Q$ be a parabolic subgroup of $\mathrm{GL}_{N}$ with $L$ as
its Levi subgroup.  Then
$$\pi_{\mathrm{gen}} \cong
  \mathrm{ind}_{P(\mathbb{R})}^{\mathrm{GL}_{N}(\mathbb{R})} (\pi'
  \otimes e^{\nu}) \cong
  \mathrm{ind}_{Q(\mathbb{R})}^{\mathrm{GL}_{N}(\mathbb{R})} \mathrm{ind}_{(P \cap
      L)(\mathbb{R})}^{L(\mathbb{R})} (\pi' \otimes
  e^{\nu}) \cong
 \mathrm{ind}_{Q(\mathbb{R})}^{\mathrm{GL}_{N}(\mathbb{R})} (\pi_{I}
 \otimes \pi_{II} \otimes \pi_{III})$$
appears as a subquotient of a principal series representation
$\mathrm{ind}_{B(\mathbb{R})}^{\mathrm{GL}_{N}(\mathbb{R})} \pi_{0}$  of
$\mathrm{GL}_{N}(\mathbb{R})$.  We are free to conjugate this
principal series representation  by an element
of $\mathrm{GL}_{N}(\mathbb{R})$ in order to permute the
$\mathrm{GL}_{1}(\mathbb{R})$ factors of
$\pi_{0}$.  The factors of $\pi_{0}$ are given by
(\ref{Iform1}), (\ref{Iform2}) and (\ref{IIform}) (\emph{cf.} Lemma
\ref{biglem} for the factors coming from $\pi_{II}$ and
$\pi_{III}$).  Clearly, the  factors of $\pi_{0}$ are paired off so
that the tensor product of some permutation of them is fixed under
$\vartheta$.  Taking $\pi_{0}$ to be this $\vartheta$-fixed tensor
product allows us to
conclude that $\pi_{\mathrm{gen}}$ is infinitesimally equivalent to a
subquotient of
a $\vartheta$-stable principal series representation
$\mathrm{ind}_{B(\mathbb{R})}^{\mathrm{GL}_{N}(\mathbb{R})} \pi_{0}$.
The principal series
$\mathrm{ind}_{B(\mathbb{R})}^{\mathrm{GL}_{N}(\mathbb{R})} \pi_{0}$
is represented by $M(\xi_{p})$
where $\xi_{p} \in \Xi(\O, {^\vee}
\mathrm{GL}_{N}^{\Gamma})^\vartheta$ is as in Lemma \ref{prinsame}
and so $M(\xi_{p})^{\thicksim} = M(\xi_{p})^{+}$

Finally, suppose that $\pi_{I} \otimes \pi_{II} \otimes \pi_{III}$ is
infinitesimally equivalent to a subquotient of an induced
representation which is described by (\ref{almprin}).  Then we may
argue as in the previous paragraph,
except that now exactly one of the factors of $\pi_{0}$, stemming from
$\pi_{III}$, is a relative discrete
series representation $D_{2j + 1}$ of $\mathrm{GL}_{2}(\mathbb{R})$
for a positive integer $j$.  After possibly permuting the factors
of $\pi_{0}$, the factor $D_{2j+1}$ may be assumed to occupy the
middle block of $\mathrm{GL}_{N}$ ($N$ is even by Lemma \ref{biglem}).  The
remaining factors of $\pi_{0}$ are quasicharacters of
$\mathrm{GL}_{1}(\mathbb{R})$ which are paired off as before, so that
we may assume that $\pi_{0}$ is $\vartheta$-stable (Lemma
\ref{gl2wa}).
The representation of $\mathrm{GL}_{N}(\mathbb{R})$
which is parabolically induced from $\pi_{0}$ using a standard
parabolic subgroup, is then $\vartheta$-stable.  This representation
may be represented by $M(\xi_{p})$ for some $\xi_{p} \in \Xi(\O, {^\vee}
\mathrm{GL}_{N}^{\Gamma})^\vartheta$ or as $M(x,y)$ for the equivalent
Atlas parameter (Lemma \ref{XXXi}).  The representation
$M(x,y)$ has the same form as (\ref{instages}), with $x = x'$, $\pi_{0}$
replacing $\pi_{M}$, $y$ replacing $y'$, and the infinitesimal character
of $\pi_{0}$ replacing $\lambda$.  The arguments following
(\ref{instages}) apply equally well to $M(x,y)$ and so we conclude
that
$$M(\xi_{p})^{\thicksim} = M(x,y)^{\thicksim} = M(x,y)^{+} = M(\xi_{p})^{+}.$$
\end{proof}
\begin{prop}
\label{wasign2}
Suppose $\xi \in \Xi(\O, {^\vee}
\mathrm{GL}_{N}^{\Gamma})^\vartheta$.
Then
  $$M(\xi)^{\thicksim}(\vartheta)  = (-1)^{l^{I}(\xi) -
  l^{I}_{\vartheta}(\xi)} \ M(\xi)^{+} (\vartheta)$$
and
$$\pi(\xi)^{\thicksim}(\vartheta)  = (-1)^{l^{I}(\xi) -
  l^{I}_{\vartheta}(\xi)} \ \pi(\xi)^{+} (\vartheta).$$
\end{prop}
\begin{proof}
Let $\pi_{\mathrm{gen}} = \pi(\xi_{0})$ be the
generic representation of Lemma \ref{uniquegeneric} and let $\xi_{p}
$ be as in Lemma \ref{biglem1}.  Then the proof follows  the
proof of Proposition \ref{wasign1} exactly, although it is not as
straightforward to show that $l^{I}(\xi_{p}) =
l^{I}_{\vartheta}(\xi_{p})$ when $M(\xi_{p})$ is not a principal
series representation.  Let $(x,y)$ be the Atlas parameter equivalent
to $\xi_{p}$.  When $M(\xi_{p})$ is not a principal series
representation it was noted in the proof of Lemma \ref{biglem1} that
$N$ is even and that
the Cartan involution corresponding to $x = x'$ acts on roots as
$-s_{\alpha_{N/2}}$ (see (\ref{cartanaction})).  From (\ref{intlength}),
(\ref{thetalength}) and the fact that $\alpha_{N/2}$ is also simple root
of type 1  in $R_{\vartheta}(\mathrm{GL}_{N},
H)$ (see (\ref{simplekappa})),
we compute
\begin{align*}
l^{I}(\xi_{p}) &= -\frac{1}{2} \left( \left| \left\{\alpha \in R^{+}(\lambda) :
-s_{\alpha_{N/2}} (\alpha) \in R^{+}(\lambda) \right\}\right| +
\dim\left(H^{-s_{\alpha_{N/2}}}\right) \right)\\
& = - \frac{1}{2}(1+1)\\
& = -\frac{1}{2} \left( \left| \left\{\alpha \in
R^{+}_{\vartheta}(\lambda)  :
-s_{\alpha_{N/2}} (\alpha) \in R^{+}_{\vartheta}(\lambda) \right\} \right|+
\dim\left(\left(H^{\vartheta}\right)^{-s_{\alpha_{N/2}}}\right) \right)\\
& = l^{I}_{\vartheta}(\xi_{p}).
\end{align*}
\end{proof}

It is worth observing that Theorem \ref{twistpairing} now has the
following simple form, reminiscent of Theorem \ref{ordpairing}.
\begin{cor}
  \label{twistpairingfinal}
  The pairing \eqref{pair2}, defined by \eqref{pairdef2},
  satisfies
$$\langle M(\xi)^{\thicksim}, \mu(\xi')^{+} \rangle = \delta_{\xi, \xi'}$$
and
  $$\langle \pi(\xi)^{\thicksim}, P(\xi')^{+} \rangle = (-1)^{d(\xi)} \,
  \delta_{\xi, \xi'}
$$
for $\xi, \xi' \in \Xi(\O, {^\vee}\mathrm{GL}_{N}^{\Gamma})^\vartheta$.
Equivalently,
\begin{equation}
  \label{twist15.13a}
m_{r}^{\thicksim}(\xi',\xi)  = (-1)^{d(\xi) - d(\xi')}
c_{g}^{\vartheta}(\xi, \xi'),
\nomenclature{$m_{r}^{\thicksim}(\xi',\xi)$}{}
\end{equation}
where $m_{r}^{\thicksim}(\xi', \xi)$ is defined by the decomposition
\begin{equation}
M(\xi)^{\thicksim} = \sum_{\xi' \in \Xi(\mathcal{O},
{^\vee}\mathrm{GL}_{N}^{\Gamma})^{\vartheta}}
m^{\thicksim}_{r}(\xi',\xi) \, \pi(\xi')^{\thicksim}.
\label{twistmult3}
\end{equation}
in  $K\Pi(\mathcal{O}, \mathrm{GL}_{N}(\mathbb{R}), \vartheta)$.
\end{cor}
\begin{proof}
The first assertion is an immediate consequence of Proposition
\ref{wasign2}.  For
the second assertion, we return to decomposition (\ref{twistmult1}).
Substituting the Whittaker extensions of Proposition \ref{wasign2} into this
decomposition and comparing with (\ref{twistmult3}), we deduce that
$$m_{r}^{\vartheta}(\xi', \xi) = (-1)^{l^{I}(\xi') -
  l^{I}_{\vartheta}(\xi') -( l^{I}(\xi) - l^{I}_{\vartheta}(\xi))}\,
m_{r}^{\thicksim}(\xi', \xi).$$
Substituting this expression into the identity of Proposition
\ref{p:twist} we see that the first assertion is equivalent to
$$m_{r}^{\thicksim}(\xi', \xi) =  (-1)^{l^{I}(\xi) -
  l^{I}(\xi') } \, c_{g}^{\vartheta}(\xi, \xi').$$
This identity is equivalent to (\ref{twist15.13a}), as
$$l^{I}(\xi) - d(\xi) = l^{I}(\xi') - d(\xi')$$
is a constant independent of $\xi, \xi' \in
\Xi(\mathcal{O},
{^\vee}\mathrm{GL}_{N}^{\Gamma})^{\vartheta}$ (Proposition B.1 \cite{AMR1}).
\end{proof}

Another consequence of Proposition \ref{wasign2} is that the endoscopic
lifting map
$\mathrm{Lift}_{0}$  (\ref{twistendlift}) is equal to the
endoscopic transfer map $\mathrm{Trans}_{G}^{\mathrm{GL}_{N} \rtimes \vartheta}$ used
in Arthur's definition (\ref{spectrans}) of
$\etaA_{\psi_{G}}$.  This is a crucial step in
the comparison of $\etaA_{\psi_{G}}$ and $\etaABV_{\psi_{G}}$.
\begin{cor}
  \label{cortrans}
  Suppose $G$ is a simple twisted endoscopic group as in Section
  \ref{twistendsec}.  Suppose further that $S_G \subset
  X(\O_{G},\LG)$ is a
  ${^\vee}G$-orbit and let
  $\epsilon(S_{G})  \subset X(\O, \LGL)$ be the
  ${^\vee}\mathrm{GL}_{N}$-orbit of the image of $S_{G}$ under
  $\epsilon$ (\ref{epinclusion1}). Then
  \begin{enumerate}[label={(\alph*)}]
  \item \begin{equation*}
    \label{endlift1}
    \mathrm{Lift}_{0} (\eta^{\mathrm{loc}}_{S_{G}}(\upsigma)(\delta_{q}))
    = M(\epsilon(S_{G}), 1)^{\thicksim},
    \end{equation*}

  \item  $$\Lift_0 =\Trans_{G}^{\mathrm{GL}_{N} \rtimes \vartheta}$$
    on $K_{\mathbb{C}}\Pi(\O_{G}, G(\mathbb{R}, \delta_{q}))^{\mathrm{st}}$.

    \end{enumerate}

\end{cor}

  \begin{proof}
    The first assertion is  an immediate consequence of Propositions
    \ref{twistimlift} and \ref{wasign2}.  The second assertion follows
    from the identity
    $$\Trans_{G}^{\mathrm{GL}_{N} \rtimes \vartheta} ( \eta_{S_{G}}^{\mathrm{loc}}(\delta_{q})) =
    M(\epsilon(S_{G}), 1)^{\thicksim}$$
(\cite{AMR}*{(1.0.3)}, \cite{Mezo2}) and the fact that the stable
    virtual characters $\eta_{S_{G}}^{\mathrm{loc}}(\delta_{q})$ form a basis
    for $K_{\mathbb{C}} \Pi(\O_{G}, G(\mathbb{R},
    \delta_{q}))^{\mathrm{st}}$ as $S_{G}$ runs
over the ${^\vee}G$-orbits in  $X(\O_{G},\LG)$.
\end{proof}

\section{The comparison of $\Pi_{\psi_{G}}^{\mathrm{Ar}}$ and $\Pi_{\psi_{G}}^{\mathrm{ABV}}$ for
  regular infinitesimal character}
\label{equalapacketreg}

In this section we prove the main results comparing $\etaA_{\psi_{G}}$
with $\etaABV_{\psi_{G}}$ ((\ref{main1}), (\ref{main2})).
We shall work under the assumptions of Section \ref{twistendsec}.  In
particular, $\psi_{G}$ and $\psi = \epsilon \circ \psi_{G}$ are A-parameters
with respective infinitesimal characters $\O_{G}$ and
$\O$.  The assumption on the infinitesimal characters is that
they are regular with respect to $\mathrm{GL}_{N}$.  This assumption
shall be removed in the next section.

The  definition of $\etaA_{\psi_{G}}$ was outlined in
(\ref{spectrans}).  Let us provide a few more details from
\cite{Arthur}.  The key lemma is
\begin{lem}
  \label{2.2.2}
  Let  $S_{\psi} \subset X(\O,
  {^\vee}\mathrm{GL}_{N}^{\Gamma})$ be the
  ${^\vee}\mathrm{GL}_{N}$-orbit corresponding to $\phi_{\psi}$ (\cite{ABV}*{Proposition 6.17}, (\ref{eq:orbitbijection})).
\begin{enumerate}[label={(\alph*)}]
\item  There
    exist integers $n_{S}$ such that
    \begin{equation}
\label{piwhittdecomp}
\pi(S_{\psi}, 1)^{\thicksim} = \sum_{(S, 1)  \in
\Xi(\O, {^\vee} \mathrm{GL}_{N})^{\vartheta} } n_{S} \,
M(S, 1)^{\thicksim}
\end{equation}
    in $K\Pi(\O, \mathrm{GL}_{N}(\mathbb{R}), \vartheta)$.

\item  Moreover, for every $S$ such that $n_{S} \neq 0$ in
  (\ref{piwhittdecomp}) there exists a unique
${^\vee}G$-orbit $S_{G} \subset X(\O_{G},
{^\vee}G^{\Gamma})$  which is carried to $S$ under $\epsilon$.

\item Writing
  $$S = \epsilon(S_{G})$$ for the orbits in part (b), we have
\begin{equation}\label{piwhittdecomp1}
  \begin{aligned}
\pi(S_{\psi}, 1)^{\thicksim} &=
 \mathrm{Trans}_{G}^{\mathrm{GL}_{N} \rtimes \vartheta}\left(
 \sum_{S_{G}}  n_{\epsilon(S_{G})} \, \eta_{S_{G}}^{\mathrm{loc}} (\delta_{q})
 \right)\\
  &=
 \mathrm{Lift}_{0} \left(
 \sum_{S_{G}}  n_{\epsilon(S_{G})} \, \eta_{S_{G}}^{\mathrm{loc}} (\delta_{q})
 \right) .
\end{aligned}
\end{equation}

\end{enumerate}
\end{lem}
\begin{proof}
By virtue of Proposition \ref{wasign2}, (\ref{piwhittdecomp})  is
equivalent to a decomposition
$$\pi(S_{\psi}, 1)^{+} = \sum_{(S, 1)  \in
\Xi(\O, {^\vee} \mathrm{GL}_{N})^{\vartheta} } n_{S}' \,
M(S, 1)^{+}$$
of Atlas extensions.   The latter decomposition follows from
(\ref{twistmult1}) and Lemma \ref{invertmultmat}.
The existence of the orbit $S_{G}$ in part (b) is established on the
first page of the proof of \cite{Arthur}*{Lemma 2.2.2}.   The uniqueness
of the orbit follows from Lemma \ref{twistinj}.
Part (c) is a consequence of Corollary \ref{cortrans}.
\end{proof}

Arthur's definition of $\etaA_{\psi_{G}}$ is  easiest to state
when the endoscopic group $G$ is not equal to $\mathrm{SO}_{N}$ for
even $N$.  In this case, the group $G$ has no outer automorphisms and
the defining equation (\ref{spectrans}) is the same as
(\ref{piwhittdecomp1}).  It follows that
\begin{equation}
\label{Arthurstabchar}
\etaA_{\psi_{G}} = \sum_{S_{G}} n_{\epsilon(S_{G})} \,
\eta^{\mathrm{loc}}_{S_{G}}(\delta_{q})  \in K_{\mathbb{C}}
\Pi(\O_{G}, G(\mathbb{R}, \delta_{q}))^{\mathrm{st}}
\nomenclature{$\etaA_{\psi_{G}}$}{stable virtual character defining an Arthur packet}
\end{equation}
(\emph{cf.} \cite{Arthur}*{(2.2.12)}).  By definition, the Arthur packet
$\Pi_{\psi_{G}}^{\mathrm{Ar}}
\nomenclature{$\Pi_{\psi_{G}}^{\mathrm{Ar}}$}{Arthur's packet}
$ consists of
those irreducible characters in $\Pi(\O_{G}, G(\mathbb{R},
\delta_{q}))$ which occur with non-zero multiplicity when
(\ref{Arthurstabchar}) is expressed as a linear combination in the
basis of irreducible characters.

In the case that $N$ is even and $G = \mathrm{SO}_{N}$, the stable
virtual character $\etaA_{\psi_{G}}$ is defined to be invariant under
the action of the
outer automorphisms induced by the orthogonal group
$\mathrm{O}_{N}$ (\cite{Arthur}*{\emph{pp.} 12, 41}).  Fix
\begin{equation}
  \label{tildew}
\tilde{w}
\in \mathrm{O}_{N} - \mathrm{SO}_{N}.
\nomenclature{$\tilde{w}$}{}
\end{equation}
The orthogonal group acts on geometric parameters in $X(\O_{G},
{^\vee}G^{\Gamma})$ in a straightforward manner, sending
them to geometric parameters in $X(\tilde{w}\cdot \O_{G},
{^\vee}G^{\Gamma})$.  The stable virtual character
$$\frac{1}{2}(\eta_{S_{G}}^{\mathrm{loc}}(\delta_{q})  + \eta_{\tilde{w} \cdot
  S_{G}}^{\mathrm{loc}}(\delta_{q}) )  \in
K_{\mathbb{C}} \Pi(\O_{G}, G(\mathbb{R},
\delta_{q}))^{\mathrm{st}} \oplus K_{\mathbb{C}} \Pi(\tilde{w} \cdot
\O_{G}, G(\mathbb{R},
\delta_{q}))^{\mathrm{st}}$$
is $\mathrm{O}_{N}$-invariant by design.  Extending the domain of
$\mathrm{Trans}_{G}^{\mathrm{GL}_{N} \rtimes \vartheta}$ to the space
on the right, equations (\ref{spectrans}) and (\ref{piwhittdecomp1}) imply
\begin{equation*}
\label{Arthurstabchar1}
\etaA_{\psi_{G}} = \sum_{S_{G}} \frac{n_{\epsilon(S_{G}) }}{2} \,
(\eta^{\mathrm{loc}}_{S_{G}}(\delta_{q}) + \eta^{\mathrm{loc}}_{\tilde{w}\cdot
S_{G}}(\delta_{q})).
\end{equation*}
This is a virtual character in $K_{\mathbb{C}}
\Pi(\O_{G}, G(\mathbb{R}, \delta_{q}))^{\mathrm{st}} \oplus  K_{\mathbb{C}}
\Pi(\tilde{w} \cdot \O_{G}, G(\mathbb{R}, \delta_{q}))^{\mathrm{st}}$
and the Arthur packet $\Pi_{\psi_{G}}$  consists of the
irreducible characters in its support.

\begin{thm}
  \label{finalthm}
  \begin{enumerate}[label={(\alph*)}]

\item  If $G$ is not isomorphic to $\mathrm{SO}_{N}$ for even $N$ then
  $$\etaA_{\psi_{G}} = \eta_{\psi_{G}}^{\mathrm{mic}}(\delta_{q}) =
  \etaABV_{\psi_{G}}   \quad\mbox{ and }\quad
  \Pi_{\psi_{G}}^{\mathrm{Ar}} = \Pi_{\psi_{G}}^{\mathrm{ABV}}.$$
\item   If $N$ is even and $G \cong \mathrm{SO}_{N}$ then
  \begin{align*}
  \eta_{\psi_{G}} &= \frac{1}{2} \left(\eta_{\psi_{G}}^{\mathrm{mic}}(\delta_{q}) +
  \eta^{\mathrm{mic}}_{\mathrm{Int}(\tilde{w}) \circ \psi_{G}}(\delta_{q})
  \right)  = \frac{1}{2} \left(\etaABV_{\psi_{G}} +
  \etaABV_{\mathrm{Int}(\tilde{w}) \circ \psi_{G}}
  \right)
\end{align*}
  {and }
   \begin{align*}
    \Pi_{\psi_{G}}^{\mathrm{Ar}} &= \Pi_{\psi_{G}}^{\mathrm{ABV}} \cup \,
 \Pi_{\mathrm{Int}(\tilde{w})   \circ   \psi_{G}}^{\mathrm{ABV}}
  \end{align*}
  where the union is disjoint.
\end{enumerate}
  \end{thm}
  \begin{proof}
    This just involves putting together the pieces. Let $\xi=(S_{\psi},1)$
    as in Corollary \ref{etaplus2}.


    $$
    \begin{aligned}
\mathrm{Lift}_{0} \, (\eta^{\mathrm{mic}}_{\psi_{G}}(\delta_{q})) &=  \mathrm{Lift}_{0} \,(
\eta^{\mathrm{mic}}_{\psi_{G}}(\upsigma) (\delta_{q}))\quad(\text{by } \eqref{nosigma})\\
&=(-1)^{l^{I}(\xi)-l^{I}_{\vartheta}(\xi)} \pi(\xi)^{+}
\quad(\text{Corollary }\ref{etaplus2})
\\
&=\pi(\xi)^\thicksim\quad(\text{Proposition }\ref{wasign2})\\
&=\pi(S_{\psi},1)^\thicksim\\
&=  \mathrm{Trans}_{G}^{\mathrm{GL}_{N} \rtimes
  \vartheta}\left( \sum_{S_{G}}  n_{\epsilon(S_{G})} \,
  \eta_{S_{G}}^{\mathrm{loc}} (\delta_{q}) \right)
\quad(\mathrm{Lemma }\, \ref{2.2.2})\\
& =  \mathrm{Lift}_{0}  \left( \etaA_{\psi_{G}} \right)\quad(\mathrm{Corollary }\ \ref{cortrans}(b)).\\
\end{aligned}
$$
The equality of the stable virtual characters follows from the injectivity of
$\mathrm{Lift}_{0}$  (Proposition \ref{injlift2}).
The equality of packets follows immediately.

For part (b), identical reasoning leads to the identity
$$\mathrm{Lift}_{0} (\eta^{\mathrm{mic}}_{\psi_{G}}(\delta_{q})) =
\mathrm{Trans}_{G}^{\mathrm{GL}_{N} \rtimes \vartheta} \left(
\sum_{S_{G}} n_{\epsilon({S}_{G})} (\eta^{\mathrm{loc}}_{S_{G}}(\delta_{q}) +
\eta^{\mathrm{loc}}_{\tilde{w} \cdot   S_{G}}(\delta_{q})) \right) =
\mathrm{Lift}_{0} (\etaA_{\psi_{G}}).$$
The first assertion of part (b) follows from the injectivity of
$\mathrm{Lift}_{0}$ as before.
The irreducible characters in the support of $\sum_{S_{G}} n_{\epsilon(S_{G})}
\eta^{\mathrm{loc}}_{S_{G}}(\delta_{q})$  are those in the packet $\Pi_{\psi_{G}}^{\mathrm{ABV}}$.
These irreducible characters lie in $\Pi(\O_{G},
G(\mathbb{R}, \delta_{q}))$.
Similarly the irreducible characters in the support of
$\sum_{S_{G}} n_{\epsilon(S_{G})}
\eta^{\mathrm{loc}}_{\tilde{w} \cdot S_{G}}(\delta_{q})$ are those in
$\Pi_{\mathrm{Int}(\tilde{w}) \circ \psi_{G}}^{\mathrm{ABV}}$. These irreducible
characters lie in $\Pi(\tilde{w} \cdot \O_{G},
G(\mathbb{R}, \delta_{q}))$. The regularity of the infinitesimal
character $\O_{G}$ implies $\O_{G} \cap
(\tilde{w}\cdot \O_{G}) = \emptyset$, otherwise there exists
$g \in {^\vee}\mathrm{SO}_{N}$ and $\lambda \in \O_{G} \cap
{^\vee}\mathfrak{h}^{\vartheta}$ such that
$\mathrm{Ad}(\tilde{w}g) \lambda = \lambda$.  This would imply $\tilde{w}g \in
{^\vee}H^{\vartheta} \subset {^\vee}\mathrm{SO}_{N}$, contradicting
the definition of $\tilde{w}$.  In
consequence,
$$\Pi(\O_{G},
G(\mathbb{R}, \delta_{q})) \cap \Pi(\tilde{w} \cdot \O_{G},
G(\mathbb{R}, \delta_{q})) = \emptyset$$
and
$$\Pi_{\psi_{G}}^{\mathrm{ABV}} \cap \,  \Pi_{\mathrm{Int}(\tilde{w})
  \circ \cdot  \psi_{G}}^{\mathrm{ABV}} = \emptyset.$$
This proves the final assertion.
\end{proof}

\section{The comparison of $\Pi_{\psi_{G}}^{\mathrm{Ar}}$ and $\Pi_{\psi_{G}}^{\mathrm{ABV}}$ for
  singular infinitesimal character}
\label{equalapacketsing}

To conclude our comparison of stable virtual characters,  we  retain
the setup of the previous
section, but without the hypothesis of regularity on
the infinitesimal character.  In other words, the orbits $\O$ and
$\O_{G}$ are now allowed to be orbits of singular
infinitesimal characters and the reader should think of them as such.
In order to prove something like Theorem
\ref{finalthm} for singular $\O$, we must extend the
pairing of Theorem \ref{twistpairing} and extend the twisted
endoscopic lifting (\ref{twistendlift}) to
include representations with singular infinitesimal character.  The
main tool for this extension is the
\emph{Jantzen-Zuckerman translation principle}, which we refer to simply as \emph{translation}.
In essence the  Jantzen-Zuckerman
translation principle allows one to transfer results for regular infinitesimal
character to results for singular infinitesimal character.  Applying
this principle to the results of the previous section will allow us to
compare $\Pi_{\psi_{G}}$ with $\Pi_{\psi_{G}}^{\mathrm{ABV}}$ with no restriction on
the infinitesimal character.

The reader is assumed to have some familiarity with the Jantzen-Zuckerman
translation principle, which for us begins with the existence of a
regular orbit $\O' \subset {^\vee}\mathfrak{gl}_{N}$ and a
\emph{translation datum}
$\mathcal{T}$ from $\O$ to $\O'$ (\cite{ABV}*{Definition 8.6, Lemma 8.7}).
\nomenclature{$\mathcal{T}$}{translation datum}
 A key feature of the translation
datum is that if $\O$ is the ${^\vee}\mathrm{GL}_{N}$-orbit
of $\lambda \in {^\vee}\mathfrak{h}$
then $\O'$ is the ${^\vee}\mathrm{GL}_{N}$-orbit of
\begin{equation}
\label{lambdaprime}
\lambda' = \lambda + \lambda_{1} \in {^\vee}\mathfrak{h}
\end{equation}
 where $\lambda_{1} \in X_{*}(H)$ is regular and  dominant with respect to the
 positive system of $R^{+}(\mathrm{GL}_{N},H)$.   We may and shall
 take $\lambda_{1}$ to be the sum of the positive roots.  In this way,
each of $\lambda$, $\lambda_{1}$ and $\lambda'$ are fixed by
$\vartheta$.  The translation
 datum $\mathcal{T}$
 induces a ${^\vee}\mathrm{GL}_{N}$-equivariant morphism
\begin{equation}
\label{ftee}
f_{\mathcal{T}} :  X(\O', {^\vee}\mathrm{GL}_{N}^{\Gamma}) \rightarrow
X(\O, {^\vee}\mathrm{GL}_{N}^{\Gamma}) \nomenclature{$f_{\mathcal{T}}$}{}
\end{equation}
of geometric parameters  (\cite{ABV}*{Proposition 8.8}).  The morphism
has connected fibres of fixed dimension, a fact we shall use
when comparing orbit dimensions.  The
${^\vee}\mathrm{GL}_{N}$-equivariance of (\ref{ftee}) is tantamount to
a coset map commuting with left-multiplication by ${^\vee}\mathrm{GL}_{N}$
(\cite{ABV}*{(6.10)(b)}).
Since both $\lambda$ and $\lambda'$ are fixed by $\vartheta$, it is
just as easy to see that the action of $\vartheta$ commutes with the
same coset map.  We leave this exercise to the reader, taking for
granted the resulting $({^\vee}\mathrm{GL}_{N} \rtimes \langle
\vartheta \rangle)$-equivariance of (\ref{ftee}).

According to \cite{ABV}*{Proposition 7.15}, the morphism $f_{\mathcal{T}}$ induces an
inclusion
\begin{equation}
\label{fstar}
f^{*}_{\mathcal{T}}: \Xi(\O,{^\vee}\mathrm{GL}_{N}^{\Gamma})
\hookrightarrow \Xi(\O',{^\vee}\mathrm{GL}_{N}^{\Gamma})\nomenclature{$f^{*}_{\mathcal{T}}$}{}
\end{equation}
of complete geometric parameters.
The $\vartheta$-equivariance of
(\ref{ftee}) implies that this inclusion restricts to an inclusion (denoted by the same symbol)
$$f^{*}_{\mathcal{T}}: \Xi(\O,{^\vee}\mathrm{GL}_{N}^{\Gamma})^\vartheta
\hookrightarrow \Xi(\O',{^\vee}\mathrm{GL}_{N}^{\Gamma})^\vartheta. $$
The (Jantzen-Zuckerman) translation functor (\cite{AvLTV}*{(17.8j)})
$$T_{\lambda'}^{\lambda} = T_{\lambda +  \lambda_{1}}^{\lambda}
\nomenclature{$T_{\lambda'}^{\lambda}$}{Jantzen-Zuckerman translation}
$$
is an exact functor on
a category of Harish-Chandra modules, which we shall often regard
as a homomorphism
\begin{equation}
  \label{transfunct}
  T_{\lambda + \lambda_{1}}^{\lambda}: K \Pi(\O',
  \mathrm{GL}_{N}(\mathbb{R}) \rtimes \langle \vartheta \rangle )
  \rightarrow  K \Pi(\O,
  \mathrm{GL}_{N}(\mathbb{R}) \rtimes \langle \vartheta \rangle )
\end{equation}
of Grothendieck groups.
It is  surjective  (\cite{AvLTV}*{Corollary 17.9.8}).
This translation functor is an extended version of the usual
translation functor (\cite{AvLTV}*{(16.8f)}), which we
also  denote by
\begin{equation}
  \label{ordtrans}
  T_{\lambda + \lambda_{1}}^{\lambda}: K \Pi(\O',
  \mathrm{GL}_{N}(\mathbb{R}) )
  \rightarrow  K \Pi(\O,
  \mathrm{GL}_{N}(\mathbb{R}) ).
\end{equation}
Let us take a moment to make (\ref{transfunct}) more precise.  The sum of the
positive roots $\lambda_{1}$ is the infinitesimal character of a
finite-dimensional representation of $\mathrm{GL}_{N}(\mathbb{R})$.
Therefore, $\lambda_{1}$ is the differential of a $\vartheta$-fixed
quasicharacter $\Lambda_{1}$ of the
split real diagonal torus $H(\mathbb{R})$, which matches the weight of this
finite-dimensional representation. The quasicharacter
$\Lambda_{1}$ may be extended to a quasicharacter $\Lambda_{1}^{+}$ of the
semi-direct product
$H(\mathbb{R}) \rtimes \langle \vartheta \rangle$  by
setting
\begin{equation}
  \label{lambdaoneplus}
  \Lambda_{1}^{+}(\vartheta) = 1.
\end{equation}

We define translation in the extended setting of (\ref{transfunct})
using this representation of the extended group.  Since the extension
is evident here we continue to write
$T_{\lambda+\lambda_1}^{\lambda}$ instead of $T_{\lambda+\Lambda_1^+}^{\lambda}$.

In the ordinary setting of (\ref{ordtrans}) we have
\begin{align*}
\pi(\xi) &= T_{\lambda + \lambda_{1}}^{\lambda} \left(
  \pi(f^{*}_{\mathcal{T}}(\xi)) \right),\\
  \nonumber  M(\xi) &= T_{\lambda + \lambda_{1}}^{\lambda} \left(
  M(f^{*}_{\mathcal{T}}(\xi)) \right), \quad \xi \in \Xi(\O,
   {^\vee}\mathrm{GL}_{N}^{\Gamma})
\end{align*}
(\cite{AvLTV}*{Corollary 16.9.4, 16.9.7 and 16.9.8}, or \cite{ABV}*{ Theorem 16.4 and
   Proposition 16.6}).
We \emph{define} the Atlas extensions of
$\pi(\xi)$ and $M(\xi)$, with $\xi \in
\Xi(\O,{^\vee}\mathrm{GL}_{N}^{\Gamma})^\vartheta$, by
\begin{align*}
  \label{atlassing}
\pi(\xi)^{+} &= T_{\lambda + \lambda_{1}}^{\lambda}
  (\pi(f^{*}_{\mathcal{T}}(\xi))^{+})\\
   M(\xi)^{+} &= T_{\lambda + \lambda_{1}}^{\lambda}
  (M(f^{*}_{\mathcal{T}}(\xi))^{+}) .
\end{align*}
(To be careful, one should verify that this definition does not conflict
with Section \ref{extrepsec} when $\O$ is regular. This amounts
to the observation that the translate of the $(\mathfrak{h}, (H \cap
K_{\delta}) \rtimes \langle \vartheta \rangle)$-module underlying an Atlas
extension remains trivial on $\vartheta$.  Justification for this
observation is given in proof of Proposition \ref{whittowhit}).

With the definition of Atlas extensions in place, the discussion of
Section \ref{grtwisted} is valid, and we see that
$T_{\lambda + \lambda_{1}}^{\lambda}$ factors to a homomorphism of $K
\Pi(\O, \mathrm{GL}_{N}(\mathbb{R}), \vartheta)$ (see
(\ref{twistgroth1})).  We use the same notation
$T_{\lambda + \lambda_{1}}^{\lambda}$ to denote the functor of Harish-Chandra modules,
and either of the earlier homomorphisms.  The reader will be reminded
of the context when it is important.

The definition of a Whittaker extension does not depend on the
regularity of the infinitesimal character.   The following proposition
shows that  translation sends Whittaker extensions to
Whittaker extensions.
\begin{prop}
  \label{whittowhit}
Suppose $\xi \in \Xi( \O, {^\vee}
\mathrm{GL}_{N}^{\Gamma})^\vartheta$.  Then (as Harish-Chandra
modules)
$$T_{\lambda + \lambda_{1}}^{\lambda}\left(
M(f^{*}_{\mathcal{T}}(\xi))^{\thicksim} \right) =
M(\xi)^{\thicksim},$$
and
$$T_{\lambda + \lambda_{1}}^{\lambda} \left( \pi(f^{*}_{\mathcal{T}}(\xi))^{\thicksim}\right) =
\pi(\xi)^{\thicksim}.$$
\end{prop}
\begin{proof}
Lemma \ref{uniquegeneric} does not require $\O$ to be
regular, and so we may choose $\xi_{0} \in  \Xi( \O, {^\vee}
\mathrm{GL}_{N}^{\Gamma})^\vartheta$ to be the unique parameter
such that $\pi(\xi_{0})$ is generic and embeds as a
subrepresentation of $M(\xi)$.  (Here, we are implicitly assuming that we are
working with actual admissible representations (or Harish-Chandra
modules) rather
than equivalence classes.) Since the orbit $S_{\xi_{0}} \subset
X(\O, {^\vee}\mathrm{GL}_{N}^{\Gamma})$ of $\xi_{0}$ is open
(see the proof of Proposition \ref{conjq}), it is an immediate
consequence of the definition of
$f^{*}_{\mathcal{T}}$ (\cite{ABV}*{(7.16)(b)}) that
$f^{*}_{\mathcal{T}}(S_{\xi_{0}})$ is open and therefore
that $\pi(f^{*}_{\mathcal{T}}(\xi_{0}))$ is generic.

According to Lemma \ref{biglem1},  $\pi(f^{*}_{\mathcal{T}}(\xi_{0}))$ embeds into a
standard representation $M(\xi_{p}')$ for some $\xi_{p}' \in
\Xi( \O', {^\vee}
\mathrm{GL}_{N}^{\Gamma})^\vartheta$, which satisfies $M(\xi_{p}')^{\thicksim} =
M(\xi_{p}')^{+}$.
Furthermore,  $\pi(f^{*}_{\mathcal{T}}(\xi_{0}))$ occurs as a
subrepresentation with  multiplicity one, and
$\pi(f^{*}_{\mathcal{T}}(\xi_{0}))^{\thicksim}$ is a subrepresentation of
$M(\xi_{p}')^{\thicksim}$ (Lemma \ref{uniquegeneric}).  Applying the
exact functor $T_{\lambda + \lambda_{1}}^{\lambda}$
(of Harish-Chandra modules), we see that
$T_{\lambda + \lambda_{1}}^{\lambda}\left(
\pi(f^{*}_{\mathcal{T}}(\xi_{0}))^{\thicksim} \right)$ is
a subrepresentation of $T_{\lambda + \lambda_{1}}^{\lambda} \left(
M(\xi_{p}')^{\thicksim} \right)$.

Suppose first that $M(\xi_{p}')$ is a principal series
representation.  By \cite{AvLTV}*{Corollary 17.9.7}
$T_{\lambda + \lambda_{1}}^{\lambda}
\left( M(\xi_{p}')^{\thicksim} \right)$ is an extension of
a principal series representation.  Indeed, $M(\xi_{p}')^{\thicksim} =
M(\xi_{p}')^{+}$ (Lemma \ref{prinsame}) and is parabolically induced
from a quasicharacter of a
split Cartan subgroup extended by $\langle \vartheta \rangle$--the value of this
quasicharacter on $\vartheta$ being one (Section \ref{extrepsec}).
  \cite{AvLTV}*{Corollary 17.9.7} tells us that $T_{\lambda + \lambda_{1}}^{\lambda}
\left( M(\xi_{p}')^{\thicksim} \right)$ is parabolically induced from
the tensor product of the aforementioned quasicharacter with the inverse of
$\Lambda_{1}^{+}$ as in
(\ref{lambdaoneplus}) (\cite{AvLTV}*{Theorem 17.7.5}).  This
justifies   $T_{\lambda + \lambda_{1}}^{\lambda}
\left( M(\xi_{p}')^{\thicksim} \right)$ being an extended principal
series representation, but more can be said.  In view of
(\ref{lambdaoneplus}),
translation by $(\Lambda_{1}^{+})^{-1}$  does not affect the value of the
quasicharacter on $\vartheta$.  Consequently its value on $\vartheta$
is still one.  The arguments of Lemma \ref{prinsame} therefore apply to
$T_{\lambda + \lambda_{1}}^{\lambda} \left( M(\xi_{p}')^{\thicksim} \right)$ as
they do for $M(\xi_{p})^{\thicksim}$ and we deduce
\begin{equation}
  \label{whitrans}
  \omega = \omega \circ T_{\lambda + \lambda_{1}}^{\lambda} \left(
  M(\xi_{p}')^{\thicksim} \right) (\vartheta)
  \end{equation}
for the Whittaker functional $\omega$ defined by (\ref{whittfun}).

If $M(\xi_{p}')$ is not a principal series representation then it is
of the form (\ref{nearps}), which is a parabolically induced
representation, essentially from a relative discrete series
representation on $\mathrm{GL}_{2}(\mathbb{R})$.  Such a representation
may still be regarded as being induced, albeit not parabolically
induced,  from a quasicharacter of a
non-split Cartan subgroup. The earlier arguments from \cite{AvLTV}*{ Corollary
17.9.7} apply.  We leave it to the reader, to verify that
(\ref{whitrans}) holds in any case.  In consequence of
(\ref{whitrans}) and the exactness of $T_{\lambda + \lambda_{1}}^{\lambda}$,
$$\omega \circ T_{\lambda + \lambda_{1}}^{\lambda}\left(
\pi(f^{*}_{\mathcal{T}}(\xi_{0}))^{\thicksim} \right)(\vartheta) =
\omega \circ T_{\lambda + \lambda_{1}}^{\lambda} \left(
  M(\xi_{p}')^{\thicksim} \right) (\vartheta)_{|\pi(\xi_{0})} = \omega.$$
This proves that $ T_{\lambda + \lambda_{1}}^{\lambda}\left(
\pi(f^{*}_{\mathcal{T}}(\xi_{0}))^{\thicksim} \right)$ is the
Whittaker extension of
$$T_{\lambda + \lambda_{1}}^{\lambda}(\pi(f^{*}_{\mathcal{T}}(\xi_{0})) =
\pi(\xi_{0}),$$ that is
\begin{equation}
  \label{genwhitt}
T_{\lambda + \lambda_{1}}^{\lambda}\left(
\pi(f^{*}_{\mathcal{T}}(\xi_{0}))^{\thicksim} \right) =
\pi(\xi_{0})^{\thicksim}.
\end{equation}

To complete the proposition we embed
$\pi(f^{*}_{\mathcal{T}}(\xi_{0}))^{\thicksim}$ as a subrepresentation
of $M(f^{*}_{\mathcal{T}}(\xi))^{\thicksim}$ using Lemma
\ref{uniquegeneric}.  Applying the exact functor $T_{\lambda + \lambda_{1}}^{\lambda}$
(of Harish-Chandra modules), we see that (\ref{genwhitt}) is a
subrepresentation of $T_{\lambda + \lambda_{1}}^{\lambda}\left(
M(f^{*}_{\mathcal{T}}(\xi))^{\thicksim} \right)$.  Since a Whittaker
functional $\omega$ of  $T_{\lambda + \lambda_{1}}^{\lambda}\left(
M(f^{*}_{\mathcal{T}}(\xi)) \right) = M(\xi)$ restricts to a non-zero
Whittaker functional $\omega_{| \pi(\xi_{0})}$ on  $\pi(\xi_{0})$ and
$$c \, \omega  = \omega \circ T_{\lambda + \lambda_{1}}^{\lambda}\left(
M(f^{*}_{\mathcal{T}}(\xi))^{\thicksim} \right) (\vartheta)$$
for $c = \pm 1$, we deduce in succession that
$$c \, \omega_{| \pi(\xi_{0})}  = \omega \circ T_{\lambda + \lambda_{1}}^{\lambda}\left(
M(f^{*}_{\mathcal{T}}(\xi))^{\thicksim} \right)
(\vartheta)_{| \pi(\xi_{0})} = \omega \circ \pi(\xi_{0})^{\thicksim}(\vartheta) =
\omega_{| \pi(\xi_{0})},$$
$c = 1$ and
$$\omega  = \omega \circ T_{\lambda + \lambda_{1}}^{\lambda}\left(
M(f^{*}_{\mathcal{T}}(\xi))^{\thicksim} \right) (\vartheta).$$
The final equation implies
$$T_{\lambda + \lambda_{1}}^{\lambda}\left(
M(f^{*}_{\mathcal{T}}(\xi))^{\thicksim} \right) =
M(\xi)^{\thicksim}.$$
Since $\pi(f^{*}_{\mathcal{T}}(\xi))^{\thicksim}$ and
$\pi(\xi)^{\thicksim}$ are the unique irreducible
quotients of $M(f^{*}_{\mathcal{T}}(\xi))^{\thicksim}$ and
$M(\xi)^{\thicksim}$ respectively,
and $T_{\lambda + \lambda_{1}}^{\lambda}$ is exact (on Harish-Chandra modules),
the proposition follows.
\end{proof}

Our translation datum $\mathcal{T}$ for $\mathrm{GL}_{N}$ is defined by
(\ref{lambdaprime}), in which both $\lambda$ and $\lambda'$ are fixed
by the endoscopic datum $\mathrm{Int}(s) \circ \vartheta$.  For this
reason  (\ref{lambdaprime}) also determines a translation datum
$\mathcal{T}_{G}$ from
${^\vee}G$-orbits  $\O_{G}$ to   $\O'_{G}$ for
the twisted endoscopic group $G$ (\cite{ABV}*{Definition 8.6 (e)}).
Just as
for $\mathrm{GL}_{N}$, we have maps
\begin{align*}
\label{fteeG}
 f_{\mathcal{T}_{G}} &:  X(\O', {^\vee}G^{\Gamma}) \rightarrow
X(\O, {^\vee}G^{\Gamma}) \\
\nonumber f^{*}_{\mathcal{T}_{G}}&: \Xi(\O,{^\vee}G^{\Gamma})
\hookrightarrow \Xi(\O',{^\vee}G^{\Gamma})
\end{align*}
and the  translation functor $T_{\lambda +
  \lambda_{1}}^{\lambda}$ which satisfies
$$\pi(\xi) = T_{\lambda + \lambda_{1}}^{\lambda}\left(
\pi(f^{*}_{\mathcal{T}_{G}}(\xi)) \right), \quad \xi \in
\Xi(\O_{G}, {^\vee}G^{\Gamma})$$
(\cite{ABV}*{Proposition 16.6}, \cite{AvLTV}*{Section 16}).

The translation data $\mathcal{T}$ and $\mathcal{T}_{G}$ allow us to
transport properties of our pairings at regular infinitesimal
character (Proposition \ref{twistpairingfinal}) to the same properties
for pairings at singular  infinitesimal character.
\begin{prop}
  \label{finalpairing}
  Define the pairing
\begin{equation}\label{eq:finalpairing}
\langle \cdot, \cdot \rangle:  K \Pi(\O,
\mathrm{GL}_{N}(\mathbb{R}), \vartheta)
\times K X(\O, {^\vee}\mathrm{GL}_{N}^{\Gamma}, \upsigma)
\rightarrow \mathbb{Z}
\end{equation}
by
$$\langle M(\xi)^{\thicksim}, \mu(\xi')^{+} \rangle = \delta_{\xi, \xi'}.$$
Then
  $$\langle \pi(\xi)^{\thicksim}, P(\xi')^{+} \rangle = (-1)^{d(\xi)} \,
  \delta_{\xi, \xi'}$$
where $\xi, \xi' \in
\Xi(\O, {^\vee}\mathrm{GL}_{N}^{\Gamma})^\vartheta$.
\end{prop}
\begin{proof}
We first sketch the proof for the ordinary pairing of Theorem
\ref{ordpairing} in \cite{ABV}, which does not involve twisting by
$\vartheta$.  This will allow us to point out the portions of the
proof that must be modified in the twisted setting.
In the ordinary case there are no Whittaker or Atlas extensions, and
the identity to be proven is (\ref{pairingmc})
$$ m_r(\xi_{1}, \xi_{2}) = (-1)^{d(\xi_{1}) - d(\xi_{2})} \,
c_{g}(\xi_{2}, \xi_{1}), \quad  \xi_{1}, \xi_{2} \in \Xi(\O,
{^\vee}G^{\Gamma})$$
for the possibly singular
orbit $\O$.   The idea of the proof is to show that both
sides of (\ref{pairingmc}) are invariant under translation.  Starting with
the right-hand side of (\ref{pairingmc}), we use
\cite{ABV}*{Proposition 8.8 (b)}, which provides an exact functor
from $\mathcal{P}(X(\O,{^\vee}G^{\Gamma}))$ to
$\mathcal{P}(X(\O',{^\vee}G^{\Gamma}))$  satisfying
$$P(\xi) \mapsto P(f^{*}_{\mathcal{T}_{G}}(\xi))$$
and
\begin{equation}
\label{geomatrix}
c_{g}(f^{*}_{\mathcal{T}_{G}}(\xi_{1}), f^{*}_{\mathcal{T}_{G}}(\xi_{2})) =
c_{g}(\xi_{1}, \xi_{2}), \quad \xi_{1}, \xi_{2} \in \Xi(\O,
{^\vee}G^{\Gamma}).
\end{equation}
The invariance of the left-hand side of (\ref{pairingmc})
$$m_r(f^{*}_{\mathcal{T}_{G}}(\xi_{1}), f^{*}_{\mathcal{T}_{G}}(\xi_{2})) =
m_r(\xi_{1}, \xi_{2}),  \quad \xi_{1}, \xi_{2} \in \Xi(\O,
{^\vee}G^{\Gamma}),$$
is given by  \cite{ABV}*{Proposition 16.6 and (16.5)(d)}, which
rely on the translation functor (\cite{ABV}*{(16.3)}).
All that is now needed
to prove (\ref{pairingmc}) for the possibly singular orbit $\O$
is to line up the equations
\begin{equation}\label{lineup}
\begin{aligned}
 m_r(\xi_{1}, \xi_{2}) &=
m_r(f^{*}_{\mathcal{T}_{G}}(\xi_{1}), f^{*}_{\mathcal{T}_{G}}(\xi_{2}))\\
& = (-1)^{d(f^{*}_{\mathcal{T}_{G}}(\xi_{1})) -
  d(f^{*}_{\mathcal{T}_{G}}(\xi_{2}))} \,
c_{g}(f^{*}_{\mathcal{T}_{G}}(\xi_{2}), f^{*}_{\mathcal{T}_{G}}(\xi_{1}))\\
& = (-1)^{(d(\xi_{1}) -d) - (d(\xi_{2})-d )} \, c_{g}(\xi_{2}, \xi_{1})\\
& = (-1)^{d(\xi_{1}) - d(\xi_{2})} \, c_{g}(\xi_{2},
\xi_{1}).
\end{aligned}
\end{equation}
In the third equation, we have used  \cite{ABV}*{ (7.16)(b)} and the
dimension $d$ of the connected fibres of $f^{*}_{\mathcal{T}_{G}}$ to describe
the orbit dimensions.

Let us repeat the preceding proof in the twisted setting.
The desired analogue of (\ref{geomatrix}) is
\begin{equation}
\label{geomatrix1}
c_{g}(f^{*}_{\mathcal{T}}(\xi_{1})_{\pm},
f^{*}_{\mathcal{T}}(\xi_{2})_{\pm}) = c_{g}(\xi_{1\pm}, \xi_{2\pm}),
\quad \xi_{1}, \xi_{2} \in \Xi(\O,
      {^\vee}\mathrm{GL}_{N}^{\Gamma})^\vartheta
\end{equation}
(see (\ref{extsheafmult})). As before,  \cite{ABV}*{Proposition 8.8 (b)} ensures this as long as
\begin{equation}
\label{plustoplus}
  P(\xi)^{+} \mapsto P(f^{*}_{\mathcal{T}}(\xi))^{+}, \quad \xi \in
  \Xi(\O, {^\vee}\mathrm{GL}_{N}^{\Gamma})^\vartheta
  \end{equation}
(\cite{ABV}*{Proposition 7.15 (b)}).
This may be seen as follows.  The complete geometric parameter
$f^{*}_{\mathcal{T}} (\xi)$ determines a $({^\vee}\mathrm{GL}_{N}
\rtimes \langle \vartheta \rangle)$-equivariant local system, and the
perverse sheaf $P(\xi)^{+}$ is mapped to the
intermediate extension  of this local system (to its
${^\vee}\mathrm{GL}_{N}$-orbit closure) (\cite{ABV}*{(7.10)(d)},
\cite{bbd}*{\emph{p.} 110}).
Let us call the resulting perverse sheaf $P^{+}$.   We would like
$P^{+} = P(f^{*}_{\mathcal{T}}(\xi))^{+}$ and this holds when the
$({^\vee}\mathrm{GL}_{N} \rtimes \langle \vartheta
\rangle)$-equivariant irreducible constructible sheaf
$\mu(f^{*}_{\mathcal{T}}(\xi))^{+}$ occurs in the decomposition of
$P^{+}$ (\emph{cf.} (\ref{extsheafmult})).  The latter property is
true, for $P^{+}$ and $\mu(f^{*}_{\mathcal{T}}(\xi))^{+}$ are obtained
from the same $({^\vee}\mathrm{GL}_{N} \rtimes \langle \vartheta
\rangle)$-equivariant local system by intermediate extension and
extension by zero respectively (\emph{cf}.  \cite{ABV}*{(7.11)(b)}).
 This justifies (\ref{plustoplus}) and therefore also
 (\ref{geomatrix1}). By definition (\ref{twistgmult}) and (\ref{geomatrix1})
\begin{equation}
\label{geomatrix2}
c_{g}^{\vartheta}(f^{*}_{\mathcal{T}}(\xi_{1}),
f^{*}_{\mathcal{T}}(\xi_{2})) = c_{g}^{\vartheta}(\xi_{1}, \xi_{2}),
\quad \xi_{1}, \xi_{2} \in \Xi(\O,
      {^\vee}\mathrm{GL}_{N}^{\Gamma})^\vartheta.
\end{equation}
Moving to the representation-theoretic multiplicities, we appeal to
the
translation functor $T_{\lambda + \lambda_{1}}^{\lambda}$  for
$\mathrm{GL}_{N} \rtimes \langle \vartheta \rangle$.
In $K\Pi(\O', \mathrm{GL}_{N}(\mathbb{R}), \vartheta)$ we have
\begin{align}
  \label{twistmult4}
M(f^{*}_{\mathcal{T}}(\xi_{2}))^{\thicksim} = \sum_{\xi_{1} \in \Xi(\O,
{^\vee}\mathrm{GL}_{N}^{\Gamma})^\vartheta}
m^{\thicksim}_{r}(f^{*}_{\mathcal{T}}(\xi_{1}),f^{*}_{\mathcal{T}}
(\xi_{2}))& \, \pi(f^{*}_{\mathcal{T}}(\xi_{1}) )^{\thicksim}\\
\nonumber&+\  \sum_{\xi'}
m^{\thicksim}_{r}(\xi',f^{*}_{\mathcal{T}}
(\xi_{2})) \, \pi(\xi')^{\thicksim}
\end{align}
where $\xi'$  are those parameters in $\Xi(\O',
{^\vee}\mathrm{GL}_{N}^{\Gamma})^\vartheta$ which do not lie in
the image of (\ref{fstar}).
Applying $T_{\lambda + \lambda_{1}}^{\lambda}$ to (\ref{twistmult4})
has the effect of annihilating the second sum on the right
 (\cite{AvLTV}*{Corollary 17.9.4 and 17.9.8}).  By Proposition \ref{whittowhit}, the remaining
terms are
$$M(\xi_{2})^{\thicksim} = \sum_{\xi_{1} \in \Xi(\O,
{^\vee}\mathrm{GL}_{N}^{\Gamma})^\vartheta}
m^{\thicksim}_{r}(f^{*}_{\mathcal{T}}(\xi_{1}),f^{*}_{\mathcal{T}}
(\xi_{2})) \, \pi(\xi_{1})^{\thicksim}$$
and this equation implies
\begin{equation}
\label{repmatrix1}
m_r^{\thicksim}(f^{*}_{\mathcal{T}}(\xi_{1}),
f^{*}_{\mathcal{T}}(\xi_{2})) = m_r^{\thicksim}(\xi_{1}, \xi_{2}),
\quad \xi_{1}, \xi_{2} \in \Xi(\O,
      {^\vee}\mathrm{GL}_{N}^{\Gamma})^\vartheta.
\end{equation}
Using equations (\ref{geomatrix2}) and (\ref{repmatrix1}), and
replacing $m_r$ and $c_{g}$  with $m_r^{\thicksim}$ and
$c_{g}^{\vartheta}$ respectively in (\ref{lineup}), we deduce that
(\ref{twist15.13a}) holds for the possibly singular orbit $\O$.
\end{proof}
Proposition \ref{finalpairing} is the final version of the twisted
pairing, and we use it to extend the definition of endoscopic lifting
$\mathrm{Lift}_{0}$
to include singular infinitesimal characters ((\ref{twistloweps}),
(\ref{twistendlift})).  In fact, all of the remaining results used in Section
\ref{equalapacketreg} easily carry over to the more general setting,
except for the injectivity of $\mathrm{Lift}_{0}$ (Proposition
\ref{injlift2}).
In particular, using the pairing
(\ref{eq:finalpairing})
in the proof of
Proposition    \ref{twistimlift}, we see that
for any  ${^\vee}G$-orbit $S_G \subset
  X(\O_{G},\LG)$ we still have
$$\mathrm{Lift}_{0} \left(\eta^{\mathrm{loc}}_{S_{G}}(\upsigma)(\delta_{q})\right)
    = M(\epsilon(S_{G}), 1)^{\thicksim}.$$
It is explained in  \cite{Arthur}*{\emph{p.} 31} that $\mathrm{Lift}_{0}$
is injective when $G$ is not isomorphic to $\mathrm{SO}_{N}$ for even
$N$.  However, when $G \cong \mathrm{SO}_{N}$ for even $N$, the
endoscopic lifting map is only injective on $\mathrm{O}_{N}$-orbits
\cite{Arthur}*{\emph{pp.} 12, 31}).   That is to say,
$$\mathrm{Lift_{0}}\left( \eta^{\mathrm{loc}}_{S_{1}}(\delta_{q})\right)= \mathrm{Lift}_{0}\left(
\eta_{S_{2}}^{\mathrm{loc}}(\delta_{q})\right)$$
for ${^\vee}G$-orbits of complete geometric parameters if and only if
$S_{2} = \tilde{w} \cdot S_{1}$  for $\tilde{w}$ as in
(\ref{tildew}).  One might hope to retain injectivity by restricting
the infinitesimal character of $S_{1}$ to lie in $\O_{G}$,
but this too fails as it is not difficult to construct singular
examples in which $\O_{G} = \tilde{w}\cdot \O_{G}$.

Recall decomposition (\ref{piwhittdecomp})
$$\pi(S_{\psi}, 1)^{\thicksim} = \sum_{(S, 1)  \in
\Xi(\O, {^\vee} \mathrm{GL}_{N}^{\Gamma})^\vartheta} n_{S} \,
M(S, 1)^{\thicksim}.$$
As in the previous section, for each
${^\vee}\mathrm{GL}_{N}$-orbit $S$ with $n_{S} \neq 0$ there
exists a ${^\vee}G$-orbit $S_{G} \subset
X(\O_{G},{^\vee}G^{\Gamma})$  such that $\epsilon(S_{G}) = S$.
The difference now is that when $G$ is an even
special orthogonal group the orbit $S_{G}$ may not be uniquely
determined in $X(\O_{G}, {^\vee}G^{\Gamma})$.  The lack of
uniqueness forces us to weaken  Theorem \ref{finalthm} in this context
in that the ABV-packets in part (b) below are no longer necessarily disjoint.
  \begin{thm}
    \label{finalthm1}
      \begin{enumerate}[label={(\alph*)}]
\item  If $G$ is not isomorphic to $\mathrm{SO}_{N}$ for even $N$ then
  $$\etaA_{\psi_{G}} = \eta_{\psi_{G}}^{\mathrm{mic}}(\delta_{q}) =
  \etaABV_{\psi_{G}}   \quad\mbox{and}\quad
  \Pi_{\psi_{G}}^{\mathrm{Ar}} = \Pi_{\psi_{G}}^{\mathrm{ABV}}.$$
\item   If $N$ is even and $G \cong \mathrm{SO}_{N}$ then
  \begin{align*}
  \eta_{\psi_{G}} &= \frac{1}{2} \left(\eta_{\psi_{G}}^{\mathrm{mic}}(\delta_{q}) +
  \eta^{\mathrm{mic}}_{\mathrm{Int}(\tilde{w}) \circ \psi_{G}}(\delta_{q})
  \right)  = \frac{1}{2} \left(\etaABV_{\psi_{G}} +
  \etaABV_{\mathrm{Int}(\tilde{w}) \circ \psi_{G}}
  \right)
  \end{align*}
   and
   \begin{align*}\Pi_{\psi_{G}}^{\mathrm{Ar}} &= \Pi_{\psi_{G}}^{\mathrm{ABV}} \cup \,
 \Pi_{\mathrm{Int}(\tilde{w})   \circ   \psi_{G}}^{\mathrm{ABV}}.
  \end{align*}
\end{enumerate}

  \end{thm}
  \begin{proof}
The proof of the first assertion is completely the same as the proof
of Theorem \ref{finalthm}(a), since the injectivity of
$\mathrm{Lift}_{0}$ holds.  Suppose therefore that $N$ is even and $G
\cong \mathrm{SO}_{N}$.
As in the proof  of
Theorem \ref{finalthm} we have
\begin{align*}
\mathrm{Lift}_{0} \left( \eta^{\mathrm{mic}}_{\psi_{G}}(\delta_{q}) +
\eta^{\mathrm{mic}}_{\mathrm{Int}(\tilde{w}) \circ \psi_{G}}(\delta_{q}) \right)
& = 2 \pi(S_{\psi}, 1)^{\thicksim}\\
& = \sum_{j} 2n_{S} \, M(S, 1)^{\thicksim}\\
& =  \mathrm{Lift}_{0}  \left( \sum_{S_{G}} n_{\epsilon(S_{G})} ( \eta^{\mathrm{loc}}_{S_{G}}(\delta_{q})+
\eta^{\mathrm{loc}}_{\tilde{w} \cdot S_{G}}(\delta_{q})) \right).
\end{align*}
Since $\mathrm{Lift}_{0}$ is injective on $\mathrm{O}_{N}$-orbits of
stable virtual characters, the second assertion follows.
\end{proof}

\section{The comparison of Problems B-E}
\label{btoe}

Theorem \ref{finalthm1} is a comparison of the
solutions to \hyperlink{proba}{Problem A} of Arthur and
Adams-Barbasch-Vogan.  Let us compare the remaining problems of the
introduction.

\hyperlink{probe}{Problem E}, concerning the unitarity of the
representations in the Arthur packets,  stands apart from Problems B-D.  It is
also easy to dispense with.  Arthur proves that $\Pi_{\psi_{G}}$
consists of unitary representations (\cite{Arthur}*{Theorem 2.2.1 (b)}),
and so by Theorem \ref{finalthm1}, every packet $\Pi_{\psi_{G}}^{\mathrm{ABV}}$
also consists of unitary representations.  The unitarity of the subset
of special unipotent representations is proven by different methods in
\cite{BMSZ}.

For problems B-D, we review Arthur's approach first.  The stable
virtual character $\eta_{\psi_{G}}^{\mathrm{Ar}}$ is  written
\begin{equation}
\label{7.1.2}
\eta_{\psi_{G}}^{\mathrm{Ar}} = \sum_{\sigma \in \tilde{\Sigma}_{\psi_{G}}} < s_{\psi_{G}}, \sigma > \sigma
\end{equation}
as in
\cite{Arthur}*{(7.1.2)}.  Here, $\tilde{\Sigma}_{\psi_{G}}$ is a
finite set of non-negative integral linear combinations
$$\sigma = \sum_{\pi \in  \Pi_{\mathrm{unit}}(G(\mathbb{R}))} m(\sigma, \pi) \, \pi$$
of irreducible unitary characters of $G(\mathbb{R}) = G(\mathbb{R}, \delta_q)$.  Furthermore, there is an injective map from
$\tilde{\Sigma}_{\psi_{G}}$ into the set of those
quasicharacters of $$A_{\psi_{G}}={^\vee}G_{\psi_{G}}/
({^\vee}G_{\psi_{G}})^{0}$$ which are trivial on the centre of
${^\vee}G$.  The injection is denoted by
$$\sigma \mapsto < \cdot, \sigma>.$$
The element $s_{\psi_{G}}$ is the image of
$$\small\psi_{G} \left(1, \begin{bmatrix} -1 & 0\\ 0 &
  -1 \end{bmatrix} \right)$$\normalsize
in $A_{\psi_{G}}$.  The element $s_{\psi_{G}}$ is clearly of order
two.   It is easy to rewrite (\ref{7.1.2})
as

\begin{equation}
\label{7.4.1}
\eta_{\psi_{G}}^{\mathrm{Ar}} = \sum_{\pi \in \Pi_{\psi_{G}}^{\mathrm{Ar}}} \left( \sum_{\sigma
  \in \tilde{\Sigma}_{\psi_{G}}}  m(\sigma, \pi) \, <s_{\psi_{G}},
\sigma> \right) \pi
\end{equation}
(\emph{cf.} \cite{Arthur}*{Proposition 7.4.3 and (7.4.1)}).
By defining a finite-dimensional representation
\begin{equation}
\label{artprobb}
\tau_{\psi_{G}} (\pi) = \bigoplus_{\sigma \in
  \tilde{\Sigma}_{\psi_{G}}} m(\sigma, \pi) \, <\cdot, \sigma>,
  \end{equation}
equation \eqref{7.4.1} becomes
\begin{equation}
\label{etaPsisum}
\eta_{\psi_{G}}^{\mathrm{Ar}} = \sum_{\pi \in \Pi_{\psi_{G}}^{\mathrm{Ar}}} \mathrm{Tr}
\left(\tau_{\psi_{G}}(\pi)(s_{\psi_{G}})\right) \pi.
\end{equation}
The finite-dimensional representations defined in (\ref{artprobb}) provide a solution to \hyperlink{probb}{Problem B}.  Equation (\ref{etaPsisum}) is close to a  complete resolution of
\hyperlink{probc}{Problem C}. We also need to show that the
quasicharacters $<\cdot, \sigma>$ occurring in a given $\tau_{\psi_{G}}(\pi)$
have the same value $\varepsilon_{\pi}$ at $s_{\psi_{G}}$.  In this way
the trace $\mathrm{Tr} \left(\tau_{\psi_{G}}(\pi)(s_{\psi_{G}})\right)
$ would reduce to $\varepsilon_{\pi} \cdot \dim( \tau_{\psi_{G}}(\pi))$ as
expected.  We  return to this point when we compare with
$\eta^{\mathrm{ABV}}_{\psi_{G}} = \eta_{S_{\psi_{G}}}^{\mathrm{mic}}(\delta_{q})$ (\ref{etapsiabv}) below.

\hyperlink{probd}{Problem D} concerns endoscopic lifting from an
endoscopic group $G'$ of $G$.  The endoscopic group $G'$ is defined to
be a quasisplit form of a complex reductive group whose dual ${^\vee}G'$ is the identity component of the
centralizer in ${^\vee}G$ of a semisimple element $s \in {^\vee}G$
(\emph{cf.} Section \ref{endosec}  and \cite{Arthur}*{Theorem 2.2.1(b)}).
Furthermore, the element $s$ is taken to centralize the image of $\psi_{G}$, and there is a natural
embedding $\epsilon': {^\vee}(G')^{\Gamma} \hookrightarrow
{^\vee}G^{\Gamma}$.  Arthur's solution to \hyperlink{probd}{Problem D}
tells us that if
$$\psi_{G} = \epsilon' \circ \psi_{G'}$$
for an A-parameter $\psi_{G'}$ then there exists a stable virtual
character $\eta_{\psi_{G'}}$ on $G'(\mathbb{R})$ such that
\begin{equation}
  \label{7.4.1b}
\mathrm{Trans}_{G'}^{G}(\eta_{\psi_{G'}}) = \sum_{\pi \in \Pi_{\psi_{G}}^{\mathrm{Ar}}} \mathrm{Tr}
\left(\tau_{\psi_{G}}(\pi)(s_{\psi_{G}} \bar{s})\right) \pi
\end{equation}
(\cite{Arthur}*{Theorem 2.2.1}).  Here, $\bar{s} \in A_{\psi_{G}}$ is
the coset of $s$, and $\mathrm{Trans}_{G'}^{G}$ denotes the standard
endoscopic lifting of Shelstad
(\cite{shelstad_endoscopy}).
Observe that
(\ref{etaPsisum}) is obtained from (\ref{7.4.1b}) by taking $s = 1$.

Now let us look at Problems B-D from the perspective of \cite{ABV}.  Each $\pi \in \Pi_{\psi_{G}}^{\mathrm{ABV}}$ is of the form $\pi(\xi)$ for a unique complete geometric parameter  $\xi = (S_{\xi}, \tau_{\xi})$.   Using this, we set
$$\tau_{\psi_{G}}^{\mathrm{ABV}}(\pi) = \tau_{S_{\psi_{G}}}^{\mathrm{mic}}(P(\xi))$$
as in (\ref{miclocalsys}).  This is a solution to \hyperlink{probb}{Problem B}.  The solution to \hyperlink{probc}{Problem C} is then given by (\ref{etapsi1}), which we may write as
\begin{equation*}
\label{equationc}
\eta_{\psi_{G}}^{\mathrm{ABV}} = \sum_{\pi \in \Pi_{\psi_{G}}^{\mathrm{ABV}}}
(-1)^{d(\pi) - d(S_{\psi_{G}})}
\ \dim\left(\tau^{\mathrm{ABV}}_{\psi_{G}}(\pi) \right) \, \pi,
\end{equation*}
where $d(\pi) = d(S_{\xi})$  for $\pi = \pi(\xi)$ as above.  The solution to \hyperlink{probd}{Problem D}
is given by \cite{ABV}*{Theorem 26.25}.  Translated into the setting of (\ref{7.4.1b}), it reads as
\begin{equation}
\label{equationd}
\mathrm{Lift}_{0}^{G}\left(\eta_{\psi_{G'}}^{\mathrm{ABV}}\right) = \sum_{\pi \in \Pi_{\psi_{G}}^{\mathrm{ABV}}}
(-1)^{d(\pi) - d(S_{\psi_{G}})}
\ \mathrm{Tr} \left(\tau^{\mathrm{ABV}}_{\psi_{G}}(\pi)(\bar{s}) \right) \, \pi
\end{equation}
\cite{ABV}*{Definition 24.15 and (26.17)(f)}.  Here, $\mathrm{Lift}_{0}^{G}$ is the standard endoscopic lifting map of  \cite{ABV}*{Definition 26.18}, which is defined on the stable virtual characters of  $G'(\mathbb{R})$ and take values in the virtual characters of $G(\mathbb{R})$.

We wish to compare Arthur's solutions to Problems B-D with those of \cite{ABV}. This amounts to comparing  (\ref{7.4.1b}) with (\ref{equationd}).
For this comparison, we shall, for the sake of simplicity, assume that
$$G  = \mathrm{SO}_{N}, \quad N \mbox{ odd}$$
from now on.
 This assumption avoids the irksome complications arising from even special orthogonal groups in Theorem \ref{finalthm1} (b).  Under our assumption Theorem  \ref{finalthm1} tells us that the solutions of Arthur and Adams-Barbasch-Vogan to \hyperlink{proba}{Problem A} are identical.

In comparing (\ref{7.4.1b}) with (\ref{equationd}), we must choose our endoscopic groups judiciously.  Recall that $A_{\psi_{G}}$  is the component group of
the centralizer  in ${^\vee}G$ of the image of $\psi_{G}$.  The explicit description of this centralizer in \cite{Arthur}*{(1.4.8)} makes it clear that every element $\bar{s} \in A_{\psi_{G}}$ has a diagonal representative $\dot{s}$ in the centralizer with eigenvalues $\pm 1$.  The endoscopic group $G'(\dot{s})$ determined by $\dot{s}$ is a direct product $G_{1}'(\dot{s}) \times G_{2}'(\dot{s})$ in which each of the  two factors is  a special orthogonal group of odd rank  (\cite{Arthur}*{\emph{pp.} 13-14}).  The A-parameter $\psi_{G'(\dot{s})}$ decomposes accordingly as a product $\psi_{G'_{1}(\dot{s})} \times \psi_{G'_{2}(\dot{s})}$ of A-parameters (\cite{Arthur}*{pp. 31, 36}).  Similarly, Arthur's stable virtual character $\eta_{\psi_{G'(\dot{s})}}$ is defined as the tensor product $\eta_{\psi_{G'_{1}(\dot{s})}}^{\mathrm{Ar}} \otimes \eta_{\psi_{G'_{2}(\dot{s})}}^{\mathrm{Ar}}$ (\cite{Arthur}*{Remark 2 of Theorem 2.2.1}).   Hence, a particular instance of (\ref{7.4.1b}) reads as
\begin{equation}
  \label{7.4.1c}
\mathrm{Trans}_{G'}^{G}(\eta_{\psi_{G'_{1}(\dot{s})}}^{\mathrm{Ar}} \otimes \eta^{\mathrm{Ar}}_{\psi_{G'_{2}(\dot{s})}}) = \sum_{\pi \in \Pi_{\psi_{G}}} \mathrm{Tr}
\left(\tau_{\psi_{G}}(\pi)(s_{\psi_{G}} \bar{s})\right) \pi.
\end{equation}
We now turn to rewriting the left-hand side of (\ref{7.4.1c}) so as to match it with the left-hand side of (\ref{equationd}).
First, it is noted on  \cite{ABV}*{\emph{p.} 289} that $\mathrm{Trans}_{G'}^{G}$ differs from $\mathrm{Lift}_{0}^{G}$ by no more than a scalar multiple.  To see that this scalar is one, it suffices to compare values on pseudopackets (\cite{ABV}*{Lemma 18.11}).  By applying parabolic induction, one is reduced to comparing values on tempered L-packets (\cite{ABV}*{Proposition 26.7}, \cite{ShelstadIII}*{Corollary 11.7}).  In this comparison the summands corresponding to the unique generic representation in a tempered L-packet are seen to be equal (and trivial) in both the expansions of $\mathrm{Trans}_{G'}^{G}$ (\cite{ShelstadIII}*{Theorem 11.5}) and $\mathrm{Lift}_{0}^{G}$ (\cite{ABV}*{Proposition 13.12, Corollary 10.7, Corollary 9.12}).  Hence, the scalar is one and $\mathrm{Trans}_{G'}^{G} = \mathrm{Lift}_{0}^{G}$.

Second, using the arguments in the proof of Corollary \ref{cor:LeviABVpacket}, we see that
$$\eta_{\psi_{G'(\dot{s})}}^{\mathrm{ABV}} = \eta_{\psi_{G'_{1}(\dot{s})}}^{\mathrm{ABV}} \otimes \eta_{\psi_{G'_{2}(\dot{s})}}^{\mathrm{ABV}}.$$
Third, since $G'_{1}(\dot{s})$ and $G_{2}'(\dot{s})$ are both odd rank special orthogonal groups, Theorem \ref{finalthm1} (a) tells us that
$$\eta_{\psi_{G'_{j}(\dot{s})}}^{\mathrm{Ar}} = \eta_{\psi_{G'_{j}(\dot{s})}}^{\mathrm{ABV}}, \quad j = 1,2.$$
Taking these three observations together we conclude
\begin{align*}
\mathrm{Trans}_{G'}^{G}\left(\eta_{\psi_{G'(\dot{s})}}^{\mathrm{Ar}} \right)&= \mathrm{Lift}_{0}^{G}\left(\eta_{\psi_{G_{1}'(\dot{s})}}^{\mathrm{Ar}}
\otimes \eta_{\psi_{G_{2}'(\dot{s})}}^{\mathrm{Ar}} \right)\\
& = \mathrm{Lift}_{0}^{G} \left(\eta_{\psi_{G'_{1}(\dot{s})}}^{\mathrm{ABV}} \otimes
\eta_{\psi_{G'_{2}(\dot{s})}}^{\mathrm{ABV}}\right)\\
& = \mathrm{Lift}_{0}^{G}\left(\eta_{\psi_{G'(\dot{s})}}^{\mathrm{ABV}}\right).
\end{align*}
It is now immediate from (\ref{7.4.1b}) and (\ref{equationd}) that
$$ \sum_{\pi \in \Pi_{\psi_{G}}} \mathrm{Tr}
\left(\tau_{\psi_{G}}(\pi)(s_{\psi_{G}} \bar{s})\right) \pi
= \sum_{\pi \in \Pi_{\psi_{G}}}
(-1)^{d(\pi) - d(S_{\psi_{G}})}
\ \mathrm{Tr} \left(\tau^{\mathrm{ABV}}_{\psi_{G}}(\pi)(\bar{s}) \right) \, \pi$$
for any $\bar{s} \in A_{\psi}$.
By the linear independence of characters on $G(\mathbb{R})$
$$\mathrm{Tr}
\left(\tau_{\psi_{G}}(\pi)(s_{\psi_{G}} \bar{s})\right)
= (-1)^{d(\pi) - d(S_{\psi_{G}})}
\ \mathrm{Tr} \left(\tau^{\mathrm{ABV}}_{\psi_{G}}(\pi)(\bar{s}) \right)$$
for any $\bar{s} \in A_{\psi}$.  This may be regarded as an equality between virtual (quasi)characters on $A_{\psi}$ (\emph{cf.} (\ref{artprobb})). By appealing to the linear independence of these (quasi)characters we conclude that
$$\tau_{\psi_{G}}(\pi)(s_{\psi_{G}}) = (-1)^{d(\pi) - d(S_{\psi_{G}})}$$
and
$$\tau_{\psi_{G}}(\pi) = \tau^{\mathrm{ABV}}_{\psi_{G}}(\pi).$$
The former equation gives a complete solution to Arthur's approach to \hyperlink{probc}{Problem C}.

This completes our solution of Problems B-D for odd rank special
orthogonal groups.  A similar argument holds for symplectic and even orthogonal groups,
keeping in mind
the element $\tilde{w}$ of Theorem \ref{finalthm1} (b) when comparing virtual
characters on $G(\mathbb{R})$ or $G'(\mathbb{R})$.  We leave the details to the interested reader.
\printnomenclature


\begin{bibdiv}
\begin{biblist}

\bib{snowbird}{incollection}{
      author={Adams, Jeffrey},
       title={Guide to the {A}tlas software: computational representation
  theory of real reductive groups},
        date={2008},
   booktitle={Representation theory of real reductive {L}ie groups},
      series={Contemp. Math.},
      volume={472},
   publisher={Amer. Math. Soc., Providence, RI},
       pages={1\ndash 37},
         url={http://dx.doi.org/10.1090/conm/472/09235},
      review={\MR{2454331}},
}

\bib{adamsnotes}{misc}{
      author={Adams, Jeffrey},
       title={Computing twisted klv polynomials},
        date={2017},
}

\bib{Arthur84}{incollection}{
      author={Arthur, James},
       title={On some problems suggested by the trace formula},
        date={1984},
   booktitle={Lie group representations, {II} ({C}ollege {P}ark, {M}d.,
  1982/1983)},
      series={Lecture Notes in Math.},
      volume={1041},
   publisher={Springer, Berlin},
       pages={1\ndash 49},
}

\bib{Arthur89}{article}{
      author={Arthur, James},
       title={Unipotent automorphic representations: conjectures},
        date={1989},
        ISSN={0303-1179},
     journal={Ast\'erisque},
      number={171-172},
       pages={13\ndash 71},
        note={Orbites unipotentes et repr{\'e}sentations, II},
      review={\MR{1021499}},
}

\bib{Arthur08}{incollection}{
      author={Arthur, James},
       title={Problems for real groups},
        date={2008},
   booktitle={Representation theory of real reductive {L}ie groups},
      series={Contemp. Math.},
      volume={472},
   publisher={Amer. Math. Soc., Providence, RI},
       pages={39\ndash 62},
}

\bib{Arthur}{book}{
      author={Arthur, James},
       title={The endoscopic classification of representations},
      series={American Mathematical Society Colloquium Publications},
   publisher={American Mathematical Society, Providence, RI},
        date={2013},
      volume={61},
        ISBN={978-0-8218-4990-3},
        note={Orthogonal and symplectic groups},
      review={\MR{3135650}},
}

\bib{ABV}{book}{
      author={Adams, Jeffrey},
      author={Barbasch, Dan},
      author={Vogan, David~A., Jr.},
       title={The {L}anglands classification and irreducible characters for
  real reductive groups},
      series={Progress in Mathematics},
   publisher={Birkh\"auser Boston, Inc., Boston, MA},
        date={1992},
      volume={104},
        ISBN={0-8176-3634-X},
         url={http://dx.doi.org/10.1007/978-1-4612-0383-4},
      review={\MR{1162533 (93j:22001)}},
}

\bib{Adams-Fokko}{article}{
      author={Adams, Jeffrey},
      author={du~Cloux, Fokko},
       title={Algorithms for representation theory of real reductive groups},
        date={2009},
        ISSN={1474-7480},
     journal={J. Inst. Math. Jussieu},
      volume={8},
      number={2},
       pages={209\ndash 259},
         url={https://doi.org/10.1017/S1474748008000352},
      review={\MR{2485793}},
}

\bib{Adams-Johnson}{article}{
      author={Adams, Jeffrey},
      author={Johnson, Joseph~F.},
       title={Endoscopic groups and packets of nontempered representations},
        date={1987},
        ISSN={0010-437X},
     journal={Compositio Math.},
      volume={64},
      number={3},
       pages={271\ndash 309},
         url={http://www.numdam.org/item?id=CM_1987__64_3_271_0},
      review={\MR{918414 (89h:22022)}},
}

\bib{AMR}{article}{
      author={Arancibia, Nicol\'{a}s},
      author={M{\oe}glin, Colette},
      author={Renard, David},
       title={Paquets d'{A}rthur des groupes classiques et unitaires},
        date={2018},
        ISSN={0240-2963},
     journal={Ann. Fac. Sci. Toulouse Math. (6)},
      volume={27},
      number={5},
       pages={1023\ndash 1105},
         url={https://doi.org/10.5802/afst.1590},
      review={\MR{3919547}},
}

\bib{AMR1}{misc}{
      author={Arancibia, Nicolás},
      author={Moeglin, Colette},
      author={Renard, David},
       title={Paquets d'arthur des groupes classiques et unitaires},
        date={2017},
}

\bib{arancibia_characteristic}{article}{
      author={Arancibia~Robert, Nicol\'{a}s},
       title={Characteristic cycles, micro local packets and packets with
  cohomology},
        date={2022},
        ISSN={0002-9947},
     journal={Trans. Amer. Math. Soc.},
      volume={375},
      number={2},
       pages={997\ndash 1049},
         url={https://doi-org.proxy.library.carleton.ca/10.1090/tran/8492},
      review={\MR{4369242}},
}

\bib{AV}{article}{
      author={Adams, Jeffrey},
      author={Vogan, David~A., Jr.},
       title={{$L$}-groups, projective representations, and the {L}anglands
  classification},
        date={1992},
        ISSN={0002-9327},
     journal={Amer. J. Math.},
      volume={114},
      number={1},
       pages={45\ndash 138},
         url={https://doi.org/10.2307/2374739},
      review={\MR{1147719}},
}

\bib{AV92}{article}{
      author={Adams, Jeffrey},
      author={Vogan, David~A., Jr.},
       title={{$L$}-groups, projective representations, and the {L}anglands
  classification},
        date={1992},
     journal={Amer. J. Math.},
      volume={114},
      number={1},
       pages={45\ndash 138},
}

\bib{AVParameters}{incollection}{
      author={Adams, Jeffrey},
      author={Vogan, David~A., Jr.},
       title={Parameters for twisted representations},
        date={2015},
   booktitle={Representations of reductive groups},
      series={Progr. Math.},
      volume={312},
   publisher={Birkh\"{a}user/Springer, Cham},
       pages={51\ndash 116},
      review={\MR{3495793}},
}

\bib{AvLTV}{article}{
      author={Adams, Jeffrey~D.},
      author={van Leeuwen, Marc A.~A.},
      author={Trapa, Peter~E.},
      author={Vogan, David~A., Jr.},
       title={Unitary representations of real reductive groups},
        date={2020},
        ISSN={0303-1179},
     journal={Ast\'{e}risque},
      number={417},
       pages={viii + 188},
         url={https://doi-org.proxy-um.researchport.umd.edu/10.24033/ast},
      review={\MR{4146144}},
}

\bib{BB}{article}{
      author={Beilinson, Alexandre},
      author={Bernstein, Joseph},
       title={Localisation de {$g$}-modules},
        date={1981},
        ISSN={0151-0509},
     journal={C. R. Acad. Sci. Paris S\'er. I Math.},
      volume={292},
      number={1},
       pages={15\ndash 18},
      review={\MR{610137 (82k:14015)}},
}

\bib{bbd}{incollection}{
      author={Be{\u\i}linson, A.~A.},
      author={Bernstein, J.},
      author={Deligne, P.},
       title={Faisceaux pervers},
        date={1982},
   booktitle={Analysis and topology on singular spaces, {I} ({L}uminy, 1981)},
      series={Ast\'erisque},
      volume={100},
   publisher={Soc. Math. France, Paris},
       pages={5\ndash 171},
      review={\MR{751966}},
}

\bib{borel}{inproceedings}{
      author={Borel, A.},
       title={Automorphic {$L$}-functions},
        date={1979},
   booktitle={Automorphic forms, representations and {$L$}-functions ({P}roc.
  {S}ympos. {P}ure {M}ath., {O}regon {S}tate {U}niv., {C}orvallis, {O}re.,
  1977), {P}art 2},
      series={Proc. Sympos. Pure Math., XXXIII},
   publisher={Amer. Math. Soc., Providence, R.I.},
       pages={27\ndash 61},
}

\bib{Bourbaki}{book}{
      author={Bourbaki, Nicolas},
       title={Lie groups and {L}ie algebras. {C}hapters 4--6},
      series={Elements of Mathematics (Berlin)},
   publisher={Springer-Verlag, Berlin},
        date={2002},
        note={Translated from the 1968 French original by Andrew Pressley},
}

\bib{Boreletal}{book}{
      author={Borel, A.},
      author={Grivel, P.-P.},
      author={Kaup, B.},
      author={Haefliger, A.},
      author={Malgrange, B.},
      author={Ehlers, F.},
       title={Algebraic {$D$}-modules},
      series={Perspectives in Mathematics},
   publisher={Academic Press, Inc., Boston, MA},
        date={1987},
      volume={2},
}

\bib{BMSZ}{misc}{
      author={Barbasch, Dan~M.},
      author={jun Ma, Jia},
      author={Sun, Binyong},
      author={Zhu, Chen-Bo},
       title={Special unipotent representations: orthogonal and symplectic
  groups},
        date={2021},
}

\bib{Lunts}{book}{
      author={Bernstein, Joseph},
      author={Lunts, Valery},
       title={Equivariant sheaves and functors},
      series={Lecture Notes in Mathematics},
   publisher={Springer-Verlag, Berlin},
        date={1994},
      volume={1578},
}

\bib{Carter}{article}{
      author={Carter, R.~W.},
       title={Conjugacy classes in the {W}eyl group},
        date={1972},
     journal={Compositio Math.},
      volume={25},
       pages={1\ndash 59},
}

\bib{cliftonetal}{misc}{
      author={Cunningham, Clifton},
      author={Fiori, Andrew},
      author={Moussaoui, Ahmed},
      author={Mracek, James},
      author={Xu, Bin},
       title={Arthur packets for $p$-adic groups by way of microlocal vanishing
  cycles of perverse sheaves, with examples},
        date={2021},
}

\bib{Christie-Mezo}{incollection}{
      author={Christie, Aaron},
      author={Mezo, Paul},
       title={Twisted endoscopy from a sheaf-theoretic perspective},
        date={2018},
   booktitle={Geometric aspects of the trace formula},
      series={Simons Symp.},
   publisher={Springer, Cham},
       pages={121\ndash 161},
}

\bib{casshah}{article}{
      author={Casselman, William},
      author={Shahidi, Freydoon},
       title={On irreducibility of standard modules for generic
  representations},
        date={1998},
     journal={Ann. Sci. \'{E}cole Norm. Sup. (4)},
      volume={31},
      number={4},
       pages={561\ndash 589},
}

\bib{GM}{book}{
      author={Goresky, Mark},
      author={MacPherson, Robert},
       title={Stratified {M}orse theory},
      series={Ergebnisse der Mathematik und ihrer Grenzgebiete (3) [Results in
  Mathematics and Related Areas (3)]},
   publisher={Springer-Verlag, Berlin},
        date={1988},
      volume={14},
}

\bib{Hotta}{book}{
      author={Hotta, Ryoshi},
      author={Takeuchi, Kiyoshi},
      author={Tanisaki, Toshiyuki},
       title={{$D$}-modules, perverse sheaves, and representation theory},
      series={Progress in Mathematics},
   publisher={Birkh\"{a}user Boston, Inc., Boston, MA},
        date={2008},
      volume={236},
        ISBN={978-0-8176-4363-8},
         url={https://doi.org/10.1007/978-0-8176-4523-6},
        note={Translated from the 1995 Japanese edition by Takeuchi},
      review={\MR{2357361}},
}

\bib{Knapp94}{incollection}{
      author={Knapp, A.~W.},
       title={Local {L}anglands correspondence: the {A}rchimedean case},
        date={1994},
   booktitle={Motives ({S}eattle, {WA}, 1991)},
      series={Proc. Sympos. Pure Math.},
      volume={55},
   publisher={Amer. Math. Soc., Providence, RI},
       pages={393\ndash 410},
}

\bib{Knapp}{book}{
      author={Knapp, Anthony~W.},
       title={Representation theory of semisimple groups},
      series={Princeton Mathematical Series},
   publisher={Princeton University Press, Princeton, NJ},
        date={1986},
      volume={36},
        note={An overview based on examples},
}

\bib{beyond}{book}{
      author={Knapp, Anthony~W.},
       title={Lie groups beyond an introduction},
      series={Progress in Mathematics},
   publisher={Birkh\"{a}user Boston, Inc., Boston, MA},
        date={1996},
      volume={140},
}

\bib{Kostant78}{article}{
      author={Kostant, Bertram},
       title={On {W}hittaker vectors and representation theory},
        date={1978},
     journal={Invent. Math.},
      volume={48},
      number={2},
       pages={101\ndash 184},
}

\bib{Kaletha-Minguez}{misc}{
      author={Kaletha, Tasho},
      author={Minguez, Alberto},
      author={Shin, Sug~Woo},
      author={White, Paul-James},
       title={Endoscopic classification of representations: Inner forms of
  unitary groups},
        date={2014},
}

\bib{KS}{article}{
      author={Kottwitz, Robert~E.},
      author={Shelstad, Diana},
       title={Foundations of twisted endoscopy},
        date={1999},
        ISSN={0303-1179},
     journal={Ast\'erisque},
      number={255},
       pages={vi+190},
      review={\MR{1687096}},
}

\bib{Knapp-Vogan}{book}{
      author={Knapp, Anthony~W.},
      author={Vogan, David~A., Jr.},
       title={Cohomological induction and unitary representations},
      series={Princeton Mathematical Series},
   publisher={Princeton University Press, Princeton, NJ},
        date={1995},
      volume={45},
        ISBN={0-691-03756-6},
         url={https://doi.org/10.1515/9781400883936},
      review={\MR{1330919}},
}

\bib{Langlands}{incollection}{
      author={Langlands, R.~P.},
       title={On the classification of irreducible representations of real
  algebraic groups},
        date={1989},
   booktitle={Representation theory and harmonic analysis on semisimple {L}ie
  groups},
      series={Math. Surveys Monogr.},
      volume={31},
   publisher={Amer. Math. Soc., Providence, RI},
       pages={101\ndash 170},
         url={https://doi.org/10.1090/surv/031/03},
      review={\MR{1011897}},
}

\bib{LV}{article}{
      author={Lusztig, George},
      author={Vogan, David~A., Jr.},
       title={Singularities of closures of {$K$}-orbits on flag manifolds},
        date={1983},
     journal={Invent. Math.},
      volume={71},
      number={2},
       pages={365\ndash 379},
}

\bib{LV2014}{article}{
      author={Lusztig, George},
      author={Vogan, David~A., Jr.},
       title={Quasisplit {H}ecke algebras and symmetric spaces},
        date={2014},
        ISSN={0012-7094},
     journal={Duke Math. J.},
      volume={163},
      number={5},
       pages={983\ndash 1034},
         url={https://doi.org/10.1215/00127094-2644684},
      review={\MR{3189436}},
}

\bib{MR4}{article}{
      author={M\oe~glin, Colette},
      author={Renard, David},
       title={Sur les paquets d'{A}rthur des groupes unitaires et quelques
  cons\'{e}quences pour les groupes classiques},
        date={2019},
        ISSN={0030-8730},
     journal={Pacific J. Math.},
      volume={299},
      number={1},
       pages={53\ndash 88},
  url={https://doi-org.proxy.library.carleton.ca/10.2140/pjm.2019.299.53},
      review={\MR{3947270}},
}

\bib{Mezo}{article}{
      author={Mezo, Paul},
       title={Character identities in the twisted endoscopy of real reductive
  groups},
        date={2013},
        ISSN={0065-9266},
     journal={Mem. Amer. Math. Soc.},
      volume={222},
      number={1042},
       pages={vi+94},
}

\bib{Mezo2}{article}{
      author={Mezo, Paul},
       title={Tempered spectral transfer in the twisted endoscopy of real
  groups},
        date={2016},
     journal={J. Inst. Math. Jussieu},
      volume={15},
      number={3},
       pages={569\ndash 612},
}

\bib{Moeglin-11}{incollection}{
      author={M{\oe}glin, C.},
       title={Multiplicit\'e 1 dans les paquets d'{A}rthur aux places
  {$p$}-adiques},
        date={2011},
   booktitle={On certain {$L$}-functions},
      series={Clay Math. Proc.},
      volume={13},
   publisher={Amer. Math. Soc., Providence, RI},
       pages={333\ndash 374},
      review={\MR{2767522 (2012f:22033)}},
}

\bib{Mok}{article}{
      author={Mok, Chung~Pang},
       title={Endoscopic classification of representations of quasi-split
  unitary groups},
        date={2015},
        ISSN={0065-9266},
     journal={Mem. Amer. Math. Soc.},
      volume={235},
      number={1108},
       pages={vi+248},
         url={http://dx.doi.org/10.1090/memo/1108},
      review={\MR{3338302}},
}

\bib{MR3}{incollection}{
      author={Moeglin, Colette},
      author={Renard, David},
       title={Sur les paquets d'{A}rthur des groupes classiques et unitaires
  non quasi-d\'{e}ploy\'{e}s},
        date={2018},
   booktitle={Relative aspects in representation theory, {L}anglands
  functoriality and automorphic forms},
      series={Lecture Notes in Math.},
      volume={2221},
   publisher={Springer, Cham},
       pages={341\ndash 361},
      review={\MR{3839702}},
}

\bib{MR1}{article}{
      author={Moeglin, Colette},
      author={Renard, David},
       title={Sur les paquets d'{A}rthur des groupes classiques r\'{e}els},
        date={2020},
        ISSN={1435-9855},
     journal={J. Eur. Math. Soc. (JEMS)},
      volume={22},
      number={6},
       pages={1827\ndash 1892},
         url={https://doi-org.proxy.library.carleton.ca/10.4171/jems/957},
      review={\MR{4092900}},
}

\bib{Sha80}{article}{
      author={Shahidi, Freydoon},
       title={Whittaker models for real groups},
        date={1980},
     journal={Duke Math. J.},
      volume={47},
      number={1},
       pages={99\ndash 125},
}

\bib{Sha81}{article}{
      author={Shahidi, Freydoon},
       title={On certain {$L$}-functions},
        date={1981},
     journal={Amer. J. Math.},
      volume={103},
      number={2},
       pages={297\ndash 355},
}

\bib{shelstad}{article}{
      author={Shelstad, D.},
       title={Characters and inner forms of a quasi-split group over {${\bf
  R}$}},
        date={1979},
     journal={Compositio Math.},
      volume={39},
      number={1},
       pages={11\ndash 45},
}

\bib{Shelstad08}{incollection}{
      author={Shelstad, D.},
       title={Tempered endoscopy for real groups. {I}. {G}eometric transfer
  with canonical factors},
        date={2008},
   booktitle={Representation theory of real reductive {L}ie groups},
      series={Contemp. Math.},
      volume={472},
   publisher={Amer. Math. Soc., Providence, RI},
       pages={215\ndash 246},
}

\bib{ShelstadIII}{article}{
      author={Shelstad, D.},
       title={Tempered endoscopy for real groups. {III}. {I}nversion of
  transfer and {$L$}-packet structure},
        date={2008},
     journal={Represent. Theory},
      volume={12},
       pages={369\ndash 402},
}

\bib{Shelstad12}{article}{
      author={Shelstad, D.},
       title={On geometric transfer in real twisted endoscopy},
        date={2012},
     journal={Ann. of Math. (2)},
      volume={176},
      number={3},
       pages={1919\ndash 1985},
}

\bib{shelstad_endoscopy}{incollection}{
      author={Shelstad, Diana},
       title={Orbital integrals, endoscopic groups and
  {$L$}-indistinguishability for real groups},
        date={1983},
   booktitle={Conference on automorphic theory ({D}ijon, 1981)},
      series={Publ. Math. Univ. Paris VII},
      volume={15},
   publisher={Univ. Paris VII},
     address={Paris},
       pages={135\ndash 219},
      review={\MR{MR723184 (85i:22019)}},
}

\bib{springer}{book}{
      author={Springer, T.~A.},
       title={Linear algebraic groups},
     edition={Second},
      series={Progress in Mathematics},
   publisher={Birkh\"{a}user Boston, Inc., Boston, MA},
        date={1998},
      volume={9},
        ISBN={0-8176-4021-5},
         url={https://doi.org/10.1007/978-0-8176-4840-4},
      review={\MR{1642713}},
}

\bib{SmithiesTaylor}{article}{
      author={Smithies, Laura},
      author={Taylor, Joseph~L.},
       title={An analytic {R}iemann-{H}ilbert correspondence for semi-simple
  {L}ie groups},
        date={2000},
     journal={Represent. Theory},
      volume={4},
       pages={398\ndash 445},
}

\bib{Tadic}{incollection}{
      author={Tadi\'{c}, Marko},
       title={{${\rm GL}(n,\mathbb C)\sphat$} and {${\rm GL}(n,\mathbb
  R)\sphat$}},
        date={2009},
   booktitle={Automorphic forms and {$L$}-functions {II}. {L}ocal aspects},
      series={Contemp. Math.},
      volume={489},
   publisher={Amer. Math. Soc., Providence, RI},
       pages={285\ndash 313},
}

\bib{Taibi}{article}{
      author={Ta\"{i}bi, Olivier},
       title={Dimensions of spaces of level one automorphic forms for split
  classical groups using the trace formula},
        date={2017},
     journal={Ann. Sci. \'{E}c. Norm. Sup\'{e}r. (4)},
      volume={50},
      number={2},
       pages={269\ndash 344},
}

\bib{ICIII}{article}{
      author={Vogan, David~A.},
       title={Irreducible characters of semisimple {L}ie groups. {III}. {P}roof
  of {K}azhdan-{L}usztig conjecture in the integral case},
        date={1983},
     journal={Invent. Math.},
      volume={71},
      number={2},
       pages={381\ndash 417},
}

\bib{Vogan78}{article}{
      author={Vogan, David~A., Jr.},
       title={Gelfand-{K}irillov dimension for {H}arish-{C}handra modules},
        date={1978},
     journal={Invent. Math.},
      volume={48},
      number={1},
       pages={75\ndash 98},
}

\bib{greenbook}{book}{
      author={Vogan, David~A., Jr.},
       title={Representations of real reductive {L}ie groups},
      series={Progress in Mathematics},
   publisher={Birkh\"{a}user, Boston, Mass.},
        date={1981},
      volume={15},
}

\bib{ICIV}{article}{
      author={Vogan, David~A., Jr.},
       title={Irreducible characters of semisimple {L}ie groups. {IV}.
  {C}haracter-multiplicity duality},
        date={1982},
        ISSN={0012-7094},
     journal={Duke Math. J.},
      volume={49},
      number={4},
       pages={943\ndash 1073},
         url={http://projecteuclid.org/euclid.dmj/1077315538},
      review={\MR{683010}},
}

\bib{vogan_local_langlands}{incollection}{
      author={Vogan, David~A., Jr.},
       title={The local {L}anglands conjecture},
        date={1993},
   booktitle={Representation theory of groups and algebras},
      series={Contemp. Math.},
      volume={145},
   publisher={Amer. Math. Soc.},
     address={Providence, RI},
       pages={305\ndash 379},
      review={\MR{MR1216197 (94e:22031)}},
}

\end{biblist}
\end{bibdiv}


\end{document}